\documentclass[hyp,reqno]{amsbook}
\usepackage{hyperref}
\usepackage{enumerate,amscd}
\hypersetup{nesting=true,debug=true,naturalnames=true}
\usepackage{graphicx,amssymb,upref}
\usepackage[utf8]{inputenc}
\usepackage{tikz-cd}
\usepackage{mathtools}
\usepackage{subcaption} 
\usepackage{array}

\usepackage{stmaryrd}

\usepackage{slashed}
\usepackage{mxedruli}
\usepackage{mathabx}
\usepackage{musicography}

\newcommand{\R}{\mathbb{R}}
\newcommand{\C}{\mathbb{C}}

\newcommand{\Z}{\mathbb{Z}}

\newcommand{\N}{\mathbb{N}}

\newcommand{\ch}{ \mathrm{ch}}

\newcommand{\rd}{\mathrm{d}}

\newcommand{\Dom}{\mathrm{Dom}\,}

\newcommand{\ind}{\mathrm{i} \mathrm{n} \mathrm{d} \,}

\newcommand{\coker}{\mathrm{c} \mathrm{o} \mathrm{k} \mathrm{e}
\mathrm{r}\,}
\newcommand{\Aut}{\mathrm{Aut}}
\newcommand{\Hom}{\mathrm{Hom}}
\newcommand{\End}{\mathrm{End}}

\renewcommand{\epsilon}{\varepsilon}

\newcommand{\im}{\mathrm{i} \mathrm{m} \,}

\newcommand{\e}{\mathrm{e}}

\DeclareFontFamily{OT1}{pzc}{}
\DeclareFontShape{OT1}{pzc}{m}{it}{<-> s * [1.1] pzcmi7t}{}
\DeclareMathAlphabet{\mathpzc}{OT1}{pzc}{m}{it}

\let\<\langle
\let\>\rangle
\usepackage[all,pdf]{xy}
\UseComputerModernTips

\let\uml\"
\usepackage{xparse}
\NewDocumentCommand{\xleftrightarrows}{ O{}O{} }{%
\mathrel{%
\vcenter{\hbox{%
\begin{tikzpicture}
  \node[minimum width=0.5cm,minimum height=1ex,anchor=south,align=center] (a){\text{\vphantom{hg}#1}\\[0.5ex] \vphantom{hg}#2};
  \draw[->] ([yshift=-0.35ex]a.east) -- ([yshift=-0.35ex]a.west);
  \draw[<-] ([yshift=0.35ex]a.east) -- ([yshift=0.35ex]a.west);
\end{tikzpicture}
}}%
}%
}

\DeclareRobustCommand{\ghani}{%
 \text{\mxedr .g}
}

\newtheorem{theorem}{Theorem}[section]

\newtheorem{proposition}[theorem]{Proposition}
\newtheorem{conjecture}[theorem]{Conjecture}
\newtheorem{corollary}[theorem]{Corollary}
\newtheorem{lemma}[theorem]{Lemma}
\newtheorem{thm*}{Theorem}
\newtheorem{cor*}[thm*]{Corollary}

\theoremstyle{definition}
\newtheorem{example}[theorem]{Example}
\newtheorem{remark}[theorem]{Remark}
\newtheorem{definition}[theorem]{Definition}

\date{\today}
\begin{document}

\author{
  Magnus Goffeng,
  Alexey Kuzmin}
\title{Index theory of hypoelliptic operators on Carnot manifolds}

\address{
Magnus Goffeng,\newline
\indent Centre for Mathematical Sciences\newline 
\indent Lund University\newline 
\indent Box 118, SE-221 00 Lund\newline 
\indent Sweden\newline
\newline
\indent Alexey Kuzmin,\newline
\indent Department of Mathematical Sciences\newline 
\indent Chalmers University of Technology and \newline 
\indent University of Gothenburg\newline 
\indent SE-412 96 Gothenburg\newline 
\indent Sweden\newline}
\email{magnus.goffeng@math.lth.se, vagnard.k@gmail.com}

\maketitle

\begin{abstract}
We study the index theory of hypoelliptic operators on Carnot manifolds -- manifolds whose Lie algebra of vector fields is equipped with a filtration induced from sub-bundles of the tangent bundle. A Carnot pseudodifferential operator, elliptic in the calculus of Melin and van Erp-Yuncken, is hypoelliptic and Fredholm. Under some geometric conditions, we compute its Fredholm index by means of operator $K$-theory. These results extend the work of Baum-van Erp for co-oriented contact manifolds to a methodology for solving this index problem geometrically on Carnot manifolds.

Under the assumption that the Carnot manifold is regular, i.e. has isomorphic osculating Lie algebras in all fibres, and its generic coadjoint orbits are flat, the methodology derived from Baum-van Erp's work is developed in full detail. In this case, we develop $K$-theoretical dualities computing the Fredholm index by means of geometric $K$-homology a la Baum-Douglas. The duality involves a Hilbert space bundle of flat orbit representations. Explicit solutions to the index problem for Toeplitz operators and operators of the form ``$\Delta_H+\gamma T$" are computed in geometric $K$-homology, extending results of Boutet de Monvel and Baum-van Erp, respectively, from co-oriented contact manifolds to regular polycontact manifolds. 

The existence and the precise form of the geometric duality constructed for the Carnot calculus relies on the representation theory in the flat coadjoint orbits of the osculating Lie groupoid. We address the technical issue of constructing a Hilbert space bundle of representations associated to the flat coadjoint orbits via Kirillov's orbit method. The construction intertwines the index theory of Carnot operators to characteristic classes constructed from the Carnot structure further clarifying the two opposite spin$^c$-structures appearing in Baum-van Erp's solution to the index problem on contact manifolds. 
\end{abstract}
    
\tableofcontents

\part{Introduction}

This monograph is a study in the index theory of hypoelliptic operators on Carnot manifolds. We consider operators elliptic in the Carnot calculus of Melin \cite{melinoldpreprint} and van Erp-Yuncken \cite{vanErp_Yuncken}. Examples of such operators include generalized Toeplitz operators \cite{bdmguille,engliszhanghigher,melroseeptein}, certain Hörmander sum of squares \cite{androerp,horhypo,horsumsquares} and in a graded sense also Bernstein-Gelfand-Gelfand complexes \cite{bggoriginal,morecap,Dave_Haller1}. Except in the case of a trivial Carnot structure, these operators are not elliptic in the classical sense and will in general not belong to the classical pseudodifferential calculus of Kohn-Nirenberg.  We solve the index problem geometrically (in the sense of Baum-Douglas \cite{Baum_Douglas,BDbor}) for Heisenberg elliptic operators assuming that the Carnot manifold is regular (the isomorphism class of the osculating Lie algebra is constant), and its generic coadjoint orbits are flat (the Kirillov form is non-degenerate modulo the center in some orbit). The precise condition needed we call $\pmb{F}$-regularity, see Definition \ref{osculatingdef2} below. The solution to the index problem is formulated in terms of $K$-homology and relies on a technical construction in representation theory, occupying a third of this work, of a bundle of those representations corresponding to the flat coadjoint orbits under Kirillov's orbit method. 

Our main results extend work of Baum-van Erp \cite{baumvanerp} from contact manifolds to Carnot manifolds as above, and following \cite[Theorem in Introduction]{baumvanerp}, we can summarize our results in the following theorem. Consider the index problem of describing the $K$-homology class $[D]\in K_*(X)$ of a Heisenberg elliptic operator $D:C^\infty(X;E)\to C^\infty(X;F)$ on a compact Carnot manifold $X$ satisfying the assumptions above.

\begin{thm*}
\label{firstdedcsofsol}
The index problem for an $H$-elliptic $[D]\in K_*(X)$ is solved by
$$[D]=(p_\Gamma)_*(\llbracket\sigma_H(D)\rrbracket\cap [\Gamma_X])\in K_*(X).$$
\end{thm*}

Here $p_\Gamma:\Gamma_X\to X$ denotes the locally trivial bundle of flat coadjoint orbits in each fibre of the osculating Lie groupoid and $\llbracket\sigma_H(D)\rrbracket$ is a $K$-theory element on $\Gamma_X$ constructed from the Heisenberg symbol. The class $[\Gamma_X]\in K_*(\Gamma_X)\equiv K^*(C_0(\Gamma_X))$ is the fundamental class for the spin$^c$-structure induced from a natural embedding of $\Gamma_X$ into the vector bundle dual to the osculating Lie algebroid of $X$. A bit more precisely, $\llbracket\sigma_H(D)\rrbracket$ can be represented by an elliptic complex on $\Gamma_X$ and $\llbracket\sigma_H(D)\rrbracket\cap [\Gamma_X]$ is the associated compactly supported $K$-cycle. A consequence of Theorem \ref{firstdedcsofsol} is the following computation of the index.

\begin{thm*} 
\label{indexformfirstdedcsofsol}
The index of an $H$-elliptic $D$ is given by
$$\ind(D)=\int_{\Gamma_X} \ch\left(\ind[\sigma_H(D)\otimes \mathcal{H}]\right)\wedge \e^{c_1(\mathfrak{M}(\mathcal{H}))}\wedge \mathrm{Td}(\Gamma_X),$$
\end{thm*}

Here $\ind[\sigma_H(D)\otimes \mathcal{H}]$ is an index class on $\Gamma_X$ constructed from a certain deformation of the Heisenberg symbol of $D$, invertible outside $\Gamma_X$, acting on a bundle $\mathcal{H}\to \Gamma_X$ of Hilbert spaces whose fibre in a point corresponds to the point under the Kirillov orbit method and finally $\mathfrak{M}(\mathcal{H})\to \Gamma_X$ is a line bundle corresponding to metaplectic corrections in the bundle $\mathcal{H}$ of flat orbit representations. Neither the class $\llbracket\sigma_H(D)\rrbracket\in K^0(\Gamma_X)$ of Theorem \ref{firstdedcsofsol} nor the class $[\sigma_H(D)\otimes \mathcal{H}]\in K^0(\Gamma_X)$ of Theorem \ref{indexformfirstdedcsofsol} are uniquely determined; the two classes however have uniquely determined images under the duality map $K^*(\Gamma_X)\to K_*(X)$, $\xi\mapsto (p_\Gamma)_*(\xi\cap [\Gamma_X])$.

We return to Theorem \ref{firstdedcsofsol} and \ref{indexformfirstdedcsofsol} in Section \ref{subsec:mainresults} below to give further details, more precise forward references to the body of this text and examples showing how they are used in computations.

In case that $X$ has the trivial Carnot structure, so $\Gamma_X=T^*X$ with its canonical spin$^c$-structure, Theorem \ref{indexformfirstdedcsofsol} is the statement of Atiyah-Singer's index theorem \cite{atiyahIII} and Theorem \ref{firstdedcsofsol} recovers Baum-Douglas' solution to the index problem for elliptic operators \cite[Part 5]{Baum_Douglas}. If $X$ has the Carnot structure induced from a co-oriented contact structure, so $\Gamma_X=(X\dot{\cup}(-X))\times \R_{>0}$ as spin$^c$-manifolds, Theorem \ref{firstdedcsofsol} coincides with Baum-van Erp's solution to the index problem for Heisenberg elliptic operators \cite{baumvanerp}. Also in \cite{baumvanerp}, a non-unique lifting class $\llbracket\sigma_H(D)\rrbracket$ is utilized.

\section{Background and context}

The study of hypoelliptic operators has a long history, dating back to works of Kolmogorov \cite{kolmogorovfirst}, Malgrange \cite{malgrange57}, Hörmander \cite{horhypo}, Rothschild-Stein \cite{rothschildstein} and Helffer-Nourrigat \cite{helfnougat} to name but a few. A notable result is Hörmander's sum of squares theorem \cite{horsumsquares} stating that a ``sum of squares'' $\sum_j X_j^2$ is hypoelliptic if the collection of vector fields $(X_j)_j$ generates the Lie algebra of all vector fields. The ideas of pseudodifferential calculus was first introduced into the hypoelliptic realm by Folland-Stein \cite{follandstein} who used the idea of freezing coefficients in the $\bar{\partial}_b$-operator in a boundary point of a strictly pseudoconvex domain. The idea of Folland-Stein was in short to model the tangent space of said boundary point with a Heisenberg group encoding the local geometry and using a symbol calculus in which the symbols are homogeneous convolution operators on a graded Lie group. This technique was later improved in \cite{rothschildstein}. Operators locally modelled on homogeneous convolution operators on graded Lie groups have since been an intense object of study, see for instance, \cite{christgelleretal,cum89,fischruzh,taylorncom}. The special case of contact manifolds, locally modelled on the Heisenberg group, is particularly well studied \cite{bealsgreiner,dynin75,dynin76,melroseeptein,pongemonograph,taylorncom}.\\

The underlying geometric objects we consider in this monograph are Carnot manifolds: manifolds $X$ equipped with a filtration 
 of the tangent bundle $TX$ by smooth subbundles
\begin{equation}
\label{alsdaodjbaofjbewgowub}
0 = T^{0}X \subset T^{-1}X \subset T^{-2}X \ldots \subset T^{-r+1}X\subset T^{-r}X = TX, 
\end{equation}
such that all inclusions are strict and for any vector fields $Y \in C^\infty(X,T^iX)$, $Y' \in C^\infty(X,T^jX)$ we have that $[Y, Y'] \in C^\infty(X,T^{i+j}X)$. The number $r$ is called the depth of $X$. In this work the reader can find more details on and examples of Carnot manifolds in Part \ref{sec:carnotmfds}. The term Carnot manifold is in accordance with  \cite{choiponge,Dave_Haller1,Dave_Haller2,mohsen2}, while in  \cite{Morimoto,sadeghhigson,vanErp_Yunckentangent,vanErp_Yuncken} it is called a filtered manifold and in \cite{melinoldpreprint}, a filtration as in \eqref{alsdaodjbaofjbewgowub} is called a Lie filtration. The Lie bracket on the tangent fields induces the structure of a graded nilpotent Lie algebra on each fibre of the graded tangent bundle 
$$\mathfrak{t}_HX:=\oplus_j T^{-j}X/T^{-j+1}X.$$
Equipped with the zero anchor mapping, $\mathfrak{t}_HX$ is a Lie algebroid on $X$ that integrates to a Lie groupoid $T_HX\to X$. One calls $\mathfrak{t}_HX$ the osculating Lie algebroid and $T_HX$ the osculating Lie groupoid of $X$ as the fibre algebraically encodes the filtration. As a fibre bundle, $T_HX\to X$ can be identified with the tangent bundle $TX\to X$ upon choosing a splitting $\mathfrak{t}_HX\cong TX$. A Carnot manifold is \emph{regular} if there is a graded nilpotent Lie algebra $\mathfrak{g}$ such that $\mathfrak{t}_HX_x\cong \mathfrak{g}$ as graded Lie algebras in all points $x\in X$.

The motivating idea of the Carnot calculus is that the algebra of differential operators on $X$ is naturally filtered by the order of operators, but in certain geometric situations it is more natural to filter by the Carnot order obtained from declaring the order of differentiation along a vector field in $C^\infty(X;T^{-j}X)$ to be $\geq j$. See for instance \cite{Dave_Haller1,vanerpsimplest} for motivating examples. We write $\mathcal{DO}^m_H(X)$ for the space of differential operators of Carnot order $\leq m$. The associated symbol calculus takes place in the space $\mathcal{DO}^m_H(X)/\mathcal{DO}^{m-1}_H(X)$ which can naturally be identified with the space $C^\infty(X;\mathcal{U}_m(\mathfrak{t}_HX))$ of sections of the bundle $\mathcal{U}_m(\mathfrak{t}_HX)\to X$ of degree $m$ homogeneous elements of the universal enveloping Lie algebra $\mathcal{U}(\mathfrak{t}_HX)$.

A pseudodifferential calculus associated with the symbol calculus of Carnot operators was first studied by Melin \cite{melinoldpreprint}, based on work of Hörmander-Melin \cite{hormelin}. The constructions of Melin were further developed by van Erp-Yuncken \cite{vanErp_Yuncken} building on their tangent groupoid construction \cite{vanErp_Yunckentangent}. The work of Melin \cite{melinoldpreprint} and van Erp-Yuncken \cite{vanErp_Yuncken} extends  \cite{bealsgreiner,dynin75,dynin76,melroseeptein,pongemonograph,taylorncom} from contact manifolds and more generally Heisenberg manifolds, and \cite{christgelleretal,cum89,fischruzh} from the case of Carnot manifolds locally isomorphic to a nilpotent Lie group, to general Carnot manifolds. Further refinements of van Erp-Yuncken's work include those by Dave-Haller \cite{Dave_Haller1,Dave_Haller2}, Ewert \cite{eskeewertthesis,eskeewertpaper} and Mohsen \cite{mohsen1}. The reader can find a review of this body of work, and further examples, in Part \ref{part:pseudod} of this monograph. We write $\Psi^m_H(X;E,F)$ for the space of order $m$ Carnot pseudodifferential operators from the vector bundle $E$ to the vector bundle $F$, or in short an $m$:th order Carnot operator. A Carnot operator which is invertible modulo lower order terms is said to be $H$-elliptic.\\

\emph{The problem we study in this work is the index problem for $H$-elliptic operators.} By results from \cite{Dave_Haller1}, that builds on \cite{christgelleretal}, $H$-ellipticity is characterized by a Rockland type condition and relates to maximal hypoellipticity by recent work \cite{androerp}. A Carnot operator $D$ has an associated operator valued symbol parametrized by all representations $\widehat{T_HX}\setminus X$ (where $X$ is identified with the trivial representations of each fibre) and $H$-ellipticity is equivalent to the principal symbol in the relevant calculus taking invertible values. As such, the pseudodifferential calculus and more or less standard constructions show that $H$-ellipticity implies hypoellipticity and Fredholm properties. The reader can find more details in Part \ref{part:pseudod} and references listed there. \\

The index problem will be solved by means of $K$-theoretical methods under an additional regularity assumptions on the Carnot manifold. By the index problem, we are asking for a feasible description of the $K$-homology class $[D]\in K_*(X)$ of an $H$-elliptic pseudodifferential operator $D:C^\infty(X;E)\to C^\infty(X;F)$ on a closed Carnot manifold. We shall take feasible to mean a solution in terms of Baum-Douglas' geometric $K$-homology \cite{Baum_Douglas,BDbor}. As such we solve the ``Baum-Douglas index problem'' for the $K$-homology class associated with an $H$-elliptic operator. For further discussion on the ``Baum-Douglas index problem'', see its first appearance \cite[Part 5]{Baum_Douglas} or \cite{baumvanerp,DGIII} for further applications thereof. Indeed, our solution in spirit follows the corresponding problem on co-oriented contact manifolds by Baum-van Erp \cite{baumvanerp}. 

Let us discuss previous work in this direction. In the case of contact manifolds, or more generally Heisenberg manifolds, there is an extensive body of work due to van Erp \cite{vanerpsimplest,vanerpfirst,vanerpannI,vanerpannII}, building in parts on work of Melrose-Epstein \cite{melroseeptein}. Let us bluntly summarize the works \cite{vanerpannI,vanerpannII} in an index formula. If $D:C^\infty(X;E)\to C^\infty(X;F)$ is an $H$-elliptic operator, there is an associated $K$-theory class $[\sigma_H(D)]\in K_0(C^*(T_HX))$ defined from the Carnot symbol and a difference class construction. By results of Nistor \cite{nistorsolvabel}, there is a Connes-Thom isomorphism $\psi:K^*(T^*X)=K_*(C^*(TX))\to K_*(C^*(T_HX))$. The index formula of van Erp \cite{vanerpannI,vanerpannII} states that for an $H$-elliptic $D$ on a co-oriented contact manifold, 
\begin{equation}
\label{vanerpfirstformalaal}
\ind(D)=\int_{T^*X}\ch[\psi^{-1}[\sigma_H(D)]]\wedge \mathrm{Td}(T^*X).
\end{equation}
We note that the class $\psi^{-1}[\sigma_H(D)]$ was described in more detail in \cite{vanerpannII} for co-oriented contact manifolds. As discussed in \cite[Example 6.5.3]{baumvanerp}, the index formula \eqref{vanerpfirstformalaal} comes with the drawback of not behaving well under cap products with $K$-theory when describing $\psi^{-1}[\sigma_H(D)]$ as in \cite{vanerpannII}, an issue solved in \cite{baumvanerp}.  

The results of van Erp was later further refined by Baum-van Erp \cite{baumvanerp} who gave a geometric solution to the index problem of $H$-elliptic operators on co-oriented contact manifolds. The index problem for $H$-elliptic operators on co-oriented contact manifolds has also been solved in \cite{gorovanerp} by means of cyclic cohomology. We describe the solution from \cite{baumvanerp} in quite some detail as it gives more context to the main results listed in the next subsection. A contact manifold is a Carnot manifold with depth $r=2$ such that $H:=T^{-1}X\subseteq TX$ has codimension $1$ and the Lie bracket induces a non-degenerate two form $\mathcal{L}:H\wedge H\to TX/H$. If $TX/H$ is trivializable, the contact structure is said to be co-oriented. In this case, there is a complex structure on $H$ adapted to $\mathcal{L}$ and we can consider the (symmetric) Fock bundle $\mathpzc{F}_0:=\bigoplus_{k=0}^\infty H^{\otimes_\C^{\rm sym} k}$ over $X$. On the other hand, the osculating Lie groupoid $T_HX\to X$ is a bundle of Heisenberg groups and in each fiber the space of representations take the form $\widehat{T_HX_x}=H_x^*\coprod \R^\times \theta_x$ where $\theta$ is a one form with $H=\ker\theta$ and $\mathcal{L}$ can be identified with $\rd \theta$. We note that $\widehat{T_HX}$ is a bundle over $X$ with non-Hausdorff fibers (any open set meeting $H_x^*$ contains a punctured neighborhood of $0$ in the nonabelian representations $\R^\times \theta_x$). Baum-van Erp's solution to the index problem for an $H$-elliptic operator $D$ on a co-oriented contact manifolds can be collated in the following algorithm:
\begin{enumerate}
\label{bveaogo}
\item[Step 1:]  {\bf Identifying a Zariski open Hausdorff subset $\Gamma_X\subseteq \widehat{T_HX}$}\\
The fibre bundle $\Gamma_X:=\R^\times \theta\to X$ is a Zariski open Hausdorff subset of $\widehat{T_HX}$. There is a Hilbert space bundle $\mathpzc{F}\to \Gamma_X$ defined by $\mathpzc{F}|_{X\times \R_{>0}\theta}=\mathpzc{F}_0$ and $\mathpzc{F}|_{X\times \R_{<0}\theta}=\bar{\mathpzc{F}}_0$. For $\xi\in \Gamma_X$, the corresponding representation is defined by the Bargmann representation of $T_HX_x$ on $\mathpzc{F}_\xi$. Moreover, the map $K_0(C_0(\Gamma_X;\mathbb{K}(\mathpzc{F})))\to K_0(C^*(T_HX))$, induced from the inclusion $C_0(\Gamma_X;\mathbb{K}(\mathpzc{F})))\hookrightarrow C^*(T_HX)$, is surjective. 
\item[Step 2:] {\bf Localizing the index class to $\Gamma_X$}\\
There is a distributional multiplier $\tilde{\sigma}$ -- a compactly supported almost homogeneous fiberwise distribution -- lifting the Carnot symbol $\sigma_H(D)$ of $D$ in a suitable sense. The associated represented family $(\pi_\xi(\tilde{\sigma}))_{\xi\in \Gamma_X}$ acts as densely defined endomorphisms of $\mathpzc{F}$. As Baum-van Erp argues in \cite{baumvanerp}, the family $(\pi_\xi(\tilde{\sigma}))_{\xi\in \Gamma_X}$ can be continuously deformed to be invertible outside a compact of $\Gamma_X$, and represented by a $K$-theory element in $K_0(C_0(\Gamma_X;\mathbb{K}(\mathpzc{F})))$; this is a pre-image of $[\sigma_H(D)]$ under the surjection $K_0(C_0(\Gamma_X;\mathbb{K}(\mathpzc{F})))\to K_0(C^*(T_HX))$. 
\item[Step 3:] {\bf Finite rank approximation of the localized index class on $\Gamma_X$}\\
Under Morita invariance and Bott periodicity, $K_0(C_0(\Gamma_X;\mathbb{K}(\mathpzc{F})))\cong K^1(X)\oplus K^1(X)$, so we  can represent the class $[\sigma_H(D)]\in K_0(C^*(T_HX))$ by a pre-image $[u]=[u_+]\oplus [u_-]\in K^1(X)\oplus K^1(X)$ under 
$$K^1(X)\oplus K^1(X)\cong K_0(C_0(\Gamma_X;\mathbb{K}(\mathpzc{F})))\to K_0(C^*(T_HX)).$$ 
\end{enumerate}
Baum-van Erp proves that 
\begin{equation}
\label{solutiontoindex}
[D]=[X]\cap [u_+]-[X]\cap [u_-],
\end{equation}
solves the index problem for $D$. In terms of geometric Baum-Douglas cycles, the class $[X]\cap [u_+]-[X]\cap [u_-]$ is computable from clutching constructions. The reader should note that $\Gamma_X\subseteq T^*X$ carries a canonical spin$^c$-structure and $\Gamma_X\cong (X\dot{\cup}(-X))\times \R_{>0}$ as spin$^c$-manifolds which relates the sign appearing in \eqref{solutiontoindex} to Theorem \ref{indexformfirstdedcsofsol}.

The current state of the art in the index problem for $H$-elliptic operators on a general Carnot manifold is that the index formula \eqref{vanerpfirstformalaal} is valid. The formula was seen conjectured in \cite{Dave_Haller1} and was proven by Mohsen \cite{mohsen1} using adiabatic deformations. It was also proven by Ewert \cite{eskeewertthesis,eskeewertpaper} using different methods. The problems discussed in \cite[Example 6.5.3]{baumvanerp} persists -- this index formula does not behave well under cap products. We do mention that Mohsen \cite{mohsen2} has described the class $\psi^{-1}[\sigma_H(D)]$ in further details by means of an operator valued symbol constructed at a symbolic level. Further progress was made recently \cite{androerp,mohsen3}. The general direction of recent work in the index theory of $H$-elliptic operators has been to push the limits of how little is needed from a Carnot structure in order to develop general abstract index results in $K$-theory. Despite the beauty of these endeavours, we take the opposite direction in this work by imposing the assumption that the Carnot structure is regular and obtain more explicit solutions.

\section{Main results}
\label{subsec:mainresults}

The main results in this work can be summarized as a careful analysis of the algorithm on page \pageref{bveaogo}, extracted from the work of Baum-van Erp \cite{baumvanerp}, in the case of more general Carnot manifolds. We often restrict our attention to regular Carnot manifolds, that by results of Morimoto \cite{Morimoto} is equivalent to $\mathfrak{t}_HX\to X$ being a locally trivial bundle of Lie algebras, we write $\mathfrak{g}$ for the fibre and $\mathsf{G}$ for the integrating simply connected nilpotent Lie group. For the purposes of this subsection, we only consider regular Carnot manifolds. The reader can find more context for the regularity assumption in the examples of Section \ref{sec:examcarn} below.

As can be seen from the algorithm on page \pageref{bveaogo}, a geometric solution can be obtained by suitably localizing the index class $[\sigma_H(D)]\in K_*(C^*(T_HX))$ of an $H$-elliptic operator to a suitable open dense Hausdorff subset of $\widehat{T_HX}$. We shall impose a condition guaranteeing the existence of a suitable $\Gamma_X\subseteq \widehat{T_HX}$. The representations of each fibre of $T_HX$ can be described by the Kirillov orbit method, and $\widehat{T_HX}$ is a quotient of $(\mathfrak{t}_HX)^*$ by the fibrewise coadjoint action of $T_HX$. Write $\mathfrak{g}$ for the fibre of $\mathfrak{t}_HX$ and let $\mathfrak{z}$ denote its center. For notational simplicity, we say that a coadjoint orbit $\mathcal{O}\subseteq \mathfrak{g}^*$ is flat if there is a $\xi\in \mathcal{O}$ such that the Kirillov form 
$$\omega_\xi(Y,Y'):=\xi[Y,Y'], \quad Y, Y'\in \mathfrak{g},$$
descends to a non-degenerate form on $\mathfrak{g}/\mathfrak{z}$. \emph{We here use an abbreviated terminology: in the literature a flat orbit refers to an orbit $\mathcal{O}$ which is an affine subspace of $\mathfrak{g}^*$ while we only refer to a flat orbit $\mathcal{O}$ when it is an affine subspace modelled on $\mathfrak{z}^\perp$.} We let $\Gamma$ denote the set of flat orbits in the Lie algebra $\mathfrak{g}$ and $\Gamma_X\subseteq \widehat{T_HX}$ the set of flat orbits. We can identify $\Gamma$ with a Zariski open subset of $\mathfrak{z}^*$. If $\Gamma\neq \emptyset$, then $\Gamma\subseteq \hat{\mathsf{G}}$ is an open, dense, Hausdorff subset. Examples of when $\Gamma\neq \emptyset$ can be found in Section \ref{subsectionexamples} below. Regularity of the Carnot structure ensures that $p_\Gamma:\Gamma_X\to X$ forms a locally trivial bundle with fibre $\Gamma$. \emph{We say that a Carnot manifold $X$ is $\pmb{F}$-regular if $X$ is regular and $\Gamma_X\neq \emptyset$ with the push-forward map $K^*(\Gamma_X)\to K^*(X)$ being surjective}, see Definition \ref{osculatingdef2} below. For examples of $\pmb{F}$-regular Carnot manifolds, see Section \ref{sec:examcarn}.

While the moral of our geometric solution is to localize the index class $[\sigma_H(D)]\in K_*(C^*(T_HX))$ to $\Gamma_X$, we need to understand the representations belonging to $\Gamma_X\subseteq \widehat{T_HX}$ in order for further computations in $K$-theory. We consider the ideal of flat orbit representations $I_X\subseteq C^*(T_HX)$ defined from 
$$I_X:=\left\{a\in C^*(T_HX): \; \pi(a)=0\; \forall \pi\in \widehat{T_HX}\setminus \Gamma_X\right\}.$$
The ideal $I_X$ is a continuous trace algebra with spectrum $\Gamma_X$. To apply the ideal $I_X$ we need to construct a bundle of Hilbert spaces $\mathcal{H}\to \Gamma_X$ with $I_X\cong C_0(\Gamma_X;\mathbb{K}(\mathcal{H}))$. The mere existence of $\mathcal{H}\to \Gamma_X$ reduces to the statement $\delta_{\rm DD}(I_X)=0\in \check{H}^3(\Gamma_X,\Z)$ on Dixmier-Duoady invariants, but since the Hilbert space bundle makes its explicit appearance in the geometric solution to the index problem we wish for an explicit construction of $\mathcal{H}\to \Gamma_X$.

Based on works of Pedersen \cite{pedersen84,pedersen88,pedersen89,pedersen94}, Lion \cite{lion76} and Vergne \cite{vergneCR70,vergneBullMathFra72}, we develop the necessary structural representation theory results to study the ideal of flat orbit representations in the group $C^*$-algebra of a nilpotent Lie group. These results extend also to a construction of $\mathcal{H}\to \Gamma_X$ on Carnot manifolds and to a range of associated deformation groupoids associated with Nistor's Connes-Thom isomorphism \cite{nistorsolvabel}, see also Section \ref{subsec:nistorctsubsn}. The following theorem summarizes the salient points in our representation theoretical constructions. The proof of the following theorem covers most of the Parts \ref{sec:reptheory} and \ref{gropodoprar}, and the call for explicit constructions in its proof is what makes this monograph so long.

\begin{thm*}
\label{maintechforrep}
Let $X$ be an $\pmb{F}$-regular Carnot manifold. Then there exists a Hilbert space bundle $\mathcal{H}\to \Gamma_X$ and a $C_0(\Gamma_X)$-linear $*$-isomorphism
$$\pi_{\musFlat}:I_X\to C_0(\Gamma_X,\mathbb{K}(\mathcal{H}))).$$
It holds that:
\begin{itemize}
\item[i)] The unitary equivalence class of the Hilbert space bundle $\mathcal{H}\to \Gamma_X$ and the $*$-isomorphism $\pi_{\musFlat}$ is unique up to a line bundle on $\Gamma_X$. 
\item[ii)] The inclusion $I_X\hookrightarrow C^*(T_HX)$ induces a surjection in $K$-theory
$$K_*(I_X)\to K_*(C^*(T_HX)).$$
\item[iii)] There is a line bundle $\mathfrak{M}(\mathcal{H})\to \Gamma_X$, uniquely determined up to isomorphism, such that for any other such $\mathcal{H}'$ and $\pi_{\musFlat}'$, there is an isomorphism
$$\mathcal{H}\otimes \mathfrak{M}(\mathcal{H})\cong \mathcal{H}'\otimes \mathfrak{M}(\mathcal{H}')$$
compatible with the $(I_X,C_0(\Gamma_X))$-bimodule structure and making the following diagram commutative
\[
\begin{tikzcd}
K_*(I_X)\arrow[r,"\otimes \mathcal{H}"]\arrow[dd]& K^*(\Gamma_X)\arrow[r, "\otimes \mathfrak{M}(\mathcal{H})"]&K^*(\Gamma_X)\arrow[dd, "j_!"] \\
&&\\
K_*(C^*(T_HX)) & & K^*(T^*X)\arrow[ll, "\psi"],
\end{tikzcd}
\]
where $\psi$ denotes Nistor's Connes-Thom isomorphism and $j:\Gamma_X\hookrightarrow T^*X$ denotes the inclusion.
\end{itemize}
\end{thm*}

The proof of Theorem \ref{maintechforrep} is quite delicate and covers a large portion of the body of the text, mainly in Parts \ref{sec:reptheory} and \ref{gropodoprar}. We present two constructions of the Hilbert space bundle $\mathcal{H}\to \Gamma_X$ of flat orbit representations: in Proposition \ref{groupoidftincentrals1} a direct construction gluing together locally defined bundles of representations is given and in Theorem \ref{trivialdldaaddo} we give an indirect construction using deformations. Item i) of Theorem \ref{maintechforrep} is a standard result for continuous trace algebras, we recall it below as Lemma \ref{lem:desciriddtriviala}. Item ii) of Theorem \ref{maintechforrep} follows from Corollary \ref{onsurjinteroffm} below. The proof of item iii) of Theorem \ref{maintechforrep} follows from Theorem \ref{maincomputationforisg}, and the uniqueness statement follows in conjunction with item i) of Theorem \ref{maintechforrep}. The reader can consult Proposition \ref{metaplecticlinecorrection} and Lemma \ref{superimportanttechnicallemma} for further descriptions of the metaplectic line bundle $\mathfrak{M}(\mathcal{H})$.\\

The main player in our index result is dualities in Kasparov's $KK$-theory, and especially the interplay between dualities defined from $K$-orientations and those defined from Carnot structures. There are well known duality maps 
$$\mathsf{PD}^{\rm an}: K^*(T^*X)\to K_*(X),\quad\mbox{and}\quad \mathsf{PD}^{\rm geo}:K^*(T^*X)\to K_*^{\rm geo}(X),$$
that by \cite{baumvanerpII} are compatible via the analytic assembly of geometric cycles $\gamma:K_*^{\rm geo}(X)\to K_*(X)$. There is also a duality mapping 
$$\mathsf{PD}^{\rm an}_H: K_*(C^*(T_HX))\to K_*(X),$$
defined from adiabatic deformations. The duality map $\mathsf{PD}^{\rm an}_H$ relates the index class of the Carnot symbol to the $K$-homology class of an $H$-elliptic operator via the identity: 
$$\mathsf{PD}^{\rm an}_H[\sigma_H(D)]=[D],$$
see Proposition \ref{chooseanopheisenberg}. Using Theorem \ref{maintechforrep}, we also introduce a geometric duality defined from the Carnot structure:
$$\mathsf{PD}^{\rm geo}_H: K_*(C^*(T_HX))\to \tilde{K}_*^{\rm geo}(X).$$
The codomain of the duality mapping $\mathsf{PD}^{\rm geo}_H$ is a variation of  Baum-Douglas' geometric $K$-homology. This variation, that we denote by $\tilde{K}_*^{\rm geo}(X)$, is constructed from geometric cycles in which the manifold is not necessarily compact, but rather the compact support is ensured from coefficients in elliptic complexes similarly to the constructions in \cite{emermeyer}. There is an inclusion of the set of isomorphism classes of ordinary Baum-Douglas cycles into the set of isomorphism classes of Baum-Douglas cycles with coefficients in the elliptic complexes that induce a natural isomorphism $K_*^{\rm geo}(X)\cong \tilde{K}_*^{\rm geo}(X)$; the inverse is obtained from a clutching construction. The duality map $\mathsf{PD}^{\rm geo}_H$ is defined from 
$$\mathsf{PD}^{\rm geo}_H(x):=[(\Gamma_X,\xi\otimes \mathfrak{M}(\mathcal{H}),p_\Gamma)],$$
where $p_\Gamma:\Gamma_X\to X$ is the projection, and $\xi$ is an elliptic complex on $\Gamma_X$ that maps to $x$ under the composition 
\begin{equation}
\label{comoampadadm}
K^0(\Gamma_X)\xrightarrow{\otimes \mathcal{H}} K_0(I_X)\to K_0(C^*(T_HX)),
\end{equation}
for some bundle $\mathcal{H}$ as in Theorem \ref{maintechforrep}. Existence of $\xi$ and $\mathcal{H}$ is ensured by Theorem \ref{maintechforrep}.

\begin{remark}
\label{commenonmodedl}
To justify this change of model for $K$-homology, the reader can consider Baum-Douglas' solution \cite[Part 5]{Baum_Douglas} to their index problem for an elliptic operator $D:C^\infty(X;E)\to C^\infty(X;F)$ on a closed manifold. While Baum-Douglas solves the index problem by a clutching construction producing a cycle on $S(T^*X\oplus 1_\R)$, it could equally well be solved without involving clutching by a cycle of the form $(T^*X, (p^*E,p^*F,\sigma(D)),p)$ viewed as a geometric cycle with coefficients in the elliptic complexes, for $p:T^*X\to X$ denoting the projection. In general $\Gamma_X$ is non-compact, so the index problem for $H$-elliptic operators is in general solved by a non-compact cycle without an obvious place to clutch.
\end{remark}

The duality mappings discussed above are related in Theorem \ref{commutingtetraederthm} of the text producing a commuting tetraeder. We recall its statement here.

\begin{thm*}
\label{commtererar}
Let $X$ be a compact $\pmb{F}$-regular Carnot manifold. Then all morphism in the following diagram are isomorphisms and the diagram commutes
\[
\begin{tikzcd}
& K_*(C^*(T_H X)) \arrow{ddddl}[swap]{\mathsf{PD}_H^{\rm geo}}\arrow[ddddr, "\mathsf{PD}_H^{\rm an}"] & \\
&&\\
& K^*(T^* X) \arrow[ddl, "\mathsf{PD}^{\rm geo}"] \arrow{ddr}[swap]{\mathsf{PD}^{\rm an}} \arrow[uu, "\psi"] & \\
&&\\
\tilde{K}_*^{\rm geo}(X) \arrow[rr, "\gamma"] & & K_*(X)
\end{tikzcd}
\]
If $X$ is a Carnot manifold which is not necessarily $\pmb{F}$-regular, the diagram above with the upper left arrow removed only consists of isomorphisms and commutes.
\end{thm*}

The commuting tetraeder of Theorem \ref{commtererar} produces a solution in the abstract to the index problem for $H$-elliptic operators. For a contact manifold, Theorem \ref{commtererar} can be found on \cite[page 36]{baumvanerp}.

\begin{cor*}
\label{indexofmrulaknon}
Let $D:C^\infty(X;E)\to C^\infty(X;F)$ be an $H$-elliptic operator on a compact $\pmb{F}$-regular Carnot manifold. Then the Baum-Douglas index problem for the class 
$$[D]\in K_0(X),$$
is solved by the geometric cycle $\mathsf{PD}_H^{\rm geo}[\sigma_H(D)]$. More precisely, if we have an elliptic complex $\xi$ on $\Gamma_X$ such that $[\xi]\mapsto [\sigma_H(D)]$ under the mapping \eqref{comoampadadm} for some bundle $\mathcal{H}$ as in Theorem \ref{maintechforrep}, then 
$$[D]=\gamma\left([(\Gamma_X, \xi\otimes \mathfrak{M}(\mathcal{H}),p_\Gamma)]\right)\in K_*(X),$$
and 
$$\ind(D)=\int_{\Gamma_X}\ch[\xi]\wedge \e^{c_1(\mathfrak{M}(\mathcal{H}))}\wedge \mathrm{Td}(\Gamma_X).$$
\end{cor*}

Corollary \ref{indexofmrulaknon} extends the index formula in \cite{baumvanerp} from co-oriented contact manifolds to $\pmb{F}$-regular Carnot manifolds. Indeed, in the co-oriented contact case, the Fock bundle $\mathpzc{F}\to \Gamma_X$ defines a bundle of flat orbit representations with $\mathfrak{M}(\mathpzc{F})=1$ and $\Gamma_X=(X\dot{\cup}(-X))\times \R_{>0}$ as spin$^c$-manifolds. 

\begin{remark}
The index formula of Corollary \ref{indexofmrulaknon} extends also to graded Rockland sequences as defined in \cite{Dave_Haller1}. This covers the (curved) Bernstein-Gelfand-Gelfand sequences of \cite{morecap}. For more details, see the discussion below in Section \ref{sec:grafadpknapdnarock}.
\end{remark}

\begin{remark}
The index formula \eqref{vanerpfirstformalaal} was proven by Mohsen \cite{mohsen1} and the class $\psi^{-1}[\sigma_H(D)]$ was computed in \cite{mohsen2} by means of an operator valued symbol constructed at a symbolic level. A crucial step in the computation of \cite{mohsen2} was a modification of the osculating Lie algebroid. It is of interest to note that this modification always produces a Lie algebra admitting flat orbits; we discuss this modification in Example \ref{mohsenconstruct}. It is possible to perform this modification at the level of Carnot manifolds, but our construction of such a modification produces other technical issues and we have omitted all details beyond Example \ref{mohsenconstruct}.
\end{remark}

Despite the beauty in the solution of Baum-van Erp  \cite{baumvanerp} being indisputable, there is in practice a $K$-theoretical unwieldiness to both their solution and Corollary \ref{indexofmrulaknon}. In order to discuss these computational issues, let us give a version of the algorithm on page \pageref{bveaogo} -- for solving the index problem -- adapted to $\pmb{F}$-regular Carnot manifolds:
\begin{enumerate}
\label{blablaalgorithm}
\item[Step 1:]  {\bf Identifying a Zariski open Hausdorff subset $\Gamma_X\subseteq \widehat{T_HX}$}\\
The bundle $\Gamma_X\to X$ of flat orbits is a Zariski open Hausdorff subset of $\widehat{T_HX}$ with a Hilbert space bundle $\mathcal{H}\to \Gamma_X$ and fulfills the same properties as in Step 1 of the algorithm on page \pageref{bveaogo}: that is, the corresponding ideal $I_X\subseteq C^*(T_HX)$ has vanishing Dixmier-Douady invariant and induces a surjection $K_*(I_X)\to K_*(C^*(T_HX))$.
\item[Step 2:] {\bf Localizing the index class to $\Gamma_X$}\\
The distributional multiplier $\tilde{\sigma}$ -- lifting the Carnot symbol $\sigma_H(D)$ of $D$ -- acts as a densely defined endomorphism of $\mathcal{H}$. Surjectivity of $K_*(I_X)\to K_*(C^*(T_HX))$ ensures that the family $(\pi_\xi(\tilde{\sigma}))_{\xi\in \Gamma_X}$ can be continuously deformed to a $\hat{\sigma}$ which is invertible outside a compact of $\Gamma_X$, therefore producing a pre-image of $[\sigma_H(D)]\in K_*(C^*(T_HX))$ as a $K$-theory element $\ind[\hat{\sigma}]\in K_0(C_0(\Gamma_X;\mathbb{K}(\mathcal{H})))$. 
\item[Step 3:] {\bf Finite rank approximation of the localized index class on $\Gamma_X$}\\
Under Morita invariance, $K_*(C_0(\Gamma_X;\mathbb{K}(\mathcal{H})))\cong K^*(\Gamma_X)$, so we can represent the class $[\sigma_H(D)]\in K_0(C^*(T_HX))$ by approximating the deformed symbol $\hat{\sigma}$ with an elliptic complex $\xi_D$ on $\Gamma_X$. 
\end{enumerate}
For $\pmb{F}$-regular Carnot manifolds, Step 1 is automatically solved by Theorem \ref{maintechforrep} but we nevertheless include it as we believe a similar approach is possible for other Carnot manifolds upon identifying the right Zariski open Hausdorff subset $\Gamma_X\subseteq \widehat{T_HX}$. However, Step 2 and Step 3 involve considerable amounts of constructions in $K$-theory already for co-oriented contact manifolds as in \cite{baumvanerp}. We give two examples of full solutions beyond the realm of contact manifolds. The computational complexity of Step 2 and Step 3 is immense for osculating Lie groups of high step length, so we restrict to step length two. An exception is found in Theorem \ref{indexforengel} where we prove a vanishing result for a depth three example.

\subsection{Two examples of index theorems} We consider two examples generalizing an index theorem of Baum-van Erp from \cite{baumvanerp} and an index theorem of Boutet de Monvel \cite{bdmindex} and Epstein-Melrose \cite{melroseeptein}, respectively, from co-oriented contact manifolds to regular polycontact manifolds. \\

A polycontact structure on a manifold $X$ is a subbundle $H\subseteq TX$, with Levi bracket $\mathcal{L}:H\wedge H\to TX/H$, $Y\wedge Y'\mapsto [Y,Y']+H$, such that for any non-zero $\xi\in H^\perp\subseteq T^*X$ the two-form 
$$\omega_\xi(Y,Y'):=\xi\circ\mathcal{L}(Y\wedge Y'),$$
is non-degenerate on $H$. A polycontact manifold is a depth $2$ Carnot manifold $X$ with step two osculating Lie groupoid $T_HX=H\oplus TX/H$, whose Carnot structure is defined from a polycontact structure $H\equiv T^{-1}X\subseteq TX$. If $X$ is a contact manifold, the sphere bundle $S(H^\perp)\to X$ is a two-fold covering which is trivial (in which case $S(H^\perp)=X\dot{\cup}(-X)$) if and only if the contact structure is co-oriented. Let $p:H^\perp\setminus X\to X$ denote the projection. We can consider $(\omega_\xi)_{\xi\in H^\perp\setminus X}$ as a symplectic form $\omega$ on $p^*H\to H^\perp\setminus X$. Since $\omega$ is non-degenerate, we can choose an adapted complex structure on $p^*H$. In this case, there is an abundance of flat orbits as we have the identity
$$\Gamma_X=H^\perp\setminus X\subseteq T^*X.$$
In particular, the set of flat orbits is a suspension of the  compact spin$^c$-manifold $S(H^\perp)\subseteq T^*X$. We can take the bundle of flat orbit representations in Theorem \ref{maintechforrep} to be the (symmetric) Fock bundle 
$$\mathpzc{F}:=\bigoplus_{k=0}^\infty (p^*H)^{\otimes_\C^{\rm sym} k}\to H^\perp\setminus X.$$

The reader can find the following index theorem stated as Theorem \ref{charhellfofodotwocordonkwithgeoindex} in the bulk of the text. The characterization of $H$-ellipticity provided below in Theorem \ref{indexbacaocm} is found in Theorem \ref{charhellfofodo} in the bulk of the text and it is by \cite[Theorem 1.1]{horcharnice} equivalent to hypoellipticity with loss of one derivative. 

\begin{thm*}
\label{indexbacaocm}
Let $X$ be a Riemannian regular polycontact manifold with polycontact bundle $H$ and $E\to X$ a complex vector bundle. Let $\Delta_H\in \mathcal{DO}^2_H(X;E)$ denote an associated sub-Laplacian operator. 

Assume that $\gamma\in C^\infty(H^\perp,p^*\End(E))$ is a fibrewise polynomial such that for any $\xi\in H^\perp\setminus X$ and $k\in \N$ the vector bundle morphism 
$$\gamma_k(\xi): (p^*H)^{\otimes_\C^{\rm sym} k}\otimes E_{\xi}\to (p^*H)^{\otimes_\C^{\rm sym} k}\otimes E_\xi,$$
is an isomorphism, where $\gamma_k$ is as in Proposition \ref{charhellfofodotwo}. Consider a differential operator $D_\gamma\in \mathcal{DO}^2_H(X;E)$ such that $D_\gamma-\Delta_H$ is a differential operator of classical order $1$ whose classical symbol satisfies 
$$\sigma^1(D_\gamma-\Delta_H)|_{H^\perp}=\gamma.$$ 

The operator $D_\gamma$ is $H$-elliptic and its $K$-homology class is, for large enough $N>>0$, represented by the geometric cycle 
$$(S(H^\perp)\times S^1,E_\gamma^N,p),$$
where $E_\gamma^N\to S(H^\perp)\times S^1$ is defined from clutching 
$$\bigoplus_{k=0}^N (p^*H)^{\otimes_\C^{\rm sym} k}\times [0,1]\to S(H^\perp)\times [0,1].$$
along the automorphism 
$$\bigoplus_{k=0}^N\gamma_k: \bigoplus_{k=0}^N (p^*H)^{\otimes_\C^{\rm sym} k}\circlearrowleft.$$
\end{thm*}

\begin{remark}
\label{commedoned}
Theorem \ref{indexbacaocm} extends results for contact manifolds from \cite{baumvanerp} to polycontact manifolds. It is not clear how interesting index theorems such as Theorem \ref{indexbacaocm} are when the depth is larger than two. In Theorem \ref{indexforengel}, we show that for a four-dimensional manifold with an Engel structure -- a depth three Carnot structure -- and an analogue of $D_\gamma$ has vanishing index. This result is more or less immediate from work of Helffer \cite{helffercomp} and extends to any Carnot manifold such that $T^{-2}X/T^{-1}X$ is a line bundle and $\mathfrak{t}_HX_x$ having rank $>2$ in all points $x\in X$. Compare also to the discussion in Remark \ref{akfnaodns} of the vanishing of the index of $D_\gamma$ for step length $>2$.
\end{remark}

Let us also consider an index theorem for Toeplitz operators on polycontact manifolds. If $X$ is the boundary of a strictly pseudoconvex domain $\Omega$ in a complex manifold, it inherits a co-oriented contact structure. There is (under a technical assumption always valid if $\dim(X)\geq 5$) a distinguished projection $P\in \Psi^0_H(X)$ called the Szegö projection. The image $H^2(X):=PL^2(X)$ consists of all $L^2$-functions on $X=\partial \Omega$ with a holomorphic extension to $\Omega$. A classical object of study is that of Toeplitz operators
$$T_u:=PuP:H^2(X;\C^N)\to H^2(X;\C^N),$$
defined from a matrix valued $u\in C(X;M_N(\C))$. The index theory of Toeplitz operators on boundaries of strictly pseudoconvex domains was described by Boutet-de Monvel \cite{bdmindex}. The result was extended by Epstein-Melrose \cite{melroseeptein} to co-oriented contact manifolds. Their index theory in fact comes from a $K$-homology class $[P]:=[(L^2(X;E),2P-1)]\in K_1(X)$, that was described in \cite{baumvanerpbdmindex}.

We consider projections $P\in \Psi^0_H(X;E)$, for some vector bundle $E$, on more general Carnot manifolds $X$. We restrict to the case that $X$ has depth $2$. It was proven in \cite{vanerpszego} that if $X$ has depth $2$, projections in $\Psi^0_H(X)$ with nontrivial symbol exist if and only if $X$ is polycontact. And in this case, $P$ or $1-P$ is a Hermite operator, i.e. one of the two Carnot symbols $\sigma_H(P)$ and $\sigma_H(1-P)$ is supported in the flat orbits $\Gamma_X=H^\perp\setminus X$. We note that if $P\in \Psi^0_H(X;E)$ is a Hermite operator and $\sigma_H^0(P)$ is a projection, then $\sigma_H^0(P)(E\otimes \mathpzc{F})$ defines a complex vector bundle on $\Gamma_X=H^\perp\setminus X$. The following index theorem is found as Theorem \ref{bdmdmdam} in the body of the text and extends the index theorems \cite{baumvanerpbdmindex,bdmindex} from boundaries of strictly pseudoconvex domains to regular polycontact manifolds.

\begin{thm*}
Let $X$ be a regular polycontact manifold defined from a polycontact bundle $H$ and $E\to X$ a complex vector bundle. Assume that $P\in \Psi^0_H(X;E)$ is an idempotent Hermite operator, so $(L^2(X;E),2P-1)$ is an odd $K$-homology cycle. The associated $K$-homology class is denoted by $[P]:=[(L^2(X;E),2P-1)]\in K_1(X)$. 

The $K$-homology class $[P]$ is represented by the geometric cycle 
$$(S(H^\perp),\sigma_H^0(P)(E\otimes \mathpzc{F}),p).$$
In particular, for any $u\in C(X,GL_N(\C))$, the Toeplitz operator
$$PuP:PL^2(X;E\otimes \C^N)\to PL^2(X;E\otimes \C^N),$$
is Fredholm and its index is computed from the formula
$$\ind(PuP)=\int_{S(H^\perp)} \ch[u]\wedge \ch[\sigma_H^0(P)(E\otimes\mathpzc{F})]\wedge \mathrm{Td}(H^\perp).$$
\end{thm*}

\section{Summary of contents}

\subsection*{Part \ref{sec:reptheory}}
In the second part of the monograph we discuss the representation theory of nilpotent Lie groups. While this is a well studied subject, and well understood due to Kirillov's orbit method, we focus on global aspects of bundles of representations over open Hausdorff subsets of the representation space. A crucial role is played by the work of Pedersen \cite{pedersen84,pedersen88,pedersen89,pedersen94}, Lion \cite{lion76} and Vergne \cite{vergneCR70,vergneBullMathFra72}. The individual results in this part of the work are far from novel, but we were not able to find references piecing together the results in the way needed for index theory which explains much of the length of this part. The reader well acquainted with the representation theory of nilpotent Lie groups can skip this part of the monograph on a first reading and return to it when needed. 

In Section \ref{subsec:reptheory} we review constructive aspects of representation theory of a nilpotent simply connected Lie group via Kirillov's orbit method seen through the procedure of geometric quantization. In Section \ref{subsecfinestrat} we discuss Pedersen's fine stratification of $\mathfrak{g}^*$ and the spectrum $\widehat{\mathsf{G}}$ disassembling the spectrum into a finite hierarchy of Hausdorff subsets. In Section \ref{subsecflatorbitssos} we further analyze the flat orbits and recall the various characterizations of flat orbits and show that the associated bundle of representations glue together. In Section \ref{sec:autoromom} we discuss the automorphisms of a nilpotent simply connected Lie group and its action on the set of flat orbits. In Section \ref{ctstructrifoeod} we study the continuous trace algebra structure of the ideal $I_{\mathsf{G}}$ defined from the open set of flat orbits and its $\Aut(\mathsf{G})$-action. In Section \ref{subsectionexamples} we give examples of Carnot-Lie groups and describe the bundle of flat orbit representations. 

\subsection*{Part \ref{gropodoprar}}
In the third part of this work we study groupoids, primarily those associated with bundles of nilpotent Lie groups and related deformations. The motivation to do so is that the tangent bundle of a Carnot manifold can be seen as a groupoid globally describing the nilpotent, osculating Lie group structure on each fibre. Homological information about $H$-elliptic operators on Carnot manifolds is naturally contained in $K_*(C^*(T_HX))$ and the aim of this part is to set the stage for computations in such $K$-groups. The three main results in this part are in terms of the osculating Lie groupoid $T_HX$ the following:
\begin{itemize}
\item firstly, the known Connes-Thom isomorphism of Nistor \cite{nistorsolvabel} producing isomorphisms $K_*(C^*(T_HX))\cong K^*(T^*X)$;
\item secondly, a Connes-Thom isomorphism for the ideal of flat orbits $K_*(I_X)\cong K^*(\Xi_X)$ compatible with Nistor's Connes-Thom isomorphism and with the Morita equivalence defined from a bundle of flat orbit representations;
\item thirdly, that the inclusion map $K_*(I_X) \rightarrow K_*(C^*(T_HX))$ is surjective when $X$ is $\pmb{F}$-regular. 
\end{itemize}

 In Section \ref{sec:liegrohoad} we review definitions, basic facts and examples from the theory of Lie groupoids and algebroids. We describe central decomposition of a Lie group as a twisted semidirect product in Section \ref{sec:groupoidforflatorbits} and these results are extended to a more general groupoid situation in Section \ref{subsec:loctrivalandnda}. In Section \ref{subsec:nistorctsubsn} we discuss Nistor's Connes-Thom isomorphism $K_*(C^*(T_HX))\cong K^*(T_HX)$ and its restriction to $K_*(I_X) \simeq K^*(\Xi_X)$ -- the construction relies on classical constructions of adiabatic groupoids, parabolical adiabatic groupoids and also a cocycle deformation thereof. At the end of the section we prove surjectivity of the inclusion map $K_*(I_X) \rightarrow K_*(C^*(\mathcal{G}))$. In Section \ref{sec:connesthomandadiaofofd} we study the compatibility of these isomorphisms showing that the Connes-Thom isomorphism $K_*(I_X) \simeq K^*(\Gamma_X)$ differs from the Morita equivalence defined from a bundle of flat orbit representations $\mathcal{H}$ up to a line bundle -- the metaplectic correction line bundle $\mathfrak{M}(\mathcal{H})$. 

\subsection*{Part \ref{sec:carnotmfds}}

The underlying geometric object of study in this monograph is Carnot manifolds. In the fourth part we give a general overview of Carnot manifolds. We explore the general properties of depth two Carnot manifolds and numerous other examples. An important object we recall in this part is the osculating tangent groupoid $T_HX\to X$ a fibre bundle where each fibre is a simply connected nilpotent Lie group. We introduce the notion of $\pmb{F}$-regularity of a Carnot manifold in which case the fibre of $T_HX\to X$ is (up to isomorphism) the same Lie group and it admits flat orbits. The geometry of Carnot manifolds is well studied, for instance in \cite{choiponge,Dave_Haller1,Dave_Haller2,melinoldpreprint,Morimoto,sadeghhigson,vanErp_Yunckentangent}. In this part we mainly set notations and provide examples. In Section \ref{subsec:carnot} we recall the geometry of Carnot manifolds. In Section \ref{sec:examcarn} the reader can find several examples of Carnot manifolds. In Section \ref{subsec:parabolictanget} we recall the construction of the parabolic tangent groupoid from \cite{vanErp_Yunckentangent}.

\subsection*{Part \ref{part:pseudod}}

In this part of the work we study operators in the Carnot calculus on general Carnot manifolds. We use the Carnot calculus of Melin \cite{melinoldpreprint} and van Erp-Yuncken \cite{vanErp_Yuncken} which was further studied by Dave-Haller \cite{Dave_Haller1, Dave_Haller2}. Here we provide the technical set up for the  studying index theory in the Carnot calculus. 

We recall the Carnot calculus of van Erp-Yuncken \cite{vanErp_Yuncken} in Section \ref{pseudcalcalc}. We also provide a range of different examples: for instance the operators $\Delta_H+\gamma T$ studied by \cite{baumvanerp}, Bernstein-Gelfand-Gelfand complexes, and Szegö projections. We also recall the graded calculus of Dave-Haller \cite{Dave_Haller1}. In Section \ref{subsec:hellleleldlda} we set up the represented symbol calculus following Dave-Haller \cite{Dave_Haller1, Dave_Haller2} and revisit the examples $\Delta_H+\gamma T$, the Bernstein-Gelfand-Gelfand complexes, and the Szegö projections. Using the Rockland condition from \cite{Dave_Haller1}, and applying standard techniques of pseudodifferential calculus, the Fredholm theory of elliptic pseudodifferential operators readily extends to $H$-elliptic operators on a general Carnot manifold. 

The main technical contributions of Part \ref{part:pseudod} are found in Section \ref{secondnaaoaoac}. Here we refine analytic properties of the symbol calculus of $H$-elliptic operators. We prove, for $\pmb{F}$-regular manifolds, that the symbols of $H$-elliptic operators behave well on the $C_0(\Gamma_X)$-Hilbert $C^*$-module of flat orbit representations, allowing us to define $K$-theory cycles and moving freely between bounded and unbounded models thereof. For $\pmb{F}$-regular manifolds, we characterize $H$-ellipticity via the condition that the principal symbol restricted to the flat orbits is invertible as a multiplier on the $C_0(\Gamma_X)$-Hilbert $C^*$-module of flat orbit representations (suitably modified with respect to the order of the symbol). 

In Section \ref{sec:ktheoomon} we utilize the technical tools of Section \ref{secondnaaoaoac} to study the $K$-theoretical invariants in $K_*(C^*(T_HX))$ and $K^*(\Gamma_X)$. For the purposes of index theory, a goal is to construct a suitable lift of the symbol class $[\sigma_H(D)]\in K_*(C^*(T_HX))$ to an element of $K^*(\Gamma_X)$, i.e. an elliptic complex on $\Gamma_X$. This is a two step process: first we need to localize the support of $[\sigma_H(D)]\in K_*(C^*(T_HX))$ to $\Gamma_X$, i.e. construct a pre-image under $K_*(I_X)\to  K_*(C^*(T_HX))$, and then we need to approximate by a finite rank elliptic complex, i.e. implement the Morita equivalence $K_*(I_X)\cong K^*(\Gamma_X)$. We prove that this is possible in the abstract, and provide a range of techniques to do so. On a polycontact manifold we carry out these computations in the examples coming from $\Delta_H+\gamma T$-operators and projections of Hermite type.

\subsection*{Part \ref{dualityapapadoda}}

In the last part of this work we put the various pieces together. The goal of this part is to set the index problem for $H$-elliptic operators in the Carnot calculus in a $K$-homological context via the dualities discussed above in Theorem \ref{commtererar} and apply the $K$-theory computations from Part \ref{part:pseudod}. 

In Section \ref{geomedadonaodn} we develop a variation of geometric $K$-homology that uses geometric cycles modelled on noncompact spin$^c$-manifolds but places the compactness condition on the vector bundle data. Section \ref{geomedadonaodn} closely follows standard developments in Baum-Douglas geometric $K$-homology, cf. the related work \cite{emermeyer}, with the additional technical details from non-compactness being addressed using techniques of Bunke \cite{bunkerelative} and Kucerovsky \cite{dantheman}. The noncompact model for geometric $K$-homology is put to use in Section \ref{sec:ordinadadp} in order to study $K$-homological dualities. In spirit, the constructions of Section \ref{sec:ordinadadp} closely follows \cite{baumvanerp} however in its generality it relies heavily on the developments in Part \ref{sec:reptheory}, \ref{gropodoprar}, \ref{sec:carnotmfds}, and \ref{part:pseudod}. In Section \ref{sec:indexingener} we conclude index theorems for $H$-elliptic operators that after the set up of Section \ref{sec:ordinadadp} are readily deduced.

This work ends with Section \ref{sec:grafadpknapdnarock} where we discuss the problem of solving the index problem for graded Rockland sequences. In an abstract setting, its solution is similar to that of $H$-elliptic operators. However, in concrete settings such as BGG-complexes we believe it fruitful to study further, especially in an equivariant context.

\section{Acknowledgements}

MG was supported by the Swedish Research Council Grant VR 2018-0350. 
We thank Bo Berndtsson for interesting discussions about weakly pseudoconvex domains, Ryszard Nest for discussions about everything between index theory and cyclic cohomology, Alexander Gorokhovsky for explaining his work with Erik van Erp and the algebra of Carnot symbols on a contact manifold and Heiko Gimperlein for discussions leading up to Remark \ref{regularidremak}. We are also grateful for interesting discussions with Robin Deeley, Eske Ewert, Omar Mohsen, Jonathan Rosenberg and Robert Yuncken. We thank Bernard Helffer for helping us locate the manuscript \cite{helffercomp}.

\part{Simply connected nilpotent Lie groups}
\label{sec:reptheory}

\begin{center}
{\bf Introduction to part}
\end{center}
Each fibre in the tangent bundle of a Carnot manifold carries the structure of a nilpotent Lie group (see more below in Part \ref{sec:carnotmfds}) -- the osculating Lie group in the fibre. The symbol calculus relates the Carnot calculus on a Carnot manifold to the representation theory of the osculating Lie group. Ellipticity in the calculus is detectable via the Rockland condition, a condition on the action of the symbols in representations of the osculating Lie group. In this section, we review the prerequisites of nilpotent Lie groups and their representation theory that are necessary for understanding the index theory on Carnot manifolds. We have allowed for a generous interpretation of what counts as a \emph{necessary prerequisite}. At the cost of a few extra pages, we have presented the material with a broad perspective that is not directly used later on, but shed light on the $\pmb{F}$-regularity condition used in the main results of this work. 

The description of the representations is well known from Kirillov theory. What we will mainly be concerned with is the global structure of the flat orbit representations and how they transform under automorphisms. The general goal is to construct global bundles of flat orbit representations to be used in the index theory of $H$-elliptic operators -- operators elliptic in the Carnot calculus. We suspect that the individual pieces in this part of the monograph pose no surprise to the reader familiar with nilpotent Lie groups, but in lack of a reference for the construction and transformation properties of the global bundle of flat orbit representations we have included a large amount of details. 

In Section \ref{subsec:reptheory} we review constructive aspects of representation theory of a nilpotent simply connected Lie group via Kirillov's orbit method seen through the procedure of geometric quantization, following work of Pedersen \cite{pedersen84,pedersen88,pedersen89,pedersen94}. Geometric quantization is a general method producing a Hilbert space representation from certain symplectomorphisms on a symplectic manifold with a prequantum line bundle and a polarization. 
The basis for such constructions is the fact that simply connected nilpotent groups $\mathsf{G}$ possess a large family of symplectic manifolds where $\mathsf{G}$ acts by symplectomorphisms, namely the coadjoint orbits. We polarize them using the Vergne polarization \cite{vergneCR70,vergneBullMathFra72}. For global considerations, there are explicit unitary intertwiners constructed by Lion \cite{lion76} that we later use to glue together bundles of representations over different parts of the representation space.

In Section \ref{subsecfinestrat} we discuss Pedersen's fine stratification of $\mathfrak{g}^*$ and the spectrum $\widehat{\mathsf{G}}$. Given a Jordan-Hölder basis, $\mathfrak{g}^*$ disassembles into a finite hierarchy of Hausdorff subsets $\Xi_1, \Xi_2, \ldots, \Xi_m$ such that $\Xi_j$ is Zariski open in $\mathfrak{g}^*\setminus \bigcup_{i < j} \Xi_i$. Each $\Xi_j$ is invariant under the coadjoint action which is free and proper on it, so the quotient spaces $\Gamma_j$ stratifies the spectrum $\widehat{\mathsf{G}}$. This stratification allows to assemble representations of $\mathsf{G}$ into bundles of representations over the fine stratas. Thus, instead of a set of Hilbert space representations numerated by coadjoint orbits, we obtain a finite sequence of representations on Hilbert $C^*$-modules $\mathcal{H}_i$ numerated by the fine stratas. 

In Section \ref{subsecflatorbitssos} we further analyze the flat orbits and recall the various characterizations of flat orbits. We prove that the  set of flat orbits is covered by the top fine stratas when varying the Jordan-Hölder bases. The set of flat orbits has worse "local" properties than top fine stratas - there are bundles of polarizations on the top fine strata that need not extend over the set of flat orbits. We do however show that the associated bundle of representations glue together. The set of flat orbits however has better "global" properties - it is $\Aut(\mathsf{G})$-invariant. We recall the property of a simply connected nilpotent lie group to have flat orbits. 

In Section \ref{sec:autoromom} we discuss the automorphisms of a nilpotent simply connected Lie group. The group $\mathsf{Aut}(\mathsf{G})$ acts on the spectrum preserving the set of flat orbits. A class of automorphisms of particular interest comes from dilations, or equivalently gradings; it plays a crucial role in the pseudodifferential calculus on Carnot manifolds. Nilpotent Lie groups equipped with a dilation are called Carnot-Lie groups. 

In Section \ref{ctstructrifoeod} we study the continuous trace algebra structure of the ideal $I_\Gamma$ defined from the open set of flat orbits. The automorphism group of $\mathsf{G}$ acts on $I_\Gamma$, and we are particularly interested in constructing an equivariant Morita equivalence $I_\Gamma\sim_M C_0(\Gamma)$, or more precisely a bundle of flat orbit representations which is covariant for an action of a subgroup of automorphisms. Such constructions produce induced bundles of flat orbits on Carnot manifolds. 
In Section \ref{subsectionexamples} we give examples of Carnot-Lie groups and describe the bundle of flat orbit representations.

\section{Representation theory}
\label{subsec:reptheory}
The spectrum of a nilpotent Lie group can be described as the space of coadjoint orbits by Kirillov theory. For a nilpotent Lie group $\mathsf{G}$ we let $\hat{\mathsf{G}}$ denotes its spectrum, i.e. the set of irreducible unitary representations of $\mathsf{G}$ equipped with the Fell topology. More details on the representation theory of nilpotent groups can be found in for instance \cite{Corwin_Greenleaf,kirillovbook,pukan67jfa, pukan67,pukanens71}. The coadjoint action of $\mathsf{G}$ on $\mathfrak{g}^*$ is given by
\[ \mathsf{Ad}^* : \mathsf{G}  \rightarrow \mathrm{End}(\mathfrak{g}^*), \quad [\mathsf{Ad}^*(g)(\xi)](X)= \xi( \mathsf{Ad}(g^{-1})(X)),\]
for $g\in \mathsf{G}$, $\xi \in \mathfrak{g}^*$ and $X\in \mathfrak{g}$ where $\mathsf{Ad}$ denotes the adjoint action of $\mathsf{G}$ on its Lie algebra. The orbits of this action are called coadjoint orbits. An element $\xi \in \mathfrak{g}^*$ defines a coadjoint orbit by
\[ \mathcal{O}_\xi = \mathsf{Ad}^*(\mathsf{G}) \{ \xi \}. \]
By construction, the mapping $\mathsf{G} \to \mathcal{O}_\xi$, $g\mapsto \mathsf{Ad}^*(g)\xi$ defines a diffeomorphism $\mathsf{G}/\mathsf{Stab}(\xi)\cong \mathcal{O}_\xi$. Here $\mathsf{Stab}(\xi):=\{g\in \mathsf{G}:\, \mathsf{Ad}^*(g)\xi=\xi\}$ is the stabilizer of $\xi$. The mapping $\mathsf{G} \to \mathcal{O}_\xi$ identifies $T_\xi\mathcal{O}_\xi$ with $\mathfrak{g}/\mathsf{stab}(\xi)$ where $\mathsf{stab}(\xi)=\{X\in \mathfrak{g}:\, \xi([X,Y])=0\; \forall Y\in \mathfrak{g}\}$ denotes the Lie algebra of $\mathsf{Stab}(\xi)$.

Kirillov's orbit method describes a procedure which identifies $\widehat{\mathsf{G}}$ and the space of coadjoint orbits $\mathfrak{g}^* / \mathsf{Ad}^*$. One can make $\widehat{\mathsf{G}}$ into a topological space by equipping $\widehat{\mathsf{G}}$ with the Fell topology and $\mathfrak{g}^* / \mathsf{Ad}^*$ with the quotient topology. Then the mapping, described by Kirillov's theory, is a homeomorphism, see \cite{Corwin_Greenleaf}. Let us describe one possible way to set up this correspondence by means of polarizations of the orbits. This is described in full detail in \cite[Section 1]{pedersen88}.

Fix a coadjoint orbit $\mathcal{O}\in \mathfrak{g}^*/\mathsf{Ad}^*$. For $\xi\in \mathcal{O}$, the $2$-form $\omega_\xi(X,Y):=\xi([X,Y])$ on $\mathfrak{g}$ is $\mathsf{Stab}(\xi)$-invariant and its radical is exactly $\mathsf{stab}(\xi)$. The form $\omega_\xi$ is called the Kirillov form in $\xi$. The identification of $T_\xi\mathcal{O}_\xi$ with $\mathfrak{g}/\mathsf{stab}(\xi)$ induces a $\mathsf{G}$-equivariant symplectic form $\omega_{\mathcal{O}}$ on the coadjoint orbit $\mathcal{O}$. This construction depends only on $\mathcal{O}$ and not on the choice of $\xi\in \mathcal{O}$.

We define the prequantization of a coadjoint orbit $\mathcal{O}\in \mathfrak{g}^*/\mathsf{Ad}^*$ following \cite[Section 1.3]{pedersen88}. Observe that for any $\xi\in \mathfrak{g}^*$, the restriction $\xi|_{\mathsf{stab}(\xi)}$ is a Lie algebra character so $\chi_{0,\xi}:=\mathrm{exp}(i\xi|_{\mathsf{stab}(\xi)}):\mathsf{Stab}(\xi)\to U(1)$ is a well defined Lie group character. We can define a $\mathsf{G}$-equivariant hermitean line bundle $L_\mathcal{O}\to \mathcal{O}$ by fixing a $\xi\in \mathcal{O}$ and identifying the $\mathsf{G}$-equivariant hermitean line bundle $\mathsf{G}\times_{\chi_{0,\xi}}\C\to \mathsf{G}/\mathsf{Stab}(\xi)$ with a line bundle on $\mathcal{O}\cong \mathsf{G}/\mathsf{Stab}(\xi)$. It is readily verified that this line bundle does not depend on the choice of $\xi\in \mathcal{O}$. The smooth sections $C^\infty(\mathcal{O},L_\mathcal{O})$ can for any $\xi\in \mathcal{O}$ be identified with the space of $\mathsf{Stab}(\xi)$-invariant sections of $C^\infty(\mathsf{G},\C_{\chi_{0,\xi}})$, i.e. as a smooth function $\tilde{f}:\mathsf{G}\to \C$ such that 
$$\tilde{f}(gh)=\tilde{f}(g)\chi_{{0,\xi}}(h)^{-1},\quad \forall g\in \mathsf{G}, \; h\in \mathsf{Stab}(\xi).$$
We define a hermitean connection $\nabla_\mathcal{O}$ on $L_\mathcal{O}$ by for $v\in C^\infty(\mathcal{O},T\mathcal{O})$, $f\in C^\infty(\mathcal{O},L_\mathcal{O})$ and $\xi\in \mathcal{O}$ setting 
$$\nabla_{\mathcal{O},v} f(\xi)=[X.\tilde{f}](e)-i\xi(X)\tilde{f}(e),$$
where $\tilde{f}\in C^\infty(\mathsf{G})$ is the $\mathsf{Stab}(\xi)$-invariant section of $C^\infty(\mathsf{G},\C_{\chi_{0,\xi}})$ representing $f$ and $X\in \mathfrak{g}$ is any preimage under $\mathfrak{g}\to \mathfrak{g}/\mathsf{stab}(\xi)$ of $v_\xi$ under the canonical identification $T_\xi\mathcal{O}\cong \mathfrak{g}/\mathsf{stab}(\xi)$. The curvature of $\nabla_\mathcal{O}$ is $-i\omega_{\mathcal{O}}$. We define the prequantization of $\mathcal{O}$ as the pair $(L_\mathcal{O},\nabla_{\mathcal{O}})$.

A subbundle $F\subseteq T\mathcal{O}$ is said to be a polarization of $\mathcal{O}$ if it is a $\mathsf{G}$-equivariant involutive Lagrangian subbundle, i.e. $[F,F]\subseteq F$ and $F^\perp=F$ in the symplectic form $\omega_\mathcal{O}$. Since the action is transitive, a polarization is determined by a fibre $F_\xi\subseteq \mathfrak{g}/\mathsf{stab}(\xi)$ for any $\xi\in \mathcal{O}$. Indeed, for a fixed $\xi\in \mathcal{O}$, a subspace $\mathfrak{h}\subseteq \mathfrak{g}$ induces a polarization $F$ of $\mathcal{O}$ with $F_\xi=\mathfrak{h}/ \mathsf{stab}(\xi)$ if and only if $\mathfrak{h}$ is a real algebraic polarization of $\xi$, i.e. $\mathfrak{h}$ is a Lie subalgebra which is maximally isotropic with respect to $\omega_\xi$.  Note that if $\mathfrak{h}$ is maximally isotropic with respect to $\omega_\xi$ it follows that $\mathsf{stab}(\xi)\subseteq \mathfrak{h}$. Existence of polarizations of $\mathcal{O}$ thus follows from \cite{kirillovbook} and below we focus on the Vergne polarization (see Equation \eqref{vergndendoded}).

For a polarization $F$ of $\mathcal{O}$, define 
$$C^\infty_F(\mathcal{O},L_\mathcal{O}):=\{f\in C^\infty(\mathcal{O},L_\mathcal{O}): \; \nabla_{\mathcal{O},v}f=0\; \forall v\in C^\infty(\mathcal{O},F)\}.$$
If we fix a $\xi\in \mathcal{O}$, the polarization $F$ defines a real algebraic polarization $\mathfrak{h}_\xi$ of $\xi$. Let $\mathsf{G}_\xi:=\mathrm{exp}(\mathfrak{h}_\xi)\subseteq \mathsf{G}$. Since $\mathfrak{h}_\xi$ is isotropic for $\omega_\xi$, $\xi|_{\mathfrak{h}_\xi}$ is a Lie algebra character and $\chi_{F,\xi}:=\mathrm{exp}(\xi|_{\mathfrak{h}_\xi}):\mathsf{H}_\xi \to U(1)$ is a Lie group character. Write $\C_{\chi_{F,\xi}}$ for the associated representation of $\mathsf{H}_\xi$. We let $C^\infty(\mathsf{G},\C_{\chi_{F,\xi}})^{\mathsf{H}_\xi}$ denote the invariant subspace for the diagonal $\mathsf{H}_\xi$-action $h.f(g):=\chi_{F,\xi}(h)f(gh)$. The mapping 
\begin{equation}
\label{thevximap}
V_\xi:C^\infty(\mathsf{G},\C_{\chi_{F,\xi}})^{\mathsf{H}_\xi}\to C^\infty_F(\mathcal{O},L_\mathcal{O}), \quad V_\xi f(g\xi):=f(g),
\end{equation}
is a $\mathsf{G}$-equivariant isomorphism. We can identify $C^\infty(\mathsf{G},\C_{\chi_{F,\xi}})^{\mathsf{H}_\xi}$ with the smooth sections of the hermitean line bundle $L_{F,\xi}:=\mathsf{G}\times_{\chi_{F,\xi}}\C\to \mathsf{G}/\mathsf{H}_\xi$. As such, $V_\xi$ induces a pre-Hilbert space structure on 
$$C^\infty_{c,F}(\mathcal{O},L_\mathcal{O}):=V_\xi(C^\infty_c(\mathsf{G}/\mathsf{H}_\xi,L_{F,\xi})).$$ We denote the Hilbert space completion of this space by $\mathcal{H}_{F,\xi}(\mathcal{O})$ which carries a unitary $\mathsf{G}$-action. We shall later recall results of Pedersen showing that the Hilbert space $\mathcal{H}_{F,\xi}(\mathcal{O})$ can be identified $L^2(\R^{\dim(\mathcal{O})/2})$ and that the associated representation of $\mathfrak{g}$ is as first order differential operators with polynomial coefficients. It is immediate from the construction that $\mathcal{H}_{F,\xi}(\mathcal{O})=\ind_{\mathsf{H}_\xi}^{\mathsf{G}}(\chi_{F,\xi})$ and we can deduce the following result from Kirillov's orbit method (see \cite{kirillovbook} or \cite{Corwin_Greenleaf}).

\begin{proposition}
\label{unitarywowowod}
For any polarization $F$ of $\mathcal{O}$ and a $\xi\in \mathcal{O}$, the $\mathsf{G}$-representation $\mathcal{H}_{F,\xi}(\mathcal{O})$ is irreducible and depends on the choice of $F$ and $\xi$ only up to unitary equivalence. 
\end{proposition}

That the $\mathsf{G}$-representation $\mathcal{H}_{F,\xi}(\mathcal{O})$ is irreducible follows from the maximality of a polarization. That the representation up to unitary equivalence is independent of polarization follows from the next lemma due to Lion \cite{lion76}, that will play an important role also later in the monograph. We refer its proof to \cite{lion76}. For more details, see also \cite{lion80}.

\begin{lemma}
\label{lionsintertwiners}
Assume that $\mathfrak{h}_1$ and $\mathfrak{h}_2$ are two real algebraic polarizations at $\xi\in \mathfrak{g}^*$, inducing polarizations $F_1$ and $F_2$ of $\mathcal{O}_\xi$. Let $\mathsf{H}_1$ and $\mathsf{H}_2$ denote the two associated subgroups of $\mathsf{G}$ and define 
\begin{align*}
\mathfrak{L}_{\mathfrak{h}_1,\mathfrak{h}_2}^{(0)}:&C^\infty_c(\mathsf{G}/\mathsf{H}_2,L_{F_2,\xi}) \to C^\infty(\mathsf{G}/\mathsf{H}_1,L_{F_1,\xi}),\\
&\mathfrak{L}_{\mathfrak{h}_1,\mathfrak{h}_2}^{(0)}f(g):=\int_{\mathsf{H}_1/\mathsf{H}_1\cap \mathsf{H}_2}f(gh)\chi_{F_1,\xi}(h)\rd h.
\end{align*}
Then, up to a positive scalar,  $\mathfrak{L}_{\mathfrak{h}_1,\mathfrak{h}_2}^{(0)}$ defines a unitary equivalence $\mathcal{H}_{F_2,\xi}(\mathcal{O}_\xi)\xrightarrow{\sim} \mathcal{H}_{F_1,\xi}(\mathcal{O}_\xi)$ such that if $\mathfrak{h}_3$ is a third real algebraic polarization at $\xi$, it holds that 
$$\mathfrak{L}_{\mathfrak{h}_3,\mathfrak{h}_2}^{(0)}\mathfrak{L}_{\mathfrak{h}_2,\mathfrak{h}_1}^{(0)}\mathfrak{L}_{\mathfrak{h}_1,\mathfrak{h}_3}^{(0)}=\e^{\frac{\pi i}{4}\mathrm{Mas}(\mathfrak{h}_1,\mathfrak{h}_2,\mathfrak{h}_3)},$$
where $\mathrm{Mas}(\mathfrak{h}_1,\mathfrak{h}_2,\mathfrak{h}_3)$ denotes the Maslov triple index defined as the signature of the quadratic form on on $\mathfrak{h}_1\oplus \mathfrak{h}_2\oplus \mathfrak{h}_3$ defined by 
$$(X_1,X_2,X_3)\mapsto \xi([X_1,X_2])+\xi([X_2,X_3])+\xi([X_3,X_1]).$$
\end{lemma}

The Lion transform from Lemma \ref{lionsintertwiners} will play an important role in constructing global bundles of representations. However, for gluing the Maslov triple index produces issues that needs correction. We correct the Lion transform by an $\eta$-invariant studied in \cite{capleem}.

\begin{definition}
\label{definieofeta}
Let $V$ be a symplectic vector space equipped with an adapted complex structure $J$. For two Lagrangian subspaces $L_1,L_2\subseteq V$, define the self-adjoint, closed, densely defined operator 
$$D(L_1,L_2):L^2([0,1];V)\dashrightarrow L^2([0,1],V),$$
by $D(L_1,L_2):=-J\partial_t$ acting on the domain 
$$\Dom(D(L_1,L_2)):=\{f\in H^1([0,1];V): \; f(0)\in L_1, \, f(1)\in L_2\}.$$

The $\eta$-invariant of the pair $(L_1,L_2)$ is defined as 
$$\eta(L_1,L_2):=\eta(D(L_1,L_2)),$$
the $\eta$-invariant of the self-adjoint operator $D(L_1,L_2)$. 
\end{definition}

The $\eta$-invariant of pairs of Lagrangians was studied in \cite[Chapter 6]{capleem}. We remark that the adapted complex structure $J$ appearing in Definition \ref{definieofeta} is in relation to the Maslov triple index auxiliary. 

\begin{remark}
\label{commentaboutpairs}
Let $\pmb{L}(V)$ denote the set of Lagrangian subspaces of $V$. The set $\pmb{L}(V)$ carries a transitive action of $U(V)$ with isotropy group $O(V)$, so fixing a Lagrangian subspace we can identify $\pmb{L}(V)\cong U(V)/O(V)$. Note here that the adapted complex structure induces both the structure of a complex Hermitean vector space and an inner product space on $V$. In particular, if we take a Lagrangian subspace $L_1$ then for any other Lagrangian subspace $L_2$ there is an $A\in U(V)$ such that $L_2=AL_1$ and $A$ is uniquely determined up to $O(V)$. The reader can compare to the discussion in \cite[Chapter 6]{capleem}. By \cite[Proposition 6.1]{capleem}, it holds for $A\in U(V)$ and $L\in \pmb{L}(V)$ that 
$$\eta(L,AL)=\mathrm{Tr}(g(A)),$$
where $g$ is the Borel function 
$$g(\e^{i\theta})=
\begin{cases}
1-2\theta/\pi, \; &\theta\in (0,\pi),\\
0, \; &\theta\in \{0,\pi\},\\
-1-2\theta/\pi, \; &\theta\in (-\pi,0). 
\end{cases}$$
Therefore, writing the eigenvalues of $A$ as $(\e^{i\theta_j})_{j=1}^{\dim_\R(V)/2}$, we have that 
$$\e^{\frac{\pi i}{4}\eta(L,AL)}=\mathrm{e}^{\frac{\pi i}{4} d(A)}\det(A)^{-1/2},$$
where 
$$d(A):=\#\{j: \theta_j\in (0,\pi)\}-\#\{j: \theta_j\in (\pi,0)\}+2\#\{j: \theta_j=\pi\}$$ 
and we choose the branch cut of the square root so that $(\e^{i\theta})^{-1/2}=\e^{-i\theta/2}$ for $\theta\in (-\pi,\pi]$.
\end{remark}

It follows from \cite[Proposition 8.3]{capleem} that for three Lagrangians,
$$\mathrm{Mas}(L_1,L_2,L_3)=\eta(L_1,L_2)+\eta(L_2,L_3)+\eta(L_3,L_1).$$
We conclude the following result in combination with Lemma \ref{lionsintertwiners}

\begin{proposition}
\label{lionsintertwinerswitheta}
For two real algebraic polarizations $\mathfrak{h}_1$ and $\mathfrak{h}_2$ at $\xi\in \mathfrak{g}^*$, define the corrected Lion transform
\begin{align*}
\mathfrak{L}_{\mathfrak{h}_1,\mathfrak{h}_2}:&C^\infty_c(\mathsf{G}/\mathsf{H}_2,L_{F_2,\xi}) \to C^\infty(\mathsf{G}/\mathsf{H}_1,L_{F_1,\xi}),\quad \mathfrak{L}_{\mathfrak{h}_1,\mathfrak{h}_2}:=\e^{\frac{\pi i}{4}\eta(\mathfrak{h}_1,\mathfrak{h}_2)}\mathfrak{L}_{\mathfrak{h}_1,\mathfrak{h}_2}^{(0)},
\end{align*}
where the $\eta$-invariant is computed in the symplectic vector space $\mathfrak{g}/\mathsf{stab}(\xi)$. Then, up to a positive scalar,  $\mathfrak{L}_{\mathfrak{h}_1,\mathfrak{h}_2}$ defines a unitary equivalence $\mathcal{H}_{F_2,\xi}(\mathcal{O}_\xi)\xrightarrow{\sim} \mathcal{H}_{F_1,\xi}(\mathcal{O}_\xi)$ such that if $\mathfrak{h}_3$ is a third real algebraic polarization at $\xi$, it holds that 
$$\mathfrak{L}_{\mathfrak{h}_3,\mathfrak{h}_2}\mathfrak{L}_{\mathfrak{h}_2,\mathfrak{h}_1}\mathfrak{L}_{\mathfrak{h}_1,\mathfrak{h}_3}=1.$$
\end{proposition}

It follows from Proposition \ref{unitarywowowod} and Lemma \ref{lionsintertwiners} that the map 
\begin{equation}
\label{kricorgeom}
\mathfrak{g}^*/\mathsf{Ad}^*\to \hat{G},\quad \mathcal{O}\mapsto [\mathcal{H}_{F,\xi}(\mathcal{O})],
\end{equation}
for some choice of $\xi\in \mathcal{O}$ and a polarization $F\to \mathcal{O}$ for $\xi$, is well defined and coincides with the Kirillov correspondence as constructed in  \cite{browndual,kirillovbook}.

\begin{theorem} 
The mapping in Equation \eqref{kricorgeom} defines a homeomorphism $\mathfrak{g}^*/\mathsf{Ad}^*\cong \widehat{\mathsf{G}}$ that coincides with the Kirillov orbit method described in \cite{browndual,kirillovbook}.
\end{theorem}

\section{The fine stratification of the spectrum}
\label{subsecfinestrat}
The spectrum $\widehat{\mathsf{G}}\cong \mathfrak{g}^*/\mathrm{Ad}^*$ is in general a non-Hausdorff space. However, a choice of Jordan-Hölder basis produces a stratification of $\widehat{\mathsf{G}}$ into locally closed Hausdorff subspaces. We shall recall the construction, due to Pedersen \cite{pedersen89}, of this so called fine stratification and construct a Hilbert space bundle of representations on each strata. The stratification is described in further detail in \cite{pedersen88, pedersen89}, and is complemented by an efficient summary in \cite{LipmanRosenberg}. The $C^*$-algebraic ramifications were studied in \cite{Beltita_Ludwig_Fourier} where the Hilbert space bundles were implicitly constructed and showed to give a filtration of $C^*(\mathsf{G})$ by means of a sequence of ideals whose subsequent quotients are stably isomorphic to continuous functions on the strata.

Let $n = \dim \mathfrak{g}$. We say that a basis $\{X_1, \ldots, X_n \}$ of $\mathfrak{g}$ is Jordan-Hölder if the subspaces
\[ \mathfrak{g}_j = \text{span} \{X_1, \ldots, X_j \}, \ j = 1,\ldots, n \]
are ideals in $\mathfrak{g}$. The associated sequence of ideals
$$0=\mathfrak{g}_0\subseteq \mathfrak{g}_1\subseteq \mathfrak{g}_2\subseteq \cdots\subseteq \mathfrak{g}_{n-1}\subseteq \mathfrak{g}_n=\mathfrak{g},$$
is called a Jordan-Hölder flag. The Jordan-Hölder flag determines the Jordan-Hölder basis up to a lower triangular change of basis. We denote the dual basis by $\{\xi_1, \ldots, \xi_n \}$. We tacitly assume that the center of $\mathfrak{g}$ coincides with $\mathfrak{g}_j$ for some $j$ (namely the dimension of the center).

Given $\xi \in \mathfrak{g}^*$ and $k \in \{1,\ldots,n\}$, define the set of jump indices
\[ J_\xi^k = \{ j \in \{1, \ldots, k \} \text{ } \vert \text{ } \mathfrak{g}_j \not \subset \mathsf{stab}(\xi \vert_{\mathfrak{g}_k}) + \mathfrak{g}_{j-1} \}. \]
Let $\mathcal{E}_n:=\{(\epsilon_1,\ldots, \epsilon_n): \epsilon_k\subseteq \{1,\ldots, k\}\}$. Given $\epsilon \in\mathcal{E}_n$, we define
\[ \Xi_\epsilon := \{ \xi \in \mathfrak{g}^* \ | \ (J_\xi^1, \ldots, J_\xi^n) = \epsilon \}. \]
We set $\mathcal{E}_n(\mathsf{G}):=\{\epsilon\in \mathcal{E}_n: \Xi_\epsilon\neq \emptyset\}$. The set $\Xi_\epsilon$ is locally closed for any $\epsilon\in\mathcal{E}_n(G)$. We can in fact order $(\Xi_\epsilon)_{\epsilon\in \mathcal{E}_n(\mathsf{G})}$ as $\Xi_1,\Xi_2,\ldots$ so that $\Xi_1\subseteq \mathfrak{g}^*$ is Zariski open and $\Xi_j$ is Zariski open in $\mathfrak{g}^*\setminus \cup_{i<j}\Xi_i$. More specifically, there is an explicit ordering on $\mathcal{E}_n$ constructed in \cite{pedersen94} corresponding to this ordering of strata. We can decompose $\mathfrak{g}^*$ as a disjoint union 
$$\mathfrak{g}^*=\cup_{\epsilon\in \mathcal{E}_n}\Xi_\epsilon.$$

The jump indices are $\mathsf{Ad}^*$-invariant, i.e.
\[ J_\xi^k = J_{\mathsf{Ad}^*_g(\xi)}^k, \ \xi \in \mathfrak{g}^*, \ g \in \mathsf{G}, \]
so the coadjoint action restricts to an action of $\mathsf{G}$ on $\Xi_\epsilon$. We define 
\[ \Gamma_\epsilon :=\Xi_\epsilon/\mathsf{G} \subseteq  \mathfrak{g}^* / \mathsf{Ad}^*.\]
Due to $\mathsf{Ad}^*$-invariance of the jump indices we can define the jump index of an orbit $\mathcal{O} \in \mathfrak{g}^* / \mathsf{Ad}^*$ by setting $J_\mathcal{O}^k := J_\xi^k$ for an arbitrary $\xi \in \mathcal{O}$. Note that
\[ \Gamma_\epsilon = \{ \mathcal{O} \in \mathfrak{g}^* / \mathsf{Ad}^* \ | \ (J_\mathcal{O}^1, \ldots, J_\mathcal{O}^n) = \epsilon \}. \]
The set $\Gamma_\epsilon$ is a Hausdorff and locally closed subset of $\widehat{\mathsf{G}}=\mathfrak{g}^*/\mathsf{Ad}^*$. We can decompose $\widehat{\mathsf{G}}$ as a disjoint union
$$\widehat{\mathsf{G}}=\cup_{\epsilon\in \mathcal{E}_n}\Gamma_\epsilon.$$
The family $\{ \Gamma_\epsilon \}_{\epsilon \in \mathcal{E}_n(\mathsf{G})}$ is called fine stratification of $\widehat{\mathsf{G}}$. We remark that the constructions above depends on the choice of Jordan-Hölder flag, but not the Jordan-Hölder basis representing the flag.

\begin{proposition}
\label{gammasplitting}
We consider an $\epsilon=(\epsilon_1,\ldots,\epsilon_n)\in \mathcal{E}_n(\mathsf{G})$, write $\epsilon_n=\{j_1<j_2<\cdots<j_{2d}\}$ and define 
$$\Lambda_\epsilon:=\{\xi\in \Xi_\epsilon: \xi_{j_1}=\xi_{j_2}=\cdots=\xi_{j_{2d}}=0\}.$$
Then for any $\xi\in \Xi_\epsilon$, the set $\mathcal{O}_\xi\cap \Lambda_\epsilon$ contains exactly one point, call it $p_\epsilon(\xi)$. The map $p_\epsilon:\Xi_\epsilon\to \Lambda_\epsilon$ is in fact a trivializable fibre bundle with fibre $\R^{2d}$, where the trivialization respects the fibrewise symplectic form, that fits into a commuting diagram:
\[
\begin{tikzcd}
\Xi_\epsilon \arrow[d, "p_\epsilon"] \arrow[rd] &  \\
\Lambda_\epsilon \arrow[r] &\Gamma_\epsilon,
\end{tikzcd}
\]
where the diagonal map is the projection map and the horizontal map is the composition of the projection map and the inclusion $\Lambda_\epsilon\subseteq \Xi_\epsilon$. Moreover, the map $\Lambda_\epsilon \to\Gamma_\epsilon$ is a homeomorphism.
\end{proposition}

We refer the reader to \cite{pedersen88, pedersen89} for a proof of this statement. We will denote the trivialization of 
$p_\epsilon:\Xi_\epsilon\to \Lambda_\epsilon$  by $\Psi_\epsilon :\Lambda_\epsilon \times \mathbb{R}^{2d} \rightarrow \Xi_\epsilon$. By the previous proposition, we can assume that for every $\xi \in \Lambda_\epsilon$, $\Psi_\xi = \Psi_\epsilon(\xi, \cdot)$ is a symplectomorphism from $\mathbb{R}^{2d}$ with the standard symplectic structure onto the coadjoint orbit $\mathcal{O}_\xi$. The results of \cite{pedersen88, pedersen89} show that on $\Xi_\epsilon$ there are global ``canonical coordinates'' $(q_1^\epsilon,\ldots, q_d^\epsilon,p_1^\epsilon,\ldots, p_d^\epsilon)\subseteq C^\infty(\Xi_\epsilon)$ (implementing $\Psi_\epsilon$) that are rational functions in the $\Lambda_\epsilon$-direction and polynomial in the $\mathbb{R}^{2d}$-direction.

Let us now construct a bundle of representations over each strata in $\widehat{\mathsf{G}}$. We fix a Jordan-Hölder basis as above. Recall the notations $\mathfrak{g}_j = \text{span} \{X_1, \ldots, X_j \}$. The Vergne polarization (with respect to the Jordan-Hölder basis) of $\xi\in \mathfrak{g}^*$ is defined by 
\begin{equation}
\label{vergndendoded}
\mathfrak{h}_V(\xi):=\sum_{j=1}^n \mathsf{stab}(\xi|_{\mathfrak{g}_j}).
\end{equation}
This real algebraic polarization was first used by Vergne, see \cite{bernatconzevergne72,vergneCR70, vergneBullMathFra72}. Since each $\mathfrak{g}_j$ is an ideal in $\mathfrak{g}$, it is clear that $\mathfrak{h}_V(\xi)\subseteq \mathfrak{g}$ is a subalgebra which is isotropic with respect to $\omega_\xi$. We refer the proof of the fact that the codimension of $\mathfrak{h}_V(\xi)$ is $\#J^n_\xi/2=\dim(\mathcal{O}_\xi)/2$ to \cite[Proposition 1.1.2, Chapter IV]{bernatconzevergne72}. This proves that $\mathfrak{h}_V(\xi)$ is a real algebraic polarization of $\xi$. For $\xi\in \Xi_\epsilon$, define the closed subgroup 
$$\mathsf{H}_V(\xi):=\mathrm{exp}(\mathfrak{h}_V(\xi)).$$

\begin{proposition}
\label{vbfv}
Let $X_1,\ldots, X_n$ be a Jordan-Hölder basis, take an $\epsilon\in \mathcal{E}_n(\mathsf{G})$ and define 
$$F_{V,\epsilon}:=\{(\xi,X)\in \Xi_\epsilon\times \mathfrak{g}: \; X\in\mathfrak{h}_V(\xi)\}.$$
Then the projection mapping $p_1:F_{V,\epsilon}\to \Xi_\epsilon$ defines a real stably trivializable vector bundle of rank $\#\epsilon_n/2$. Moreover, for any orbit $\mathcal{O}\in \Gamma_\epsilon$ and $\xi\in \mathcal{O}$, letting $q_\xi:\mathsf{G}\to \mathcal{O}$ denote the map $g\mapsto g\xi$,  it holds that 
$$F_{V,\epsilon}(\mathcal{O},\xi):=(q_\xi)_*p_1^{-1}(\mathcal{O})\subseteq T\mathcal{O}$$ 
is a polarization of $\mathcal{O}$. 
\end{proposition}

\begin{proof}
It follows from \cite[Theorem 5.1.1]{pedersen89} that $F_{V,\epsilon}$ is a vector bundle. For $\xi\in \mathfrak{g}^*$, we set 
$$K_V(\xi):=\{j\in \{1,\ldots,m\}: \mathfrak{g}_j\not \subset \mathfrak{h}_V(\xi) + \mathfrak{g}_{j-1}\}.$$
By the argument in \cite[Subsection 1.4]{pedersen89}, the set $(X_j)_{j\in K_V(\xi)}$ forms a basis (even of subalgebra type) for $\mathfrak{g}/\mathfrak{h}_V(\xi)$. By the arguments in \cite[Subsection 4.2]{pedersen89}, we have for $\xi\in \Xi_\epsilon$ that 
$$K_V(\xi)=K_\epsilon:=\{j\in \{1,\ldots, n\}: \epsilon_{j-1}\not \subset\epsilon_j\}.$$
We conclude that $K_V(\xi)=K_\epsilon$ is independent of $\xi\in \Xi_\epsilon$. In particular, $(X_j)_{j\in K_\epsilon}$ viewed as constant sections spans a trivial subbundle of $\Xi_\epsilon\times \mathfrak{g}$ which is complementary to $F_{V,\epsilon}$ in all fibres. Since $F_{V,\epsilon}$ is a complement of a trivializable subbundle in a trivial bundle,  $F_{V,\epsilon}$ is a stably trivializable vector bundle.
\end{proof}

\begin{proposition}
\label{rhoepsilonconst}
Let $\epsilon\in \mathcal{E}_n(\mathsf{G})$. There is a smooth map 
$$\rho_\epsilon:\Xi_\epsilon\times \mathsf{G}\to \Xi_\epsilon\times \mathsf{G},$$
commuting with the projections onto $\Xi_\epsilon$ and such that for each $\xi\in \Xi_\epsilon$
\begin{enumerate}
\item $\rho_\epsilon(\xi,\cdot)|_{\mathsf{H}_V(\xi)}=\mathrm{id}_{\mathsf{H}_V(\xi)}$.
\item $\rho_\epsilon(\xi,\cdot):\mathsf{G}\to \mathsf{H}_V(\xi)$.
\item $\rho_\epsilon(\xi,gh)=\rho_\epsilon(\xi,g)h$ for all $h\in \mathsf{H}_V(\xi)$.
\end{enumerate}
\end{proposition}

\begin{proof}
Let $P_\epsilon:\Xi_\epsilon\times \mathfrak{g}\to F_{V,\epsilon}$ denote the vector bundle projection along the complement constructed in the proof of Proposition \ref{vbfv}. By a slight abuse of notation, we identify the projection map $P_\epsilon$ with the associated smooth idempotent $P_\epsilon\in C^\infty(\Xi_\epsilon,\End(\mathfrak{g}))$ such that $P_\epsilon(\xi)\mathfrak{g}=h_V(\xi)$ for all $\xi$. Define $\mathrm{exp}_\epsilon:\Xi_\epsilon\times \mathfrak{g}\to \Xi_\epsilon\times \mathsf{G}$ by 
$$\mathrm{exp}_\epsilon(\xi,X):=(\xi,\mathrm{exp}((1-P_\epsilon(\xi))X)\mathrm{exp}(P_\epsilon(\xi)X)).$$
The fact that the exponential is a diffeomorphism implies that $\mathrm{exp}_\epsilon$ is a diffeomorphism. We denote its inverse by $\log_\epsilon:\Xi_\epsilon\times \mathsf{G}\to \Xi_\epsilon\times \mathfrak{g}$.

We define the map $\rho_\epsilon$ by the property that 
$$\log_\epsilon\circ \rho_\epsilon \circ \mathrm{exp}_\epsilon=P_\epsilon,$$
as mappings $\Xi_\epsilon\times \mathfrak{g}\to \Xi_\epsilon\times \mathfrak{g}$. It is readily verified that this mapping has the required properties. 
\end{proof}

\begin{proposition}
\label{bundleofhspspac}
Let $X_1,\ldots, X_n$ be a Jordan-Hölder basis, take an $\epsilon\in \mathcal{E}_n(\mathsf{G})$ and define 
$$\tilde{\mathcal{H}}_{V,\epsilon}:=\{(\xi,f): \xi\in \Xi_\epsilon, \; f\in\mathcal{H}_{F_{V,\epsilon}(\mathcal{O},\xi),\xi}(\mathcal{O}_\xi)\}.$$
Then there is a unique topology on $\tilde{\mathcal{H}}_{V,\epsilon}$ such that 
\begin{enumerate}
\item The projection mapping $p_1:\tilde{\mathcal{H}}_{V,\epsilon}\to \Xi_\epsilon$ defines a trivializable bundle of Hilbert spaces with fibre 
$$p_1^{-1}(\xi)=\mathcal{H}_{F_{V,\epsilon}(\mathcal{O},\xi),\xi}(\mathcal{O}_\xi)\cong L^2(\R^{d}),$$ 
where $d=\#\epsilon_n/2$. 
\item For all compact $K\subseteq \Xi_\epsilon$, the space of sections $C(K,\tilde{\mathcal{H}}_{V,\epsilon})$ densely contains the space of sections of the form $\xi\mapsto (\xi,V_\xi f)$ for $f\in C^\infty(K\times \mathsf{G})$ such that $f(\xi,gh)=\chi_{\mathfrak{h}_V(\xi),\xi}(h)f(\xi, g)$ for all $\xi\in K$, $g\in \mathsf{G}$, $h\in \mathsf{H}_V(\xi)$, and that $f$ is of Schwarz class modulo $\mathsf{H}_V$.
\item The fibrewise $\mathsf{G}$-action on $\mathcal{H}_{F_{V,\epsilon}(\mathcal{O},\xi),\xi}(\mathcal{O}_\xi)$ extends to a strongly continuous, $C(\Xi_\epsilon)$-linear $\mathsf{G}$-action on the space of continuous section $C(\Xi_\epsilon,\tilde{\mathcal{H}}_{V,\epsilon})$. 
\end{enumerate}
\end{proposition}

\begin{proof}
By \cite[Theorem 5.1.1]{pedersen89} there exists rational functions $q_{1}^{\epsilon},\ldots q_{d}^{\epsilon}$ on $\mathfrak{g^*}$ whose restrictions to $\Xi_\epsilon$ are smooth such that the mappings
$$q_{j,\epsilon}:\Xi_\epsilon\times \mathsf{G}\to \C, \quad q_{j,\epsilon}(\xi,g)=q_j^\epsilon(g\xi),$$
are smooth and satisfy that $q_{j,\epsilon}(\xi,\cdot)\in C^\infty(\mathsf{G})$ are $\mathsf{H}_V(\xi)$-invariant. For a multi-index $\alpha\in \N^d$, we let $h_\alpha$ denote the associated Hermite function on $\R^d$. Define the smooth functions $\tilde{q}_{\alpha,\epsilon}:\Xi_\epsilon\times \mathsf{G}\to \C$ by  
$${q}_{\alpha,\epsilon}(\xi,g)=h_\alpha(q_{1,\epsilon}(\xi,g),\ldots, q_{d,\epsilon}(\xi,g))\chi_{\mathfrak{h}_V(\xi),\xi}(\rho_\epsilon(g)),$$
where $\rho_\epsilon$ is the function constructed in Proposition \ref{rhoepsilonconst}. 

By Proposition \ref{rhoepsilonconst}, we have that 
$$\chi_{\mathfrak{h}_V(\xi),\xi}(\rho_\epsilon(gh))=\chi_{\mathfrak{h}_V(\xi),\xi}(\rho_\epsilon(g))\chi_{\mathfrak{h}_V(\xi),\xi}(h),$$
for all $\xi\in \Xi_\epsilon$, $g\in \mathsf{G}$, $h\in \mathsf{H}_V(\xi)$. It now follows from the properties of $q_{1}^{\epsilon},\ldots q_{d}^{\epsilon}$ that for each compact $K\subseteq \Xi_\epsilon$, the function $q_{\alpha,\epsilon}\in C^\infty(K\times \mathsf{G})$ satisfies that 
$${q}_{\alpha,\epsilon}(\xi,gh)=\chi_{\mathfrak{h}_V(\xi),\xi}(h){q}_{\alpha,\epsilon}(\xi,g),$$ 
for all $\xi\in K$, $g\in \mathsf{G}$, $h\in \mathsf{H}_V(\xi)$, and that $\tilde{q}_{\alpha,\epsilon}$ is of Schwarz class modulo $\mathsf{H}_V$. It follows from the construction, and that $q_{1}^{\epsilon},\ldots q_{d}^{\epsilon}$ define canonical coordinates on each orbit, that for each $\xi$, the sequence $(V_\xi\tilde{q}_{\alpha,\epsilon}(\xi,\cdot))_{\alpha\in \N^d}$ is an ON-basis for $\mathcal{H}_{F_{V,\epsilon}(\mathcal{O},\xi),\xi}(\mathcal{O}_\xi)$.

We define a topology on $\tilde{\mathcal{H}}_{V,\epsilon}$ by for each compact $K\subseteq \Xi_\epsilon$ declaring a section $f:K\to \tilde{\mathcal{H}}_{V,\epsilon}$ to be continuous if and only if the function 
$$\tilde{f}:K\to L^2(\R^d),\quad \tilde{f}(\xi)=\sum_{\alpha\in \N^d}\left\langle V_\xi{q}_{\alpha,\epsilon}(\xi,\cdot), f(\xi)\right\rangle h_\alpha,$$
is continuous. By construction, Item 1 and 2 of the theorem holds. We note that the map $f\mapsto \tilde{f}$ defines a global trivialization $T_\epsilon:\tilde{\mathcal{H}}_{V,\epsilon}\xrightarrow{\sim} \Xi_\epsilon \times L^2(\R^d)$. Item 3 follows from the next lemma. 
\end{proof}

We let $C^\infty(\Xi_\epsilon,\mathcal{U}(L^2(\R^d)))$ denote the topological group of mappings $\Xi_\epsilon \to \mathcal{U}(L^2(\R^d))$ that are smooth for the strong operator topology.

\begin{lemma}
\label{descofreprefield}
Let $X_1,\ldots, X_n$ be a Jordan-Hölder basis, take an $\epsilon\in \mathcal{E}_n(\mathsf{G})$ and consider the trivialization 
$$T_\epsilon:\tilde{\mathcal{H}}_{V,\epsilon}\to \Xi_\epsilon\times L^2(\R^d),$$
constructed in the proof of Proposition \ref{bundleofhspspac}, where $d=\#\epsilon_n/2$. Let $\pi_{\epsilon,0}=(\pi_\xi)_{\xi\in \Xi_\epsilon}$ denote the family of unitary representations of $\mathsf{G}$ on $L^2(\R^d)$ induced from the fibrewise $\mathsf{G}$-action on $\tilde{\mathcal{H}}_{V,\epsilon}$ and $T_\epsilon$. 

There exists smooth functions $a_{\epsilon,j}\in C^\infty(\Xi_\epsilon\times \R^d\times \mathfrak{g})$, for $j=0,\ldots, d$, that are rational in $\xi\in \Xi_\epsilon$, polynomial in $t\in \R^d$ and linear in $X\in \mathfrak{g}$, such that ,
$$(D\pi_\xi(X)f)(t)=ia_{\epsilon,0}(\xi,t,X)f(t)+\sum_{j=1}^da_{\epsilon,j}(\xi,t,X)\frac{\partial f}{\partial t_j}(t),$$
for $f\in \mathcal{S}(\R^d)$, $\xi\in \Xi_\epsilon$ and $X\in \mathfrak{g}$. Moreover, there exists $h_\epsilon\in C^\infty(\Xi_\epsilon\times \R^d\times \mathsf{G})$ and $k_\epsilon\in C^\infty(\Xi_\epsilon\times \R^d\times \mathsf{G},\R^d)$ that are rational in $\xi\in \Xi_\epsilon$, polynomial in $t\in \R^d$ and polynomial in $g\in \mathsf{G}$ such that  
$$(\pi_\xi(g) f)(t)=\mathrm{e}^{ih_\epsilon(\xi,t,g)}f(t+k_\epsilon(\xi,t,g)).$$

In particular, $\pi_{\epsilon,0}$ defines a continuous homomorphism 
$$\pi_{\epsilon,0}:\mathsf{G}\to C^\infty(\Xi_k,\mathcal{U}(L^2(\R^d))).$$
\end{lemma}

We omit the proof of this lemma and refer the reader to \cite[Lemma 2.8 and Proposition 2.11]{LipmanRosenberg}. See also \cite[Theorem 2.7.2]{pedersen89}.

To place this construction in a more $C^*$-algebraic context, we will define a $C_0(\Xi_\epsilon)$-Hilbert $C^*$-module that will capture the bundle $\tilde{\mathcal{H}}_{V,\epsilon}$. 
\label{thehilbmdmoslsl}
We define the space $\tilde{\mathpzc{H}}^\infty_{V,\epsilon,c}$ to consist of all $f\in C^\infty(\Xi_\epsilon\times \mathsf{G})$ such that $ f(\xi,g)=\chi_{\mathfrak{h}_V(\xi),\xi}(h)f(g)$ for all $\xi\in \Xi_\epsilon$, $g\in \mathsf{G}$, $h\in \mathsf{H}_V(\xi)$, and that $f$  is compactly supported in the $\Xi_\epsilon$-direction and of Schwarz class modulo $\mathsf{H}_V$ in the $\mathsf{G}$-direction. On $\tilde{\mathpzc{H}}^\infty_{V,\epsilon,c}$, we define a $C_0(\Xi_\epsilon)$-valued inner product as 
$$\langle f_1,f_2\rangle(\xi):=\int_{\mathsf{G}/\mathsf{H}_V(\xi)} \overline{f_1(\xi,g)}f_2(\xi,g)\rd \mu_\xi, \quad\mbox{for $f_1,f_2\in\tilde{\mathpzc{H}}^\infty_{V,\epsilon,c}$}, $$
where $\mu_\xi$ is the Lebesgue measure on $\mathsf{G}/\mathsf{H}_V(\xi)$ defined from exponentiating the Lebesgue measure on $\mathfrak{g}/\mathfrak{h}_V(\xi)$ defined from the Jordan-Hölder basis. In fact, $\langle f_1,f_2\rangle\in C^\infty_c(\Xi_\epsilon)$ for $f_1,f_2\in\mathpzc{H}^\infty_{V,\epsilon,c}$. We let $\tilde{\mathpzc{H}}_{V,\epsilon}$ denote the completion of $\tilde{\mathpzc{H}}^\infty_{V,\epsilon,c}$ as a $C_0(\Xi_\epsilon)$-Hilbert module.

We note that any $f\in \tilde{\mathpzc{H}}^\infty_{V,\epsilon,c}$ canonically defines a smooth compactly supported section of $\tilde{\mathcal{H}}_{V,\epsilon}$ via the map $V_\xi$ from Equation \eqref{thevximap}. We denote the associated injection by 
\begin{equation}
\label{vtildedef}
\tilde{V}_\epsilon:\tilde{\mathpzc{H}}^\infty_{V,\epsilon,c}\to C^\infty_c(\Xi_\epsilon, \tilde{\mathcal{H}}_{V,\epsilon}).
\end{equation}

\begin{proposition}
\label{modulevsbundle}
Let $\epsilon\in \mathcal{E}_n(\mathsf{G})$. The map $\tilde{V}_\epsilon$ extends to a unitary isomorphism of $C_0(\Xi_\epsilon)$-Hilbert $C^*$-modules 
$$\tilde{V}_\epsilon:\tilde{\mathpzc{H}}_{V,\epsilon}\to C_0(\Xi_\epsilon, \tilde{\mathcal{H}}_{V,\epsilon}).$$
\end{proposition}

\begin{proof}
Recall the construction of $(q_{\alpha,\epsilon})_{\alpha\in \N^d}$ and the global trivialization $T_\epsilon$ of $\tilde{\mathcal{H}}_{V,\epsilon}$ from the proof of Proposition \ref{bundleofhspspac}. The functions $q_{\alpha,\epsilon}$ defines elements $\langle q_{\alpha,\epsilon}| \in \mathcal{L}_{C_0(\Xi_\epsilon)}(\tilde{\mathpzc{H}}_{V,\epsilon},C_0(\Xi_\epsilon))$. We set $| q_{\alpha,\epsilon}\rangle:= \langle q_{\alpha,\epsilon}|^*\in \mathcal{L}_{C_0(\Xi_\epsilon)}(C_0(\Xi_\epsilon), \tilde{\mathpzc{H}}_{V,\epsilon})$ and note that 
$$\langle q_{\alpha,\epsilon}| q_{\beta,\epsilon}\rangle=\delta_{\alpha,\beta}\in C_b(\Xi_\epsilon)=\mathcal{M}(C_0(\Xi_\epsilon)).$$ 
We can write the identity of  $\tilde{\mathpzc{H}}_{V,\epsilon}$ as the strictly convergent sum
\begin{equation}
\label{decompuais}
\mathrm{id}_{\tilde{\mathpzc{H}}_{V,\epsilon}}=\sum_{\alpha\in \N^d} | q_{\alpha,\epsilon}\rangle \langle q_{\alpha,\epsilon}|.
\end{equation}
It is readily verified that the composition $T_\epsilon \tilde{V}_\epsilon$ takes the form 
$$T_\epsilon \tilde{V}_\epsilon f=\sum_{\alpha\in \N^d}\left\langle q_{\alpha,\epsilon}, f\right\rangle h_\alpha,$$
for $f\in \tilde{\mathpzc{H}}^\infty_{V,\epsilon,c}$. Using Equation \eqref{decompuais}, we conclude that $T_\epsilon \tilde{V}_\epsilon:\tilde{\mathpzc{H}}_{V,\epsilon}\to C_0(\Xi_\epsilon,L^2(\R^d))$ is a unitary isomorphism of $C_0(\Xi_\epsilon)$-Hilbert $C^*$-modules. The proposition now follows from the fact that $T_\epsilon$ is a global (unitary) trivialization. 
\end{proof}

\begin{proposition}
\label{reponen}
Let $\epsilon\in \mathcal{E}_n(\mathsf{G})$. Left translation of $\mathsf{G}$ on functions in $\tilde{\mathpzc{H}}^\infty_{V,\epsilon,c}$ defines an action of $\mathsf{G}$ on $\tilde{\mathpzc{H}}_{V,\epsilon}$ as unitary, adjointable $C_0(\Xi_\epsilon)$-linear mappings. This action integrates to a $*$-homomorphism
$$\tilde{\pi}_{V,\epsilon}:C^*(\mathsf{G})\to \mathrm{End}^*_{C_0(\Xi_\epsilon)}(\tilde{\mathpzc{H}}_{V,\epsilon}),$$
and the unitary isomorphism $\tilde{V}_\epsilon:\tilde{\mathpzc{H}}_{V,\epsilon}\to C_0(\Xi_\epsilon, \tilde{\mathcal{H}}_{V,\epsilon})$ from Proposition \ref{modulevsbundle} is $\mathsf{G}$-equivariant. The unitary $\tilde{V}_\epsilon$ in particular induces a $\mathsf{G}$-equivariant $*$-isomorphism 
$$\mathbb{K}_{C_0(\Xi_\epsilon)}(\tilde{\mathpzc{H}}_{V,\epsilon})\cong C_0(\Xi_\epsilon, \mathbb{K}(\tilde{\mathcal{H}}_{V,\epsilon})).$$
\end{proposition}

\begin{proof}
By construction, $\tilde{V}_\epsilon$ is $\mathsf{G}$-equivariant as a mapping $\tilde{\mathpzc{H}}^\infty_{V,\epsilon,c}\to C^\infty_c(\Xi_\epsilon, \tilde{\mathcal{H}}_{V,\epsilon})$. The proposition follows from density.
\end{proof}

The constructions above take place over $\Xi_\epsilon$, but it can readily be seen from Kirillov's orbit method that the unitary equivalence class of the $\mathsf{G}$-representation on each fibre remains constant along each orbit. We now restrict our constructions to bundles and modules over the stratas $\Gamma_\epsilon$. As explained in Proposition \ref{gammasplitting}, the Jordan-Hölder basis defines a homeomorphism $\Gamma_\epsilon\cong \Lambda_\epsilon$ and we can thusly consider $\Gamma_\epsilon$ as a subset of $\Xi_\epsilon$. We define the locally trivial bundle of Hilbert spaces
$$\mathcal{H}_{V,\epsilon}:=\tilde{\mathcal{H}}_{V,\epsilon}|_{\Gamma_\epsilon}\to \Gamma_\epsilon,$$
and the $C_0(\Gamma_\epsilon)$-Hilbert $C^*$-module 
$$\mathpzc{H}_{V,\epsilon}:=\tilde{\mathpzc{H}}_{V,\epsilon}\otimes_{C_0(\Xi_\epsilon)}C_0(\Gamma_\epsilon).$$
We note that Proposition \ref{modulevsbundle} implies that $\tilde{V}_\epsilon$ from Equation \eqref{vtildedef} induces a $\mathsf{G}$-equivariant unitary isomorphism of $C_0(\Gamma_\epsilon)$-Hilbert $C^*$-modules
\begin{equation}
\label{modudldlde}
V_\epsilon:\mathpzc{H}_{V,\epsilon}\to C_0(\Gamma_\epsilon, \mathcal{H}_{V,\epsilon}).
\end{equation}
The unitary $V_\epsilon$ in particular induces a $\mathsf{G}$-equivariant $*$-isomorphism 
$$\mathbb{K}_{C_0(\Gamma_\epsilon)}(\mathpzc{H}_{V,\epsilon})\cong C_0(\Gamma_\epsilon, \mathbb{K}(\mathcal{H}_{V,\epsilon})).$$

As above, we order the strata $\Gamma_1, \Gamma_2,\ldots$ in the fine stratification as above. We use the same notation for $\Xi_1,\Xi_2,\ldots $, $\mathcal{H}_{V,1},\mathcal{H}_{V,2},\ldots$ and $\mathpzc{H}_{V,1},\mathpzc{H}_{V,2},\ldots$. We define the ideals $J_k\subseteq C^*(\mathsf{G})$ as
$$J_k:=\{a\in C^*(\mathsf{G}): \pi(a)=0\forall \pi\in \widehat{\mathsf{G}}\setminus \cup_{j\leq k}\Gamma_j\},$$
where we identify each $\Gamma_j$ with a subset of representations under the Kirillov correspondence. We note that these ideals fit into a sequence of ideal inclusions
\begin{equation}
\label{ideaslsls}
0\vartriangleleft J_1\vartriangleleft J_2 \vartriangleleft\cdots \vartriangleleft C^*(\mathsf{G}).
\end{equation}
We also define the $C^*$-algebras $I_k:=J_k/J_{k-1}$.

\begin{proposition}
\label{spectrumofthejs}
The spectrum of the ideals $J_k\vartriangleleft C^*(\mathsf{G})$ are given by 
$$\widehat{J_k}=\cup_{j\leq k}\Gamma_j\subseteq \hat{\mathsf{G}},$$
and the map $\widehat{J_k} \to \widehat{J_{k+1}}$ induced from the ideal inclusion $J_k\to J_{k+1}$ is given by the inclusion 
$$\cup_{j\leq k}\Gamma_j\to \cup_{j\leq k+1}\Gamma_j.$$
\end{proposition}

\begin{proof}
This follows from \cite[Proposition 3.2.1]{Dixmierbook}.
\end{proof}

\begin{proposition}
\label{repsandosisosos}
The $*$-homomorphism 
$$\pi_{V,k}:C^*(\mathsf{G})\to \mathrm{End}^*_{C_0(\Gamma_k)}(\mathpzc{H}_{V,k}),$$
induced from the $*$-homomorphism $\pi_{V,k}$ of Proposition \ref{reponen} satisfies the following:
\begin{enumerate}
\item $\ker\pi_{V,k}=J_{k-1}$
\item $\pi_{V,k}$ restricted to $J_k$ defines a surjection
$$\pi_{V,k}:J_k\to \mathbb{K}_{C_0(\Gamma_k)}(\mathpzc{H}_{V,k}).$$
\item $\pi_{V,k}$ induces an isomorphism 
$$I_k\cong \mathbb{K}_{C_0(\Gamma_k)}(\mathpzc{H}_{V,k}).$$
\end{enumerate}
In particular, any global trivialization $\mathcal{H}_{V,k}\cong \Gamma_k\times L^2(\R^d)$ together with the unitary from Equation \eqref{modudldlde} together with $\pi_{V,k}$ produces canonical $*$-homomorphisms fitting into the commuting diagram
\begin{equation}
\label{commutingikcstag}
\begin{tikzcd}
I_k \arrow[d] \arrow[r,"\subseteq" ] &  C^*(\mathsf{G})/J_{k-1} \arrow[d] \\\
C_0(\Gamma_k,\mathbb{K}(L^2(\R^d))) \arrow[r, "\subseteq"] &C_b(\Gamma_k,\mathbb{K}(L^2(\R^d))).
\end{tikzcd}
\end{equation}
where the left vertical map is a $*$-isomorphism. 
\end{proposition}

\begin{proof}
We note that item 2 follows from item 1 and 3. To prove item 1, we note that by construction, 
\begin{align*}
\ker\pi_{V,k}&=\{a\in C^*(\mathsf{G}): \pi(a)=0 \ \forall \pi\in \Gamma_k\}=\\
&=\{a\in C^*(\mathsf{G}): \pi(a)=0 \ \forall \pi\in \hat{\mathsf{G}}\setminus \cup_{j<k}\Gamma_j\}=J_{k-1}.
\end{align*}
In the middle step used that $\Gamma_k\subseteq \hat{\mathsf{G}}\setminus \cup_{j<k}\Gamma_j$ is dense (since $\Xi_k\subseteq \mathfrak{g}^*\setminus \cup_{j<k}\Xi_j$ is Zariski open).

Let us prove item 3. It follows from item 1 and the construction that $\pi_{V,k}$ induces a monomorphism $I_k\to \mathbb{K}_{C_0(\Gamma_k)}(\mathpzc{H}_{V,k})$. Proposition \ref{modulevsbundle} and Proposition \ref{bundleofhspspac} implies that there exists an isomorphism $\mathbb{K}_{C_0(\Gamma_k)}(\mathpzc{H}_{V,k})\cong C_0(\Gamma_k,\mathbb{K}(L^2(\R^d)))$, so the spectrum of $\mathbb{K}_{C_0(\Gamma_k)}(\mathpzc{H}_{V,k})$ is the locally compact Hausdorff space $\Gamma_k$. The spectrum of $I_k$ is $\Gamma_k$ by Proposition \ref{spectrumofthejs} and \cite[Proposition 3.2.1]{Dixmierbook}. The inclusion $I_k\to \mathbb{K}_{C_0(\Gamma_k)}(\mathpzc{H}_{V,k})$ induces the identity at the level of spectrum by construction, and item 3 follows. 
\end{proof}

\begin{remark}
\label{denseityemramr}
We note that since the top strata $\Gamma_1$ is dense, the map $C^*(\mathsf{G})\to C_b(\Gamma_1,\mathbb{K}(L^2(\R^d)))$ is a $*$-monomorphism. 
\end{remark}

The results we have discussed were stated in a slightly less explicit form in \cite{Beltita_Ludwig_Fourier}. In \cite{Beltita_Ludwig_Fourier} it was proven that for any $\epsilon \in \mathcal{E}_n(\mathsf{G})$ there is a Hilbert space $\mathcal{H}_\epsilon$ and for every $\xi \in \Gamma_\epsilon$ a representation $\pi_\xi : \mathsf{G} \rightarrow \mathbb{B}(\mathcal{H}_\epsilon)$ such that $\xi= [\pi_\xi]$ and the map $\xi \mapsto \pi_\xi(a)$ is norm continuous for every $a \in C^*(\mathsf{G})$. The following map is a noncommutative Fourier transform of an element $f \in C^*(\mathsf{G})$:
\[ \mathsf{F} : C^*(\mathsf{G}) \rightarrow l^\infty(\widehat{\mathsf{G}}, \mathbb{K}(\oplus_{\epsilon\in \mathcal{E}_n(\mathsf{G})}\mathcal{H}_\epsilon)), \ \mathsf{F}(f)(\xi) = \pi_\xi(f) \in \mathbb{K}(\mathcal{H}_\epsilon). \]
In \cite{Beltita_Ludwig_Fourier} it was shown that $C^*(\mathsf{G})$ can be identified with a certain subalgebra of $l^\infty(\widehat{\mathsf{G}}, \mathbb{K})$. In particular, they constructed the ideals from Equation \eqref{ideaslsls}, proved that the ideal inclusions came from the ordering on $\mathcal{E}_n(\mathsf{G})$ defined in \cite{pedersen94} and proved that there was an isomorphism
\[ I_\epsilon \cong C_0(\Gamma_\epsilon, \mathbb{K}(\mathcal{H}_\epsilon)). \]

\section{Flat orbits}
\label{subsecflatorbitssos}
We will be particularly interested in a special class of nilpotent groups, possessing so called flat orbits. The existence of flat orbits ensure that each orbit in $\Gamma_1$ is flat. The set of flat orbits, when non-empty, is in practice readily computed and poses a natural candidate for a Zariski open subset of $\hat{\mathsf{G}}$ that we later on will use in index theory. 

\begin{theorem}
\label{flatofofdodthm}
Let $\mathcal{O}$ be a coadjoint orbit of a simply connected nilpotent Lie group $G$ with Lie algebra $\mathfrak{g}$. Denote the center of the Lie algebra by $\mathfrak{z}$ and its annihilator in $\mathfrak{g}^*$ by $\mathfrak{z}^\bot$. The following are equivalent:
 \begin{enumerate}
\item There exists a $\xi \in \mathcal{O}$ such that $\mathcal{O} = \xi + \mathfrak{z}^\bot$.
\item The representation $\pi \in \widehat{\mathsf{G}}$ corresponding to $\mathcal{O}$ is square integrable modulo $\mathfrak{z}$.
\item There is a $\xi\in \mathcal{O}$ such that $\omega_\xi$ is non-degenerate on $\mathfrak{g}/\mathfrak{z}$.
\end{enumerate}
\end{theorem}

We refer the proof of this theorem to \cite{Moore_Wolf}, \cite{Corwin_Greenleaf}.

\begin{definition}
Let $\mathcal{O}$ be a coadjoint orbit of a simply connected nilpotent Lie group $\mathsf{G}$. If $\mathcal{O}$ satisfies any of the equivalent conditions of Theorem \ref{flatofofdodthm}, we say that $\mathcal{O}$ is a flat orbit. We write $\Gamma$ for the set of flat orbits.

If there exists a flat orbit, i.e. $\Gamma\neq \emptyset$, we say that $\mathsf{G}$ admits flat orbits. 
\end{definition}

We remark that the notion of a flat orbit in some places in the literature refers to the orbit being an affine space of the annihilator of the stabilizer Lie algebra. See for instance in \cite{Corwin_Greenleaf,LipmanRosenberg}. We exclusively use the term flat orbit for the case when the stabilizer Lie algebra is the center.

\begin{theorem}
Let $\mathsf{G}$ be a simply connected nilpotent Lie group. Then $\mathsf{G}$ admits flat orbits if and only if any of the following equivalent conditions holds:
\
\begin{description}
\item [Geometric] There exists $\xi \in \mathfrak{g}^*$ such that $\mathcal{O}_\xi$ is flat, i.e. $\mathcal{O}_\xi = \xi + \mathfrak{z}^\bot$.
\item [Analytic] There exists $\pi \in \widehat{\mathsf{G}}$ which is square integrable modulo $\mathfrak{z}$.
\item [Algebraic] The center of the universal enveloping algebra of $\mathfrak{g}$ is isomorphic to the universal enveloping algebra of $\mathfrak{z}$.
\end{description}
\end{theorem}

We refer the proof of this theorem to \cite{Moore_Wolf}, \cite{Corwin_Greenleaf}.

A flat orbit is therefore, as a subset of $\mathfrak{g}^*$, an affine space modelled on $\mathfrak{z}^\perp$. We remark that if $\mathcal{O}$ is a flat orbit, then $\dim(\mathcal{O})=\dim(\mathfrak{g})-\dim(\mathfrak{z})$. In particular, the center having even codimension is a necessary condition for admitting flat orbits, it is by Example \ref{uppertriangularexample} not sufficient.

\begin{proposition}
Let $\mathsf{G}$ be a simply connected nilpotent Lie group. Choose the Jordan-Hölder basis $\{X_1,\ldots, X_n)$ used in the fine stratification to satisfy that $\mathfrak{z}$ is spanned by $X_1,\ldots, X_{n-2d}$ where $2d=\mathrm{codim}(\mathfrak{z})$. Then $\mathsf{G}$ admits flat orbits if and only there exists an $\epsilon\in \mathcal{E}_n(\mathsf{G})$ with $\#\epsilon_n=2d$. In this case,
$$\Gamma=\bigcup_{\epsilon: \#\epsilon_n=2d}\Gamma_\epsilon.$$
\end{proposition}

\begin{proof}
Since $\mathfrak{g}_{n-2d}=\mathfrak{z}$, we have for any $\xi\in \mathfrak{g}^*$ and $k=1,\ldots, n$ that 
$$\mathfrak{z}\cap \mathfrak{g}_k=\mathfrak{g}_{\min(k,n-2d)}\subseteq \mathsf{stab}(\xi \vert_{\mathfrak{g}_k}).$$ 
Therefore $J^k_\xi\cap\{1,\ldots, n-2d\}=\emptyset$ for any $\xi$. The proposition is therefore equivalent to the statement that for $\xi\in \mathfrak{g}^*$, $\mathcal{O}_\xi$ is flat if and only if $J^n_\xi=\{n-2d+1,\ldots, n\}$. Clearly, for $\xi\in \mathfrak{g}^*$ it holds that $J^n_\xi=\{n-2d+1,\ldots, n\}$ if and only if $\mathsf{stab}(\xi)=\mathfrak{g}_{n-2d}=\mathfrak{z}$. The proposition follows from the fact the annihilator of $\omega_\xi$ is $\mathsf{stab}(\xi)$. 
\end{proof}

\begin{remark}
Whenever non-empty, the set $\Gamma$ is the top strata in the coarse stratification of $\hat{\mathsf{G}}$ (see more in \cite{Beltita_Ludwig_Fourier}). There are several examples where $\Gamma$ only consists of a single fine strata (see examples in Section \ref{subsectionexamples} below). A notable example where $\Gamma$ consists of multiple fine stratas is the complexification of the Heisenberg group, see Example \ref{complexheisenbergexample}.
\end{remark}

For any $\xi\in \mathfrak{g}^*$, the antisymmetric form $\omega_\xi$ descends to an antisymmetric form on $\mathfrak{g}/\mathfrak{z}$. It is readily verified that $\xi$ defines a flat orbit if and only if $\omega_\xi$ is a symplectic form on $\mathfrak{g}/\mathfrak{z}$. Since the center is an ideal, it is also readily verified that the Pfaffian 
$$\mathsf{Pf}(\xi) := \mathsf{Pf}(\omega_\xi|_{\mathfrak{g}/\mathfrak{z}}),\quad \xi\in \mathfrak{g}^*,$$ 
is constant along orbits. So $\mathsf{Pf}$ defines a function $\mathsf{Pf}:\mathfrak{g}^*/\mathsf{Ad}^*\to \R$. We conclude the following proposition.

\begin{proposition}
\label{descripfiogogda}
Let $\mathsf{G}$ be a simply connected nilpotent Lie group. Then $\mathsf{G}$ admits flat orbits if and only if  
$$\mathsf{Pf}(\mathcal{O})\neq 0,$$
for some orbit $\mathcal{O}\in \mathfrak{g}^*/\mathsf{Ad}^*$. If $\mathsf{G}$ admits flat orbits, the set of flat orbits $\Gamma\subseteq \hat{\mathsf{G}}$ can be described as 
$$\Gamma=\{\mathcal{O}\in\mathfrak{g}^*/\mathsf{Ad}^*:\mathsf{Pf}(\mathcal{O})\neq 0\}.$$
If we identify $\mathfrak{z}^*$ with a subspace of $\mathfrak{g}^*$ via the Jordan-Hölder basis, and $\mathsf{G}$ admits flat orbits we can identify $\Gamma$ with the Zariski open dense subset $\Lambda\subseteq \mathfrak{z}^*$ defined by
$$\Lambda:=\{\xi\in \mathfrak{z}^*: \mathsf{Pf}(\omega_\xi|_{\mathfrak{g}/\mathfrak{z}})\neq 0\}=\bigcup_{\epsilon: \#\epsilon_n=2d}\Lambda_\epsilon.$$
\end{proposition}

This proposition shows that the flat orbits form a Hausdorff subset of $\hat{\mathsf{G}}$ which is stratified by the stratas in the fine stratification for which $\#\epsilon_n=2d$.

\begin{remark}
\label{flatorbisintosform}
Let us describe the Pfaffian function in slightly more detail on $\mathfrak{z}^*$ in terms of the Jordan-Hölder basis $X_1,\ldots, X_n$ fixed throughout our constructions. We assumed that $X_1,\ldots, X_{\dim \mathfrak{z}}$ span $\mathfrak{z}$ and we will further assume that for some $p\leq \dim(\mathfrak{z})$, $X_1,\ldots, X_{p}$ span $\mathfrak{z}\cap [\mathfrak{g},\mathfrak{g}]$. Let $\xi_1,\ldots, \xi_n$ denote the dual basis for $\mathfrak{g}^*$. We obtain a collection of $2$-forms $\omega_1,\ldots, \omega_p:(\mathfrak{g}/\mathfrak{z})\wedge (\mathfrak{g}/\mathfrak{z})\to \R$ by setting $\omega_j(X,Y):=\xi_j[X,Y]=\omega_{\xi_j}[X,Y]$. For any $\xi\in \mathfrak{z}^*$, we have that 
$$\omega_\xi=\sum_{j=1}^p \xi(X_j)\omega_j.$$
We conclude that on $\mathfrak{z}^*$, the Pfaffian function is the polymial of degree $\mathrm{codim}(\mathfrak{z})/2$ given by 
$$\mathsf{Pf}(\xi)=\mathsf{Pf}\left(\sum_{j=1}^p \xi(X_j)\omega_j\right).$$
The polynomial $\mathsf{Pf}(\xi)$ depends only on the coordinates $\xi_1,\ldots, \xi_p$. We note that this polynomial is identically $0$ if $\mathrm{codim}(\mathfrak{z})$ is odd. Moreover, $G$ admits flat orbits if and only if this polynomial is non-vanishing on $\mathfrak{z}^*$. 
\end{remark}

\begin{proposition}
\label{vergenpolonflat}
Let $\mathsf{G}$ be an $n$-dimensional simply connected nilpotent Lie group admitting flat orbits and $\epsilon \in \mathcal{E}_n(\mathsf{G})$ such that $\#\epsilon_n=\mathrm{codim}(\mathfrak{z})$. Adopting the notations from Remark \ref{flatorbisintosform} filter $\mathfrak{g}/\mathfrak{z}$ by the Jordan-Hölder basis, and for $j=1,\ldots ,\mathrm{codim}(\mathfrak{z})$ we write $(\mathfrak{g}/\mathfrak{z})_j$ for the linear span of $X_{\dim(\mathfrak{z})+1},\ldots, X_{\dim(\mathfrak{z})+j}$. The restriction of the bundle of Vergne polarization from Proposition \ref{vbfv} to the strata $\Lambda_\epsilon$ in the transversal $\Lambda$ for the flat orbits, i.e. $F_V|_{\Lambda_\epsilon}\to \Lambda_\epsilon$, takes the form 
$$F_V|_{\Lambda_\epsilon}=\left\{(\xi,X)\in \Lambda_\epsilon\times \mathfrak{g}: \; X\in \mathfrak{z}\oplus \sum_{j=1}^{\mathrm{codim}(\mathfrak{g})} \ker(\omega_\xi|_{(\mathfrak{g}/\mathfrak{z})_j})\right\}.$$
\end{proposition}

\begin{proof}
Recall the definition of the Vergne polarization from Equation \eqref{vergndendoded}. It is clear that for $k=1,\ldots, \mathrm{dim}(\mathfrak{z})$, 
$$\mathsf{stab}(\xi|_{\mathfrak{g}_k})=\mathfrak{g}_k\subseteq \mathfrak{z},$$
with equality at $k=\mathrm{dim}(\mathfrak{z})$. The proposition therefore follows upon noting that for $\xi\in \Lambda_\epsilon\subseteq \mathfrak{z}^*$ we have for $j=1,\ldots ,\mathrm{codim}(\mathfrak{z})$ that
$$\mathsf{stab}(\xi|_{\mathfrak{g}_{j+\mathrm{dim}(\mathfrak{z})}})=\mathfrak{z}\oplus \ker(\omega_\xi|_{(\mathfrak{g}/\mathfrak{z})_j}),$$
where we have identified $\mathfrak{g}=\mathfrak{z}\oplus \mathfrak{g}/\mathfrak{z}$ as vector spaces. 
\end{proof}

\begin{proposition}
\label{totalspaceofflatorbits}
Let $\mathsf{G}$ be an $n$-dimensional simply connected nilpotent Lie group admitting flat orbits. Define 
$$\Xi:=\left\{\xi\in \mathfrak{g}^*: \mathcal{O}_\xi\mbox{ is flat }\right\}=\bigcup_{\epsilon: \#\epsilon_n=2d}\Xi_\epsilon\subseteq \mathfrak{g}^*.$$
Then the quotient by the coadjoint action defines a smoothly trivializable vector bundle 
$$\Xi\to \Gamma,$$
with fibre $\mathfrak{z}^\perp\subseteq \mathfrak{g}^*$. 
\end{proposition}

\begin{proof}
By definition, 
$$\Xi=\{\xi\in \mathfrak{g}^*: \mathcal{O}_\xi\in \Gamma\}=\{\xi\in \mathfrak{g}^*: \mathcal{O}_\xi=\xi+\mathfrak{z}^\perp\}.$$
If we identify $\Gamma=\Lambda$, we see that 
$$\psi:\Lambda\times \mathfrak{z}^\perp\to \Xi, \quad \psi(\xi,\eta):=\xi+\eta,$$
defines a diffeomorphism that equips $\Xi$ with the structure of a trivializable smooth vector bundle.
\end{proof}

\begin{definition}
\label{idealofflatorbits}
Let $\mathsf{G}$ be a simply connected nilpotent Lie group. We define the ideal $I_\mathsf{G}\subseteq C^*(\mathsf{G})$ as 
$$I_\mathsf{G}:=\{a\in C^*(\mathsf{G}): \pi(a)=0\;\forall \pi\in \hat{\mathsf{G}}\setminus \Gamma\}.$$
We call $I_\mathsf{G}$ the ideal of flat orbits.
\end{definition}

\begin{lemma}
\label{tracelemma}
Let $\mathsf{G}$ be a simply connected nilpotent Lie group. For any $a\in C^\infty_c(\mathsf{G})$, the mapping 
$$\mathrm{Tr}_a:\Gamma\to \C, \quad [\pi]\mapsto \mathrm{Tr}(\pi(a)),$$
is continuous.
\end{lemma}

\begin{proof}
Since $\Gamma$ is a strata in the coarse stratification of $\hat{\mathsf{G}}$, the lemma follows from \cite[Lemma 4.4.4]{pedersen84}.

\end{proof}

\begin{proposition}
\label{ctforflat}
Let $\mathsf{G}$ be a simply connected nilpotent Lie group. Then $I_\mathsf{G}$ is a continuous trace algebra over $\Gamma$. 
\end{proposition}

\begin{proof}
The $C^*$-algebra is a continuous trace algebra since it has Hausdorff spectrum and each element of the dense subspace $I_\mathsf{G}\cap C^\infty_c(\mathsf{G})\subseteq I_\mathsf{G}$ has continuous trace by Lemma \ref{tracelemma}. 
\end{proof}

We shall return to study the structure of $I_G$ further below in Section \ref{ctstructrifoeod}.

\section{Automorphisms and their action on $\hat{\mathsf{G}}$}
\label{sec:autoromom}

In order to understand the global structure of a Carnot manifold, we must describe how automorphisms of the nilpotent Lie group act on the representations. Let $\Aut(\mathsf{G})$ denote the group of automorphisms of a group $\mathsf{G}$ and let $\Aut(\mathfrak{g})$ denote the group of automorphisms of a Lie algebra $\mathfrak{g}$. We first note that if $\mathsf{G}$ is a simply connected nilpotent Lie group, the exponential map induces a group isomorphism $\Aut(\mathsf{G})\cong \Aut(\mathfrak{g})$. For an automorphism $\varphi\in \Aut(\mathfrak{g})$, we let $\varphi^* : \mathfrak{g}^* \rightarrow \mathfrak{g}^*$ denote its dual.

\begin{proposition}
\label{autoongamma}
Let $\mathsf{G}$ be a simply connected Lie group with Lie algebra $\mathfrak{g}$. The action of $\Aut(\mathfrak{g})$ on $\mathfrak{g}^*$ maps coadjoint orbits to coadjoint orbits. In particular, the $\Aut(\mathfrak{g})$-action defines a continuous action on $\widehat{\mathsf{G}}$ defined by 
$$\varphi.\mathcal{O}_\xi := \mathcal{O}_{\varphi^*(\xi)}\equiv \varphi^*(\mathcal{O}_\xi),$$ 
for every $\varphi \in \Aut(\mathfrak{g})$ and $\xi\in \mathfrak{g}^*$. 

Moreover, the set of flat orbits $\Gamma\subseteq \hat{\mathsf{G}}$ is invariant for the $\Aut(\mathfrak{g})$-action. In other words, if $\mathcal{O} \in \Gamma$ is flat and $\varphi\in \Aut(\mathfrak{g})$ then $\varphi.\mathcal{O}\in \Gamma$ is also flat.
\end{proposition}

We note that in general, a strata in the fine stratification can fail to be preserved by a group automorphism that does not preserve the Jordan-Hölder flag. 

\begin{proof}
The first part of the proposition is an immediate consequence of the fact that the inner automorphisms, i.e. the adjoint action, is a normal subgroup of $\Aut(\mathfrak{g})$. That an arbitrary automorphism preserves the space of flat orbits is a consequence of that an arbitrary automorphism preserves the center of the Lie group. In other words, we have for $\xi\in \Xi$ that
$$\mathcal{O}_\xi=\xi + \mathfrak{z}^\bot\qquad\Rightarrow\qquad \varphi.\mathcal{O}_\xi=\mathcal{O}_{\varphi^*\xi}=\varphi^*\xi + \mathfrak{z}^\bot.$$
\end{proof}

We make use of the following convention for gradings of nilpotent Lie algebras. 

\begin{definition}
Let $\mathfrak{g}$ be a nilpotent Lie algebra. A grading on $\mathfrak{g}$ is a decomposition indexed by negative integers
$$\mathfrak{g}=\bigoplus_{p<0}\mathfrak{g}_p,$$
with $[\mathfrak{g}_p,\mathfrak{g}_q]\subseteq \mathfrak{g}_{p+q}$ for any $p,q<0$.
\end{definition}

If we have a grading on the Lie algebra $\mathfrak{g}$ of a simply connected nilpotent Lie group $\mathsf{G}$, we let $\Aut_{\rm gr}(\mathfrak{g})$ denote the space of graded Lie algebra automorphisms of $\mathfrak{g}$ and $\Aut_{\rm gr}(\mathsf{G})$ its image under the group isomorphism $\Aut(\mathsf{G})\cong \Aut(\mathfrak{g})$. We make the following remark about gradings in relation to certain one-parameter families of automorphisms. Following \cite{follandsubell}, we say that a one-parameter group of automorphisms $\delta:\R_+\to \Aut(\mathfrak{g})$ is a dilation if $\dot{\delta}(1)\in \End(\mathfrak{g})$ is diagonalizable with positive integer eigenvalues.

\begin{proposition}
\label{gradingvsdilationsprop}
Let $\mathfrak{g}$ be a nilpotent Lie algebra. There is a one-to-one correspondence between dilations and gradings on $\mathfrak{g}$ given by the following two constructions that are mutually inverse. 
\begin{itemize}
\item If $\delta$ is a dilation, and we set $\mathfrak{g}_p:=\ker(\dot{\delta}(1)+p)$, then
$$\mathfrak{g} = \bigoplus_{p < 0} \mathfrak{g}_p,$$
defines a grading of $\mathfrak{g}$.
\item If $\mathfrak{g} = \bigoplus_{p < 0} \mathfrak{g}_p$ is a grading of $\mathfrak{g}$, then 
$$\delta_t:=\bigoplus_{p < 0} t^{-p}\,\mathrm{id}_{\mathfrak{g}_p},\quad t>0,$$
defines a dilation of $\mathfrak{g}$.
\end{itemize}
If $\mathfrak{g}$ is graded by a dilation $\delta$, then 
$$\Aut_{\rm gr}(\mathfrak{g})=\{\varphi\in \Aut(\mathfrak{g}): \, [\dot{\delta}(1),\varphi]=0\}=\{\varphi\in \Aut(\mathfrak{g}): \, \delta_t\circ\varphi=\varphi\circ \delta_t\; \forall t>0\}.$$
\end{proposition}

\begin{remark}
\label{lackofgradingremark}
Not all nilpotent Lie algebras admit gradings. Indeed, by \cite{dyer70} there is a nine dimensional nilpotent Lie algebra not admitting a dilation. However,  any quotient of the free graded Lie algebra of any rank by a homogeneous ideal admits a dilation.

We remark that this showcases a subtle, yet important, distinction between gradings and filterings on Lie algebras. Existence of a finite filtration by ideals on a Lie algebra is equivalent to it being nilpotent. Moreover, any automorphism of  a nilpotent Lie algebra preserves the filtration defined from the descending central series. In the geometric setting, we see this phenomena appear in the Examples \ref{standardcarnot} and \ref{standardcarnotmodlattice}.
\end{remark}

\begin{definition}
A Carnot-Lie group is a pair $(\mathsf{G},\delta)$ consisting of a nilpotent Lie group $\mathsf{G}$ equipped with a choice of dilation $\delta$ (or equivalently a grading of the Lie algebra). We often drop the $\delta$ from the notation when it is clear from context.
\end{definition}

\begin{proposition}
\label{gammapartialdef}
Let $\mathsf{G}$ be a simply connected Carnot-Lie group and let $\Gamma$ denote the set of flat orbits of $\mathsf{G}$. The action of the dilation group preserves $\Gamma$ and the set 
$$\Gamma_{\partial}:=\Gamma/\delta(\R_+),$$
is a well defined topological space. Moreover, if $\mathsf{G}$ is non-abelian, the $\R_+$-action on $\Gamma$ is free and proper, and the homeomorphism $\Gamma\cong \Lambda$ of Proposition \ref{descripfiogogda} induces a homeomorphism
$$\Lambda_\partial\xrightarrow{\sim} \Gamma_\partial,$$
where $\Lambda_\partial$ is the smooth algebraic hypersurface in $\mathfrak{z}^*$ defined from
$$\Lambda_\partial=\{\xi\in \mathfrak{z}^*: \mathsf{det}(\omega_\xi|_{\mathfrak{g}/\mathfrak{z}})=1\}\subseteq \Lambda.$$
In particular, $\Gamma_{\partial}$ is canonically homeomorphic to $\Lambda_\partial$. 
\end{proposition}

We remark that since the homeomorphism $\Gamma\cong \Lambda$ depends on a choice of Jordan-Hölder basis, so does the homeomorphism $\Lambda_\partial\cong \Gamma_\partial$.

\begin{proof}
It follows from Proposition \ref{autoongamma} that $\delta$ preserves $\Gamma$. The homeomorphism $\Gamma\cong \Lambda$ of Proposition \ref{gammasplitting} intertwines the action of $\delta$ on $\Gamma$ with an $\R_+$-action on $\mathfrak{z}^*$. For the remainder of the proof, we consider non-abelian $\mathsf{G}$. If $\mathsf{G}$ is non-abelian, the $\R_+$-action on $\Gamma$ is free and proper. Indeed, we note that for $\xi\in \Lambda$, there are constants $c_0,c_1>0$ such that for $t>0$,
$$c_0\min(t^{\lambda'},t^\lambda) \leq |\mathsf{Pf}(\omega_{t.\xi}|_{\mathfrak{g}/\mathfrak{z}})|\leq c_1\max(t^{\lambda'},t^\lambda) $$
where $\lambda$ is the smallest eigenvalue of $\dot{\delta}(1)|_{\mathfrak{z}}$ and $\lambda'$ is the largest eigenvalue of $\dot{\delta}(1)|_{\mathfrak{z}}$. Since $\lambda,\lambda'\geq 1$, we conclude that there is a unique $t>0$, depending continuously on $\xi$, such that $\mathsf{Pf}(\omega_{t.\xi}|_{\mathfrak{g}/\mathfrak{z}})=\pm 1$. As such, $\Lambda/\R_+$ can be identified with 
$$\{\xi\in \mathfrak{z}^*: \mathsf{Pf}(\omega_\xi|_{\mathfrak{g}/\mathfrak{z}})=\pm 1\}=\Lambda_\partial,$$
where we in the last equality use that $\mathsf{det}(\omega_\xi|_{\mathfrak{g}/\mathfrak{z}})=|\mathsf{Pf}(\omega_\xi|_{\mathfrak{g}/\mathfrak{z}})|^2$. The determinant is a regular function outside of its zero set, so $\Lambda_\partial$ is a smooth algebraic hypersurface in $\mathfrak{z}^*$. 
\end{proof}

\begin{remark}
Proposition \ref{gammapartialdef} induces a continuous splitting of the quotient map $\Gamma\to \Gamma_\partial$. We identify $\Gamma_\partial$ with its image in $\Gamma\subseteq \hat{\mathsf{G}}$ as a transversal to the dilation action. In particular, $\Gamma_\partial\subseteq \hat{\mathsf{G}}$ is a closed Hausdorff subset.
\end{remark}

\begin{definition}
We define the $C^*$-algebra $I_{\partial,\mathsf{G}}$ as the quotient of $C^*(\mathsf{G})$ by the ideal 
$$\{a\in C^*(\mathsf{G}): \pi(a)=0\; \forall \pi \in \hat{\mathsf{G}}\setminus \Gamma_\partial\}.$$
\end{definition}

\begin{remark}
We note that the restriction of the quotient map $C^*(\mathsf{G})\to I_{\partial,\mathsf{G}}$ to $I_{\mathsf{G}}$ is surjective and that the spectrum of $I_{\partial,\mathsf{G}}$ is the closed Hausdorff subset $\Gamma_\partial\subseteq \hat{G}$.
\end{remark}

The next proposition follows from Proposition \ref{gammapartialdef}.

\begin{proposition}
\label{gammapartialsplitting}
Let $\mathsf{G}$ be a simply connected non-abelian Carnot-Lie group. There is an $\R_+$-equivariant $*$-isomorphism $\alpha_\partial:I_\mathsf{G}\to C_0(\R_+)\otimes I_{\partial,\mathsf{G}}$ (for the trivial action on $I_{\partial,\mathsf{G}}$) such that
\begin{itemize}
\item The maps induced from $\alpha_\partial$, the restricted quotient map $I_\mathsf{G}\to I_{\partial,\mathsf{G}}$, and the evaluation-at-$1$-map $C_0(\R_+)\otimes I_{\partial,\mathsf{G}}\to I_{\partial,\mathsf{G}}$ on the spectrum fits into an $\R_+$-equivariant commuting diagram
$$
\begin{tikzcd}
\R_+\times \Gamma_\partial \arrow[r,"\alpha_\partial^*" ] &  \Gamma \\\
\Gamma_\partial \arrow[ur, "\subseteq"] \arrow[u] & 
\end{tikzcd}
$$
\item These $*$-homomorphisms fit into a commuting diagram of $\R_+$-equivariant $*$-homomorphisms
$$
\begin{tikzcd}
C^*(\mathsf{G}) \arrow[r] &  I_{\partial,\mathsf{G}} \\\
I_\mathsf{G} \arrow[r, "\alpha_\partial"] \arrow[u, "\subseteq"] &C_0(\R_+)\otimes I_{\partial,\mathsf{G}} \arrow[u] .
\end{tikzcd}
$$
\end{itemize}
\end{proposition}

Let us describe how automorphisms act on the space of flat orbits in terms of its realization as the transversal $\Lambda$. Recall our convention to identify $\mathfrak{z}^*$ with a subspace of $\mathfrak{g}^*$ via the Jordan-Hölder basis. The following proposition is immediate from Proposition \ref{gradingvsdilationsprop} and \ref{gammapartialdef}. 

\begin{proposition}
\label{actionofautosonflat}
Let $\mathsf{G}$ be a simply connected, nilpotent Lie group and $\varphi\in \Aut(\mathfrak{g})$. Define the polynomial map 
$$\varphi_\Lambda^*:\Lambda\to \Lambda, \; \varphi_\Lambda^*(\xi):=(\varphi^*\xi)|_\mathfrak{z}.$$
Then $\varphi_\Lambda^*$ is a self-homeomorphism fitting into a commuting diagram wtih the homeomorphism $\Gamma\cong \Lambda$ of Proposition \ref{gammasplitting} 
\begin{equation}
\label{actiononeondal}
\begin{tikzcd}
\Lambda \arrow[d, "\cong"] \arrow[r,"\varphi_\Lambda^*" ] &  \Lambda \arrow[d, "\cong"] \\\
\Gamma \arrow[r, "\varphi^*"] &\Gamma.
\end{tikzcd}
\end{equation}
For any two $\varphi,\varphi'\in \Aut(\mathfrak{g})$, $(\varphi\circ\varphi')_\Lambda^*=\varphi'{}_\Lambda^*\circ \varphi_\Lambda^*$. 

Additionally, if $\mathsf{G}$ is a Carnot-Lie group and $\varphi\in \Aut_{\rm gr}(\mathfrak{g})$ then $\varphi_\Lambda^*$ commutes with the $\R_+$-action and induces a self-diffeomorphism 
$$\varphi_{\Lambda_\partial}^*:\Lambda_\partial \to \Lambda_\partial,$$ 
fitting into a commutative diagram with $\varphi^*:\Gamma_\partial\to \Gamma_\partial$ as in Equation \eqref{actiononeondal}.
\end{proposition}

\begin{remark}
If $\dot{\delta}(1)|_{\mathfrak{z}}$ is scalar (i.e. the dilation acts by scalars on the center) or $\varphi$ preserves the determinant function on $\Lambda$, then $\varphi_{\Lambda_\partial}^*$ is polynomial since it is the restriction of a linear map on $\mathfrak{z}^*$. 
\end{remark}

\section[The ideal of flat orbits]{The continuous trace algebra structure of the ideal of flat orbits}
\label{ctstructrifoeod}

We saw in Proposition \ref{ctforflat} that the ideal of flat orbits $I_{\mathsf{G}}$ was a continuous trace algebra with spectrum $\Gamma$. For an overview of continuous trace algebras, see \cite{Williams_Raeburn,RosenbergTwisted}. We use the notation $\delta_{\rm DD}$ for the Dixmier-Douady invariant in the third integral Cech cohomology, e.g. $\delta_{\rm DD}(I_{\mathsf{G}})\in \check{H}^3(\Gamma;\Z)$. It is well known that  $\delta_{\rm DD}(I_\mathsf{G})=0\in \check{H}^3(\Gamma,\Z)$ if and only if there exists a Hilbert space bundle $\mathcal{H}_\Gamma\to \Gamma$ and a continuous homomorphism $\pi_{\musFlat}:\mathsf{G}\to C(\Gamma,U(\mathcal{H}_\Gamma))$ to the strongly continuous sections of the unitary homomorphisms of the bundle such that for each $\xi\in \Gamma$, we have that $\xi=[\pi_{\musFlat}|_{\{\xi\}}]$. We shall see that indeed $\delta_{\rm DD}(I_\mathsf{G})=0$ and call a choice of such a bundle $\mathcal{H}_\Gamma\to \Gamma$ a bundle of flat orbit representations.

In this subsection we produce an explicit construction of a bundle of flat orbit representations and study its equivariance properties under automorphisms. While this construction in its own right might not be overly exciting, it will be used later in this work on Carnot manifolds where a global construction is less immediate. The equivariance properties will in some examples let us construct global bundles by induction from principal bundles. 

The problem of constructing bundles of flat orbit representations is interesting in light of the following conjecture of Lipman-Rosenberg. Recall that a subquotient of the $C^*$-algebra $C^*(\mathsf{G})$ is a $C^*$-algebra of the form $A:=I_1/I_2$ where $I_2\subseteq I_1\subseteq C^*(\mathsf{G})$ are ideals. By \cite{Dixmierbook}, we can identify $\hat{I}_1$ and $\hat{I}_2$ with open subsets of $\hat{\mathsf{G}}$ and $\hat{A}$ with a subset of $\hat{\mathsf{G}}$.

\begin{conjecture}[Conjecture 3.2 in \cite{LipmanRosenberg}]
\label{ddzeroconjecture}
For any connected nilpotent Lie group $\mathsf{G}$ and continuous trace subquotient $A$ of $C^*(\mathsf{G})$ (in particular $\hat{A}\subseteq \hat{\mathsf{G}}$ is a Hausdorff subset), it holds that $\delta_{\rm DD}(A)=0$. 
\end{conjecture}

We shall mainly be interested in the case that $A=I_\mathsf{G}$. It seems as if experts in the field know the case of $A=I_{\mathsf{G}}$ to fulfil Conjecture \ref{ddzeroconjecture}, but we have not been able to find a reference. We note the following cases where $\delta_{\rm DD}(I_\mathsf{G})=0$ can be concluded at this stage:
\begin{enumerate}
\item[a)] If $\mathsf{G}$ has step length $2$, Conjecture \ref{ddzeroconjecture} holds true by \cite[Theorem 3.4]{LipmanRosenberg}. In particular, $\delta_{\rm DD}(I_\mathsf{G})=0$ for all step $2$ nilpotent Lie groups $\mathsf{G}$. For examples of step $2$ Lie groups, see Subsection \ref{step2subsectionexamples}.
\item[b)] If $\Gamma$ consists of one fine strata, Proposition \ref{repsandosisosos} implies that $\delta_{\rm DD}(I_\mathsf{G})=0$. We shall see an abundance of examples below in Section \ref{subsectionexamples} where $\Gamma$ consists of one fine strata. An important special case when $\Gamma$ consists of one fine strata is when $\mathsf{G}$ has one-dimensional center and admits flat orbits (see Subsection \ref{onedcenterexasubs} below), in this case $\Gamma=\mathfrak{z}^*\setminus\{0\}$.
\item[c)] More generally than in b), if there exists a bundle of polarizations $F\to \Gamma$, with $F\subseteq \Lambda\times \mathfrak{g}$ satisfying the assumptions of Proposition \ref{vbfv}, then $\delta_{\rm DD}(I_{\mathsf{G}})=0$. Indeed, in this case the Hilbert space bundle $\mathcal{H}_\Gamma\to \Gamma$ can be constructed from $F$ as in Proposition \ref{bundleofhspspac}.
\end{enumerate}

Let us describe the structure of the continuous trace algebra $I_\mathsf{G}$ and make the construction of the bundle of flat orbit representations more explicit. First, we show that the top fine stratas, while varrying over all Jordan-Hölder bases, is an open cover of the set of flat orbits. We let $\mathfrak{JH}$ denote the set of Jordan-Hölder bases of $\mathfrak{g}$ with $\mathfrak{g}_{\dim(\mathfrak{z})}=\mathfrak{z}$. For a Jordan-Hölder basis $\pmb{B}\in \mathfrak{JH}$, we let $U_{\pmb{B}}\subseteq \Gamma$ denote the top fine strata in the fine stratification defined from $\pmb{B}$.

\begin{theorem}
\label{covconj}
Let $\mathsf{G}$ be a nilpotent simply connected Lie group admitting flat orbits. Then $(U_{\pmb{B}})_{\pmb{B}\in \mathfrak{JH}}$ is an open cover of $\Gamma$. 
\end{theorem}

To prove Theorem \ref{covconj}, we shall first need a lemma.

\begin{lemma}
\label{quotoietnoenlemma}
Let $\mathsf{G}$ be a nilpotent simply connected Lie group with a flat orbit through $\xi\in \mathfrak{z}^*\subseteq \mathfrak{g}^*$. Let $\Gamma_{\mathsf{G}}$ denote the set of flat orbits for $\mathsf{G}$. Fix an $X\in \ker(\xi)\cap \mathfrak{z}$. The following holds the Lie algebra $\bar{\mathfrak{g}}:=\mathfrak{g}/\R X$:
\begin{enumerate}
\item[i)] $\bar{\mathfrak{g}}$ admits a flat orbit through the element $\bar{\xi}\in (\bar{\mathfrak{g}})^*$ induced by $\xi$.
\item[ii)] The centre $\bar{\mathfrak{z}}$ of $\bar{\mathfrak{g}}$ is given by $\bar{\mathfrak{z}}=\mathfrak{z}/\R X$.
\item[iii)] The set $\Gamma_{\bar{\mathsf{G}}}$ of flat orbits for $\bar{\mathsf{G}}$ can be identified with a hypersurface in $\Gamma_{\mathsf{G}}$ via the injection $\bar{\mathfrak{z}}^*\hookrightarrow \mathfrak{z}^*$ dual to the projection from item ii).
\end{enumerate}
\end{lemma}

\begin{proof}
Since $\xi$ lies in a flat orbit, 
$$\mathfrak{z}/\R X=\mathsf{stab}(\xi)/\R X=\mathsf{stab}(\bar{\xi}).$$
We conclude from the series of inclusions $\bar{\mathfrak{z}}\subseteq \mathsf{stab}(\bar{\xi})=\mathfrak{z}/\R X\subseteq\bar{\mathfrak{z}}$ that $\mathfrak{z}/\R X=\bar{\mathfrak{z}}$ and that $\bar{\xi}$ lies in a flat orbit. The lemma follows.
\end{proof}

\begin{proof}[Proof of Theorem \ref{covconj}]
We prove the theorem by induction on the dimension of the centre of $\mathsf{G}$. If $\mathsf{G}$ has one-dimensional centre and flat orbits, the statement holds because $\Gamma=\mathfrak{z}^*\setminus\{0\}$ and $\Gamma=U_{\pmb{B}}$ for any $\pmb{B}\in \mathfrak{JH}$. Assume that the statement of the theorem holds true for nilpotent simply connected Lie groups admitting flat orbits with centre of dimension $\leq m-1$. Consider a nilpotent simply connected Lie group $\mathsf{G}$ admitting flat orbits and whose centre is of dimension $m$. Take a $\xi\in \Gamma_{\mathsf{G}}$ and an $X\in \ker(\xi)\cap \mathfrak{z}$. We follow the notation of Lemma \ref{quotoietnoenlemma}.

The nilpotent simply connected Lie group $\bar{\mathsf{G}}$ that $\bar{\mathfrak{g}}$ integrates to satisfy the statement of the theorem by the induction assumption, because $\bar{\mathfrak{z}}=\mathfrak{z}/\R X$ has dimension $m-1$. Therefore, there is a Jordan-Hölder basis $\pmb{B}_1=(\bar{X}_2,\ldots \bar{X}_n)$ for $\bar{\mathfrak{g}}$ such that $\bar{\xi}\in U_{\pmb{B}_1}\subseteq \Gamma_{\bar{\mathsf{G}}}$. We can lift $(\bar{X}_2,\ldots \bar{X}_n)$ to $X_2,\ldots, X_n\in \mathfrak{g}$ and, up to permutation, consider the Jordan-Hölder basis $\pmb{B}=(X_1,X_2,\ldots, X_n)$ for $\mathfrak{g}$ where $X_1=X$. 

The induction step is complete upon proving that $\Gamma_{J_\xi}^{\pmb{B}}$ contains an open subset, in which case $\Gamma_{J_\xi}^{\pmb{B}}=U_{\pmb{B}}$ and $\xi\in U_{\pmb{B}}$ since $\xi\in \Gamma_{J_\xi}^{\pmb{B}}$ by definition. Take an $\eta\in \mathfrak{z}^*$ such that $\eta(X)=1$ and $\eta(X_j)=0$ for $j>1$. By Lemma \ref{quotoietnoenlemma}.iii), it suffices to prove that $J_{\xi'+t\eta}^{\pmb{B}}$ is independent of sufficiently small $t$ for $\xi'\in U_{\pmb{B}_1}\subseteq \Gamma_{\bar{\mathsf{G}}}$ in a neighbourhood of $\bar{\xi}$. Consider the family of Jordan-Hölder bases $\pmb{B}(t)=(X_1-tX_2,X_2,\ldots, X_n)$. By definition, we have that $J_{\xi'+t\eta}^{\pmb{B}}=J_{\xi'}^{\pmb{B}(t)}$ and since the variation in $t$ only can alter jump indices in the first step, we have that $J_{\xi'}^{\pmb{B}(t)}$ is independently of sufficiently small $t$. We conclude that $J_{\xi'+t\eta}^{\pmb{B}}$ is independent of sufficiently small $t$, and $\Gamma_{J_\xi}^{\pmb{B}}$ contains an open neighbourhood around $\xi$.
\end{proof}

\begin{remark}
Consider the statement:
\begin{center}
\emph{Assume that $\mathsf{G}$ has orbits of dimension at most $2d$, and consider the space of orbits of maximal dimension $\tilde{\Gamma}:=\{\mathcal{O}\in \hat{\mathsf{G}}: \dim(\mathcal{O})=2d\}$. \\Then $(U_{\pmb{B}})_{\pmb{B}\in \mathfrak{JH}}$ is an open cover of $\tilde{\Gamma}$. }
\end{center}
Theorem \ref{covconj} shows that the above statement holds in the presence of flat orbits -- in this case $\Gamma=\tilde{\Gamma}$. Generally we have that $\cup_{\pmb{B}\in \mathfrak{JH}}U_{\pmb{B}}\subseteq \tilde{\Gamma}$ is Zariski open (since each $U_{\pmb{B}}$ is). However, the statement above fails in general as seen from the four dimensional Lie algebra spanned by $X_1,X_2,X_3,X_4$ with brackets $[X_4,X_3]=X_2$ and $[X_4,X_2]=X_1$ and all other brackets zero. For this Lie algebra we can identify $\tilde{\Gamma}=\{(\xi_1,\xi_2):(\xi_1,\xi_2) \neq 0\}$ but $U_{\pmb{B}}=\{(\xi_1,\xi_2):\xi_1\neq 0\}$ for any Jordan-Hölder basis $\pmb{B}\in \mathfrak{JH}$. This example can be found in \cite[Remark 1, p. 274]{pukan67jfa}.
\end{remark}

To simplify notation, we shall drop the notation involving jump indices, as we henceforth only use the minimal jump index that corresponds to the top strata, and instead indicate the basis. For instance, $U_{\pmb{B}}$ denotes the top fine strata in the stratification defined from the basis $\pmb{B}$, $F_{V,\pmb{B}}\to U_{\pmb{B}}$ denotes the bundle of Vergne polarizations, and $\mathcal{H}_{V,\pmb{B}}\to U_{\pmb{B}}$ the associated Hilbert space bundle of $\mathsf{G}$-representations over $U_{\pmb{B}}\subseteq \Gamma\subseteq \hat{\mathsf{G}}$. 

For any $\pmb{B}_1,\ldots, \pmb{B}_k\in \mathfrak{JH}$, we write 
\begin{align*}
I_{\mathsf{G},\pmb{B}_1,\ldots, \pmb{B}_k}:&=C_0(U_{\pmb{B}_1}\cap\cdots U_{\pmb{B}_k})I_\mathsf{G}=\\
&=\{a\in C^*(\mathsf{G}): \pi(a)=0\; \forall \pi\in \hat{\mathsf{G}}\setminus U_{\pmb{B}_1}\cap\cdots U_{\pmb{B}_k}\}.
\end{align*}
For each $\pmb{B}\in \mathfrak{JH}$, Proposition \ref{reponen} and \ref{repsandosisosos} produces canonical trivializations
$$\pi_{\pmb{B}}:I_{\mathsf{G},\pmb{B}}\to C_0(U_{\pmb{B}},\mathbb{K}(\mathcal{H}_{V,\pmb{B}})),$$
where $d=\mathrm{codim}(\mathfrak{z})/2$. For any $\pmb{B},\pmb{B}'\in \mathfrak{JH}$, the corrected Lion intertwiner from Proposition \ref{lionsintertwinerswitheta} (see also Lemma \ref{lionsintertwiners}) defines a fibrewise unitary intertwiner
$$\mathfrak{L}_{\pmb{B},\pmb{B}'}:=\mathfrak{L}_{F_{V,\pmb{B}'},F_{V,\pmb{B}}}:\mathcal{H}_{V,\pmb{B}'}|_{U_{\pmb{B}}\cap U_{\pmb{B}'}}\to \mathcal{H}_{V,\pmb{B}}|_{U_{\pmb{B}}\cap U_{\pmb{B}'}}.$$
By possibly shrinking the open sets in the cover, we can also ensure continuity. In other words, 
$$\mathsf{Ad}(\mathfrak{L}_{\pmb{B}',\pmb{B}}):C_0(U_{\pmb{B}}\cap U_{\pmb{B}'},\mathbb{K}(\mathcal{H}_{V,\pmb{B}}|_{U_{\pmb{B}}\cap U_{\pmb{B}'}}))\to C_0(U_{\pmb{B}}\cap U_{\pmb{B}'},\mathbb{K}(\mathcal{H}_{V,\pmb{B}'}|_{U_{\pmb{B}}\cap U_{\pmb{B}'}})),$$
is a unitary isomorphism fitting into the commuting diagram of $*$-isomorphisms
\small
\[
\begin{tikzcd}
C_0(U_{\pmb{B}}\cap U_{\pmb{B}'},\mathbb{K}(\mathcal{H}_{V,\pmb{B}}|_{U_{\pmb{B}}\cap U_{\pmb{B}'}})) \arrow[rr, "\mathsf{Ad}(\mathfrak{L}_{\pmb{B}',\pmb{B}})"] & & C_0(U_{\pmb{B}}\cap U_{\pmb{B}'},\mathbb{K}(\mathcal{H}_{V,\pmb{B}'}|_{U_{\pmb{B}}\cap U_{\pmb{B}'}}))\\
& I_{G,\pmb{B}, \pmb{B}'} \arrow[ul, "\pi_{\pmb{B}}"] \arrow[ur, "\pi_{\pmb{B}'}"] & 
\end{tikzcd}
\]
\normalsize
We conclude the following from Proposition \ref{lionsintertwinerswitheta}.

\begin{proposition}
\label{descriptionofddinvariant}
Let $\mathsf{G}$ be a simply connected nilpotent Lie group admitting flat orbits. Then the collection $(\mathcal{H}_{V,\pmb{B}})_{\pmb{B}\in \mathfrak{JH}}$ glues together as bundles of Hilbert spaces with a fibrewise unitary $\mathsf{G}$-action to a bundle of flat representations
$$\mathcal{H}_\Gamma\to \Gamma,$$
via the corrected Lion intertwiners $(\mathfrak{L}_{\pmb{B},\pmb{B}'})_{\pmb{B},\pmb{B}'\in \mathfrak{JH}}$ from Proposition \ref{lionsintertwinerswitheta}. In particular, the fibrewise unitary $\mathsf{G}$-action induces a $C_0(\Gamma)$-linear $*$-isomorphism
$$\pi_{\musFlat}:I_\mathsf{G}\to C_0(\Gamma,\mathbb{K}(\mathcal{H}_\Gamma)).$$
\end{proposition}

The existence of a bundle of flat orbit representations also follow from a deformation argument below in Theorem \ref{trivialdldaaddo}.

By the universal property of the group $C^*$-algebra, any group automorphism $\varphi\in \Aut(G)$ induces an automorphism $C^*(\mathsf{G})$ that we by an abuse of notation also denote by $\varphi$. Since the associated group homomorphism $\Aut(\mathsf{G})\to \Aut(C^*(\mathsf{G}))$ is injective this abuse of notation should not be a cause of confusion. By Proposition \ref{spectrumofthejs} and \ref{autoongamma}, any automorphism $\varphi\in \Aut(\mathsf{G})$ preserves the ideal $I_\mathsf{G}\subseteq C^*(\mathsf{G})$ of flat orbits. We abstractly know that $I_\mathsf{G}$ is a continuous trace algebra over $\Gamma$ with fibre $\mathbb{K}(L^2(\R^d))$, where $d=\mathrm{codim}(\mathfrak{z})/2$, and $I_\mathsf{G}$ trivializes by Proposition \ref{descriptionofddinvariant}. We now describe the action of automorphisms on $I_\mathsf{G}$ and their lifts to the bundle of flat orbit representations.

For a Jordan-Hölder basis $\pmb{B}=(X_1,\ldots, X_n)\in \mathfrak{JH}$ of $\mathfrak{g}$, with $\mathfrak{g}_{\dim\mathfrak{z}}=\mathfrak{z}$ as in Section \ref{subsecfinestrat} and \ref{subsecflatorbitssos}, recall the construction of $F_{V,\pmb{B}}\to U_{\pmb{B}}$ from Proposition \ref{vbfv}, where we write $F_{V,\pmb{B}}$ and $ U_{\pmb{B}}$ for the top strata in the basis $\pmb{B}$ (cf. Section \ref{subsecflatorbitssos}). For a $\varphi\in \Aut(\mathsf{G})$, we write $\varphi(\pmb{B})\in \mathfrak{JH}$ for the Jordan-Hölder basis
$$\varphi(\pmb{B}):=(\varphi^{-1}X_1,\ldots, \varphi^{-1}X_n).$$

\begin{proposition}
\label{covarianceforlions}
Let $\varphi,\varphi'$ be automorphisms of a simply connected nilpotent Lie group $\mathsf{G}$ admitting flat orbits. Take $\pmb{B},\pmb{B}'\in \mathfrak{JH}$. 
\begin{enumerate}
\item[a)] The mapping 
$$\varphi^*\times \varphi^{-1}:\mathfrak{g}^*\times \mathfrak{g}\to \mathfrak{g}^*\times \mathfrak{g}$$
restricts to a homeomorphism $\varphi^*\times \varphi^{-1}:F_{V,\pmb{B}}\to F_{V,\varphi(\pmb{B})}$ that fits into a commuting diagram 
$$\begin{tikzcd}
F_{V,\pmb{B}} \arrow[d] \arrow[r, "\varphi^*\times \varphi^{-1}"] &  F_{V,\varphi(\pmb{B})} \arrow[d] \\\
U_{\pmb{B}} \arrow[r, "\varphi^*"] &U_{\varphi(\pmb{B})}.
\end{tikzcd}$$
\item[b)] The map $(\xi,f)\mapsto (\varphi^*\xi,f\circ\varphi^*)$ defines a fibrewise unitary morphism 
$$t_{\pmb{B}}(\varphi):\mathcal{H}_{V,\pmb{B}}\to \mathcal{H}_{V,\varphi(\pmb{B})},$$ 
such that:
\begin{itemize}
\item The following diagram of Hilbert space bundles commutes
$$\begin{tikzcd}
\mathcal{H}_{V,\pmb{B}} \arrow[d] \arrow[r, "t_{\pmb{B}}(\varphi)"] &  \mathcal{H}_{V,\varphi(\pmb{B})} \arrow[d] \\\
U_{\pmb{B}} \arrow[r, "\varphi^*"] &U_{\varphi(\pmb{B})},
\end{tikzcd}$$
and $t_{\pmb{B}}(\varphi)$ is equivariant for the ordinary $\mathsf{G}$-action on $\mathcal{H}_{V,\pmb{B}}$ and the $\mathsf{G}$-action on $\mathcal{H}_{V,\varphi(\pmb{B})}$ pulled back along $\varphi$.
\item It holds that 
$$t_{\varphi(\pmb{B})}(\varphi')t_{\pmb{B}}(\varphi)=t_{\pmb{B}}(\varphi \varphi'),$$
as morphisms of Hilbert space bundles $\mathcal{H}_{V,\pmb{B}}\to \mathcal{H}_{V,\varphi\varphi'(\pmb{B})}$.
\item The following diagram of Hilbert space bundles commutes
$$\begin{tikzcd}
\mathcal{H}_{V,\pmb{B}}|_{U_{\pmb{B}}\cap U_{\pmb{B}'}} \arrow[d, "\mathfrak{L}_{\pmb{B}',\pmb{B}}^{(0)}"] \arrow[r, "t_{\pmb{B}}(\varphi)"] &  \mathcal{H}_{V,\varphi(\pmb{B})}|_{U_{\varphi(\pmb{B})}\cap U_{\varphi(\pmb{B}')}} \arrow[d, "\mathfrak{L}_{\varphi(\pmb{B}'),\varphi(\pmb{B})}^{(0)}"] \\\
\mathcal{H}_{V,\pmb{B}'}|_{U_{\pmb{B}}\cap U_{\pmb{B}'}}  \arrow[r, "t_{\pmb{B}'}(\varphi)"] &  \mathcal{H}_{V,\varphi(\pmb{B}')} |_{U_{\varphi(\pmb{B})}\cap U_{\varphi(\pmb{B}')}}
\end{tikzcd}$$
\end{itemize}
\item[c)] The automorphism $\varphi\in \Aut(I_\mathsf{G})$ restricts to a $*$-isomorphism $\varphi|:I_{\mathsf{G},\pmb{B}}\to I_{\mathsf{G},\varphi(\pmb{B})}$ that together with $t_{\pmb{B}}(\varphi)$ from the previous item fit into the commuting diagram of $*$-isomorphisms
$$\begin{tikzcd}
I_{\mathsf{G},\pmb{B}}\arrow[d,"\pi_{\pmb{B}}"]\arrow[r, "\varphi|"] &I_{\mathsf{G},\varphi(\pmb{B})}\arrow[d,"\pi_{\varphi(\pmb{B})}"]\\
C_0(U_{\pmb{B}},\mathbb{K}(\mathcal{H}_{V,\pmb{B}}))\arrow[r, "\mathrm{Ad}(t_{\pmb{B}}(\varphi))"] &  C_0(U_{\varphi(\pmb{B)}},\mathbb{K}(\mathcal{H}_{V,\varphi(\pmb{B})})).
\end{tikzcd}$$
\end{enumerate}
\end{proposition}

\begin{proof}
Note that $\mathfrak{h}$ is a real algebraic polarization of $\xi\in \mathfrak{g}^*$ if and only if $\varphi^{-1}\mathfrak{h}$ is a real algebraic polarization of $\varphi^*\xi\in \mathfrak{g}^*$. We also note that since $\varphi^*(\Xi_{\epsilon,\pmb{B}})=\Xi_{\epsilon,\varphi(\pmb{B})}$ for any $\epsilon$, and $\varphi^*$ is a homeomorphism on $\Gamma$, we have that $\varphi^*(U_{\pmb{B}})=U_{\varphi(\pmb{B})}$. To prove item a), we simply note that for $\xi\in U_{\varphi(\pmb{B})}$, 
$$\varphi^{-1}(\mathfrak{h}_{V,\pmb{B}}(\varphi^{-1*}\xi))=\mathfrak{h}_{V,\varphi(\pmb{B})}(\xi),$$
which implies that
\begin{align*}
(\varphi^*\times \varphi^{-1})F_{V,\pmb{B}}&=\{(\varphi^*\xi,\varphi^{-1}X): \, \xi \in U_{\pmb{B}}, X\in \mathfrak{h}_{V,\pmb{B}}(\xi)\}=\\
&=\{(\xi,X): \, \xi \in U_{\varphi(\pmb{B})}, X\in \varphi^{-1}(\mathfrak{h}_{V,\pmb{B}}(\varphi^{-1*}\xi)\}=F_{V,\varphi(\pmb{B})}.
\end{align*}
Item b) and c) are proven by similar, lengthier but straightforward, computations and we omit their proofs.
\end{proof}

To describe how to lift an automorphism to the bundle $\mathcal{H}_\Gamma$ of flat orbit representations constructed in Proposition \ref{descriptionofddinvariant}, we follow \cite{Williams_Raeburn}. We let $(u_{\pmb{B},\pmb{B}'})_{\pmb{B},\pmb{B}'\in \mathfrak{JH}}\in \check{C}^1((U_{\pmb{B}})_{\pmb{B}\in \mathfrak{JH}},U(1))$ denote the cocycle defined from 
$$u_{\pmb{B},\pmb{B}'}:=\e^{\frac{\pi i}{4}\eta(F_{V,\pmb{B}'},F_{V,\pmb{B}})}.$$ 
Here we have for notational simplicity disregarded the fact that the covering needs to be refined to ensure continuity. We define the map 
\begin{align}
\label{zetadef}
\zeta_0:\Aut(\mathsf{G})\to &\check{C}^1(((U_{\pmb{B}})_{\pmb{B}\in \mathfrak{JH}},U(1))), \\
\nonumber
 &\varphi\mapsto (\varphi^*(u_{\pmb{B},\pmb{B}'})u_{\varphi(\pmb{B}),\varphi(\pmb{B}')}^{-1})_{\pmb{B},\pmb{B}'\in \mathfrak{JH}}.
\end{align}

\begin{proposition}
\label{zetanandndlaf}
Let $\mathsf{G}$ be a simply connected nilpotent Lie group admitting flat orbits. Then the mapping $\zeta_0$ satisfies 
\begin{enumerate}
\item The range of $\zeta_0$ is contained in the Cech cocycles, i.e. $\zeta_0$ induces a mapping 
$$\zeta:\Aut(\mathsf{G})\to \check{H}^1(\Gamma,U(1)).$$
\item $\zeta_0\in Z^1(\Aut(\mathsf{G}),\check{Z}^1((U_{\pmb{B}})_{\pmb{B}\in \mathfrak{JH}},U(1)))$ is a group cocycle for the induced $\Aut(\mathsf{G})$-action on $\check{Z}^1((U_{\pmb{B}})_{\pmb{B}\in \mathfrak{JH}},U(1))$, and induces a cohomology class 
$$[\zeta_0]\in H^1(\Aut(\mathsf{G}),\check{Z}^1((U_{\pmb{B}})_{\pmb{B}\in \mathfrak{JH}},U(1))),$$
and $\zeta\in Z^1(\Aut(\mathsf{G}),\check{H}^1(\Gamma,U(1)))$ induces a cohomology class 
$$[\zeta]\in H^1(\Aut(\mathsf{G}),\check{H}^1(\Gamma,U(1))).$$
\item The class $[\zeta]\in H^1(\Aut(\mathsf{G}),\check{H}^1(\Gamma,U(1)))$ does not depend on the choice of $(u_{\pmb{B},\pmb{B}'})_{\pmb{B},\pmb{B}'\in \mathfrak{JH}}\in \check{C}^1((U_{\pmb{B}})_{\pmb{B}\in \mathfrak{JH}},U(1))$ as long as the cochain $(u_{\pmb{B},\pmb{B}'})_{\pmb{B},\pmb{B}'}$ solves 
\begin{equation}
\label{onononad}
u_{\pmb{B},\pmb{B}'}u_{\pmb{B}',\pmb{B}''}u_{\pmb{B}'',\pmb{B}}=\e^{\frac{\pi i}{4}\mathrm{Mas}(F_{V,\pmb{B}},F_{V,\pmb{B}'},F_{V,\pmb{B}''})}.
\end{equation}
\end{enumerate}
Moreover, the action of $\varphi\in \Aut(\mathsf{G})$ on $\Gamma$ lifts to a bundle morphism on $\mathcal{H}_\Gamma$ inducing $\varphi$ on $I_\mathsf{G}$ if and only if $\zeta(\varphi)=0\in \check{H}^1(\Gamma,U(1))$. Additionally, for a line bundle $L\to \Gamma$, $\varphi$ lifts to a bundle morphism on $\mathcal{H}_\Gamma\otimes L$ inducing $\varphi$ on $I_\mathsf{G}$ if and only if $\zeta(\varphi)=\varphi^*(c_1(L))c_1(L^{-1})\in \check{H}^1(\Gamma,U(1))$.
\end{proposition}

\begin{proof}
The coboundary of $\zeta_0(\varphi)$ is the Cech cocycle given by 
\begin{align*}
\varphi^*(u_{\pmb{B},\pmb{B}'})&u_{\varphi(\pmb{B}),\varphi(\pmb{B}')}^{-1}\varphi^*(u_{\pmb{B'},\pmb{B}''})u_{\varphi(\pmb{B}'),\varphi(\pmb{B}'')}^{-1}\varphi^*(u_{\pmb{B}'',\pmb{B}})u_{\varphi(\pmb{B}''),\varphi(\pmb{B})}^{-1}=\\
&=t_{\pmb{B}}(\varphi)(\mathfrak{L}_{\pmb{B},\pmb{B}'}\mathfrak{L}_{\pmb{B}',\pmb{B}''}\mathfrak{L}_{\pmb{B}'',\pmb{B}})t_{\pmb{B}}(\varphi)^*\mathfrak{L}_{\varphi(\pmb{B}),\varphi(\pmb{B}')}\mathfrak{L}_{\varphi(\pmb{B}'),\varphi(\pmb{B}'')}\mathfrak{L}_{\varphi(\pmb{B}''),\varphi(\pmb{B})}=1,
\end{align*}
by Proposition \ref{covarianceforlions}. Item 1 follows. To prove item 2, we note that the group coboundary $\delta\zeta_0(\varphi,\varphi')$ is given by 
\begin{align*}
\delta&\zeta_0(\varphi,\varphi')=\varphi^*\zeta_0(\varphi')\zeta_0(\varphi\varphi')^{-1}\zeta_0(\varphi)=\\
&=((\varphi\varphi')^*(u_{\pmb{B},\pmb{B}'})\varphi^*u_{\pmb{B},\pmb{B}'}^{-1}(\varphi\varphi')^*(u_{\pmb{B},\pmb{B}'})^{-1}u_{\pmb{B},\pmb{B}'}\varphi^*(u_{\pmb{B},\pmb{B}'})u_{\pmb{B},\pmb{B}'}^{-1})_{\pmb{B},\pmb{B}'\in \mathfrak{JH}}=1.
\end{align*}
To prove item 3), we note that if $(u'_{\pmb{B},\pmb{B}'})_{\pmb{B},\pmb{B}'\in \mathfrak{JH}}\in \check{C}^1((U_{\pmb{B}})_{\pmb{B}\in \mathfrak{JH}},U(1))$ is another choice of $U(1)$-cochain solving \eqref{onononad}, then $(u'_{\pmb{B},\pmb{B}'}u^{-1}_{\pmb{B},\pmb{B}'})_{\pmb{B},\pmb{B}'\in \mathfrak{JH}}\in \check{Z}^1((U_{\pmb{B}})_{\pmb{B}\in \mathfrak{JH}},U(1))$. Denote the associated cohomology class by $x\in \check{H}^1(\Gamma,U(1))$ and let $\zeta'\in Z^1(\Aut(\mathsf{G}),\check{H}^1(\Gamma,U(1)))$ denote the cocycle constructed from $(u'_{\pmb{B},\pmb{B}'})_{\pmb{B},\pmb{B}'\in \mathfrak{JH}}$. Clearly, $\zeta'(\varphi)\zeta(\varphi)^{-1}=\varphi^*(x)x^{-1}$ so 
$$[\zeta]=[\zeta']\quad\mbox{in}\quad H^1(\Aut(\mathsf{G}),\check{H}^1(\Gamma,U(1))).$$

To prove the final remark, we note that Proposition \ref{covarianceforlions} (and a computation involving the definition of $\mathcal{H}_\Gamma$) implies that $\varphi$ lifts to $\mathcal{H}_\Gamma$ if and only if there exists a group cocycle $(\tau_{\varphi,\pmb{B}})_{\pmb{B}\in \mathfrak{JH}}$ such that 
$$\varphi^*(u_{\pmb{B},\pmb{B}'})u_{\varphi(\pmb{B}),\varphi(\pmb{B}')}=\tau_{\varphi,\pmb{B}}\tau_{\varphi,\pmb{B}'}^{-1}.$$
In particular,  $\varphi$ lifts to $\mathcal{H}_\Gamma$ if and only if $\zeta(\varphi)=[\delta (\tau_{\varphi,\pmb{B}})_{\pmb{B}\in \mathfrak{JH}}]=0$. The argument for the case when $\zeta(\varphi)=\varphi^*(c_1(L))c_1(L^{-1})$, for a line bundle $L\to \Gamma$, goes similarly.
\end{proof}

\begin{proposition}
\label{liftingcor}
Let $\mathsf{G}$ be a simply connected nilpotent Lie group admitting flat orbits. Assume that $\ghani\subseteq \Aut(\mathsf{G})$ is a subgroup. The following are equivalent
\begin{enumerate}
\item $[\zeta|_\ghani]=0\in H^1(\ghani,\check{H}^1(\Gamma,U(1)))$
\item Up to tensoring $\mathcal{H}_\Gamma$ with a line bundle $L\to \Gamma$ (solving $\zeta(\varphi)=\varphi^*(c_1(L))c_1(L^{-1})$ for each $\varphi\in \ghani$), each $\varphi\in \ghani$ lifts to $\mathcal{H}_\Gamma$ and the choice can locally in $\ghani$ be made so that the lift depends strongly continuously on $\varphi$.
\end{enumerate}
\end{proposition}

\begin{proof}
If $[\zeta|_\ghani]=0\in H^1(\ghani,\check{H}^1(\Gamma,U(1)))$, there exists an $x\in \check{H}^1(\Gamma,U(1))$ with $\zeta(\varphi)=\varphi^*(x)x^{-1}$ for all $\varphi\in \ghani$. There is a line bundle $L\to \Gamma$ with $c_1(L)$ mapped to $x$ under $\check{H}^1(\Gamma,U(1))\cong \check{H}^2(\Gamma,\Z)$. By Proposition \ref{zetanandndlaf}, we can lift all $\varphi$ in a strongly continuous manner. The converse follows in the same way.
\end{proof}

\begin{definition}
\label{modifiedhgammspace}
Let $\mathsf{G}$ be a simply connected nilpotent Lie group admitting flat orbits and $\ghani\subseteq \Aut(\mathsf{G})$ a subgroup such that $[\zeta|_\ghani]=0\in H^1(\ghani,\check{H}^1(\Gamma,U(1)))$. 
\begin{enumerate}
\item Let $L_\ghani\to \Gamma$ be a line bundle such that $\zeta(\varphi)=\varphi^*(c_1(L_\ghani))c_1(L_\ghani)^{-1}$ for each $\varphi\in \ghani$, and define 
$$\mathcal{H}_{\Gamma,L_\ghani}:=\mathcal{H}_\Gamma\otimes L_\ghani.$$
\item For each $\varphi\in \ghani$, pick lifts $\alpha(\varphi)$ to $\mathcal{H}_{\Gamma,L_\ghani}$ as in Proposition \ref{liftingcor}. Define the map 
$$c_\ghani:\ghani\times \ghani\to C(\Gamma,U(1)), \quad (\varphi,\varphi')\mapsto \alpha(\varphi)\varphi^*(\alpha(\varphi'))\alpha(\varphi\varphi')^{-1}.$$
\end{enumerate}
\end{definition}

\begin{proposition}
\label{thecocycleforhandnd}
Let $\mathsf{G}$ be a simply connected nilpotent Lie group admitting flat orbits. Assume that $\ghani\subseteq \Aut(\mathsf{G})$ is a subgroup such that $[\zeta|_\ghani]=0\in H^1(\ghani,\check{H}^1(\Gamma,U(1)))$ and construct $\mathcal{H}_{\Gamma,L_\ghani}$ and $c_\ghani$ as in Definition \ref{modifiedhgammspace}. Then $c_\ghani$ satisfies the following:
\begin{enumerate}
\item $c_\ghani\in Z^2(\ghani,C(\Gamma,U(1)))$ is a group cocycle for the induced $\ghani$-action on $C(\Gamma,U(1))$. 
\item The class $[c_\ghani]\in H^2(\ghani,C(\Gamma,U(1)))$ is independent of the choice of $L_\ghani$ and lifts $\alpha(\varphi)$.
\item The lifts of $\ghani$ to $\mathcal{H}_{\Gamma,L_\ghani}$ can be choosen to define a strongly continuous $\ghani$-action if and only if  $[c_\ghani]=0\in H^2(\ghani,C(\Gamma,U(1)))$.
\end{enumerate}
\end{proposition}

\begin{proof}
Item 1 is a short computation. Item 2 follows from noting that the lift $\alpha$ is unique up to a $C(\Gamma,U(1))$-factor, so two different choices will give rise to cohomologous $c_\ghani$. 

To prove item 3, we first note that the lifts $\alpha(\varphi)$ are uniquely determined by $\varphi$ up to a $C(\Gamma,U(1))$-factor. We have that $[c_\ghani]=0\in H^2(\ghani,C(\Gamma,U(1)))$ if and only if there is a $\tau_\ghani:\ghani\to C(\Gamma,U(1))$ such that 
$$c_\ghani(\varphi,\varphi)=\tau_\ghani(\varphi)\varphi^*(\tau_\ghani(\varphi'))\tau_\ghani(\varphi\varphi')^{-1}.$$ 
Existence of such a $\tau_\ghani$ is clearly equivalent to existence of a modification of the lifts $\alpha(\varphi)$ by a $C(\Gamma,U(1))$-factor such that $\alpha(\varphi)\varphi^*\alpha(\varphi')=\alpha(\varphi\varphi')$.
\end{proof}

\begin{theorem}
\label{charofequimorita}
Let $\mathsf{G}$ be a simply connected nilpotent Lie group admitting flat orbits and $\ghani\subseteq \Aut(\mathsf{G})$. Then the following are equivalent:
\begin{enumerate}
\item There exists a Hilbert space bundle $\mathcal{H}\to \Gamma$ with a fibrewise strongly continuous $\mathsf{G}$-action, with the unitary equivalence class of the representation in the fibre over $\xi\in \Gamma$ is exactly $\xi$, such that the $\ghani$-action on $\Gamma$ lifts to a strongly continuous action on $\mathcal{H}\to \Gamma$ and a $\ghani$-equivariant $*$-isomorphism
$$\pi_{\musFlat}:I_\mathsf{G}\to C_0(\Gamma,\mathbb{K}(\mathcal{H})).$$
\item It holds that $[\zeta|_\ghani]=0\in H^1(\ghani,\check{H}^1(\Gamma,U(1)))$ and $[c_\ghani]=0\in H^2(\ghani,C(\Gamma,U(1)))$.
\item There exists a $C_0(\Gamma)$-linear $\ghani$-equivariant Morita equivalence 
$$I_\mathsf{G}\sim_M C_0(\Gamma).$$
\end{enumerate}
\end{theorem}

\begin{proof}
That item 1 and 3 are equivalent follows from Kasparov's stabilization theorem and the definition of equivariant Morita equivalence. Clearly item 1 implies item 2, and item 2 implies item 1 by Proposition \ref{thecocycleforhandnd} (with $\mathcal{H}=\mathcal{H}_{\Gamma,L_\ghani}$).
\end{proof}

\begin{remark}
Let us take a moment to summarize the constructions above. To construct an equivariant bundle of flat representations on a simply connected nilpotent Lie group one first constructs $\mathcal{H}_\Gamma$ as in Proposition \ref{descriptionofddinvariant}. Secondly, if $\ghani\subseteq \Aut(\mathsf{G})$ satisfies $[\zeta|_\ghani]=0\in H^1(\ghani,\check{H}^1(\Gamma,U(1)))$ we can modify $\mathcal{H}_\Gamma$ by a line bundle to $\mathcal{H}_{\Gamma,L_\ghani}$ as in Proposition \ref{liftingcor} and lift the action of $\ghani$ to $\mathcal{H}_{\Gamma,L_\ghani}$. If $[c_\ghani]=0\in H^2(\ghani,C(\Gamma,U(1))$, Proposition \ref{thecocycleforhandnd} allow us modify each $\alpha(\varphi)$, $\varphi\in \ghani$, by a scalar valued function such that $\ghani$ acts on $\mathcal{H}_{\Gamma,L_\ghani}$. 
\end{remark}

\begin{remark}
In \cite{crockerkumjianetal97}, an equivariant Brauer group was defined. For a second countable group $\ghani$ acting on a second countable space $T$, the equivariant Brauer group $\mathrm{Br}_\ghani(T)$ is defined as the $T$-linear $\ghani$-equivariant Morita equivalence classes of $\ghani$-continuous trace algebras $A$ over $T$. The product in $\mathrm{Br}_\ghani(T)$ is defined by $\otimes_{C_0(T)}$ and the unit by $C_0(T)$. In \cite[Theorem 5.1]{crockerkumjianetal97}, it was proven that there is a filtration $0\subseteq B_1\subseteq B_2\subseteq B_3=\mathrm{Br}_\ghani(T)$ and morphisms 
\begin{align*}
\delta_{\rm DD} : B_3=\mathsf{Br}_\ghani(T) &\rightarrow \widecheck{H}^3(T, \mathbb{Z}),\\
\zeta : B_2=\ker\delta_{\rm DD} &\rightarrow H^1(\ghani, \widecheck{H}^1(T, U(1))),\\
c: B_1=\ker\zeta &\rightarrow H^2(\ghani, C(T, U(1))),
\end{align*}
such that $A\sim_M C_0(T)$ ($\ghani$-equivariantly over $T$) if and only if $\delta_{\rm DD}(A)=0\in H^3(T,\Z)$, $\zeta(A)=0\in H^1(\ghani,\check{H}^1(T,U(1)))$ and $c(A)=0\in H^2(\ghani,C(T,U(1)))$. As such, Theorem \ref{charofequimorita} is a special case of the results in \cite{crockerkumjianetal97}.
\end{remark}

The description above can be made substantially simpler in the special case that there is a global bundle of polarizations over all of $\Gamma$. This occurs for instance when the $\Gamma$ only consists of one fine strata or the bundle of Vergne polarizations is a constant subbundle.

\begin{proposition}
\label{bundleofpeolalss}
Let $\mathsf{G}$ be a simply connected nilpotent Lie group admitting flat orbits. Assume that there exists a sub-bundle $F\subseteq \Gamma\times \mathfrak{g}$ such that the fibre over $\xi\in \Gamma$ is a polarization of $\xi$. Then the following holds:
\begin{itemize}
\item It holds that $[\zeta]=0\in H^1(\Aut(\mathsf{G}),\check{H}^1(\Gamma,U(1)))$.
\item Define the associated Hilbert space bundle $\mathcal{H}_F\to \Gamma$, as in Proposition \ref{bundleofhspspac}, with a strongly continuous fibrewise $\mathsf{G}$-action such that $[\mathcal{H}_{F,\xi}]=\xi$ for each $\xi\in \Gamma$. Then each $\varphi\in \Aut(\mathsf{G})$ lifts to the morphism given by the composition 
$$\mathcal{H}_{F}\xrightarrow{t_F(\varphi)} \mathcal{H}_{F_\varphi}\xrightarrow{\mathfrak{L}^{(0)}_{F,F_\varphi}}\mathcal{H}_F.$$
where $F_\varphi=(\varphi^*\times \varphi^{-1})(F)\to \Gamma$ is the transformed bundle of polarizations (cf. Proposition \ref{covarianceforlions}.a) and $t_F(\varphi):\mathcal{H}_{F}\to \mathcal{H}_{F_\varphi}$ the map $(\xi,f)\mapsto (\varphi^*\xi,f\circ\varphi^*)$ (cf. Proposition \ref{covarianceforlions}.b) and  $\mathfrak{L}^{(0)}_{F,F_\varphi}: \mathcal{H}_{F_\varphi}\to\mathcal{H}_F$ is the Lion transform (see Proposition \ref{lionsintertwiners}). 
\item The cocycle 
$$c_{\Aut(\mathsf{G})}:\Aut(\mathsf{G})\times \Aut(\mathsf{G})\to C(\Gamma,U(1)),$$
is the image of $c_{F,{\rm fl}}\in Z^2(\Aut(\mathsf{G}),C(\Gamma,\Z/8))$ defined by 
$$c_{F,{\rm fl}}(\varphi,\varphi'):=\mathrm{Mas}(F_\varphi,\varphi^*F_{\varphi'},F_{\varphi\varphi'})\!\!\!\mod 8\Z.$$
\end{itemize}
\end{proposition}

\begin{proof}
The first item follows from the second item by Proposition \ref{zetanandndlaf}. The second and third item follows from Proposition \ref{lionsintertwiners}.
\end{proof}

\begin{proposition}
Let $\mathsf{G}$ be a simply connected nilpotent Lie group admitting flat orbits. Let $d:=\mathrm{codim}(\mathfrak{z})/2$. Assume that there exists a sub-bundle $F\subseteq \Gamma\times \mathfrak{g}$ such that the fibre over $\xi\in \Gamma$ is a polarization of $\xi$ and $T_F:\mathcal{H}_F\to \Gamma\times L^2(\R^d)$ be a global trivialization. Let $\ghani\subseteq \Aut(\mathsf{G})$ be a subgroup. Then the following are equivalent:
\begin{enumerate}
\item $[c_{\rm fl}|_{\ghani}]\in \ker(H^2(\ghani,C(\Gamma,\Z/8))\to H^2(\ghani,C(\Gamma,U(1))))$ 
\item The natural homomorphism $\ghani\to \Aut(I_\mathsf{G})\cong C(\Gamma,PU(L^2(\R^d)))$ (via $\mathrm{Ad}(T_F)$) lifts to a homomorphism $\ghani\to C(\Gamma,U(L^2(\R^d)))$. 
\item The $\ghani$-action on $I_\mathsf{G}$ is induced from a strongly continuous $\ghani$-action on $\mathcal{H}_F$.
\end{enumerate}
\end{proposition}

\begin{proof}
The proposition follows from Proposition \ref{thecocycleforhandnd} and \ref{bundleofpeolalss}.
\end{proof}

\section{Examples of graded nilpotent Lie groups}
\label{subsectionexamples}

In this subsection we give some salient examples of simply connected, nilpotent Lie groups geared towards the geometric scenarios arising later in the monograph. The examples aim at exemplifying the contents of the recent subsections with a focus on describing gradings, computing the top strata $\Gamma$, the Hilbert space bundle $\mathcal{H}_\Gamma\to \Gamma$, the ideal $I_\mathsf{G}\subseteq C^*(\mathsf{G})$ and in particular how to single out a compact subgroup $\ghani\subseteq \Aut(\mathsf{G})$ lifting to an action on $\mathcal{H}_\Gamma$.

\subsection{Step 1 graded nilpotent Lie groups}
\label{step1example}
Nilpotent Lie groups of step length 1 are just the abelian Lie groups. So the term simply connected, nilpotent, graded Lie group of step length 1 is a lengthy description of the Lie group $\mathsf{G}=\R^n$, for some $n\in \N$, equipped with a vector space grading 
$$\mathfrak{g}\equiv \R^n=\bigoplus_{p<0} \R^{n_p}[p],$$
where we write $\R^{n_p}[p]$ to indicate the space $\R^{n_p}$ where the grading on this subspace is $p$. Clearly any such decomposition with $\sum_{p<0} n_p=n$ defines a grading of the Lie algebra $\R^n$. Since $\mathsf{G}=\R^n$ is abelian, we have that 
$$\Aut(\mathsf{G})=GL_n(\R)\quad\mbox{and}\quad \Aut_{\rm gr}(\mathsf{G})=\bigoplus_{p<0} GL_{n_p}(\R)$$
We can also conclude that
$$\hat{\mathsf{G}}=\mathfrak{g}^*=\Gamma=\Xi,$$
which only consists of one strata. In particular, $C^*(\mathsf{G})=I=C_0(\mathfrak{g}^*)=C_0(\Xi)=C_0(\Gamma)$. We have that $\mathcal{H}_V=\Gamma\times \C$ is a trivial line bundle and the $\Aut(\mathsf{G})$-action on $I$ lifts to an action on $\mathcal{H}_V$ which is trivial on the fibre, so in particular $\delta_{\rm DD}(I_\mathsf{G})=0$, $\zeta=0$, and $[c_{\rm flat}]=0\in H^2(GL_n(\R),C_0(\Gamma,\C^\times))$. This example is likely too trivial to improve the understanding of the reader, as the whole theory collapses to the standard Fourier transform on a graded vector space, but we have included it as it contains the local model for foliations as a special case (see more below in Example \ref{regfoliationexala}) and isolates the effects of choosing a grading.

\subsection{Step 2 graded nilpotent Lie groups}
\label{step2subsectionexamples}

Let us consider the case of step length $2$. This situation arises geometrically when for instance having a non-integrable length 2 filtration of a manifold, i.e. one specified subbundle of the tangent bundle which is not integrable. 

If $\mathfrak{g}$ is a nilpotent Lie algebra of step length $2$, then $\overline{\mathfrak{g}}:=\mathfrak{g}/[\mathfrak{g},\mathfrak{g}]$ is an abelian Lie algebra. In particular, $\mathfrak{g}$ as a Lie algebra is determined by a subspace $C(\mathfrak{g})\subseteq \mathfrak{g}$ and a surjective map $\omega:\overline{\mathfrak{g}}\wedge \overline{\mathfrak{g}}\to C(\mathfrak{g})$, where $\overline{\mathfrak{g}}:=\mathfrak{g}/C(\mathfrak{g})$, by 
$$[X,Y]=\omega(\overline{X},\overline{Y}),$$
where we use the notation $\overline{X}:=X+C(\mathfrak{g})$. Surjectivity of $\omega$ ensures that $[\mathfrak{g},\mathfrak{g}]=C(\mathfrak{g})$. We say that  $\mathfrak{g}$ is a 2-step nilpotent lie algebra of type $(p, q)$ if 
\[ \dim  [\mathfrak{g}, \mathfrak{g}] = p\quad\mbox{and}\quad \dim\overline{\mathfrak{g}}= q. \]
For more details, see \cite{eberlein2step}. 
Following Remark \ref{flatorbisintosform}, we let $X_1, \ldots, X_p$ denote a basis of $[\mathfrak{g},\mathfrak{g}]$ and we define $\overline{\omega}_k:\overline{\mathfrak{g}}\wedge \overline{\mathfrak{g}}\to \R$, $k=1,\ldots, p$, by 
\[ [X, Y] = \sum_{k = 1}^p \overline{\omega}_k(\overline{X},\overline{Y})X_k. \]
Surjectivity of $\omega$ ensures that the forms $\overline{\omega}_1, \ldots, \overline{\omega}_p$ are linearly independent in $(\overline{\mathfrak{g}}\wedge \overline{\mathfrak{g}})^*$. The basis induces an inner product on $\mathfrak{g}$ and we can identify $\overline{\mathfrak{g}}$ with the orthogonal complement of $C(\mathfrak{g})\subseteq \mathfrak{g}$ and $(\overline{\mathfrak{g}}\wedge \overline{\mathfrak{g}})^*$ with $\mathfrak{so}( \overline{\mathfrak{g}})$. We identify the collection $\overline{\omega}_1, \ldots, \overline{\omega}_p\in (\overline{\mathfrak{g}}\wedge \overline{\mathfrak{g}})^*$ with a collection $\omega^1, \ldots, \omega^p\in \mathfrak{so}( \overline{\mathfrak{g}})$ and let $W\subseteq \mathfrak{so}(\overline{\mathfrak{g}})$ denote the space spanned by $\omega^1, \ldots, \omega^p$. 

For an arbitrary finite-dimensional inner product space $V$ and a subspace $W\subseteq \mathfrak{so}(V)$, we can equip $V\oplus W$ with a Lie bracket by declaring $W$ to be central and for $X,Y\in V$ the element $[X,Y]\in W$ is given by that  
$$([X,Y],Z)_{\mathfrak{so}(V)}=(Z(X),Y)_V \quad\forall Z\in W.$$
We denote this Lie algebra by $V_W$. It is readily verified that $V_W$ is a nilpotent Lie algebra of step length $2$ of type $(\dim(W),\dim(V))$ with $C(V_W)=W$. The reader should note that $V_{\mathfrak{so}(V)}$ is the free step 2 nilpotent Lie algebra on $V$ and that for any $W\subseteq \mathfrak{so}(V)$ the following sequence is exact
\begin{equation}
\label{sesfortepspso}
0\to W^\perp\to V_{ \mathfrak{so}(V)}\to V_W\to 0.
\end{equation}
It was proven in \cite{eberlein2step} that this construction up to isomorphism exhausts all nilpotent Lie algebras of step length $2$. 

\begin{proposition}[Proposition 3.1.1 of \cite{eberlein2step}]
\label{charofspeepoed}
Let $\mathfrak{g}$ be a step $2$ nilpotent Lie algebra and type $(p,q)$ with a choice of inner product $(\cdot, \cdot)_{\overline{\mathfrak{g}}}$ on $\overline{\mathfrak{g}}$ and a Jordan-Hölder basis $X_1,\ldots, X_n$ with $\mathfrak{g}_p=C(\mathfrak{g})$. Define $\omega^1, \ldots, \omega^p\in \mathfrak{so}( \overline{\mathfrak{g}})$ and $W\subseteq \mathfrak{so}(\overline{\mathfrak{g}})$ as above. For the basis $w_1, \ldots, w_p\in W$ such that $(\omega^i,w_j)_{\mathfrak{so}(\overline{\mathfrak{g}})}=\delta_{i,j}$, the mapping 
$$\mathfrak{g}\to \overline{\mathfrak{g}}_W, \quad X_j\mapsto 
\begin{cases} 
w_j, \; j=1,\ldots, p, \\ 
\overline{X}_j, \; j=p+1,\ldots, n,
\end{cases}$$
is a Lie algebra isomorphism.
\end{proposition}

\begin{remark}
\label{centervscomm}
We note that Proposition \ref{charofspeepoed} implies that for a step $2$ nilpotent Lie algebra $\mathfrak{g}$, it holds that 
$$\mathfrak{z}=\{X\in \mathfrak{g}: w(\overline{X})=0\forall w\in W\},$$
where $\mathfrak{z}$ denotes the center of $\mathfrak{g}$. In particular, $C(\mathfrak{g})=\mathfrak{z}$ if and only if for any $X\in \overline{\mathfrak{g}}$, there is a $w\in W$ with $w(X)\neq 0$. 
\end{remark}

From this characterization of step $2$ nilpotent Lie algebras, \cite{eberlein2step} ensures an explicit description of $\Aut(\mathfrak{g})$. For a step $2$ nilpotent Lie algebra $\mathfrak{g}$ with a choice of inner product on $\overline{\mathfrak{g}}$, with $W\subseteq \mathfrak{so}(\overline{\mathfrak{g}})$ as above, we define the Lie group
$$H_\mathfrak{g}:=\{g\in GL(\overline{\mathfrak{g}}): g^tW g\subseteq W\}.$$
By universality, the group $GL(\overline{\mathfrak{g}})$ acts on $\overline{\mathfrak{g}}_{\mathfrak{so}(\overline{\mathfrak{g}})}$ -- the free step 2 nilpotent Lie algebra on $\overline{\mathfrak{g}}$. We define $u_H:H_\mathfrak{g}\to \Aut(\mathfrak{g})$ by noting that the short exact sequence \eqref{sesfortepspso} ensures that the $H_\mathfrak{g}$-action descends to $\mathfrak{g}\cong \overline{\mathfrak{g}}_{\mathfrak{so}(\overline{\mathfrak{g}})}/W^\perp$. The vector space $\Hom(\overline{\mathfrak{g}},C(\mathfrak{g}))$ acts as Lie algebra automorphisms on $\mathfrak{g}$ by $u_t(A)(v\oplus w):=v\oplus (Av+w)$ for $v\in \overline{\mathfrak{g}}$, $w\in C(\mathfrak{g})$ and $A\in \Hom(\overline{\mathfrak{g}},C(\mathfrak{g}))$. It is readily verified that for $g\in H_\mathfrak{g}$ and $A\in \Hom(\overline{\mathfrak{g}},C(\mathfrak{g}))$,
$$u_H(g)u_t(A)u_H(g)^{-1}=u_t(g.A),$$
where 
\begin{equation}
\label{hgactonojomo}
(g.A)(v):=gA(gv)g^t.
\end{equation} Therefore $u_t$ and $u_H$ induces a well defined group homomorphism $u_t\rtimes u_H:\Hom(\overline{\mathfrak{g}},C(\mathfrak{g}))\rtimes H_\mathfrak{g}\to \Aut(\mathfrak{g})$ where the crossed product is with respect to the $H_g$-action on $\Hom(\overline{\mathfrak{g}},C(\mathfrak{g}))$ given in \eqref{hgactonojomo}.

\begin{proposition}[Proposition 3.4.2 of \cite{eberlein2step}]
\label{autoforstep2decoss}
Let $\mathfrak{g}$ be a step $2$ nilpotent Lie algebra with a choice of inner product on $\overline{\mathfrak{g}}:=\mathfrak{g}/C(\mathfrak{g})$. Then 
$$u_t\rtimes u_H:\Hom(\overline{\mathfrak{g}},C(\mathfrak{g}))\rtimes H_\mathfrak{g}\to \Aut(\mathfrak{g}),$$
is a group isomorphism.
\end{proposition}

\begin{remark}
The moduli space of step 2 nilpotent Lie algebras of type $(p,q)$ can by \cite{eberleinother2step} be identified with $G(p,\mathfrak{so}(q,\R))/GL_q(\R)$, where $G(p,\mathfrak{so}(q,\R))$ denotes the Grassmannian manifold of $p$-dimensional subspaces in $\mathfrak{so}(q,\R)$. The automorphism group of a step 2 nilpotent Lie algebra of type $(p,q)$ is quite small (see \cite[Proposition 3.4.3]{eberlein2step}) for most values of $p$ and $q$ in the generic case (for a dense open subset of the moduli space). In particular, for most values of $p$ and $q$, step 2 nilpotent Lie algebras of type $(p,q)$ generically do not admit gradings.
\end{remark}

The geometric setting our Lie algebras arise from ensures that the Lie algebras are graded which changes what should be meant with generic. Let us describe possible gradings on step 2 nilpotent Lie algebras. 

\begin{proposition}
Let $\mathfrak{g}$ be a step $2$ nilpotent Lie algebra. If $\delta$ is a dilation on $\mathfrak{g}$ it is uniquely determined by the induced dilation $\overline{\delta}$ on $\overline{\mathfrak{g}}$ from the isomorphism $\mathfrak{g}\cong \overline{\mathfrak{g}}_W=\overline{\mathfrak{g}}\oplus C(\mathfrak{g})$ (from Proposition \ref{charofspeepoed}) and the identity 
$$\delta_t[X,Y]=\omega(\overline{\delta}_t(\overline{X}),\overline{\delta}_t(\overline{Y})),\quad\mbox{for $[X,Y]\in C(\mathfrak{g})$}.$$
For any grading on $\mathfrak{g}$, we have that $\Aut_{\rm gr}(\mathfrak{g})\subseteq u_H(H_\mathfrak{g})$ (see Proposition \ref{autoforstep2decoss} for notation).

Moreover, a dilation $\overline{\delta}$ on $\overline{\mathfrak{g}}$ is induced from a dilation $\delta$ on $\mathfrak{g}$ if and only if the grading $\overline{\mathfrak{g}}=\bigoplus_{p<0} \overline{\mathfrak{g}}_p$ constructed from $\overline{\delta}$ (as in Proposition \ref{gradingvsdilationsprop}) satisfies that
\begin{equation}
\label{lineinindedp}
\left(\sum_{k+l<p} \omega(\overline{\mathfrak{g}}_k,\overline{\mathfrak{g}}_l)\right)\bigcap \left(\sum_{k+l=p} \omega(\overline{\mathfrak{g}}_k,\overline{\mathfrak{g}}_l)\right)=0 \quad\forall p<0.
\end{equation}
\end{proposition}

\begin{proof}
The first part of the proposition follows from that a dilation is a homomorphism $\delta:\R_+\to \Aut(\mathfrak{g})$ and the Lie bracket defines a surjection $\omega:\overline{\mathfrak{g}}\wedge \overline{\mathfrak{g}}\to C(\mathfrak{g})$. 

The second part follows from noting that the condition \eqref{lineinindedp} is equivalent to that the spaces 
$$C(\mathfrak{g})_p:=\sum_{k+l=p} \omega(\overline{\mathfrak{g}}_k,\overline{\mathfrak{g}}_l),$$
define a decomposition 
$$C(\mathfrak{g})=\bigoplus_{p<0} C(\mathfrak{g})_p.$$
We can therefore define a dilation as in Proposition \ref{gradingvsdilationsprop} from grading $\mathfrak{g}$ by 
$$\mathfrak{g}_p:=\overline{\mathfrak{g}}_p\oplus C(\mathfrak{g})_p.$$
The decomposition $\mathfrak{g}=\bigoplus_{p<0} \mathfrak{g}_p$ is indeed a grading as it follows from the construction that 
$$[\mathfrak{g}_p,\mathfrak{g}_q]=\omega(\overline{\mathfrak{g}}_p,\overline{\mathfrak{g}}_q)\subseteq C(\mathfrak{g})_{p+q}\subseteq \mathfrak{g}_{p+q}.$$
\end{proof}

We now turn our attention to representation theory of a step $2$ simply connected nilpotent Lie group. We denote the Lie group by $\mathsf{G}$ and its Lie algebra by $\mathfrak{g}$. As above we fix a Jordan-Hölder basis with $\mathfrak{g}_p=C(\mathfrak{g})$ and $\mathfrak{g}_{\dim(\mathfrak{z})}=\mathfrak{z}$. We let $\omega_1,\ldots, \omega_p:(\mathfrak{g}/\mathfrak{z})\wedge(\mathfrak{g}/\mathfrak{z}) \to \R$ denote the forms induced from $\overline{\omega}_1,\ldots, \overline{\omega}_p:\overline{\mathfrak{g}}\wedge \overline{\mathfrak{g}}\to \R$. The Pfaffian function (see Proposition \ref{descripfiogogda} and Remark \ref{flatorbisintosform}) on $\mathfrak{z}^*$ can be described in the coordinates of the Jordan-Hölder basis as the polynomial of degree $\mathrm{codim}(\mathfrak{z})/2$ given by 
$$\mathsf{Pf}(\xi)=\mathsf{Pf}\left(\sum_{j=1}^p \xi(X_j)\omega_j\right).$$
As such, if $\mathsf{G}$ admits flat orbits (i.e. the Pfaffian polynomial is non-vanishing on $\mathfrak{z}^*$), the top stratum $\Gamma\subseteq \hat{\mathsf{G}}$ can be described by the transversal $\Lambda\subseteq \mathfrak{z}^*$ which is given as follows:
\begin{align}
\label{pfaffiandnadnstep2}
\Lambda & = \left\{ \xi=\sum_{j=1}^{\dim(\mathfrak{z})}\xi(X_j)\xi_j \in \mathfrak{z}^* : \mathsf{Pf}\left(\sum_{j=1}^p \xi(X_j)\omega_j\right) \neq 0 \right\} = \\ 
\nonumber
& =\left\{ \xi=\sum_{j=1}^{\dim(\mathfrak{z})}\xi(X_j)\xi_j \in \mathfrak{z}^* : \det\left(\sum_{j=1}^p \xi(X_j)\omega_j\right) \neq 0\right \} = \Lambda_C\times (\mathfrak{z}/C(\mathfrak{g}))^*,
\end{align}
where
$$\Lambda_C=\left\{ \xi=\sum_{j=1}^{p}\xi(X_j)\xi_j \in C(\mathfrak{g})^* : \mathsf{Pf}\left(\sum_{j=1}^p \xi(X_j)\omega_j\right) \neq 0 \right\}.$$

\begin{remark}
If $C(\mathfrak{g})=\mathfrak{z}$, i.e. $\overline{\mathfrak{g}}=\mathfrak{g}/\mathfrak{z}$, then we can also write 
\begin{equation}
\label{lambdaforcentercomm}
\Lambda=w_\mathfrak{z}^{-1}(W\cap GL(\overline{\mathfrak{g}})),
\end{equation}
where $w_\mathfrak{z}:\mathfrak{z}^*\to W$ is defined by $w_\mathfrak{z}(\xi):=\sum_{j=1}^{\dim(\mathfrak{z})} \xi(X_j)\omega^j$. We note that by Remark \ref{centervscomm}, $C(\mathfrak{g})=\mathfrak{z}$ if and only if for any $X\in g$, $w(\overline{X})=0$ for all $w\in W$ implies that $X\in C(\mathfrak{g})$. In light of Equation \eqref{lambdaforcentercomm}, and injectivity of the $W$-action on $\mathfrak{g}$, it is tempting to suspect that the equality $C(\mathfrak{g})=\mathfrak{z}$ implies the existence of flat orbits, but fails since even if $C(\mathfrak{g})=\mathfrak{z}$ there need not exist an individual element of $W\subseteq \mathfrak{so}(\overline{\mathfrak{g}})$ acting injectively on  $\overline{\mathfrak{g}}$. This fails in the example of Remark \ref{noflatexample} below.
\end{remark}

\begin{theorem}
Let $\mathsf{G}$ be a step 2 simply connected nilpotent Lie group of type $(p,q)$ admitting flat orbits. Set $d:=\mathrm{codim}(\mathfrak{z})/2$. We consider $\ghani=u_H(H_\mathfrak{g})\subseteq \Aut(\mathsf{G})$. Then the following holds for the ideal of flat orbits:
\begin{enumerate}
\item It holds that $[\zeta|_\ghani]=0\in H^1(\ghani,\check{H}^1(\Gamma,U(1)))$.
\item It holds that $[c_\ghani]$ is the image of the analogous cohomology class $[c_{GL(q)}]$, defined for the free step 2 simply connected nilpotent Lie group on a $q$-dimensional vector space, under the mapping 
$$H^2(GL(q),C(GL(q)/Sp(q),U(1)))\to H^2(\ghani,C(\Gamma,U(1))),$$
defined from the inclusion $\ghani\hookrightarrow GL(q)$ and the inclusion $\Gamma\hookrightarrow GL(q)\cap \mathfrak{so}(q)\cong GL(q)/Sp(q)$.
\end{enumerate}
\end{theorem}

\begin{proof}
Item 1) follows from that $[\zeta|_\ghani]$ is the pullback of 
$$[\zeta|_{GL(q)}]\in H^1(GL(q),\check{H}^1(GL(q)/Sp(q),U(1)))=0,$$ 
defined from the quotient map onto $\mathfrak{g}$ from the free  step 2 nilpotent Lie algebra on a $q$-dimensional vector space. Here $GL(q)\cap \mathfrak{so}(q)\cong GL(q)/Sp(q)$ is the space of flat orbits of the free step 2 nilpotent Lie algebra on a $q$-dimensional vector space. Indeed, by homotopy invariance we have that $\check{H}^1(GL(q)/Sp(q),U(1))\cong \Z$ and $H^1(GL(q),\Z)=0$. Item 2) is now clear from construction.
\end{proof}

We let $\overline{G}=G/[G,G]$ denote abelianization of $G$ which coincides with the simply connected Lie group associated with $\overline{\mathfrak{g}}$. We identify $\overline{G}$ with $\overline{\mathfrak{g}}$. Pick a Jordan-Hölder basis $\pmb{B}\in \mathfrak{JH}$. By the proof of Proposition \ref{vbfv}, $(X_j)_{j\in K_{\pmb{B}}}$ spans a complement to the Vergne polarization $\mathfrak{h}_{V,\pmb{B}}(\xi)$ for any $\xi\in U_{\pmb{B}}\subseteq \Lambda$. Write $K_{\pmb{B}}=\{j_1<\cdots <j_d\}$ where $d=\mathrm{codim}(\mathfrak{z})/2$. From this fact we deduce the following proposition.

\begin{proposition}
Let $G$ be a step $2$ nilpotent Lie group admitting flat orbits. Fix a Jordan-Hölder basis $\pmb{B}\in \mathfrak{JH}$ with $K_{\pmb{B}}=\{j_1<\cdots <j_d\}$, where $d=\mathrm{codim}(\mathfrak{z})/2$. Define the mapping $\rho_{\pmb{B}}:\R^d\to G$ by 
$$\rho_{\pmb{B}}(x_1,\ldots, x_d):=\exp\left(\sum_{i=1}^d x_iX_{j_i}\right).$$
Then it holds that 
\begin{itemize}
\item The composition $\R^d\xrightarrow{\rho_{\pmb{B}}}G\to \overline{G}$ is linear.
\item For each $\xi\in U_{\pmb{B}}\subseteq \Lambda$, the composition $\R^d\xrightarrow{\rho_{\pmb{B}}}G\to G/H_{V,\pmb{B}}(\xi)$ is a polynomial diffeomorphism.
\end{itemize}
\end{proposition}

\begin{proposition}
\label{descriptandstep2}
Let $G$ be a step $2$ nilpotent Lie group admitting flat orbits with a Jordan-Hölder basis $\pmb{B}\in \mathfrak{JH}$. For $\xi\in U_{\pmb{B}}\subseteq \Lambda$, the mapping 
$$T_{V,\pmb{B}}(\xi):\mathcal{H}_{F_{V,\pmb{B}}(\mathcal{O},\xi),\xi}(\mathcal{O}_\xi)\to L^2(\R^d), \quad T_{V,\pmb{B}}(\xi)f(x):=f(\rho_{\pmb{B}}(x)\xi),$$
is a well defined unitary that fits in a unitary trivialization of Hilbert space bundles over $\Gamma$
$$T_{V,\pmb{B}}:\mathcal{H}_{V,\pmb{B}}\to U_{\pmb{B}}\times L^2(\R^d), \quad T_{V,\pmb{B}}(\xi,f):=(\xi,T_{V,\pmb{B}}(\xi)f).$$

The $G$-action on the trivial Hilbert space bundle $U_{\pmb{B}}\times L^2(\R^d)$ induced by $T_{V,\pmb{B}}$ and the $G$-action on $\mathcal{H}_{V,\pmb{B}}$ (see Proposition \ref{bundleofhspspac}) takes the form of a continuous homomorphism $\tilde{\pi}_{\pmb{B}}:G\to C(U_{\pmb{B}},U(L^2(\R^d)))$ where 
$$\tilde{\pi}_{\pmb{B},\xi}(g)f(x)=\mathrm{e}^{i\xi(\log'_\xi(g)+[\log(g),\log(\rho(x))])}f(x+\log_\xi^0(g)),\quad \forall g\in G, \;\xi\in U_{\pmb{B}}, \; f\in L^2(\R^d),$$
where we write $\log^0_\xi(g)$ for the projection of $\log(g)$ onto the linear span of $(X_j)_{j\in K_{\pmb{B}}}$ (along the Vergne polarization) and $\log'_\xi(g):=\log(g)-\log^0_\xi(g)\in \mathfrak{h}_{V,\pmb{B}}(\xi)$.
\end{proposition}

This proposition is a special case of Lemma \ref{descofreprefield}.

\begin{proof}
The statements concerning $T_V$ follows from inspecting the proof of Proposition \ref{bundleofhspspac}. Let us turn to the form of the representation $\tilde{\pi}$. Note that $\log\rho(x)=\sum_{i=1}^dx_iX_{j_i}$. For $X=X_0+X'\in \mathfrak{g}$, with $X_0$ being the projection of $X$ onto the linear span of $(X_j)_{j\in K_{\pmb{B}}}$ (along the Vergne projection) and $X'\in \mathfrak{h}_{V,\pmb{B}}(\xi)$, and $x\in \R^d$, we have that 
\begin{align*}
\mathrm{e}^{X}\rho(x)=\mathrm{e}^{X_0+X'}\mathrm{e}^{\log\rho(x)}&=\mathrm{e}^{X_0+X'+\log\rho(x)+\frac{1}{2}[X,\log\rho(x)]}=\\
&=\mathrm{e}^{X_0+\log\rho(x)}\mathrm{e}^{X'+[X,\log\rho(x)]}=\\
&=\rho(x+X_0)\mathrm{e}^{X'+[X,\log\rho(x)]}, 
\end{align*}
where we have identified $X_0$ with its coordinates in $\R^d$ with respect to $(X_j)_{j\in K_{\pmb{B}}}$. Since $\mathfrak{g}$ is of step length $2$, $[X,\log\rho(x)]\in \mathfrak{z}\subseteq \mathfrak{h}_{V,\pmb{B}}(\xi)$ so $\mathrm{e}^{X'+[X,\log\rho(x)]}\in H_{V,\pmb{B}}(\xi)$. Therefore the representation $\tilde{\pi}_{\pmb{B}}$ takes the stated form. 
\end{proof}

\begin{remark}
\label{ncfotirrod}
The explicit description of the representations of flat orbits for step length $2$ from Proposition \ref{descriptandstep2} allow us to define a ``noncommutative Fourier transform'' on flat orbits $\mathsf{F}_{\pmb{B}}:C^*(G)\to C_b(U_{\pmb{B}},\mathbb{K}(L^2(\R^d)))$ as the right vertical morphism of Proposition \ref{commutingikcstag}. By Remark \ref{denseityemramr}, this ``noncommutative Fourier transform'' on flat orbits is injective. It can be defined as the closure of the linear map $\mathsf{F}_{\pmb{B}}:\mathcal{S}(G)\to C_b(U_{\pmb{B}},\mathcal{S}(\R^d\times \R^d))$ given by 
$$[\mathsf{F}_{\pmb{B}}a](\xi,x,y):=\int_{\mathfrak{h}_{V,\pmb{B}}(\xi)}[\mathrm{exp}^*a](y-x,X')\mathrm{e}^{i\xi(X'+[X'-y,x])}\mathrm{d}X'$$
where we identify $\R^d$ with the linear span of $(X_j)_{j\in K_{\pmb{B}}}$. We note that the  ``noncommutative Fourier transform'' on flat orbits and Equation \ref{pfaffiandnadnstep2} produces the isomorphism 
$$I_{G,\pmb{B}}\cong C_0((\mathfrak{z}/C(\mathfrak{g}))^*\cap U_{\pmb{B}})\otimes C_0(\Lambda_C\cap U_{\pmb{B}},\mathbb{K}(L^2(\R^d))),$$
and on the first tensor factor the noncommutative Fourier transform acts just as the ordinary Euclidean Fourier transform.
\end{remark}

We now turn our attention to more explicit examples.

\begin{example}
\label{heisebendnedcd}
The most fundamental example of a step 2 nilpotent Lie algebra is the Heisenberg group. The Heisenberg group is the local model for contact manifolds (see more in Example \ref{contactexample} below). The reader can also consult \cite{follandphasespace} for a broader view on the Heisenberg group. If $V$ is a finite dimensional real vector space with a non-degenerate 2-form $\omega$ and basis vector $Z\in \R$, there is a Lie bracket on $V\oplus \R$ such that its centre is the copy of $\R$, and for $X,Y\in  V$ 
$$[X,Y]:=\omega(X,Y)Z.$$
We often suppress the role of $Z$, and simply denote the associated Lie algebra by $V\ltimes_\omega \R$. It is clear from the construction that $V\ltimes_\omega \R$ is step $2$ nilpotent of type $(1,\dim(V))$. A choice of inner product on $V$ induces an isomorphism $V\ltimes_\omega \R\cong V_{\R\omega^1}$ by Proposition \ref{charofspeepoed}. There is a standard grading on $V\ltimes_\omega \R$ given by 
$$(V\ltimes_\omega \R)_{-1}:=V \quad\mbox{and}\quad (V\ltimes_\omega \R)_{-2}:=\R.$$
We use the notation $\mathfrak{h}_n$ for a Heisenberg Lie algebra on a $2n$-dimensional vector space equipped with a non-degenerate $2$-form and $\mathbb{H}_n$ for the associated step $2$ simply connected nilpotent Lie group. Up to (non-canonical) isomorphism, $\mathfrak{h}_n$ and $\mathbb{H}_n$ only depend on $n$. 

We conclude from Proposition \ref{autoforstep2decoss} that
$$\Aut(\mathfrak{h}_n)=V^*\rtimes ((Sp(V,\omega)\rtimes \Z/2)\times \R_{>0}).$$
See also \cite[Chapter I.2]{follandphasespace}. Here we can identify $Sp(V,\omega)\rtimes \Z/2$ with the elements $g\in GL(V)$ such that $g\omega g^t=\pm \omega$ and the generator of $\Z/2$ is an order two element $g$ such that $g\omega g^t=- \omega$ -- it can be taken as $i$ times complex conjugation upon choosing an isomorphism of $V$ with $\C^{\dim(V)/2}$ transforming $\omega$ to the standard symplectic form on $\C^{\dim(V)/2}=\R^{\dim(V)}$. The graded automorphisms, for the standard grading, are described similarly as 
$$\Aut_{\rm gr}(\mathfrak{h}_n)=(Sp(V,\omega)\rtimes \Z/2)\times \R_{>0}.$$

For the sake of simplicity, we fix a complex structure on $V$ adapted to $\omega$. This induces an inner product structure on $V$, and we can (and will) identify $\omega$ with the associated element in $\mathfrak{so}(V)$. Since $\omega$ is non-degenerate, Equation \eqref{pfaffiandnadnstep2} implies that 
\[ \Gamma = \{ \xi\in \mathfrak{z}^*: \xi \neq 0 \},\]
and $\Gamma$ consists of one fine strata. In fact, the equality $\Gamma=\mathfrak{z}^*\setminus \{0\}$ holds for all Lie groups admitting flat orbits and whose center is one-dimensional (see below in Subsection \ref{onedcenterexasubs}). We note that the action 
$$\Aut(\mathfrak{h}_n)\times \Gamma\to \Gamma,$$
of the automorphisms on the flat orbits (see Proposition \ref{actionofautosonflat}) factors over the quotient to the unique dilation on the Heisenberg Lie algebra
$$\Aut(\mathfrak{h}_n)=V^*\rtimes ((Sp(V,\omega)\rtimes \Z/2)\times \R_{>0})\to \R_{>0}.$$
By Proposition \ref{gammapartialdef}, we can identify 
\[ \Gamma_\partial = \{ \xi\in \mathfrak{z}^*: \xi(Z)=\pm 1 \}.\]

For the Heisenberg group, $\widehat{\mathbb{H}_n}=\Gamma\cup \Gamma_2$ (as sets) where $\Gamma_2=V^*$ are the characters. Therefore, the representation theory is completely described once having described the representations of flat orbits. We fix a complex structure $J$ on $V$ adapted to $\omega$ and consider $V$ as an inner product space. Set $d:=\dim(V)/2$. Note that a subspace $V_0\subseteq V$ is Lagrangian (i.e. $V_0=\{X\in V: \omega(X,Y)=0\; \forall Y\in V_0\}$) if and only if $V_0$ and $JV_0$ are orthogonal and $V=V_0+JV_0$. Fix such a subspace $V_0$, then for any basis $X_1,\ldots, X_d$ we obtain a Jordan-Hölder basis $Z,X_1,\ldots, X_1, Y_1,\ldots, Y_d$ where $Y_j:=JX_j$. If $X_1,\ldots, X_d$ is an ON-basis, we have the well known Heisenberg relations
$$[X_j,X_k]=[Y_j,Y_k]=0\quad\mbox{and}\quad [X_j,Y_k]=\delta_{jk}Z.$$ 
For $\xi\in \Lambda$, a short computation shows that 
$$\mathfrak{h}_V(\xi)=V_0\oplus \R.$$
We conclude that 
$$\mathcal{H}_V=(\mathfrak{z}^*\setminus \{0\})\times L^2(JV_0),$$
and the $G$-representation in each fiber is given by 
$$\pi_\xi(X+JY+tZ)f(x):=\mathrm{e}^{i\xi(t+\omega(X,x))}f(x+JY),\quad X,Y\in V_0,\; t\in \R, \; f\in L^2(JV_0),$$ 
where we identify $\mathfrak{z}^*=\R$ via the basis $Z$. The reader should note that the construction does not depend on the Jordan-Hölder basis in this case, but only on the complex structure and the Lagrangian subspace $V_0$. This gives an explicit isomorphism $I_G\cong C_0(\mathfrak{z}^*\setminus \{0\},\mathbb{K}(L^2(JV_0)))$ and a short exact sequence 
$$0\to C_0(\mathfrak{z}^*\setminus \{0\},\mathbb{K}(L^2(JV_0)))\to C^*(\mathbb{H}_n)\to C_0(V^*)\to 0.$$
We also note that since $\Gamma$ consists of one fine strata, Proposition \ref{bundleofpeolalss} implies that $[\zeta]=0\in H^1(\Aut(\mathbb{H}_n),\check{H}^1(\Gamma,U(1)))$ (which also follows from that $\check{H}^1(\Gamma,U(1))=\check{H}^1(\R^\times,U(1))\cong 0$.

It is now possible to describe the cocycle $c_{\Aut(\mathbb{H}_n)}$. 
Recall the definition of the flatness cocycle $c_{\Aut(\mathbb{H}_n)}\in Z^2(\Aut(\mathbb{H}_n),U(1))$ from Definition \ref{modifiedhgammspace}. 

\begin{proposition}
\label{flatnessonheisenend}
Let $\mathbb{H}_n$ denote the Heisenberg Lie group constructed from $V\ltimes_\omega\R$. The cohomology class $[c_{\Aut(\mathbb{H}_n)}]\in H^2(\Aut(\mathbb{H}_n),U(1))$ is non-trivial and belongs to the image of 
$$H^2(\Aut(\mathbb{H}_n),C(\Gamma,\Z/8))\to H^2(\Aut(\mathbb{H}_n),C(\Gamma,U(1))).$$
\end{proposition}

\begin{remark}
The action of ${\rm Sp}(V,\omega)\subseteq \Aut(\mathbb{H}_n)$ on $\Gamma$ is trivial. It is well known that  ${\rm Sp}(V,\omega)$ has a unique two fold cover $\widetilde{{\rm Sp}}(V,\omega)$ and the action on $I$ lifts to an action of $\widetilde{{\rm Sp}}(V,\omega)$ on $\mathcal{H}_V= (\mathfrak{z}^*\setminus \{0\})\times L^2(iV_0)$ -- the Segal-Shale-Weil representation. See more in \cite[Chapter 4]{follandphasespace}. In particular, $[c_{{\rm Sp}(V,\omega)}]$ is in the image of $H^2({\rm Sp}(V,\omega),\Z/2)\to H^2({\rm Sp}(V,\omega),U(1))$.
\end{remark}

\begin{proposition}
\label{flatnesscodforheise}
Let $\mathbb{H}_n$ denote the Heisenberg Lie group constructed from $V\ltimes_\omega\R$.  The action of $\widetilde{{\rm Sp}}(V,\omega)$ on $\mathcal{H}_V$ extends uniquely to an action of $\widetilde{{\rm Sp}}(V,\omega)\rtimes \Z/2$ lifting the action on $I_{\mathbb{H}_n}$. In particular, for 
\[ \ghani = (SO(V) \cap Sp(V,\omega)) \rtimes \mathbb{Z}_2 = U(V) \rtimes \mathbb{Z}_2, \]
we have that 
$$[c_{\ghani}]=0\in H^2(\ghani,C(\Gamma,U(1))).$$
In particular, there is a $\ghani$-equivariant isomorphism $I_G\cong C_0(\mathfrak{z}^*\setminus \{0\},\mathbb{K}(L^2(JV_0)))$ for a suitable unitary representation of $\ghani$ on the trivial bundle $\mathfrak{z}^*\setminus \{0\}\times L^2(JV_0)\to \mathfrak{z}^*\setminus \{0\}$.
\end{proposition}

\begin{proof}
We can also lift the action of ${\rm Sp}(V,\omega)\rtimes \Z/2\subseteq \Aut(\mathbb{H}_n)$ to a projective action on $\mathcal{H}_V$. We can decompose $\mathcal{H}_V$ over $\mathfrak{z}^*\setminus \{0\}=\R_{<0}\dot{\cup}\R_{>0}$ as $\mathcal{H}_V=\mathcal{H}_{V<}\dot{\cup}\mathcal{H}_{V>0}$. The $\Z/2$-factor interchanges the two pieces of $\mathcal{H}_V$ and we can lift its action to the fiber by defining it as complex conjugation $\mathcal{H}_{V,\xi}=L^2(JV_0)\to L^2(JV_0)=\mathcal{H}_{V,-\xi}$. This produces an action of  $\Z/2$ on $\mathcal{H}_V$ that lifts the action on $I_{\mathbb{H}_n}$. A short comparison with the construction of the Segal-Shale-Weil representation (see \cite[Equation (4.24), (4.25), (4.26)]{follandphasespace}) shows that this induces an action of $\widetilde{{\rm Sp}}(V,\omega)\rtimes \Z/2$. 

The final conclusion of the proposition follows from that $\widetilde{{\rm Sp}}(V,\omega)$ is defined from a cocycle which is cohomologically trivial on $U(V)$ (as the Weil-Shale-Weil representation restricts to a representation of $U(V)$), so the inclusion $U(V)\to {\rm Sp}(V,\omega)$ factors over an inclusion $U(V)\to \widetilde{{\rm Sp}}(V,\omega)$. 
\end{proof}
\end{example}

\begin{example}
\label{complexheisenbergexample}
Let $\mathbb{H}_{n,\C}$ be the complexification of the Heisenberg group. That is we define its Lie algebra as $\mathfrak{h}_{n,\C}:=\mathfrak{h}_n\otimes_\R\C$ viewed as a real Lie algebra. The complex Heisenberg group is the local model for certain polycontact structures constructed from complex manifolds, see Example \ref{polycontactexama} below. If $\mathfrak{h}_n$ is defined from an inner product space $V$ with a non-degenerate $2$-form $\omega$, then identifying $V_\C:=V\otimes_\R \C=V+iV=V\oplus V$, we have that $\mathfrak{h}_{n,\C}=(V\oplus V)_W$ where $W\subseteq \mathfrak{so}(V\oplus V)$ is spanned by 
$$\omega^1=\begin{pmatrix} \omega&0\\ 0&-\omega\end{pmatrix}\quad\mbox{and}\quad \omega^2=\begin{pmatrix} 0&\omega\\ \omega&0\end{pmatrix}.$$
Note that the complexification $\omega_\C$ of $\omega$ can be identified with $\omega^1+i\omega^2$. The constructions can be carried out starting from a complex vector space $V_\C$ equipped with a complex non-degenerate $2$-form $\omega_\C$, but for simplicity we start from $V$ which produces a real subspace of $V_\C$ which is a Lagrangian for the imaginary part of $\omega_\C$. 

The Lie algebra $\mathfrak{h}_{n,\C}:=\mathfrak{h}_n\otimes_\R\C$ is a step $2$ nilpotent Lie algebra of type $(2,2\dim(V))$ with $C(\mathfrak{h}_{n,\C})=\mathfrak{z}$. Most things are quite similar to the Heisenberg group reviewed above in Example \ref{heisebendnedcd},. For instance, there is a standard grading defined from $(\mathfrak{h}_{n,\C})_{-1}=V\otimes_\R \C$. Moreover, we have that $\Aut_{\rm gr}(\mathbb{H}_{n,\C})=u_H(H_{\mathfrak{h}_{n,\C}})$, where $H_{\mathfrak{h}_{n,\C}}\subseteq GL_\R(V\otimes \C)$ consists of those $g$ preserving the linear span of $\omega^1$ and $\omega^2$. The group $H_{\mathfrak{h}_{n,\C}}$ contains a subgroup of the form $\mathrm{Sp}_\C(V_\C,\omega_\C)\times U(1)\times \R_{>0}$, where $\mathrm{Sp}_\C(V_\C,\omega_\C)$ denotes the complex symplectic group on the complex symplectic space $(V_\C,\omega_\C)$, and $U(1)$ acts via complex multiplication on $V_\C$. We also have that
$$\Aut(\mathbb{H}_{n,\C})\cong \Hom_\R(V_\C,\R^2)\rtimes H_{\mathfrak{h}_{n,\C}}.$$

For $\xi=(\xi_1,\xi_2)\in \mathfrak{z}^*$, a short computation proves
$$\mathsf{Pf}(\omega_\xi)=\mathsf{Pf}(\xi_1\omega^1+\xi_2\omega^2)=(\xi_1^2+\xi_2^2)^{\dim(V)/2}.$$
Therefore, the flat orbits are computed as
$$\Gamma=\mathfrak{z}^*\setminus \{0\}\quad \mbox{and}\quad \Gamma_\partial\cong S^1.$$
The action of $\Aut(\mathbb{H}_{n,\C})$ on $\Gamma$ is linear, so the action on $\Gamma_\partial$ factors over a quotient $\Aut(\mathbb{H}_{n,\C})\to U(1)$ acting by rotation. A big difference to the real Heisenberg group is that $\Gamma$ consists of multiple fine stratas; from a choice of ON-basis on $V$ we can construct a Jordan-Hölder basis of $\mathfrak{h}_{n,\C}$ and the top fine strata in this basis will be $\Gamma$ with a coordinate axis removed. 

Following the same arguments as in Example \ref{heisebendnedcd}, for any Lagrangian $V_0\subseteq V$, the abelian Lie algebra $V_{0,\C}+\mathfrak{z}$ defines a polarization of all flat orbits, and the representation defined from $\xi\in \Gamma$ takes the form 
$$\pi_\xi(X+JY+tZ)f(x):=\mathrm{e}^{i\xi(t+\omega_\C(X,x))}f(x+JY),$$ 
for $X,Y\in V_{0,\C}$, $t\in \C$, and $f\in L^2(JV_{0,\C})$. Here we have fixed a complex structure $J$ on $V$, and where we have identified $\mathfrak{z}=\C$ and $\xi$ with an $\R$-linear functional $\C\to \R$. Similarly to Example \ref{heisebendnedcd}, we arrive at an explicit isomorphism $I_{\mathbb{H}_{n,\C}}\cong C_0(\mathfrak{z}^*\setminus \{0\},\mathbb{K}(L^2(JV_{0,\C})))$ and a short exact sequence 
$$0\to C_0(\mathfrak{z}^*\setminus \{0\},\mathbb{K}(L^2(JV_{0,\C})))\to C^*(\mathbb{H}_{n,\C})\to C_0(V^*_\C)\to 0.$$
Since there is a global polarization of the flat orbits, Proposition \ref{bundleofpeolalss} implies that $\delta_{\rm DD}(I_{\mathbb{H}_{n,\C}})=0\in H^3(\mathfrak{z}^*\setminus \{0\},\Z)$ (which also follows from the argument above) and $[\zeta]=0\in H^1(\Aut(\mathbb{H}_{n,\C}),\check{H}^1(\mathfrak{z}^*\setminus \{0\},U(1)))=H^1(\Aut(\mathbb{H}_{n,\C}),\Z)$. We moreover have that 
$$[c_{\mathrm{Sp}_\C(V_\C,\omega_\C)}]=0,$$
since the $\mathrm{Sp}_\C(V_\C,\omega_\C)$ is simply connected. In particular, we conclude from Theorem \ref{charofequimorita} that there is an $\mathrm{Sp}_\C(V_\C,\omega_\C)$-equivariant isomorphism 
$$I_{\mathbb{H}_{n,\C}}\cong C_0(\mathfrak{z}^*\setminus \{0\},\mathbb{K}(L^2(JV_{0,\C}))),$$
for a suitable unitary representation of $\mathrm{Sp}_\C(V_\C,\omega_\C)$ on the trivial bundle $\mathfrak{z}^*\setminus \{0\}\times L^2(JV_{0,\C})\to \mathfrak{z}^*\setminus \{0\}$. We note that $\mathrm{Sp}_\C(V_\C,\omega_\C)$ acts trivially on $\Gamma=\mathfrak{z}^*\setminus \{0\}$ since it acts trivially on the center of the Lie algebra. 
\end{example}

\begin{example}
Let $\mathsf{G}$ be a product of Heisenberg groups. This is the local model for the tangent group of a product of contact manifolds. For notational convenience we write $\mathsf{G}= \mathbb{H}_{n_1} \times \ldots \times \mathbb{H}_{n_k}$ and assume that each $\mathbb{H}_{n_j}$ is defined from a non-zero real inner product space $V_j$ with equipped with a non-degenerate $\omega^j\in \mathfrak{so}(V_j)$. Then $\mathsf{G}$ is of type $(p,q)$ where:
\[ p = k, \text{ and } q = 2(n_1 + \ldots n_k).\]
For notational convenience, we write $V:=\oplus_{j=1}^k V_j$ and $\mathfrak{z}=\oplus_{j=1}^k \mathfrak{z}_j$ where $\mathfrak{z}_j$ is the center of the $j$:th factor. We note that $C(\mathfrak{g})=\mathfrak{z}$ in this case. We call the grading
$$\mathfrak{g}_{-1}:=V \quad\mbox{and}\quad\mathfrak{g}_{-1}:= \mathfrak{z},$$ 
the standard grading on $\mathsf{G}$ as it it is the only grading for which all projections onto the standardly graded factors $\mathbb{H}_{n_j}$ are graded. 

Let us describe the automorphisms. By Proposition \ref{autoforstep2decoss},
$$\Aut(\mathsf{G})\cong \Hom(V,\mathfrak{z})\rtimes H_{\mathfrak{g}},$$
where we for rank reasons can write
\begin{align*}
H_\mathfrak{g}&=\{\varphi\in GL(V): \exists (\lambda_{ij}(\varphi))_{i,j=1}^k\mbox{ s.t. } \;\varphi^t\omega_i\varphi=\sum_{j=1}^k\lambda_{ij}(\varphi)\omega_j\}=\\
&=\{\varphi\in GL(V): \forall i\exists j, \lambda\neq 0\mbox{ with } n_i=n_j \mbox{ s.t. } \;\varphi^t\omega_i\varphi=\lambda\omega_j\}\cong \\
&\cong (\R_{>0}^k\times ({\rm Sp}(V_1,\omega_1)\times \cdots \times {\rm Sp}(V_k,\omega_k))\rtimes (\Z/2)^k)\rtimes F_{(n_1,\ldots,n_k)},
\end{align*}
and $F_{(n_1,\ldots,n_k)}$ is the finite group of all permutations on $k$ elements $\sigma\in S_k$ such that $n_i=n_{\sigma(i)}$ for all $i=1,\ldots, k$. If $\mathsf{G}$ is equipped with the standard grading, then 
$$\Aut_{\rm gr}(\mathsf{G})=u_H( H_{\mathfrak{g}}).$$

The flat orbits are readily described because we are dealing with a product. A similar argument as in Example \ref{heisebendnedcd} shows that 
\[ \Gamma =(\mathfrak{z}^*_1\setminus\{0\})\times \cdots \times  (\mathfrak{z}^*_k\setminus\{0\})=\{ \xi=(\xi_1,\ldots, \xi_k) \in \mathfrak{z}^* : \xi_1 \cdots \xi_k \neq 0 \}.\]
A computation shows that $\Gamma$ consists of one fine strata in a suitable Jordan-Hölder basis. The action of $\Aut(\mathsf{G})$ on $\Gamma$ factors over the quotient mapping 
$$\Aut(\mathsf{G})\to (\R_{>0}^k\times (\Z/2)^k)\rtimes F_{(n_1,\ldots,n_k)}=(\R^\times)^k\rtimes F_{(n_1,\ldots,n_k)},$$
and $(\R^\times)^k\rtimes F_{(n_1,\ldots,n_k)}$ acts on $\Gamma$ by rescaling and permutation of the coordinates.
By Proposition \ref{gammapartialdef}, we can identify 
\[  \Gamma_\partial = \{ \xi=(\xi_1,\ldots, \xi_k) \in \mathfrak{z}^* : \xi_1 \cdots \xi_k =\pm 1 \}. \]
We remark that if we choose a different grading than the standard grading, the action of $\Aut_{\rm gr}(\mathsf{G})$ on $\Gamma_\partial$ can be quite complicated.

If we pick an adapted complex structure on $V_j$ and a Lagrangian subspaces $V_{0,j}\subseteq V_j$ for each $j$, an argument as in Example \ref{heisebendnedcd} shows that we for a suitable Jordan-Hölder basis arrive at the Vergne polarization 
$$\mathfrak{h}_V(\xi)=V_0\oplus \mathfrak{z},$$
for each $\xi\in \Gamma$ where $V_0:=\bigoplus_{j=1}^k V_{0,j}$. We conclude that 
$$\mathcal{H}_V=\Gamma\times L^2(JV_0),$$
and the $\mathsf{G}$-representation in the fiber over $\xi\in \Gamma\subseteq \mathfrak{z}^*$ is given by 
$$\pi_\xi(X+JY+Z)f(x):=\mathrm{e}^{i\sum_{j=1}^k\xi_j(Z_j+\omega^j(X_j,x_j))}f(x+JY),$$ 
for $X,Y\in V_0$, $Z\in \mathfrak{z}$, and $f\in L^2(JV_0)$. 

Following the argument for the Heisenberg group from Proposition \ref{flatnesscodforheise}, we can describe the restriction of the flatness cocycle $c_{\rm flat}\in Z^2(\Aut(\mathsf{G}),C(\Gamma,\C^\times))$ to certain subgroups. We remark that Proposition \ref{flatnesscodforheise} ensures that $[c_{\rm flat}]\in H^2(\Aut(\mathsf{G}),C(\Gamma,\C^\times))$ is non-trivial as its restriction to  $H^2({\rm Sp}(V_1,\omega_1)\times \cdots \times {\rm Sp}(V_k,\omega_k),C(\Gamma,\C^\times))$ is non-trivial (but $\Z/2$-valued). We note that for 
$$\ghani_0:=(U(V_1)\times \cdots \times U(V_k))\rtimes (\Z/2)^k,$$
and an argument as in Proposition \ref{flatnesscodforheise} implies that $[c_{\ghani_0}]=0\in H^2(\ghani_0,C(\Gamma,\C^\times))$. Moreover, for 
$$\ghani:=\ghani_0\rtimes F_{(n_1,\ldots,n_k)} $$
we have by a standard argument that 
$$[c_{\ghani}]\in \im(H^2(F_{(n_1,\ldots,n_k)},C(\Gamma,\C^\times))\to H^2(\ghani,C(\Gamma,\C^\times)).$$

\end{example}

\begin{example}
\label{aquotientoffullnbyn}
For $n>1$, we consider the $2n-1$-dimensional Lie algebra $\mathfrak{g}$ spanned by $X_1,\ldots, X_n, Y_1, \ldots, Y_{n-1}$ such that
\[ [X_i, X_{i+1}] = Y_i, \quad i=1,2,\ldots, n-1\]
with all other brackets in the basis vectors being zero. The Lie algebra $\mathfrak{g}$ is nilpotent of step length $2$ and arises as the quotient of the nilpotent Lie algebra of step length $n$ consisting of all real $(n+1)\times (n+1)$ upper triangular nilpotent matrices (see more in Example \ref{uppertriangularexample}) by the ideal of all its $n-2$:nd order commutators. We have that $C(\mathfrak{g})=\mathfrak{z}=\oplus_{j=1}^{n-1}\R Y_j$, and $\mathfrak{g}$ is of type $(n-1,n)$ and is determined by $\omega^1,\ldots,\omega^{n-1}\in \mathfrak{so}(n)$ given as
\[ \omega^j = \iota_j\omega^0 q_j, \]
where $\omega^0$ is the standard symplectic form on $\R^2$, $\iota_j:\R^2\to \R^n$ is the embedding in the oriented basis $(X_j,X_{j+1})$ and $q_j:=\iota_j^*$ is the projection $\R^n\to \R^2$ onto the basis $(X_j,X_{j+1})$. We note that the associated linear span $W\subseteq \mathfrak{so}(n)$ consists of all anti-symmetric matrices $\omega=(\omega_{ij})_{i,j=1}^n$ with $\omega_{ij}=0$ if $|i-j|>1$. 

A grading on $\mathfrak{g}$ is characterized by an $n$-tuple $(l_1,\ldots, l_n)$, where $l_j\geq 1$ for all $j$ with equality for at least one $j$. The grading is defined so that $X_j$ has degree $l_j$. Indeed, for such an $n$-tuple we can write 
$$\mathfrak{g}=\bigoplus_{p<0} \mathfrak{g}_p,\quad\mbox{where}\quad \mathfrak{g}_p:=\left(\bigoplus_{j: \, l_j=p}\R X_j\right)\oplus \left(\bigoplus_{j: \, l_j+l_{j+1}=p}\R Y_j\right).$$ 
We call $(l_1,\ldots, l_k)=(1,\ldots, 1)$ the standard grading on $\mathfrak{g}$, in this grading the center has degree $2$. 

The automorphism group of $\mathfrak{g}$ is determined by Proposition \ref{autoforstep2decoss}, so
$$\Aut(\mathfrak{g})\cong \Hom(\R^n,\R^{n-1})\rtimes H_\mathfrak{g},$$
where 
$$H_\mathfrak{g}:=\{g\in GL_n(\R): (g^t\omega^i g)_{jk}=0 \;\forall i=1,\ldots, n-1, \; |j-k|>1\}.$$
In general, there are very few compact subgroups of $H_\mathfrak{g}$. In the standard grading, $\Aut_{\rm gr}(\mathfrak{g})=u_H(H_\mathfrak{g})$.

\begin{proposition}
Consider the Lie algebra $\mathfrak{g}$ from the preceding paragraphs with $n=4$, i.e. the $2$-step $7$-dimensional Lie algebra defined from the bracket relations
$$[X_1,X_2]=Y_1,\; [X_2,X_3]=Y_2, \;[X_3,X_4]=Y_3.$$ 
Let $V$ denote the linear span of $\{X_1,X_2,X_3,X_4\}$. Then for any Riemannian metric on $V$, $H_\mathfrak{g}\cap O(V)$ is finite. In particular, there are no infinite compact subgroups of $\Aut_{\rm gr}(\mathfrak{g})$ for $n=4$.
\end{proposition}

The take away from this proposition is that in order to expect interesting geometric examples modelled on the Lie algebra appearing in this example, we need to use non-compact structure groups.

\begin{proof}
We have that $H_\mathfrak{g}\cap O(V)$ is compact, so it suffices to prove that $(H_\mathfrak{g})_{(0)}\cap O(V)$ is finite, where $(H_\mathfrak{g})_{(0)}$ denotes the connected component containing the identity. A lengthier computation with the Lie algebra of $H_\mathfrak{g}$ shows that 
$$(H_\mathfrak{g})_{(0)}=\left\{g\in GL(4): g=
\begin{pmatrix} 
g_{11}&0&0&0\\
g_{21}&g_{22}&0&g_{24}\\
g_{31}&0&g_{33}&g_{34}\\
0&0&0&g_{44}\end{pmatrix}\right\}.$$
For $g\in (H_\mathfrak{g})_{(0)}$ we have that  $g\in O(V)$ if and only if 
$$\begin{cases}
g_{22}=\pm 1\\
g_{33}=\pm 1\\
g_{21}g_{22}=g_{22}g_{24}=g_{24}g_{22}=g_{34}g_{33}=0\\
g_{11}^2+g_{21}^2+g_{31}^2=1,\\
g_{24}^2+g_{34}^2+g_{44}^2=1,\\
g_{21}g_{24}+g_{31}g_{34}=0.
\end{cases}$$
The first three equations implies that $g\in (H_\mathfrak{g})_{(0)}\cap O(V)$ if and only if $g\in O(V)$ is diagonal. We conclude that $(H_\mathfrak{g})_{(0)}\cap O(V)\cong (\Z/2)^4$ is finite. 

The final conclusion follows from that if $\ghani\subseteq \Aut_{\rm gr}(\mathfrak{g})=u_H(H_\mathfrak{g})$ is compact it acts isometrically on $V$ for some Riemannian metric and by the previous argument we have that $\ghani$ is finite. 
\end{proof}

We compute the flat orbits using Equation \eqref{pfaffiandnadnstep2}:
\begin{align*}
\Gamma &= \{ \xi=(\xi_1,\ldots,\xi_{n-1}) \in \mathfrak{z}^* : \det(\xi_1 \omega^1 + \cdots+ \xi_{n-1} \omega^{n-1}) \neq 0 \} = \\
&=\begin{cases}
\emptyset, \; &\mbox{if $n$ is odd},\\
\left\{\xi=(\xi_1,\ldots,\xi_{n-1}) \in \mathfrak{z}^*: \xi_1\xi_3\cdots\xi_{n-1} \neq 0 \right\} , \; &\mbox{if $n$ is even}.
\end{cases}
\end{align*}
The space $\Gamma$ consists of multiple fine stratas. Indeed, for $n$ odd, any antisymmetric matrix is degenerate and for $n$ even induction on $n/2$ shows that $\det(\xi_1 \omega^1 + \cdots+ \xi_{n-1} \omega^{n-1})=\xi_1^2\xi_3^2\cdots\xi_{n-1}^2$. 
Using Proposition \ref{gammapartialdef}, we identify 
\[\Gamma_ \partial =
\begin{cases}
\emptyset, \; &\mbox{if $n$ is odd},\\
\left\{\xi=(\xi_1,\ldots,\xi_{n-1}) \in \mathfrak{z}^*: \xi_1\xi_3\cdots\xi_{n-1} =\pm 1 \right\} , \; &\mbox{if $n$ is even}.
\end{cases} \]

\begin{remark}
\label{noflatexample}
Consider $n=3$ for which the above construction produces the 5-dimensional  Lie algebra $\mathfrak{g}$ spanned by $X_1, X_2, X_3, Y_1, Y_2$ subject to the relations
\[ [X_1, X_2] = Y_1, \ [X_2, X_3] = Y_2,\]
and all other brackets being zero. This nilpotent Lie algebra $\mathfrak{g}$ is of step length $2$, of type $(2,3)$ with $C(\mathfrak{g})=\mathfrak{z}$ but admits no flat orbits. 
\end{remark}

We henceforth restrict our attention to $n$ being even, so that there exists flat orbits. The Vergne polarization of the flat orbits is computed using Proposition \ref{vergenpolonflat} to be 
$$\mathfrak{h}_V(\xi)=\mathfrak{h}=\mathfrak{z}+\sum_{j=1}^{n/2}\R X_{2j-1},$$
for all $\xi\in \Lambda$. Note that $\mathfrak{h}_V(\xi)=\mathfrak{h}\subseteq \mathfrak{g}$ is abelian and independent of $\xi$. Using the same argument as in Proposition \ref{descriptandstep2} we conclude that 
$$\mathcal{H}_V=\Lambda\times L^2(\R^{n/2}),$$
and the $\mathsf{G}$-representation in the fiber over $\xi\in \Gamma\subseteq \mathfrak{z}^*$ is given by 
\begin{align*}
\pi_\xi&\left(\mathrm{exp}\left(\sum_{j=1}^n x_j X_j+\sum_{l=1}^{n-1}y_jY_j \right)\right)f(t_1,t_3, \ldots, t_{n-1}):=\\
&:=\mathrm{e}^{ih(\xi,t,x,y)}f(t_1+x_1,t_3+x_3,\ldots, t_{n-1}+x_{n-1}),
\end{align*}
for $f\in L^2(\R^{n/2})$, where 
\begin{align*}
h(\xi,t,x,y)&=\xi\left(\sum_{l=1}^{n-1}y_jY_j \right)+\xi\left[\sum_{j=1}^n x_j X_j,\sum_{l=1}^{n/2}t_{2l-1}X_{2l-1}\right]=\\
&=\sum_{j=1}^{n-1}\xi_jy_j-\xi_1x_1t_1+\sum_{l=2}^{n/2}(\xi_{2l-2}x_{2l-2}-\xi_{2l-1}x_{2l})t_{2l-1}.
\end{align*}
Since there is a global polarization of the flat orbits, Proposition \ref{bundleofpeolalss} implies that $[\zeta]=0\in H^1(\Aut(\mathsf{G}),\check{H}^1(\Gamma,U(1)))$. In lack of a good description of the automorphisms, we settle for noting that there is an $\Aut(\mathsf{G})$-equivariant isomorphism $I_\mathsf{G}\cong C_0(\Gamma,\mathbb{K}(L^2(\R^{n/2})))$ for some projective $\Aut(\mathsf{G})$-action on the trivial bundle $\Gamma\times L^2(\R^{n/2})\to \Gamma$.
\end{example}

\begin{example}
\label{freestep2example}
Consider the free step $2$ nilpotent Lie group $\mathsf{G}$ on a finite-dimensional real inner product space $V$. In the notations of above, $\mathfrak{g}=V_{\mathfrak{so}(V)}$ and $\mathfrak{g}$ satisfies $C(\mathfrak{g}) = \mathfrak{z}= \mathfrak{so}(V)$. Indeed, the only nontrivial bracket is for $x, y \in V\subseteq \mathfrak{g}$ and it is given by 
$$[x, y] = x \wedge y,$$
under the identification $\mathfrak{so}(V)\cong \wedge^2 V$ defined from the inner product. The  free step $2$ nilpotent Lie algebra is of type $(n(n-1)/2,n)$ where $n:=\dim(V)$.

There is a one-to-one correspondence of vector space gradings $V=\bigoplus_{p<0} V_p$ with $V_{-1}\neq 0$ and Lie algebra gradings $\mathfrak{g}=\bigoplus_{p<0} \mathfrak{g}_p$ given by associating a grading on $\mathfrak{g}$ with a grading on $V$ by 
$$\mathfrak{g}_p:=V_p\oplus \left(\bigoplus_{j+l=p}V_j\wedge V_l\right).$$ 
The inverse construction induces a grading on $V$ from a grading on $\mathfrak{g}$ by declaring the quotient $\mathfrak{g}\to V$ to be graded. We call the grading with $V$ having constant degree $-1$ the standard grading.

By universality and Proposition \ref{autoforstep2decoss},
$$\Aut(\mathsf{G})\cong \Hom(V,\mathfrak{so}(V))\rtimes GL(V),$$
and 
$$\Aut_{\rm gr}(\mathsf{G})=u_H(GL_{\rm gr}(V)),$$
in the standard grading. Equation \eqref{pfaffiandnadnstep2} shows that we can identify
\[ \Gamma = \mathfrak{so}(V) \cap GL(V)\quad\mbox{and}\quad \Gamma_ \partial = \mathfrak{so}(V) \cap SL(V), \]
using the self-duality $\mathfrak{so}(V)^*=\mathfrak{so}(V)$. We note that $\Gamma$ is empty if $\dim(V)$ is odd and in this case readily seen to be a Zariski open subset of $\mathfrak{so}(V)$ if $\dim(V)$ is even. The action of the automorphisms of $\mathsf{G}$ factors over the surjection $\Aut(\mathsf{G})\to GL(V)$ and $GL(V)$ acts on $\Gamma= \mathfrak{so}(V) \cap GL(V)$ by matrix multiplication
$$\varphi.\xi:=\varphi^t\xi\varphi\in  \mathfrak{so}(V) \cap GL(V).$$
Since any non-degenerate two form admits a basis in which it is the standard symplectic form $\omega_0$ on $V$, the $GL(V)$-action is transitive and there is a $GL(V)$-equivariant homeomorphism 
$$GL(V)/\mathrm{Sp}(V,\omega_0)\xrightarrow{\sim} \Gamma.$$

As above, we pick a Jordan-Hölder basis for $\mathfrak{g}$ with $\mathfrak{g}_{n(n-1)/2}=\mathfrak{z}=\mathfrak{so}(V)$ (for $n=\dim(V)$). Such a Jordan-Hölder basis induces a basis of $V$ and a filtration $V_j:=\mathfrak{g}_{n(n-1)/2+j)}/\mathfrak{z}$ of $V$. We can describe the Vergne polarization of the flat orbits using Proposition \ref{vergenpolonflat} as 
$$\mathfrak{h}_V(\xi)=\mathfrak{z}+\sum_{j=1}^{\dim(V)}\ker(\xi|_{V_j}),$$
for all $\xi\in \Lambda=\mathfrak{so}(V) \cap GL(V)$. We note that $\mathfrak{h}_V(\xi)\subseteq \mathfrak{g}$ is neither abelian nor constant in $\xi\in \Lambda$. Moreover, it does not extend from the top fine strata (defined relative to the fixed Jordan-Hölder basis) to all flat orbits. We shall therefore take a more direct route using Fock-Bargmann spaces (see for instance \cite[Chapter 1.6]{follandphasespace}) to constructing a bundle of flat representations. An additional feature that comes in handy is that orthogonal automorphisms act most naturally on the Fock-Bargmann spaces. We only consider the case that $V$ is even-dimensional, so there exists flat orbits.

We fix an inner product metric on $V$, it sets up a linear bijection between antisymmetric forms and $\mathfrak{so}(V)\subseteq \End(V)$, and between symmetric forms and symmetric elements of $\End(V)$. In terms of this fixed inner product, there is a bijective correspondences between symplectic forms $\omega$ and pairs $(J,g)$ of inner products $g$ on $V$ and a complex structure $J$ on $V$ with $g(x,Jy)=-g(Jx,y)$ for all $x,y\in V$. Indeed, the pair $(J,g)$ determines the symplectic form $\omega_{(J,g)}(x,y)=g(x,Jy)$, and conversely we can given $\omega$ define $g_\omega:=(-\omega^2)^{1/2}$ and $J_\omega:=(-\omega^2)^{-1/2}\omega$. We remark that 
$$\langle x,y\rangle_\omega:=g_\omega(x,y)+i\omega(x,y),$$
defines a hermitean inner product on $V$.

Given a symplectic form $\omega$ on $V$, we can define the associated Fock-Bargmann space as $\mathpzc{F}(V,\omega):=\mathcal{O}(V)\cap L^2(V,\e^{-g_\omega})$. That is, $\mathpzc{F}(V,\omega)$ is the subspace of holomorphic functions (with respect to the complex structure $J_\omega$) in the weighted $L^2$-space $L^2(V,\e^{-g_\omega})$ defined from the weight $\e^{-g_\omega}(x):=\e^{-g_\omega(x,x)}$. One can equivalently define $\mathpzc{F}(V,\omega):=\bigoplus_{k=0}^\infty V^{*\otimes_\C^s k}$, were $V^{*\otimes_\C^s 0}=\C$ and $V^{*\otimes_\C^s k}$ is the $k$-fold symmetric tensor product of $V^*$ over $\C$ defined with respect to the complex structure $J_\omega$, with Hilbert space structure defined from $g_\omega$. 

The Bargmann-Fock representation of $\mathfrak{h}_{\dim(V)/2}=V_{\R\omega}$ (see for instance \cite[Chapter 1.6]{follandphasespace}) extends to the free step $2$ nilpotent Lie group $\pi_\omega:\mathsf{G}\to U(\mathpzc{F}(V,\omega))$ by identifying $\omega$ with a central character, using that $\Gamma=\mathfrak{so}(V) \cap GL(V)$, and by defining 
\begin{equation}
\label{twisededrep}
\pi_\omega(X)f(x):=\e^{-\frac{g_\omega(X,X)}{2}-\langle x,X\rangle_\omega}f(x+X),
\end{equation}
for $X\in X$ and $f\in \mathpzc{F}(V,\omega)$. By construction, $\pi_\omega$ is a unitary irreducible representation representing the orbit $\omega\in \Gamma$ under the Kirillov correspondence (beware that we follow the convention of identifying $\Lambda$ with the set of flat orbits $\Gamma\subseteq \mathfrak{g}^*/\mathsf{Ad}^*$). Since this construction depends continuously on $\omega$, some arguments with the Bargmann transform shows that the family of Hilbert spaces $(\mathpzc{F}(V,\omega))_{\omega\in \Gamma}$ and representations $(\pi_\omega)_{\omega\in \Gamma}$ glues together to a locally trivial bundle of Hilbert spaces $\mathcal{H}_\Gamma\to \Gamma$ with a strongly continuous representation of $\mathsf{G}$. 

The group $O(V)$ (defined with respect to the fixed inner product) is a maximal compact subgroup of $\Aut_{\rm gr}(\mathsf{G})$ and acts on the flat orbits $\Gamma$ by $\alpha.\omega:=\alpha^t \omega \alpha$. The homeomorphism $\Gamma=GL(V)\cap \mathfrak{so}(V)\cong GL(V)/\mathrm{Sp}(V,\omega_0)$ is clearly $O(V)$-equivariant for the left action on $GL(V)/\mathrm{Sp}(V,\omega_0)$. The $O(V)$-action on $\Gamma$ lifts to a strongly continuously action on the bundle $\mathcal{H}_\Gamma$, by simply letting $\alpha\in O(V)$ act as the unitary
$$\alpha:\mathpzc{F}(V,\omega)\to \mathpzc{F}(V,\alpha.\omega), \; \alpha f:=\alpha^*f.$$
This map is well defined due to the fact that $J_{\alpha.\omega}=\alpha^{-1}J_\omega \alpha$ and $g_{\alpha.\omega}=\alpha^tg_\omega \alpha$, so the action is well defined as soon as $\alpha^t=\alpha^{-1}$, i.e. $\alpha$ orthogonal.

We summarize this argument in the following proposition.

\begin{proposition}
We let $\mathsf{G}$ denote free step $2$ nilpotent Lie group on the even-dimensional inner product space $V$ with space of flat orbits $\Gamma=GL(V)\cap \mathfrak{so}(V)$ being all symplectic forms on $V$. The $O(V)$-equivariant bundle of Bargmann-Fock spaces $\mathcal{H}_\Gamma\to \Gamma$ satisfies that 
$$I_\mathsf{G}\cong C_0(\Gamma,\mathbb{K}(\mathcal{H}_\Gamma)),$$ 
as $O(V)$-equivariant $C_0(\Gamma)-C^*$-algebras via the family $(\pi_\omega)_{\omega\in \Gamma}$. In particular, 
$$\begin{cases} 
[\zeta|_{O(V)}]=0\in H^1(O(V),\check{H}^2(GL(V)\cap \mathfrak{so}(V),\Z))=0,\;\mbox{ and}\\ 
c_{O(V)}=0\in H^2(O(V),C(GL(V)\cap \mathfrak{so}(V),U(1))).\end{cases}$$
\end{proposition}

\end{example}

\subsection{Nilpotent Lie groups with one-dimensional center}
\label{onedcenterexasubs}
Let us consider another class of nilpotent Lie groups, namely those with a one-dimensional center. Given any nilpotent Lie group, a quotient by a codimension one subgroup of the center produces a nilpotent Lie group with one-dimensional center, making the case of one-dimensional center interesting. We shall also consider a construction found in \cite{mohsen2} that realize a nilpotent Lie group as a closed subgroup of a nilpotent Lie group with one-dimensional center admitting flat orbits -- this example holds potential for index theoretical considerations.

Assume that $\mathsf{G}$ is a simply connected nilpotent Lie group with one-dimensional center, and as above $\mathfrak{g}$ denotes its Lie algebra that we equipp with a choice of Jordan-Hölder basis such that $X_1$ is central. As in Remark \ref{flatorbisintosform}, there is an anti-symmetric mapping $\omega:(\mathfrak{g}/\mathfrak{z})\wedge (\mathfrak{g}/\mathfrak{z})\to \mathfrak{z}$. The group $\mathsf{G}$ admits flat orbits if and only if $\omega$ is non-degenerate, and Remark \ref{flatorbisintosform} shows that 
$$\Gamma=\mathfrak{z}^*\setminus \{0\}.$$
If $\mathsf{G}$ is graded, the associated dilation on $\mathfrak{z}$ is scalar and we can identify $\Gamma_\partial=\{\xi\in \Gamma: \xi(X_1)=\pm 1\}$. 

To simplify our discussion, we assume that $\mathsf{G}$ is graded. If $\xi_0\in \Gamma$ and $\mathfrak{h}$ is a real algebraic polarization of $\xi_0$, we can extend $\mathfrak{h}$ to a bundle of polarizations of $\Gamma$ by setting 
$$\mathfrak{h}(\xi):=\delta_t\mathfrak{h}, \quad\mbox{for $t>0$ such that $\xi=\pm \delta_t(\xi_0)$.}$$
If $\mathfrak{h}=\mathfrak{h}_V(\xi_0)$ and the Jordan-Hölder basis is homogeneous for the grading, we have that $\mathfrak{h}(\xi)=\mathfrak{h}_V(\xi)$ for all $\xi\in \Gamma$. In general, it is non-trivial to describe $\mathfrak{h}_V(\xi)$ beyond Proposition \ref{vergenpolonflat}. But we note the following concerning the ideal $I\subseteq C^*(\mathsf{G})$, if $\mathfrak{h}$ is any real algebraic polarization at $\xi_0\in \Gamma$, and $\pi_\pm:\mathsf{G}\to U(L^2(\R^d))$ denotes the representations associated with $\pm \xi_0$ (polarized by $\mathfrak{h}$) and $d=\mathrm{codim}(\mathfrak{z})/2=(\dim(\mathsf{G})-1)/2$, the dilation induces an explicit isomorphism 
$$I_\mathsf{G}\xrightarrow{\sim} C_0(\mathfrak{z}^*\setminus \{0\}, \mathbb{K}(L^2(\R^d))), \quad a\mapsto (\xi\mapsto \pi_{\mathrm{sign}(\xi)}(\delta_{t(\xi)}(a))),$$
where $\mathrm{sign}(\xi)\in \{\pm\}$ and $t(\xi)>0$ are determined by $\xi=\mathrm{sign}(\xi)t(\xi)\xi_0$. For nilpotent Lie groups with one-dimensional center, the ideal of flat orbits fits into an $\Aut(\mathfrak{g}$-equivariant short exact sequence
$$0\to C_0(\mathfrak{z}^*\setminus \{0\}, \mathbb{K}(L^2(\R^d)))\to C^*(\mathsf{G})\to C^*(\mathsf{G}/\mathsf{Z})\to 0.$$

\begin{example}
\label{engelalgebra}
We start with an example of a nilpotent Lie group with one-dimensional center lacking flat orbits. Consider the $4$-dimensional Lie algebra $\mathfrak{g}$ spanned by $Y_1$, $Y_2$, $Y_3$, and $Y_4$ subject to the relations
$$[Y_2, Y_4] = Y_1\quad\mbox{and}\quad [Y_3, Y_4] = Y_2,$$
and all other Lie brackets between the basis vectors are zero. This example was studied in \cite[Example 1.3.10 and 2.2.2]{Corwin_Greenleaf} and \cite[Chapter 3.3]{kirillovbook}. We grade $\mathfrak{g}$ by $\mathfrak{g}= \mathfrak{g}_{-3} \oplus \mathfrak{g}_{-2} \oplus \mathfrak{g}_{-1}$ where $\mathfrak{g}_{-1}$ is spanned by $Y_3, Y_4$, $\mathfrak{g}_{-2}$ is spanned by $Y_2$ and the center $\mathfrak{g}_{-3}=\mathfrak{z}$ is spanned by $Y_1$. The Lie algebra $\mathfrak{g}$ is step $3$ nilpotent with one-dimensional center. The associated Lie group $\mathsf{G}$ readily seen to be a subgroup of $\mathsf{G}_3$ (see Example \ref{uppertriangularexample}). A short computation shows that $\Aut_{\rm gr}(\mathfrak{g})\cong \R\rtimes (\R^\times)^2$ with action induced from an action on $ \mathfrak{g}_{-1}$, so the Jordan-Hölder flag defined from $Y_1, Y_2,Y_3,Y_4$ is invariant under graded automorphisms, so computations in this basis globalizes in a controlled way. Since the centre has odd codimension there exists no flat orbits. The top coarse strata will however consist of generic orbits and produce structures similar to that of the flat orbits. We include this example to indicate how the techniques of this monograph can be extended beyond the usage of flat orbits. 

We use $Y_1, Y_2,Y_3,Y_4$ as a Jordan-Hölder basis and we use the dual basis for $\mathfrak{g}^*$ with associated coordinates $\xi_1,\xi_2,\xi_3,\xi_4$. Following \cite[Example 1.3.10 and 2.2.2]{Corwin_Greenleaf} and \cite[Chapter 3.3]{kirillovbook}, the embedding $(p,q)\mapsto (0,\frac{q}{2p},0,p)$ identify 
$$\tilde{\Gamma}:=\R^\times_p\times \R_q,$$
with an open, dense Hausdorff subset of $\hat{\mathsf{G}}$ lifting to the subset $\{\xi:\xi_4\neq 0\}$. under the projection map $\mathfrak{g}^*\to \hat{\mathsf{G}}$. The inverse to the identification is given by $p(\xi_1,\xi_2,\xi_3,\xi_4)=\xi_1$ and $q(\xi_1,\xi_2,\xi_3,\xi_4)=2\xi_3\xi_1-\xi_2^2$. Each element $\xi\in \tilde{\Gamma}$ defines a two-dimensional coadjoint orbit and is polarized by the abelian subalgebra $\mathfrak{h}=\R Y_1+\R Y_2+\R Y_3$. For dimensional reasons, $\tilde{\Gamma}$ can be identified with the top coarse strata in $\hat{G}$ and $\tilde{\Gamma}$ is preserved by all of $\Aut(\mathsf{G})$. For $\xi\in \tilde{\Gamma}$, we construct the representation 
$$\pi_\xi:\mathsf{G}\to U(L^2(\R)), \quad \pi_\xi:=\ind_\mathsf{H}^\mathsf{G}\chi_\xi,$$
where $\mathsf{H}$ is the Lie group constructed from $\mathfrak{h}$,  $\chi_\xi:H\to U(1)$ the character defined from $\xi$ and where we have identified $\mathsf{G}/\mathsf{H}=\R$. 

Let $\tilde{I}_\mathsf{G}\subseteq C^*(\mathsf{G})$ denote the ideal corresponding to $\tilde{\Gamma}$. The arguments above shows that the family $(\pi_\xi)_{\xi\in \tilde{\Gamma}}$ induces an $\Aut(\mathsf{G})$-equivariant isomorphism 
$$\tilde{I}_\mathsf{G}\cong C_0(\tilde{\Gamma},\mathbb{K}(L^2(\R))),$$
for some projective $\Aut(\mathsf{G})$-action on $L^2(\R)$. In particular, $\tilde{I}_\mathsf{G}$ is an $\Aut(\mathsf{G})$-continuous trace algebra with $\delta_{\rm DD}(\tilde{I}_\mathsf{G})=0\in \check{H}^3(\tilde{\Gamma},\Z)$ and $[\zeta]=0\in H^1(\Aut(\mathsf{G}),\check{H}^2(\tilde{\Gamma},\Z))$.
\end{example}

\begin{example}
\label{uppertriangularexample}
We proceed with a family of nilpotent Lie groups with one-dimensional center lacking flat orbits. Let $\mathsf{G}_n$ denote the nilpotent Lie group of real $(n+1)\times(n+1)$ upper triangular unipotent matrices. This is a nilpotent Lie group of step length $n$. Its center is one-dimensional and consists of matrices $g=(g_{ij})_{i,j=1}^{n+1}$ such that $g_{ij}=0$ unless $i=j$ or $i=1$ and $j=n+1$ (the rigthmost, upmost entry can be non-zero). For any $\xi\in \mathfrak{z}^*\setminus \{0\}$, the Lie algebra $\mathsf{stab}(\xi)$ consists of matrices with zero entries on and below the diagonal as well as in the right most column with the exception of the rigthmost, upmost entry. In particular, $\mathrm{codim}(\mathsf{stab}(\xi))=n-1<\mathrm{codim}(\mathfrak{z})$, so there exists no flat orbits on $\mathsf{G}_n$ if $n>2$ even if $\mathrm{codim}(\mathfrak{z})$ for certain $n$ is even-dimensional.

The case $n=3$ was computed in more detail in \cite[Example 1.3.11 and 2.2.3]{Corwin_Greenleaf}. In the same way as in Example \ref{engelalgebra}, \cite[Example 1.3.11 and 2.2.3]{Corwin_Greenleaf} provides an open dense Hausdorff subset $\tilde{\Gamma}\cong\R^\times\times \R$ of $\hat{\mathsf{G}}_3$. Each coadjoint orbit defined from elements of $\tilde{\Gamma}$ are of dimension $4$ and the points of $\tilde{\Gamma}$ are all polarized by the same $4$-dimensional abelian ideal $\mathfrak{h}$. For dimensional reasons, $\tilde{\Gamma}$ can be identified with the top coarse strata in $\hat{G}_3$ and $\tilde{\Gamma}$ is preserved by all of $\Aut(\mathsf{G}_3)$. Let $\tilde{I}_{\mathsf{G}_3}\subseteq C^*(\mathsf{G}_3)$ denote the ideal corresponding to $\tilde{\Gamma}$. The arguments above shows that the family $(\pi_\xi)_{\xi\in \tilde{\Gamma}}$ induces an $\Aut(\mathsf{G}_3)$-equivariant isomorphism 
$$\tilde{I}_{\mathsf{G}_3}\cong C_0(\tilde{\Gamma},\mathbb{K}(L^2(\R^2))),$$
for some projective $\Aut(\mathsf{G}_3)$-action on $L^2(\R^2)$. In particular, $\tilde{I}_\mathsf{G}$ is an $\Aut(\mathsf{G}_3)$-continuous trace algebra with $\delta_{\rm DD}(\tilde{I}_{\mathsf{G}_3})=0\in \check{H}^3(\tilde{\Gamma},\Z)$ and $[\zeta]=0$ in the group $H^1(\Aut(\mathsf{G}_3),\check{H}^2(\tilde{\Gamma},\Z))$. It would be interesting to extend this construction to $n>3$ as it relates to parabolic geometries for $SL(n,\R)$ (see Example \ref{parapara} below).
\end{example}

\begin{example}[The Mohsen construction]
\label{mohsenconstruct}
We will now consider a construction from \cite{mohsen2} that to each simply connected nilpotent Lie group associates a larger simply connected Lie group with one-dimensional center and flat orbits. This construction is of interest since it creates flat orbits, and the modification is functorial enough to extend to bundle of Lie groups and to Carnot manifolds.

Consider a simply connected nilpotent Lie group $\mathsf{G}$ with Lie algebra $\mathfrak{g}$. The Lie algebra acts via the coadjoint action as Lie algebra homomorphisms on $\mathfrak{g}^*$ equipped with the trivial Lie bracket. We can form the Lie algebra $\mathfrak{g}^*\rtimes_{\mathsf{ad}^*} \mathfrak{g}$. More precisely, for $(\eta,X),(\eta',X')\in \mathfrak{g}^*\rtimes_{\mathsf{ad}^*} \mathfrak{g}:=\mathfrak{g}^*\oplus \mathfrak{g}$ their Lie bracket is defined by 
$$[(\eta,X),(\eta',X')]_{\mathfrak{g}^*\rtimes_{\mathsf{ad}^*} \mathfrak{g}}:=(\mathsf{ad}^*(X)\eta'-\mathsf{ad}^*(X')\eta,[X,X']).$$
The duality pairing defines a symplectic form $\omega_{\mathfrak{g}}$ on $\mathfrak{g}^*\oplus \mathfrak{g}$ which is readily verified to defined a cocycle on $\mathfrak{g}^*\rtimes_{\mathsf{ad}^*} \mathfrak{g}$. We let $\tilde{\mathfrak{g}}$ denote the associated central extension. More precisely, $\tilde{\mathfrak{g}}:=\mathfrak{g}^*\oplus \mathfrak{g}\oplus \R$ as a vector space and the Lie bracket of $((\eta,X),t),((\eta',X'),t')\in \tilde{\mathfrak{g}}=\mathfrak{g}^*\oplus \mathfrak{g}\oplus \R$ is defined by 
$$[((\eta,X),t),((\eta',X'),t')]_{\tilde{\mathfrak{g}}}:=([(\eta,X),(\eta',X')]_{\mathfrak{g}^*\rtimes_{\mathsf{ad}^*} \mathfrak{g}},\omega_\mathfrak{g}((\eta,X),(\eta',X'))).$$
Let $\tilde{\mathsf{G}}$ denote the simply connected Lie group constructed from $\tilde{\mathfrak{g}}$. The step length of $\tilde{\mathsf{G}}$ is exactly one step larger than that of $\mathsf{G}$. 

\begin{definition}
\label{mohsenmodingliealg}
Let $\mathsf{G}$ be a simply connected nilpotent Lie group with Lie algebra $\mathfrak{g}$. The Lie algebra $\tilde{\mathfrak{g}}$ constructed in the paragraph above is called the Mohsen modification of $\mathfrak{g}$ and the Lie group $\tilde{\mathsf{G}}$ is called the Mohsen modification of $\mathsf{G}$.
\end{definition}

It is readily verified that the center $\tilde{\mathfrak{z}}$ of $\tilde{\mathfrak{g}}$ is the one-dimensional summand $\R\subseteq \tilde{\mathfrak{g}}$ and that the antisymmetric mapping $\omega:\tilde{\mathfrak{g}}/\tilde{\mathfrak{z}}\wedge \tilde{\mathfrak{g}}/\tilde{\mathfrak{z}}\to \tilde{\mathfrak{z}}$ can be identified with the symplectic pairing $\omega_{\mathfrak{g}}$. In particular, we can from the discussion in the start of this subsubsection deduce that $\tilde{\mathfrak{g}}$ is a Lie algebra with one-dimensional center that admits flat orbits
$$\Gamma=\R^\times=\tilde{\mathfrak{z}}^*\setminus \{0\}.$$

The automorphism group $\Aut(\tilde{\mathfrak{g}})$ is slightly convoluted to describe and for our purposes it suffices to note that there are injections
\begin{equation}
\label{autominjecosfordoubale}
\Aut(\mathfrak{g})\hookrightarrow \Aut(\mathfrak{g}^*\rtimes_{\mathsf{ad}^*} \mathfrak{g})\quad \mbox{and}\quad \Aut(\mathfrak{g})\hookrightarrow \Aut(\tilde{\mathfrak{g}}).
\end{equation}

The Lie algebra $\tilde{\mathfrak{g}}$ can always be graded by $\tilde{\mathfrak{g}}_{-1}:=\mathfrak{g}^*\oplus \mathfrak{g}$ and $\tilde{\mathfrak{g}}_{-2}:=\R$. We call this grading the standard grading on $\tilde{\mathfrak{g}}$. In examples, when $\mathfrak{g}$ is graded it is more natural to consider the induced grading on $\tilde{\mathfrak{g}}$. Let us describe how to induce gradings from $\mathfrak{g}$ to $\tilde{\mathfrak{g}}$ (this construction is also found in \cite[Subsection 3.1]{mohsen2}). If $\mathfrak{g}=\bigoplus_{-n\leq p<0}\mathfrak{g}_p$ is a grading of $\mathfrak{g}$, we define 
$$\tilde{\mathfrak{g}}_p:=
\begin{cases}
0, & p<-n-1,\\
\R, &p=-n-1,\\
\mathfrak{g}_{p}\oplus \mathfrak{g}^*_{-n-p-1}, & p\geq -n,
\end{cases}$$
In the standard grading $\Aut(\mathfrak{g})\hookrightarrow \Aut_{\rm gr}(\tilde{\mathfrak{g}})$ and in an induced grading $\Aut_{\rm gr}(\mathfrak{g})\hookrightarrow \Aut_{\rm gr}(\tilde{\mathfrak{g}})$.

We polarize $\xi_0=1\in \R^*\setminus\{0\}$ by the real algebraic polarization 
$$\mathfrak{h}:=\R\oplus \mathfrak{g}^*.$$
For a suitable Jordan-Hölder basis and either the standard grading or an induced grading on $\tilde{\mathfrak{g}}$, the dilation reconstructs the Vergne polarization of the flat orbits from this polarization and $\mathfrak{h}_V(\xi)=\R\oplus \mathfrak{g}^*$ for any $\xi\in \Gamma$. Let us describe the flat representations constructed from the polarization $\mathfrak{h}:=\R\oplus \mathfrak{g}^*$.  We have that 
$$\mathcal{H}_V=\R^\times \times L^2(\mathsf{G}),$$
and the representation in the fibre over $\xi\in \R^\times$ is given by 
\begin{align*}
\pi_\xi&\left(\mathrm{e}^{\eta+z}g\right)f(t):=\mathrm{e}^{i\xi\left(z+\eta(\log(t))\right)}f(gt),
\end{align*}
for $g\in \mathsf{G}$, $\eta\in \mathfrak{g}^*$, $z\in \R=\tilde{\mathfrak{z}}$ and $f\in L^2(\mathsf{G})$. The reader should note that the restriction of any flat orbit representation of $\tilde{\mathsf{G}}$ to $\mathsf{G}$ is equivalent to the left regular representation of $\mathsf{G}$.

We arrive at an isomorphism $I_{\tilde{\mathsf{G}}}\cong C_0(\R^\times,\mathbb{K}(L^2(\mathsf{G})))$ and a short exact sequence 
$$0\to C_0(\R^\times,\mathbb{K}(L^2(\mathsf{G})))\to C^*(\tilde{\mathsf{G}})\to C^*(\mathfrak{g}^*\rtimes_{\mathsf{Ad}^*}\mathsf{G})\to 0.$$
The action of $\Aut(\tilde{\mathsf{G}})$ on $I_{\tilde{\mathsf{G}}}$ restricts to the ordinary fibrewise $\Aut(\mathsf{G})$-action on $I_{\tilde{\mathsf{G}}}\cong C_0(\R^\times,\mathbb{K}(L^2(\mathsf{G})))$ under the mapping $\Aut(\mathsf{G})\hookrightarrow \Aut(\tilde{\mathsf{G}})$ from Equation \eqref{autominjecosfordoubale}. Furthermore, we deduce that $[\zeta]=0\in H^1(\Aut(\tilde{\mathsf{G}}),\check{H}^2(\R^\times,\Z))=0$ and $[c_{\Aut(\mathsf{G})}]=0\in H^2(\Aut(\mathsf{G}),C(\R^\times,U(1)))=0$.

\end{example}

\part{Groupoids and twists}
\label{gropodoprar}

\begin{center}
{\bf Introduction to part}
\end{center}

The aim of the current part is to extend the structures defined in the previous part to bundles of nilpotent Lie groups and to prove the $K$-theoretical results needed for further index theoretical considerations. The motivation to do so comes from the fact that the tangent bundle of a Carnot manifold can be seen as a groupoid globally describing the osculating, nilpotent Lie group structure on each fibre. In the following table we gather the main characters of the story and compare notation for their incarnations on the level of the fiber and the level of frame bundle $P\to X$ of the bundle of Lie groups:
\begin{center}
\begin{tabular}{ | m{5cm} | c| m{5cm} | c | } 
 \hline
 Lie group (group theoretic model fiber) & $\mathsf{G}$ & Locally trivial bundle of Lie groups (e.g. the tangent groupoid of a Carnot manifold) & $\mathcal{G}$ \\ 
  \hline
 Lie algebra (Lie algebraic model fiber) & $\mathfrak{g}$ & Lie algebroid with trivial anchor map (e.g. the tangent algebroid of a Carnot manifold) & $\mathbf{g}$ \\ 
  \hline
 Center & $\mathsf{Z}$ or $\mathfrak{z}$ & Central groupoid & $\mathcal{Z}$ \\
  \hline
 Set of flat orbits & $\Gamma$ & Fiber bundle of flat orbits & $\Gamma_P$ \\
  \hline
 Total space of flat orbits & $\Xi$ & Symplectic vector bundle of flat orbits over the total space of the fiber bundle of flat orbits & $\Xi_P$ \\
 \hline
 Ideal of flat representations & $I_{\mathsf{G}}$ & Ideal of flat orbit representations in $C^*(\mathcal{G})$& $I_P$ \\
 \hline
\end{tabular}
\end{center}
Later, when $P$ is constructed from a regular Carnot manifold $X$ we indicate $X$ rather than $P$ in the notation. For much of this part, we assume the reader is familiar with $KK$-theory and the view on $KK$ as a category  \cite{cumero,higsonprimer, knudjen,kaspnov}. Often we consider diagrams of $C^*$-algebras involving morphisms from $KK$.

Homological information about $H$-elliptic operators on Carnot manifolds is naturally contained in $K_*(C^*(\mathcal{G}))$. The aim is to develop computation tools in these $K$-groups. In particular, there will be three main results in this part:
\begin{enumerate}
    \item The Connes-Thom isomorphism allow us to establish isomorphism $K_*(C^*(\mathsf{G})) \simeq K_*(C_0(\mathfrak{g}^*))$, this is known since work of Nistor \cite{nistorsolvabel}. This result extends to the level of bundles, i.e. for a locally trivial bundle of nilpotent connected simply connected Lie groups we show that $K_*(C^*(\mathcal{G})) \simeq K_*(C_0(\mathbf{g}^*))$. The technique used to construct Nistor's Connes-Thom isomorphism is that of adiabatic deformations: a groupoid that smoothly deforms $\mathcal{G}$ to $\mathbf{g}$. 
    \item We prove that the procedure of Nistor above restricts from $C^*(\mathcal{G})$ to $I_P$ and gives an isomorphism $K_*(I_P) \simeq K^*(\Xi_P)$. Both isomorphisms for $C^*(\mathcal{G})$ and for $I_P$ are of an abstract nature and do not provide concrete recipe for transforming cycles in $K_*(I_P)$ to cycles in $K^*(\Xi_P)$. Extending the proof that $\delta(I_{\mathsf{G}}) = 0$ -- by constructing a concrete Morita equivalence bundle -- to the case of a bundle of Lie groups we construct an isomorphism $K_*(I_P) \xrightarrow{Mor} K^*(\Gamma_P) \xrightarrow{Thom} K^*(\Xi_P)$ from Morita invariance and the Thom isomorphism. This isomorphism coincides with the Connes-Thom isomorphisms up to a line bundle called the metaplectic correction bundle.
    \item Finally, we study surjectivity of the map $K_*(I_P) \rightarrow K_*(C^*(\mathcal{G}))$, induced from inclusion. In summary, the classes of $H$-elliptic operators in $K_*(C^*(\mathcal{G}))$ can be parametrized by $K^*(\Gamma_P)$ under a surjectivity assumption on the wrong way map $K^*(\Gamma_P)\to K^*(X)$ as it is possible to lift to $K_*(I_P)$ and then use the Morita equivalence to transform to cycles in $K^*(\Gamma_P)$.
\end{enumerate}

In Section \ref{sec:liegrohoad} we review definitions, basic facts and examples from the theory of Lie groupoids and algebroids. We recall the construction of the $C^*$-algebra of a Lie groupoid and the definition of the Lie algebroid associated with a Lie groupoid. In Section \ref{sec:groupoidforflatorbits} we describe central decomposition of a Lie group as a twisted semidirect product. This allows to establish equivariant isomorphisms between $C^*(\mathsf{G})$ and $C^*(\mathcal{G}_\mathsf{Z}, \omega_\mathsf{Z})$, where $\mathcal{G}_\mathsf{Z}$ is a certain groupoid over the dual of the center $\mathfrak{z}^*$. For nilpotent groups this isomorphism restricts to the ideal of flat orbits $I_{\mathsf{G}}$ and provides an isomorphism $I_{\mathsf{G}} \simeq C^*(\mathcal{G}_\Gamma, \omega_\Gamma)$ where $\mathcal{G}_\Gamma=\mathcal{G}_\mathsf{Z}|_\Gamma$.
 
 In Section \ref{subsec:loctrivalandnda} we extend the structures of Section \ref{sec:liegrohoad} and \ref{sec:groupoidforflatorbits} from nilpotent Lie groups to bundles of nilpotent Lie groups. For instance, we prove analogues of central decompositions considered in Section \ref{sec:groupoidforflatorbits}. In Section \ref{subsec:nistorctsubsn} we construct Connes-Thom isomorphisms $K_*(C^*(\mathcal{G})) \simeq K_*(C_0(\mathbf{g}^*))$ and $K_*(I_P) \simeq K^*(\Gamma_P)$. The construction requires notions of classical adiabatic groupoids, parabolically blowed adiabatic groupoid and a cocycle deformation of the adiabatic groupoid. At the end of the section we characterize surjectivity of the map $K_*(I_P) \rightarrow K_*(C^*(\mathcal{G}))$ induced from the inclusion. 
 
The results of Section \ref{sec:connesthomandadiaofofd} are those of most relevance for later results on index theory. A crucial result in this section states that the  Connes-Thom isomorphism $K_*(I_X) \simeq K^*(\Gamma_X)$ differs from the Morita equivalence defined from a bundle of flat orbit representations up to a line bundle -- the metaplectic correction bundle. This result lies at the heart of the index theorem as it allow us to explicitly compute the Connes-Thom map $K_*(C^*(\mathsf{G})) \to K_*(C_0(\mathfrak{g}^*))$ by finding a preimage under $K_*(I_P) \rightarrow K_*(C^*(\mathcal{G}))$ and the applying Morita invariance. 

\section{Lie groupoids and their $C^*$-algebras}
\label{sec:liegrohoad}

A groupoid can abstractly be defined as a small category in which all morphisms are isomorphisms. The abstract definition says quite litte and should rather be interpreted as having a space (the objects) with symmetries (the morphisms) that could vary over the space. We shall solely be concerned with Lie groupoids and twists of them.

We will recall the basic definitions for groupoids for the purpose of setting notations. The reader is expected to be acquainted with groupoids. The references  \cite{debordskandpseudo, debordskandext,debordlescure,mohsen1,lobster,vanErp_Yunckentangent,vanErp_Yuncken} can be consulted for a more extensive overview of groupoids.

\begin{definition}
Let $M$ be a smooth manifold. A Lie groupoid $\mathcal{G}\rightrightarrows M$ over $M$ is a smooth manifold of morphisms $\mathcal{G}$ equipped with the following structure maps:
\begin{enumerate}
    \item A source and range maps $s,r:\mathcal{G}\rightarrow M$.
    \item A composition map $\circ : \mathcal{G}^{(2)} \rightarrow \mathcal{G}$, where 
    \[ \mathcal{G}^{(2)} = \{ (\gamma_1, \gamma_2) \in \mathcal{G} \times \mathcal{G} : r(\gamma_2) = s(\gamma_1) \}, \]
    which is associative in the sense that $(\gamma_1\circ \gamma_2)\circ \gamma_3=\gamma_1\circ (\gamma_2\circ \gamma_3)$ whenever all compositions are defined.
    \item A unit map $e : M \rightarrow \mathcal{G}$, subject to $re=se=e$, that maps each object to the identity morphism in the sense that $e(r(\gamma))\circ\gamma=\gamma \circ e(s(\gamma))=\gamma$ for each $\gamma\in G$.
    \item An inverse map $i : \mathcal{G} \rightarrow \mathcal{G}$, subject to $ri=s$, $si=r$ and $i^2=\mathrm{id}_\mathcal{G}$, which sends each morphism to its inverse in the sense that $i(\gamma)\circ \gamma=e(s(\gamma))$ and $\gamma\circ i(\gamma)=e(r(\gamma))$.
\end{enumerate}
All structure maps are required to be smooth and $s,r$ are required to be submersions.
\end{definition}

A Lie group is a Lie groupoid over a point. In the same way that a Lie group has a Lie algebra, a Lie groupoid will have a Lie algebroid.

\begin{definition}
A Lie algebroid over a smooth manifold $M$ is a triple $(\mathbf{g}, \rho, [\cdot,\cdot])$ where
\begin{itemize}
    \item $\mathbf{g}\rightarrow M$ is a vector bundle.
    \item $\rho : \mathbf{g} \rightarrow TM$ is a vector bundle map called the anchor map.
    \item $[\cdot,\cdot] : C^\infty(M,\mathbf{g}) \times C^\infty(M,\mathbf{g}) \rightarrow C^\infty(M,\mathbf{g})$ is a Lie bracket such that for any $X, Y \in C^\infty(M,\mathbf{g})$, $f \in C^\infty(M)$ one has
    \[ [X,fY] = f[X,Y] + (\rho(X)f)Y. \]
\end{itemize}
\end{definition}

\begin{definition}
\label{liealgebrofleiedldp}
Let $\mathcal{G}\rightrightarrows M$ be a Lie groupoid over $M$. The Lie algebroid $\mathsf{Lie}(\mathcal{G})$ of $\mathcal{G}$ is defined as follows:
\begin{enumerate}
    \item The vector bundle $\mathsf{Lie}(\mathcal{G}) = \ker Ds \vert_{M}$ with Lie bracket defined from identifying sections of $\mathsf{Lie}(\mathcal{G})$ with right invariant vector fields on $\mathcal{G}$.
    \item The anchor map $\rho = Dr : \mathsf{Lie}(\mathcal{G}) \rightarrow T M$.
\end{enumerate}
\end{definition}

\begin{remark}
Under certain conditions one has an analogue of Lie's Third Theorem, which associates a Lie groupoid $\mathcal{G}$ with a Lie algebroid $\mathbf{g}$ such that $\mathsf{Lie}(\mathcal{G}) = \mathbf{g}$. In general, it can fail, see the discussion in \cite[Chapter 5.6]{debordskandpseudo}.
\end{remark}

Let us describe some examples that will play an important role later on in the monograph.

\begin{example}[Pair groupoid] 
Let $M$ be a smooth manifold. The manifold $\mathcal{G}:=M \times M$ is a groupoid over $M$ by setting 
\begin{align*}
r(x,y)&:=x, \quad s(x,y):=y,\\
(x, y) \circ (y, z) &:= (x, z), \\
i(x,y) &:= (y,x), \quad e(x) := (x, x). 
\end{align*}
\end{example}

\begin{example}[Transformation groupoid] 
\label{transgroupex}
Given a Lie group $G$ acting by diffeomorphisms on a smooth manifold $M$, define the Lie groupoid $M\rtimes G: = M \times G\rightrightarrows M$ with operations
\begin{align*}
r(x,g)&:=x, \quad s(x,g):=gx,\\
(x, g) \circ (xg, h) &:= (x,gh), \\
i(x, g) &:= (xg, g^{-1}), \quad  e(x) := (x, e_G). 
\end{align*}
\end{example}

\begin{example}[Deformation groupoid]
\label{ex:defomfodmomedm}
For a closed smooth submanifold $M_0$ of a smooth manifold $M$ with a normal bundle $N$. The deformation to the normal cone is the smooth manifold
    \[ D(M_0, M) = N \times \{0\} \sqcup M \times (0,1]. \]
See \cite{debordskandblowup} for the definition of the smooth structure on $D(M_0,M)$. A smooth groupoid $\mathcal{G}$ is called a deformation groupoid if $\mathcal{G} = D(\mathcal{H}_0, \mathcal{H})$ for a groupoid $\mathcal{H}$ and closed Lie subgroupoid $\mathcal{H}_0$ with object space $M$. $\mathcal{G}$ has object space $M \times [0,1]$. For more details on this construction and the next, see \cite{debordskandblowup}. We occasionally consider the deformation groupoid construction over $[0,\infty)$ instead of $[0,1]$. 
\end{example}

\begin{example}[Adiabatic groupoid] 
\label{ex:adiagropododo}
Given a smooth groupoid $\mathcal{G}\rightrightarrows M$ one can consider its Lie algebroid $\mathsf{Lie}(\mathcal{G})$ (see definition \ref{liealgebrofleiedldp}). $\mathsf{Lie}(\mathcal{G})$ is considered to be a Lie groupoid with the vector bundle operation. It is isomorphic to the normal bundle of unit the inclusion $M \hookrightarrow \mathcal{G}$. The adiabatic groupoid $\mathcal{G}_{ad}$ is defined to be 
    \[ \mathcal{G}_{\rm ad} = D(M, \mathcal{G})\equiv \mathsf{Lie}(\mathcal{G})\times \{0\}\sqcup \mathcal{G}\times (0,1]. \]
\end{example}

\begin{example}[Tangent groupoid]
\label{ex:tangentgroupoidladladld}
The tangent groupoid $\mathbb{T} M\rightrightarrows M\times [0,1]$ of a smooth manifold $M$ is defined to be adiabatic groupoid of the pair groupoid $M \times M\rightrightarrows M$. It was first studied by Connes \cite{connesbook} as a way of conceptualizing the Atiyah-Singer index theorem. The Lie algebroid of $\mathbb{T} M$ is $TM\times [0,1]$. An important variation of the tangent groupoid is the parabolic tangent groupoid of a Carnot manifold constructed by van Erp-Yuncken (see \cite{vanErp_Yunckentangent}), we review it below in the section on Carnot manifolds, see Section \ref{subsec:parabolictanget}.
\end{example}

\subsection{Twisted groupoid $C^*$-algebras}

To encode the representations of a groupoid, one associates a $C^*$-algebra. In order to be able to describe the full range of situations arising for nilpotent groups and Carnot manifolds, we will need the generality of twisted groupoid $C^*$-algebras. Before defining these $C^*$-algebras, we recall some basic notions for groupoids. 

Given a smooth groupoid $\mathcal{G}\rightrightarrows M$, for $x \in M$ we denote the $s$-fiber of $\mathcal{G}$ by
\[ \mathcal{G}_x := s^{-1}(\gamma)=\{ \gamma \in \mathcal{G} : s(\gamma) = x \}. \]
By definition, $\mathcal{G}_x$ is a smooth submanifold of $\mathcal{G}$. Similarly, $\mathcal{G}^x:=r^{-1}(x)$ denotes the $r$-fibre in $x\in M$. For any $\eta \in \mathcal{G}_x\cap \mathcal{G}^y$ there is an associated diffeomorphism
\[ R_\eta : \mathcal{G}_y \rightarrow \mathcal{G}_x, \ R_\eta(\gamma) = \gamma \circ \eta. \]

\begin{definition}
Let $\mathcal{G}\rightrightarrows M$ be a Lie groupoid. A Haar system on $\mathcal{G}$ is a family of smooth measures $(\mu_x)_{x\in M}$, where $\mu_x$ is a measure on the $s$-fiber $\mathcal{G}_x$, such that
\begin{enumerate}
    \item For any $f \in C_c^\infty(\mathcal{G})$, the function $M\ni x \mapsto \int_{\mathcal{G}_x} f d\mu_x\in \C$ is smooth on $M$.
    \item  For any $\eta \in \mathcal{G}_x\cap \mathcal{G}^y$, the diffeomorphism $R_\gamma : G_y \rightarrow G_x$ is measure-preserving.
\end{enumerate}
\end{definition}

We note that the second condition in the definition of a Haar system encodes that our convention is to use right invariant Haar system. Any two Haar systems on a Lie groupoid $\mathcal{G}\rightrightarrows M$ differ by a smooth, positive function on $M$. Therefore, the Haar system is  essentially unique.

\begin{definition}
A smooth 2-cocycle on a Lie groupoid $\mathcal{G}$ is a smooth map $\omega : \mathcal{G}^{(2)} \rightarrow U(1)$ such that
\[ \omega(x,y)\omega(xy,z) = \omega(x,yz) \omega(y,z) \text{ whenever } (x,y), \ (y,z) \in \mathcal{G}^{(2)}, \]
\[ \omega(x, s(x)) = 1 = \omega(r(x),x). \]
\end{definition}

\begin{definition}
\label{convolution}
Let $\mathcal{G}\rightrightarrows M$ be a Lie groupoid and $\omega$ a smooth $2$-cocycle on $\mathcal{G}$. The twisted convolution algebra of $(\mathcal{G},\omega)$ is the $*$-algebra $C_c^\infty(\mathcal{G}, \omega):=C^\infty_c(\mathcal{G})$ of compactly supported smooth functions on $\mathcal{G}$ equipped with the product defined from twisted convolution
\[f \ast g(\gamma) = \int_{\mathcal{G}_{r(\gamma)}} \omega(\eta, \eta^{-1} \gamma) f(\eta) g(\eta^{-1} \gamma)\rd \mu^{r(\gamma)}, \quad f,g\in C^\infty_c(\mathcal{G},\omega),\]
and the involution is given by 
\[f^*(\gamma) = \overline{f(\gamma^{-1})\omega(\gamma, \gamma^{-1})}.\]
\end{definition}

\begin{remark}
\label{schwarszremamd1}
We say that a Lie groupoid $\mathcal{G}\rightrightarrows M$ is a polynomial Lie groupoid if $M$ and $\mathcal{G}$ are manifolds with polynomial structures (see \cite{pedersen89}), and all maps involved in the groupoid structure are polynomial.  For a smooth $2$-cocycle $\omega$ on a polynomial Lie groupoid  the Schwartz space $\mathcal{S}(\mathcal{G}, \omega):=\mathcal{S}(\mathcal{G})$ forms a $*$-algebra with the identically defined operations as that of $C_c^\infty(\mathcal{G}, \omega)$. There is in this case an inclusion $C_c^\infty(\mathcal{G}, \omega)\subseteq \mathcal{S}(\mathcal{G}, \omega)$ which defines a $*$-homomorphism with dense range. 
\end{remark}

Restriction along the unit defines a mapping $E_M:C^\infty_c(\mathcal{G},\omega)\to C^\infty_c(M)$. For $f_1,f_2\in C^\infty_c(\mathcal{G},\omega)$ we define $\langle f_1,f_2\rangle_{C_0(M)}\in C^\infty_c(M)$ by 
$$\langle f_1,f_2\rangle_{C_0(M)}:=E_M(f_1^*\ast f_2).$$
Let $\mathpzc{E}_{\mathcal{G},\omega}$ denote the completion of $C^\infty_c(\mathcal{G},\omega)$ as a $C_0(M)$-Hilbert $C^*$-module. Up to natural isomorphism of $C_0(M)$-Hilbert $C^*$-modules, $\mathpzc{E}_{\mathcal{G},\omega}$ is independent of $\omega$. Left multiplication extends to a faithful $*$-representation 
$$\lambda^\omega:C^\infty_c(\mathcal{G},\omega)\to \End_{C_0(M)}^*(\mathpzc{E}_{\mathcal{G},\omega}).$$
We call $\lambda$ the left regular representation of $\mathcal{G}$.

\begin{definition}
Let $\mathcal{G}\rightrightarrows M$ be a Lie groupoid and $\omega$ a smooth $2$-cocycle on $\mathcal{G}$. The reduced $C^*$-algebra $C_r^*(\mathcal{G}, \omega)$ is the completion of $C_c^\infty(\mathcal{G}, \omega)$ with respect to the norm
\[ \| f \|_{C_r^*(\mathcal{G}, \omega)} :=  \| \lambda^\omega(f) \|_{ \End_{C_0(M)}^*(\mathpzc{E}_{\mathcal{G},\omega})}.\]

The full $C^*$-algebra $C^*(\mathcal{G}, \omega)$ is a completion of $C_c^\infty(\mathcal{G}, \omega)$ with respect to all $*$-representations by bounded operators on a Hilbert space (for more details see \cite{renaultbook}). 

If $\omega\equiv 1$, we write $C_r^*(\mathcal{G})$ and $C^*(\mathcal{G})$ instead of $C_r^*(\mathcal{G}, 1)$ and $C^*(\mathcal{G}, 1)$, respectively.
\end{definition}

\begin{remark}
\label{schwarszremamd2}
If $\mathcal{G}\rightrightarrows M$ is a polynomial Lie groupoid with a smooth $2$-cocycle $\omega$, the inclusion $C_c^\infty(\mathcal{G}, \omega)\subseteq \mathcal{S}(\mathcal{G}, \omega)$ extends to a continuous dense inclusion $\mathcal{S}(\mathcal{G},\omega)\subseteq C^*_r(\mathcal{G},\omega)$. Similarly, there is a continuous dense inclusion $\mathcal{S}(\mathcal{G},\omega)\subseteq C^*(\mathcal{G},\omega)$. 
\end{remark}

\begin{example} 
Consider the transformation groupoid  $\mathcal{G}= M\rtimes G $ associated with a Lie group acting smoothly on a manifold. A choice of Haar system corresponds to a choice of Haar measure on $G$. The convolution product on $C_c^\infty(\mathcal{G})$ is given by
    \[ f_1 \ast f_2(x,g) = \int_G f_1(h, x) f_2(h^{-1}g, xh) \rd h, \]
    where $\rd h$ is the Haar measure on $G$. The involution is given by
    \[ f^*(g, x) = \overline{f(g^{-1}, xg)}. \]
Therefore $C_r^*(\mathcal{G})=C_0(M)\rtimes_r G$ and $C^*(\mathcal{G})=C_0(M)\rtimes G$.
\end{example}

\begin{example}
Let $M$ be a smooth manifold and consider the pair groupoid $\mathcal{G} = M \times M$. A Haar system on $\mathcal{G}$ corresponds to a choice of a positive Radon measure on $M$. We suppress the choice of this measure from the notations. The convolution on $C_c^\infty(\mathcal{G})$ is given by
\[ f_1 \ast f_2(x,y) = \int_M f_1(x, z) f_2(z, y) \rd z. \]
The involution is given by
\[ f^*(x,y) = \overline{f(y,x)}. \]
The reduced $C^*$-algebra of $\mathcal{G}$ is isomorphic to the $C^*$-algebra of compact operators $\mathbb{K}(L^2(M))$. In fact, this isomorphism is constructed by localizing the left regular representation in a point $x\in M$ and the canonical left $\mathcal{G}$-equivariant diffeomorphism $M\cong \mathcal{G}_x$.

\end{example}

\begin{theorem}
\label{restrictqotuegroudp}
Let $\mathcal{G}\rightrightarrows M$ be a Lie groupoid with a smooth $2$-cocycle $\omega$. For any subset $X\subseteq M$ we define $\mathcal{G}_X:=s^{-1}(X)\cap r^{-1}(X)$ which is a topological groupoid (but not a Lie groupoid in general). If $X\subseteq M$ is open and $\mathcal{G}$-invariant, $\mathcal{G}_X$ is an open Lie subgroupoid of $\mathcal{G}$ and $C^*(\mathcal{G}_X,\omega)$ is an ideal in $C^*(\mathcal{G},\omega)$. If $X$ is a closed $\mathcal{G}$-invariant submanifold, $\mathcal{G}_X$ is a closed Lie subgroupoid for which restriction defines an epimorphism $C^*(\mathcal{G},\omega)\to C^*(\mathcal{G}_X,\omega)$ that fits into an exact sequence of full groupoid $C^*$-algebras:
    \[ 0 \rightarrow C^*( \mathcal{G}_{M \setminus X},\omega) \rightarrow C^*(\mathcal{G},\omega) \rightarrow C^*(\mathcal{G}_X,\omega) \rightarrow 0. \]
\end{theorem}

We refer the proof of this theorem to \cite{renaultbook}. The reader should beware that in general, the exactness property in Theorem \ref{restrictqotuegroudp} does not hold for the reduced completion. The following corollary is an important consequence that we will use abundantly.

\begin{corollary}
\label{adiadefcororo}
Let $\mathcal{G}\rightrightarrows M$ be a Lie groupoid. Assume that $\omega$ is a smooth $2$-cocycle on the adiabatic groupoid $\mathcal{G}_{ad}$. Let $\mathcal{G}_{\rm ad}^\circ$ denote the restriction of $\mathcal{G}_{\rm ad}$ to $M \times [0,1)$. Then one has the following exact sequence:
$$0 \rightarrow C^*(\mathcal{G}\times (0,1),\omega) \rightarrow C^*(\mathcal{G}_{\rm ad}^\circ,\omega) \xrightarrow{{\rm ev}_0} C^*(\mathsf{Lie}(\mathcal{G}),\omega) \rightarrow 0.$$
If $\omega=1$, the short exact sequence reduces to 
\[ 0 \rightarrow C_0(0,1)\otimes C^*(\mathcal{G}) \rightarrow C^*(\mathcal{G}_{\rm ad}^\circ) \xrightarrow{{\rm ev}_0} C_0(\mathsf{Lie}(\mathcal{G})^*) \rightarrow 0, \]
where $\mathsf{Lie}(\mathcal{G})^*\to M$ denotes the dual vector bundle of the Lie algebroid.
\end{corollary}

\section{A groupoid description of the flat orbits}
\label{sec:groupoidforflatorbits}

The ideal of flat orbits of a simply connected nilpotent Lie group can be described as a twisted groupoid $C^*$-algebra. We shall see that this relates the flat orbits to a noncommutative Fourier transform of functions on the total space of the vector bundle $\Xi\to \Gamma$.

We start with some preliminary considerations. Consider a Lie group $\mathsf{G}$ with a closed central subgroup $\mathsf{Z}\subseteq \mathsf{G}$. For notational convenience, we use multiplicative notation on $\mathsf{G}$ and $\mathsf{G}/\mathsf{Z}$ but additive notation on $\mathsf{Z}$. Assume that $\tau:\mathsf{G}/\mathsf{Z}\to \mathsf{G}$ is a smooth splitting of the quotient mapping $\mathsf{G}\to \mathsf{G}/\mathsf{Z}$. We define the smooth $2$-cocycle
\[ \omega : \mathsf{G}/\mathsf{Z} \times \mathsf{G}/\mathsf{Z} \rightarrow \mathsf{Z}, \ \omega(g_1 \mathsf{Z}, g_2 \mathsf{Z}) := \tau(g_1 \mathsf{Z})\tau(g_2 \mathsf{Z}) \tau(g_1 g_2 \mathsf{Z})^{-1}. \]

\begin{proposition}
The splitting $\tau$ induces an isomorphism $\mathsf{G} \simeq \mathsf{G}/\mathsf{Z} \rtimes_{\omega} \mathsf{Z}$, where the cocycle twisted product $\mathsf{G}/\mathsf{Z} \rtimes_{\omega} \mathsf{Z}$ is the group defined as the set $\mathsf{G} / \mathsf{Z} \times \mathsf{Z}$ equipped with the group operation
\[ (g_1 \mathsf{Z}, z_1) \cdot_{\omega} (g_2 \mathsf{Z}, z_2) := (g_1g_2\mathsf{Z}, z_1+z_2+\omega(g_1 \mathsf{Z}, g_2 \mathsf{Z}) ). \]
\end{proposition}

We are interested in the situation that $\mathsf{G}$ is a simply connected nilpotent Lie group and $\mathsf{Z}$ is its center. We can define the $2$-cocycle
\begin{equation}
\label{2cocycfromxenter}
\omega_Z:\mathsf{G} / \mathsf{Z} \times \mathsf{G} / \mathsf{Z} \to \mathsf{Z},
\end{equation}
by choosing a Jordan-Hölder basis that splits $\mathsf{G} \to \mathsf{G} / \mathsf{Z}$. The reader should note that this $2$-cocycle pairs with $\mathfrak{z}^*$ to the Kirillov form on the flat orbits, cf. Remark \ref{flatorbisintosform}. By the discussion above we have that $\mathsf{G}\cong \mathsf{G} / \mathsf{Z} \times_{\omega_Z} \mathsf{Z}$. Let us translate this context into a twisted groupoid. The crucial observation is that $\mathsf{Z}=\mathfrak{z}$ whose $C^*$-algebra is isomorphic to $C_0(\mathfrak{z}^*)$ via the Fourier transform.

\begin{definition}
\label{groipodicododve}
Let $\mathsf{G}$ be a simply connected nilpotent Lie group with center $\mathsf{Z}=\mathfrak{z}$. Define the groupoid $\mathcal{G}_\mathsf{Z}\rightrightarrows \mathfrak{z}^*$ as the transformation groupoid (see Example \ref{transgroupex}) of the trivial $\mathsf{G} / \mathsf{Z}$-action on $\mathfrak{z}^*$. 

We write $\mathcal{G}_\Gamma$ for the restriction of $\mathcal{G}_\mathsf{Z}$ to the open subset of flat orbits $\Gamma\subseteq \mathfrak{z}^*$.
\end{definition}

We note that $\mathcal{G}_\mathsf{Z}^{(2)} = \mathsf{G} / \mathsf{Z} \times \mathsf{G} / \mathsf{Z} \times \mathfrak{z}^*$ and the composition is the product of the two $\mathsf{G} / \mathsf{Z}$-factors. 

\begin{definition}
\label{unitary2coclalddaadl}
Define the unitary 2-cocycle $\omega_{\mathcal{G}_\mathsf{Z}}$ on $\mathcal{G}_\mathsf{Z}$ by
\[ \omega_{\mathcal{G}_\mathsf{Z}} : \mathcal{G}_\mathsf{Z}^{(2)} = \mathsf{G} / \mathsf{Z} \times \mathsf{G} / \mathsf{Z} \times \mathfrak{z}^* \rightarrow U(1), \ \omega_{\mathcal{G}_Z}(g \mathsf{Z}, h \mathsf{Z}, \xi) = e^{-i \xi(\omega_\mathsf{Z}(g \mathsf{Z}, h \mathsf{Z}))}. \]
\end{definition}

The group $\Aut(\mathsf{G})$ acts as smooth groupoid automorphisms on $\mathcal{G}_\mathsf{Z}$ by 
$$\varphi(g\mathsf{Z},\xi):=(\varphi(g)\mathsf{Z},(\varphi^{-1})^*\xi).$$
We are here implicitly using that a groupoid automorphism preserves the center. Let us describe a lift of this action to an action on $C^*(\mathcal{G}_\mathsf{Z},\omega_{\mathcal{G}_\mathsf{Z}})$.

\begin{proposition}
Let $\mathsf{G}$ be a simply connected nilpotent Lie group with center $\mathsf{Z}$ and fix a smooth splitting $\tau:\mathsf{G}/\mathsf{Z}\to \mathsf{G}$. For $\varphi\in \Aut(\mathsf{G})$, define $\tau_\varphi:=\varphi^{-1}\circ \tau \circ \varphi$ and 
$$b_\varphi:=\tau_\varphi\cdot \tau^{-1}:\mathsf{G}/\mathsf{Z}\to \mathsf{Z}.$$
It holds that the group $\Aut(\mathsf{G})$ acts strongly continuously on $C^*(\mathcal{G}_\mathsf{Z},\omega_{\mathcal{G}_\mathsf{Z}})$ via the action 
$$\varphi.f(g\mathsf{Z},\xi)=e^{-i \xi b_{\varphi^{-1}}(g \mathsf{Z})}(\varphi^{-1})^*(f)(g\mathsf{Z},\xi).$$
\end{proposition}

\begin{proof}
Writing $\varphi^*\omega_{\mathcal{G}_\mathsf{Z}}:=\omega_{\mathcal{G}_\mathsf{Z}}\circ (\varphi\times \varphi)$ on $\mathcal{G}_\mathsf{Z}^{(2)}$, we compute that 
\begin{align*}
\varphi^*\omega_{\mathcal{G}_\mathsf{Z}}(g \mathsf{Z}, h \mathsf{Z}, \xi)=&\omega_{\mathcal{G}_\mathsf{Z}}(\varphi(g)\mathsf{Z},\varphi(h)\mathsf{Z},(\varphi^{-1})^*\xi)=e^{-i \xi(\varphi^{-1}(\omega_\mathsf{Z}(\varphi(g) \mathsf{Z},\varphi(h) \mathsf{Z}))}.
\end{align*}
If $\omega_{\mathcal{G}_\mathsf{Z}}(g \mathsf{Z}, h \mathsf{Z}, \xi)=\tau(g \mathsf{Z})\tau(h \mathsf{Z}) \tau(g h \mathsf{Z})^{-1}$ for the smooth splitting $\tau$, then 
$$\varphi^{-1}(\omega_\mathsf{Z}(\varphi(g) \mathsf{Z},\varphi(h) \mathsf{Z}))=\tau_\varphi(g \mathsf{Z})\tau_\varphi(h\mathsf{Z}) \tau_\varphi(gh\mathsf{Z})^{-1},$$ 
for the splitting $\tau_\varphi:=\varphi^{-1}\circ \tau \circ \varphi$. Since splittings are unique up to $\mathsf{Z}$, we have that 
$$b_\varphi:=\tau_\varphi\cdot \tau^{-1}:\mathsf{G}/\mathsf{Z}\to \mathsf{Z},$$
is smooth and depends smoothly on $\varphi\in \Aut(\mathsf{G})$. Moreover, 
\begin{align*}
\varphi^*\omega_{\mathcal{G}_\mathsf{Z}}(g \mathsf{Z}, h \mathsf{Z}, \xi)=&\omega_{\mathcal{G}_\mathsf{Z}}(g \mathsf{Z}, h \mathsf{Z}, \xi) e^{-i \xi b_\varphi(g \mathsf{Z})}e^{-i \xi b_\varphi(h \mathsf{Z})}e^{i \xi  b_\varphi(g h \mathsf{Z})}.
\end{align*}

The computations above implies that $\Aut(\mathsf{G})$ acts $C^*(\mathcal{G}_\mathsf{Z},\omega_{\mathcal{G}_\mathsf{Z}})$ by setting 
$$\varphi.f(g\mathsf{Z},\xi)=e^{-i \xi b_{\varphi^{-1}}(g \mathsf{Z})}(\varphi^{-1})^*(f)(g\mathsf{Z},\xi),$$
for $f\in \mathcal{S}(\mathcal{G}_\mathsf{Z},\omega_{\mathcal{G}_\mathsf{Z}})$. Indeed, for $\varphi\in \Aut(\mathsf{G})$ and $f_1,f_2\in \mathcal{S}(\mathcal{G}_\mathsf{Z},\omega_{\mathcal{G}_\mathsf{Z}})$,
\small
\begin{align*}
(\varphi.f_1 &\ast_{\omega_{\mathcal{G}_\mathsf{Z}}}  \varphi.f_2)(g, \xi)  =\\
&= \int_{\mathsf{G} / \mathsf{Z}} e^{-i \xi b_{\varphi^{-1}}(h)}f_1(\varphi^{-1}h, \varphi^*\xi) e^{-i \xi b_{\varphi^{-1}}(h^{-1}g )}f_2(\varphi^{-1}(h^{-1}g), \varphi^*\xi) \omega_{\mathcal{G}_\mathsf{Z}}(h, h^{-1}g, \xi) \rd h =\\
&=\int_{\mathsf{G} / \mathsf{Z}} f_1(\varphi^{-1}h, \varphi^*\xi) f_2(\varphi^{-1}(h^{-1}g), \varphi^*\xi)(\varphi^{-1})^*\omega_{\mathcal{G}_\mathsf{Z}}(h , h^{-1}g , \xi)\e^{i \xi  b_{\varphi^{-1}}(g)}\rd h=\\
&=\int_{\mathsf{G} / \mathsf{Z}} f_1(h', \varphi^*\xi) f_2(h'^{-1}\varphi^{-1}(g), \varphi^*\xi)\omega_{\mathcal{G}_\mathsf{Z}}(h' , h'^{-1}\varphi^{-1}(g) , \varphi^*\xi)\e^{i \xi  b_{\varphi^{-1}}(g)}\rd h'=\\
&=\varphi.(f_1 \ast_{\omega_{\mathcal{G}_\mathsf{Z}}}  f_2)(g, \xi).
\end{align*}
\normalsize
In the second last equality we changed coordinates to $h'=\varphi^{-1}h$. We arrive at a strongly continuous action $\Aut(\mathsf{G})\to \Aut(C^*(\mathcal{G}_\mathsf{Z},\omega_{\mathcal{G}_\mathsf{Z}}))$.
\end{proof}

\begin{theorem}
\label{groupoiddescropofig}
Let $\mathsf{G}$ be a simply connected nilpotent Lie group with center $\mathsf{Z}=\mathfrak{z}$. Fourier transform in the central direction defines an $\mathsf{Aut}(\mathsf{G})$-equivariant $*$-isomorphism 
\[ \mathcal{F}_Z : C^*(\mathsf{G}) \rightarrow C^*(\mathcal{G}_\mathsf{Z}, \omega_{\mathcal{G}_\mathsf{Z}}),\]
where the Fourier transform in the central direction is defined by density on $f\in \mathcal{S}(\mathsf{G})=\mathcal{S}(\mathsf{G}/\mathsf{Z}\times \mathsf{Z})$ (cf. Remark \ref{schwarszremamd1} and \ref{schwarszremamd2}) as $\mathcal{F}_\mathsf{Z}(f)\in \mathcal{S}(\mathcal{G}_\mathsf{Z}, \omega_{\mathcal{G}_Z})$ where 
\begin{equation}
\label{centralft}
\mathcal{F}_\mathsf{Z}(f)(g \mathsf{Z}, \xi) = \int_{\mathfrak{z}} e^{i\xi(t)} f(g\mathsf{Z}, t) \,\mathrm{d}t. 
\end{equation}
Moreover, this isomorphism restricts to an isomorphism $I_\mathsf{G}\cong C^*(\mathcal{G}_\Gamma,\omega_{\mathcal{G}_\mathsf{Z}})$, i.e. we have a commuting diagram
\[
\begin{tikzcd}
I_\mathsf{G} \arrow[r, "\mathcal{F}_\mathsf{Z}"]\arrow[d, "\subseteq"]  & C^*(\mathcal{G}_\Gamma,\omega_{\mathcal{G}_\mathsf{Z}}) \arrow[d, "\subseteq"] \\
C^*(\mathsf{G})\arrow[r, "\mathcal{F}_\mathsf{Z}"]& C^*(\mathcal{G}_\mathsf{Z},\omega_{\mathcal{G}_\mathsf{Z}}) 
\end{tikzcd}
\]
\end{theorem}

The last conclusion of the theorem is interpreted as an empty statement if $\Gamma$ is empty. Another way of formulating the last conclusion of the theorem is that $I_\mathsf{G}\subseteq C^*(\mathsf{G})$ consists of those elements whose Fourier transform in the central direction is supported in $\Gamma\subseteq \mathfrak{z}^*$. In Equation \eqref{centralft} we are implicitly identifying $\mathsf{G}=\mathsf{G}/\mathsf{Z}\times \mathsf{Z}$ via the splitting $\tau$, i.e. we are identifying $f(g\mathsf{Z}, t)=f(\tau(g\mathsf{Z})t)$.

\begin{proof}
For $f\in \mathcal{S}(\mathsf{G})$ it is clear that $\mathcal{F}_\mathsf{Z}(f)\in \mathcal{S}(\mathcal{G}_\mathsf{Z})$. We start with verifying the equivariance property. For $\varphi\in \Aut(\mathsf{G})$ and $f\in \mathcal{S}(\mathsf{G})$ we have that $\varphi.f(g)=f(\varphi^{-1}(g))$ which we in Equation \eqref{centralft} are implicitly identifying with 
\begin{align*}
\varphi.f(g\mathsf{Z}, t)=&f(\varphi^{-1}(\tau(g\mathsf{Z}))\varphi^{-1}(t))=f(\tau_\varphi(\varphi^{-1}(g)\mathsf{Z})\varphi^{-1}(t))=\\
=&f(\tau(\varphi^{-1}(g)\mathsf{Z})b_{\varphi}(\varphi^{-1}(g)\mathsf{Z})\varphi^{-1}(t)),
\end{align*}
where we in the last equality used that $\tau_\varphi=\tau\cdot b_\varphi$ where $b_\varphi$ takes central values. Therefore 
\begin{align*}
\mathcal{F}_\mathsf{Z}(\varphi.f)(g \mathsf{Z}, \xi) =&\int_{\mathfrak{z}} e^{i\xi(t)}f(\tau(\varphi^{-1}(g)\mathsf{Z})b_\varphi(g\mathsf{Z})\varphi^{-1}(t))\mathrm{d}t=\\
=&\int_{\mathfrak{z}} e^{i\xi(\varphi(s)-\varphi(b_\varphi(\varphi^{-1}(g)\mathsf{Z})))}f(\tau(\varphi^{-1}(g)\mathsf{Z})s)\mathrm{d}s=\\
=&e^{-i \xi b_{\varphi^{-1}}(g \mathsf{Z})}\int_{\mathfrak{z}} e^{i\varphi^*(\xi) s}f(\tau(\varphi^{-1}(g)\mathsf{Z})s)\mathrm{d}s=\left(\varphi.\mathcal{F}_\mathsf{Z}(f)\right)(g \mathsf{Z}, \xi),
\end{align*}
where we in the second identity made the change of variables $s=b_\varphi(g\mathsf{Z})\varphi^{-1}(t)$ and in the third equality we used that $\varphi\circ b_\varphi\circ \varphi^{-1}=b_{\varphi^{-1}}$.

Let us verify that $\mathcal{F}_\mathsf{Z} :\mathcal{S}(\mathsf{G})\to  \mathcal{S}(\mathcal{G}_\mathsf{Z}, \omega_{\mathcal{G}_\mathsf{Z}})$ is a $*$-homomorphism. For $f_1,f_2\in \mathcal{S}(\mathsf{G})$, we have that 
\small
\begin{align*}
(\mathcal{F}_\mathsf{Z}(f_1)  &\ast_{\omega_{\mathcal{G}_\mathsf{Z}}}  \mathcal{F}_\mathsf{Z}(f_2))(g, \xi)  = \int_{\mathsf{G} / \mathsf{Z}} \mathcal{F}_\mathsf{Z}(f_1)(h, \xi) \mathcal{F}_\mathsf{Z}(f_2)(h^{-1}g, \xi) \omega_{\mathcal{G}_Z}(h, h^{-1}g, \xi) \rd h = \\ 
& = \int_{\mathsf{G} / \mathsf{Z}} \int_{\mathfrak{z}} \int_{\mathfrak{z}} e^{i \xi(t_1 + t_2)} f_1(h, t_1) f_2(h^{-1}g, t_2) \omega_{\mathcal{G}_Z}(h, h^{-1}g, \xi) \rd t_1 \rd t_2 \rd h = \\ 
& = \int_{\mathsf{G} / \mathsf{Z}} \int_{\mathfrak{z}} \int_{\mathfrak{z}} e^{i \xi(t_1 + \omega_Z(h, h^{-1}g))} f_1(h, t_1 - t_2 - \omega_Z(h, h^{-1}g))\cdot \\
&\qquad\qquad\qquad\qquad\qquad\qquad\qquad\qquad\cdot f_2(h^{-1}g, t_2) \omega_{\mathcal{G}_Z}(h, h^{-1}g, \xi) \rd t_1 \rd t_2 \rd h = \\ 
& = \int_{\mathsf{G} / \mathsf{Z}} \int_{\mathfrak{z}} \int_{\mathfrak{z}} e^{i \xi(t_1)} f_1(h, t_1 - t_2 - \omega_Z(h, h^{-1}g)) f_2(h^{-1}g, t_2) \rd t_1 \rd t_2 \rd h = \\ 
& = \int_{\mathfrak{z}} e^{i \xi(t_1)} (f_1 \ast f_2)(g, t_1) \rd t_1 = \mathcal{F}_\mathsf{Z}(f_1 \ast f_2)(g, \xi).
\end{align*}
\normalsize
By a similar computation, we have for $f\in \mathcal{S}(\mathsf{G})$ that 
\begin{align*}
\mathcal{F}_\mathsf{Z}(f^*)  &(g, \xi)  = \int_{\mathfrak{z}} e^{i\xi(t)} \overline{f(g^{-1}\mathsf{Z}, -t+\omega_\mathsf{Z}(g,g^{-1}))} \rd t=\\
&=e^{i\xi(\omega_\mathsf{Z}(g,g^{-1}))}\int_{\mathfrak{z}} e^{-i\xi(t)} \overline{f(g^{-1}\mathsf{Z}, t)} \rd t=\\
&=e^{i\omega_{\mathcal{G}_\mathsf{Z}}((g,\xi) ,(g^{-1},\xi)))}\overline{\int_{\mathfrak{z}} e^{i\xi(t)} f(g^{-1}\mathsf{Z}, t) \rd t}=\mathcal{F}_\mathsf{Z}(f)^*(g,\xi). 
\end{align*}
This shows that $\mathcal{F}_\mathsf{Z}:\mathcal{S}(\mathsf{G})\to  \mathcal{S}(\mathcal{G}_\mathsf{Z}, \omega_{\mathcal{G}_\mathsf{Z}})\subseteq C^*(\mathcal{G}_\mathsf{Z}, \omega_{\mathcal{G}_\mathsf{Z}})$ is a $*$-homomorphism. By the universal property, the Fourier transform in the central direction extends to a $*$-homomorphism $\mathcal{F}_\mathsf{Z} : C^*(\mathsf{G}) \rightarrow C^*(\mathcal{G}_\mathsf{Z}, \omega_{\mathcal{G}_\mathsf{Z}})$. By Fourier inversion, $\mathcal{F}_\mathsf{Z}:\mathcal{S}(\mathsf{G})\to  \mathcal{S}(\mathcal{G}_\mathsf{Z}, \omega_{\mathcal{G}_\mathsf{Z}})$ is a $*$-isomorphism, with inverse being the inverse Fourier transform, which by the same argument extends to a $*$-homomorphism $\mathcal{F}_\mathsf{Z}^{-1} : C^*(\mathsf{G}) \rightarrow C^*(\mathcal{G}_\mathsf{Z}, \omega_{\mathcal{G}_\mathsf{Z}})$ that inverts $\mathcal{F}_\mathsf{Z}$. Therefore $\mathcal{F}_\mathsf{Z} : C^*(\mathsf{G}) \rightarrow C^*(\mathcal{G}_\mathsf{Z}, \omega_{\mathcal{G}_\mathsf{Z}})$ is a $*$-isomorphism.

It remains to prove that $\mathcal{F}_\mathsf{Z}$ restricts to a $*$-isomorphism $I_\mathsf{G}\cong C^*(\mathcal{G}_\Gamma,\omega_{\mathcal{G}_\mathsf{Z}})$. For $\xi\in \mathfrak{z}^*$, we write $ \lambda^{\omega_{\mathcal{G}_\mathsf{Z}}}_\xi$ for the localization of the left regular representation of $C^*(\mathcal{G}_\mathsf{Z},\omega_{\mathcal{G}_\mathsf{Z}})$ in $\xi$. We let $\pi_{\mathsf{Z},\xi}:= \lambda^{\omega_{\mathcal{G}_\mathsf{Z}}}_\xi\circ \mathcal{F}_\mathsf{Z}$ denote the associated unitary representation of $G$. Restricting $\pi_{\mathsf{Z},\xi}$ to the center leads to the dichotomy that $\pi_{\mathsf{Z},\xi}$ is either irreducible and corresponds to the flat orbit $\mathcal{O}_\xi$ (which holds if and only if $\xi\in \Lambda$) or decomposes over representations from $\hat{\mathsf{G}}\setminus \Gamma$ (which holds if and only if $\xi\in \mathfrak{z}^*\setminus \Lambda$). We conclude that
\begin{align*}
I_\mathsf{G}&=\{a\in C^*(\mathsf{G}): \lambda^{\omega_{\mathcal{G}_\mathsf{Z}}}_\xi(\mathcal{F}_\mathsf{Z}a)=0\,\forall\xi\in \mathfrak{z}^*\setminus \Lambda\}=\\
&=\{a\in C^*(\mathsf{G}): \mathcal{F}_\mathsf{Z}a\in C^*(\mathcal{G}_\Gamma,\omega_{\mathcal{G}_\mathsf{Z}})\},
\end{align*}
and the theorem follows.
\end{proof}

\section{Locally trivial bundles of nilpotent Lie groups}
\label{subsec:loctrivalandnda}

In this subsection, we consider a situation appearing on Carnot manifolds (see Section \ref{subsec:carnot} below). The tangent space of a Carnot manifold can be equipped with a certain Lie algebroid structure that integrates to a groupoid with $r=s$ and each $r$-fibre is a simply connected nilpotent Lie group. 

\begin{proposition}
\label{integrationofnilpotentalgebroid}
Let $(\mathbf{g}, \rho, [\cdot,\cdot])$ be a Lie algebroid on $M$ with anchor map $\rho=0$. Then the Lie bracket on $\mathbf{g}$ restricts to a well-defined Lie bracket on each fibre and $\mathbf{g}\to M$ is a smooth bundle with a fibrewise defined Lie algebra structure depending smoothly on the base. If there exists an $N\geq 0$ such that each fibre of $\mathbf{g}$ is at most step $N$ nilpotent, the Hausdorff-Baker-Campbell formula integrates to a Lie groupoid structure on $\mathcal{G}:=\mathbf{g}$ with  $\mathsf{Lie}(\mathcal{G}) = \mathbf{g}$.
\end{proposition}

\begin{proof}
Since $\rho=0$, the bracket $[\cdot,\cdot]$ is $C^\infty(M)$-linear and the first part of the proposition follows. The Hausdorff-Baker-Campbell formula defines a polynomial map $\mathbf{g}\times \mathbf{g}\to \mathbf{g}$ which we can consider a composition on $\mathcal{G}:=\mathbf{g}$. It is readily verified that this is a Lie groupoid for $r=s$ being the projection $\mathbf{g}\to M$, $i$ defined as the additive inverse on each fiber and $e:M\to \mathcal{G}=\mathbf{g}$ being the zero section.
\end{proof}

In the case of interest in this work, regular Carnot manifolds (see Section \ref{subsec:carnot} below), further structures are imposed. 

\begin{definition}
Let $\mathcal{G}\to M$ be a Lie groupoid with $r=s$, and simply connected $r$-fibers and let $\mathfrak{g}$ be a nilpotent Lie algebra. 

We say that $\mathcal{G}$ is a locally trivial bundle of nilpotent Lie groups of type $\mathfrak{g}$ if the Lie algebroid $\mathbf{g}:=\mathsf{Lie}(\mathcal{G})\to M$ is a locally trivial bundle of Lie algebras with typical fibre $\mathfrak{g}$, i.e. every $x\in M$ admits a neighbourhood $U$ and a Lie algebroid isomorphism $\mathbf{g}_U\cong U\times \mathfrak{g}$.

If $\mathfrak{g}$ is graded, we say that $\mathcal{G}$ is a locally trivial bundle of Carnot-Lie groups of type $\mathfrak{g}$ if the Lie algebroid $\mathbf{g}:=\mathsf{Lie}(\mathcal{G})\to M$ is graded and the local trivializations are graded.
\end{definition}

\begin{proposition}
\label{integrationofloctrivialnilpotentalgebroid}
Let $\mathcal{G}\to M$ be a locally trivial bundle of nilpotent Lie groups of type $\mathfrak{g}$, and set $\mathbf{g}:=\mathsf{Lie}(\mathcal{G})$ and $\mathsf{G}:=\mathrm{exp}(\mathfrak{g})$. Then the following holds:
\begin{enumerate}
\item The bundle $\Aut(\mathbf{g})\to M$ of fibrewise Lie algebra automorphisms is a principal $\Aut(\mathfrak{g})$-bundle and there is an isomorphism of Lie algebroids 
$$\mathbf{g}\cong \Aut(\mathbf{g})\times_{\Aut(\mathfrak{g})}\mathfrak{g}.$$
\item If $\mathfrak{g}$ is graded and $\mathcal{G}\to M$ is a locally trivial bundle of Carnot-Lie groups, the bundle $\Aut_{\rm gr}(\mathbf{g})\to M$ of fibrewise graded Lie algebra automorphisms is a principal $\Aut_{\rm gr}(\mathfrak{g})$-bundle and there is an isomorphism of Lie algebroids 
$$\mathbf{g}\cong \Aut_{\rm gr}(\mathbf{g})\times_{\Aut_{\rm gr}(\mathfrak{g})}\mathfrak{g}.$$
\item There is a groupoid isomorphism $\mathcal{G}\cong \Aut(\mathbf{g})\times_{\Aut(\mathfrak{g})}\mathsf{G}$, and if $\mathcal{G}$ is a locally trivial bundle of Carnot-Lie groups then $\mathcal{G}\cong \Aut_{\rm gr}(\mathbf{g})\times_{\Aut_{\rm gr}(\mathfrak{g})}\mathsf{G}$.
\end{enumerate}
\end{proposition}

\begin{proof}
Since $\mathbf{g}$ is locally a trivial $\mathfrak{g}$-bundle, it follows that every point $x\in M$ admits a neighbourhood $U$ such that  $\Aut(\mathbf{g})_U=\Aut(\mathbf{g}_U)\cong U\times \Aut(\mathfrak{g})$ so $\Aut(\mathbf{g})\to M$  is a principal $\Aut(\mathfrak{g})$-bundle. The remainder of the proof follows from standard verifications in local charts. 
\end{proof}

For a locally trivial bundle of nilpotent Lie groups, we shall now define bundles of flat orbits and related ideals. These results are extensions of Theorem \ref{groupoiddescropofig} from nilpotent Lie groups to locally trivial bundle of nilpotent Lie groups. First, we shall need a multitude of notations. 

\begin{definition}
Let $\mathcal{G}\to M$ be a locally trivial bundle of nilpotent Lie groups of type $\mathfrak{g}$ coming from a principal $\Aut(\mathfrak{g})$-bundle $P\to M$ as in Proposition \ref{integrationofloctrivialnilpotentalgebroid}. 
\begin{itemize}
\item We define the central subgroupoid as 
$$\mathcal{Z}:=P\times_{\Aut(\mathfrak{g})}\mathsf{Z}(\mathsf{G})\to M,$$
which is a locally trivial bundle of nilpotent Lie groups of type $\mathfrak{z}$. We identify $\mathcal{Z}$ with a subgroupoid of $\mathcal{G}$. 
\item If we pick a complement of $\mathcal{Z}\cong \mathsf{Lie}(\mathcal{Z})$ in $\mathbf{g}$, we can similarly as in Equation \eqref{2cocycfromxenter} define a smooth $2$-cocycle 
$$\omega_\mathcal{Z}:\mathcal{G}/\mathcal{Z}\times \mathcal{G}/\mathcal{Z}\to \mathcal{Z}.$$
\item We let $\mathcal{G}_{\mathcal{Z}}\to \mathcal{Z}^*$ denote the pull back of $\mathcal{G}/\mathcal{Z}$ to  $\mathcal{Z}^*$. The Lie groupoid $\mathcal{G}_{\mathcal{Z}}$ is a locally trivial bundle of nilpotent Lie groups of type $\mathfrak{g}/\mathfrak{z}$ and a short computation shows that when viewing $\mathbf{g}^*$ as a vector bundle on $\mathcal{Z}^*$ via the restriction map, we can identify
$$
\mathbf{g}^*=\mathsf{Lie}(\mathcal{G}_{\mathcal{Z}})^*.
$$ 
\item Define the $2$-cocycle 
$$\omega_{\mathcal{G}_\mathcal{Z}} : \mathcal{G}_\mathcal{Z}^{(2)} \rightarrow U(1),$$
as in Definition \ref{unitary2coclalddaadl}. 
\item We also set 
$$\Gamma_P:=P\times_{\Aut(\mathfrak{g})}\Gamma\subseteq \mathcal{Z}^*,$$
and the locally trivial bundle of nilpotent Lie groups $\mathcal{G}_{\Gamma,P}:=\mathcal{G}_{\mathcal{Z}}|_{\Gamma_P}\to \Gamma_P$ (it is again of type $\mathfrak{g}/\mathfrak{z}$). 
\item We define $\Xi_P:=P\times_{\Aut(\mathfrak{g})}\Xi$. In light of Proposition \ref{totalspaceofflatorbits}, we can consider $\Xi_P\to \Gamma_P$ as a vector bundle. In fact, a short computation shows that 
$$
\Xi_P=\mathsf{Lie}(\mathcal{G}_{\Gamma,P})^*.
$$
We note that there is an inclusion 
\begin{equation}
\label{inclusionofxi}
\Xi_P\subseteq \mathbf{g}^*.
\end{equation}
\end{itemize}
\end{definition}

Assume that $G$ admits flat orbits. We define $I_{P}$ as the $C_0$-sections of the bundle of $C^*$-algebras 
$$P\times_{\Aut(\mathfrak{g})} I_\mathsf{G}\to M,$$
and $C^*(\mathsf{G})_P$ as the $C_0$-sections of the bundle of $C^*$-algebras 
$$P\times_{\Aut(\mathfrak{g})} C^*(\mathsf{G})\to M.$$
Note that the frame bundle induces an isomorphism $C^*(\mathsf{G})_P\cong C^*(\mathcal{G})$.

\begin{proposition}
\label{groupoidftincentrals1}
Let $\mathcal{G}\to M$ be a locally trivial bundle of nilpotent Lie groups of type $\mathfrak{g}$ defined from a principal $\Aut(\mathfrak{g})$-bundle $P\to M$. Assume that $\mathsf{G}:=\mathrm{exp}(\mathfrak{g})$ admits flat orbits. Then the $C^*$-algebra $I_{P}$ is a continuous trace algebra with spectrum $\Gamma_P$ and 
$$\delta_{\rm DD}(I_P)=0\in \check{H}^3(\Gamma_P,\Z).$$ 
\end{proposition}

\begin{proof}
By local triviality of $P$, $p_\Gamma:\Gamma_P\to M$ is locally trivial. We can therefore cover $M$ by a open subsets $(U_j)_{j\in J}$ and for local trivializations $\tau_j:p_\Gamma^{-1}(U_j)\to U_j\times \Gamma$ we can cover $\Gamma_P$ by the open subsets $(\tau_j^{-1}(U_j\times U_{\pmb{B}}))_{j\in J,\pmb{B}\in \mathfrak{JH}}$. The fact that $I_\mathsf{G}$ is a continuous trace algebra, we can use the argument in Proposition \ref{descriptionofddinvariant} to lift the transition functions for $I_P$ defined on a refinement of the cover $(\tau_j^{-1}(U_j\times U_{\pmb{B}}))_{j\in J,\pmb{B}\in \mathfrak{JH}}$ to a unitary valued Cech cocycle on $\Gamma_P$. 
\end{proof}

The reader can compare Proposition \ref{groupoidftincentrals1} to Theorem \ref{trivialdldaaddo} below. The construction of Proposition \ref{groupoidftincentrals1} gives a direct method to construct the bundle of flat orbits from the $\eta$-invariants of Definition \ref{definieofeta} and the corrected Lion intertwiner of Proposition \ref{descriptionofddinvariant}, while Theorem \ref{trivialdldaaddo} proves $\delta(I_P)=0$ via a deformation argument. 

Let us provide another construction of the bundle of flat orbit representations from an induction procedure. This induction procedure is to our knowledge not always attainable as it relies on finding a suitable subgroup $\ghani\subseteq \Aut(\mathfrak{g})$ to which we can reduce $P$.

\begin{proposition}
\label{groupoidftincentrals2}
Let $\mathcal{G}\to M$ be a locally trivial bundle of nilpotent Lie groups of type $\mathfrak{g}$ as in Proposition \ref{groupoidftincentrals1}. Assume that there is a subgroup $\ghani\subseteq \Aut(\mathfrak{g})$ to which $P$ can be reduced, and satisfying $[\zeta|_\ghani]=0$ and $[c_\ghani]=0$. Then letting $\mathcal{H}\to \Gamma$ denote a Hilbert space bundle  and $\pi_{\musFlat}$ an isomorphism as in Theorem \ref{charofequimorita}, item 1), the mapping $\pi_{\musFlat,\ghani}:=\mathrm{id}\times_\ghani \pi_{\musFlat}$ defines a $C_0(\Gamma_P)$-linear isomorphism of $C^*$-algebras
$$\pi_{\musFlat,\ghani}:I_P\to C_0(\Gamma_P,\mathbb{K}(\mathcal{H}_P)),$$
where $\mathcal{H}_P=P\times_\ghani \mathcal{H}\to \Gamma_P$ is the induced Hilbert space bundle.
\end{proposition}

The proof of Proposition \ref{groupoidftincentrals2} follows from local triviality of the reduction of $P$ to $\ghani$. From the $\Aut(\mathsf{G})$-equivariance of the Fourier transform in the central direction from Theorem \ref{groupoiddescropofig}, we can deduce the following description of the $C^*$-algebras $C^*(\mathsf{G})_P$ and $I_P$ as twisted groupoid $C^*$-algebras.

\begin{proposition}
\label{groupoidftincentrals3}
Let $\mathcal{G}\to M$ be a locally trivial bundle of nilpotent Lie groups of type $\mathfrak{g}$ as in Proposition \ref{groupoidftincentrals1}. The Fourier transform in the central direction (cf. Theorem \ref{groupoiddescropofig}) defines $*$-isomorphisms 
$$\mathcal{F}_{\mathcal{Z}}:C^*(\mathsf{G})_P\to C^*(\mathcal{G}_\mathcal{Z},\omega_{\mathcal{G}_\mathcal{Z}}) \quad\mbox{and}\quad \mathcal{F}_{\mathcal{Z}}:I_{P}\to C^*(\mathcal{G}_{\Gamma,P},\omega_{\mathcal{G}_\mathcal{Z}}),$$
that fits into the commuting diagram
\[
\begin{tikzcd}
I_{P} \arrow[r, "\mathcal{F}_{\mathcal{Z}}"]\arrow[d, "\subseteq"]  & C^*(\mathcal{G}_{\Gamma,P},\omega_{\mathcal{G}_\mathcal{Z}}) \arrow[d, "\subseteq"] \\
C^*(\mathsf{G})_P\arrow[r, "\mathcal{F}_{\mathcal{Z}}"]& C^*(\mathcal{G}_\mathcal{Z},\omega_{\mathcal{G}_\mathcal{Z}}) 
\end{tikzcd}
\]
\end{proposition}

\section{Nistor's Connes-Thom isomorphism and the ideal of flat orbits}
\label{subsec:nistorctsubsn}

In this subsection we will relate the $K$-theory of $I_G$ to that of the vector bundle $\Xi\to \Gamma$ and the $K$-theory of $C^*(G)$ to that of $\mathfrak{g}^*$. The latter result is known due to work of Nistor \cite{nistorsolvabel}. What will play an important role later on in the monograph is that the results also works for bundles of simply connected nilpotent Lie groups and that they are compatible with the inclusions $I_G\hookrightarrow C^*(\mathsf{G})$ and $C_0(\Xi)\hookrightarrow C_0(\mathfrak{g}^*)$. The $K$-theory computations are carried out using an adiabatic deformation considered by Nistor, so we therefore formulate the results of this subsection in the language of groupoids.

For a simply connected nilpotent Lie group $\mathsf{G}$ we define its adiabatic deformation as 
$$\mathsf{G}_{\rm adb}:=\mathfrak{g}\times \{0\}\dot{\cup}\mathsf{G}\times (0,1].$$
The fibrewise operations makes $\mathsf{G}_{\rm adb}\rightrightarrows[0,1]$ into a groupoid. This is the adiabatic deformation (as in Example \eqref{ex:adiagropododo}) of $\mathsf{G}$ viewed as a groupoid over a point. We note that in this case, the smooth structure of $\mathsf{G}_{\rm adb}$ is defined by declaring the map $\psi:\mathfrak{g}\times [0,1]\to \mathsf{G}_{\rm adb}$, given by 
$$\psi(X,t):=\begin{cases}
(X,0), \; &t=0,\\
(\e^{tX},t), \; &t>0,\end{cases}$$
to be smooth. We also define a variation on the adiabatic groupoid pertaining the homogeneous structure of the central characters. Since the nilpotent Lie groups arising in the fibre of the osculating groupoid of a Carnot manifold are Carnot-Lie groups (see below in Section \ref{subsec:carnot}), we shall adopt our constructions to the dilation action but remark that similar constructions can be carried out for general nilpotent Lie groups. 

\begin{definition}
\label{adiaboatihomogenensosl}
The groupoid $\mathcal{G}_{\mathsf{Z},{\rm adb}}\rightrightarrows\mathfrak{z}^*\times [0,1]$, where $\mathcal{G}_{\mathsf{Z}}$ is as in Definition \ref{groipodicododve}, is defined by
$$\mathcal{G}_{\mathsf{Z},{\rm adb}}=(\mathfrak{g}/\mathfrak{z}\times \mathfrak{z}^*\times \{0\})\dot{\cup}(\mathsf{G}/\mathsf{Z}\times \mathfrak{z}^*\times (0,1]),$$
and whose smooth structure is defined by declaring the map $\psi:\mathfrak{g}/\mathfrak{z}\times \mathfrak{z}^*\times [0,1]\to \mathcal{G}_{\mathsf{Z},{\rm adb}}$, given by 
$$\psi(X,\xi,t):=\begin{cases}
(X,\xi,0), \; &t=0,\\
(\e^{\delta_tX},\delta_{t^{-1}}^*\xi, t), \; &t>0,\end{cases}$$
to be smooth. The groupoid operations on $\mathcal{G}_{\mathsf{Z},{\rm adb}}$ are the fibrewise operations. 
\end{definition} 

We note that 
$$\mathcal{G}_{\mathsf{Z},{\rm adb}}^{(2)}=(\mathfrak{g}/\mathfrak{z}\times\mathfrak{g}/\mathfrak{z}\times \mathfrak{z}^*\times \{0\})\dot{\cup}(\mathsf{G}/\mathsf{Z}\times\mathsf{G}/\mathsf{Z}\times \mathfrak{z}^*\times (0,1]).$$
The group $\Aut_{\rm gr}(\mathsf{G})$ acts as smooth groupoid automorphisms of $\mathcal{G}_{\mathsf{Z},{\rm adb}}$ by 
$$\begin{cases}
\varphi.(X,\xi,0):=(\varphi(X),(\varphi^{-1})^*\xi,0),\;&t=0,\\
\varphi.(g,\xi,t):=(\varphi(g),(\varphi^{-1})^*\xi,t),\; &t>0.
\end{cases}$$

\begin{definition}
\label{unitary2coclalddaadladiabatic}
Consider the groupoid 
$$\mathcal{G}_{\mathsf{Z},{\rm sqadb}}:=\mathcal{G}_{\mathsf{Z},{\rm adb}}\times [0,1]\to \rightrightarrows\mathfrak{z}^*\times [0,1]^2.$$ 
Define the mapping
$$\omega_{\mathcal{G}_{\mathsf{Z},{\rm sqadb}}} : \mathcal{G}_{\mathsf{Z},{\rm sqadb}}^{(2)} \rightarrow U(1),$$
by 
$$
\begin{cases}
\omega_{\mathcal{G}_{\mathsf{Z},{\rm sqadb}}}(X , Y, \xi,0,s) = \e^{i s\omega_\xi(X,Y)}=\e^{i s\xi[X,Y]},\\
\omega_{\mathcal{G}_{\mathsf{Z},{\rm sqadb}} }(g , h, \xi,t,s) = e^{-is \xi.\omega_\mathsf{Z}(g, h)},\quad t>0.
\end{cases}$$
\end{definition}

We note that the restriction of $\omega_{\mathcal{G}_{\mathsf{Z},{\rm sqadb}}}$ to $s=1$ and $t>0$ coincides with $\omega_{\mathcal{G}_{\mathsf{Z}} }$, the cocycle from Definition \ref{unitary2coclalddaadl}. We also note that $\omega_{\mathcal{G}_{\mathsf{Z},{\rm sqadb}}}|_{s=0}=1$.

\begin{proposition}
Let $\mathsf{G}$ be a simply connected nilpotent Lie group and $\mathcal{G}_{\mathsf{Z},{\rm sqadb}}$ be constructed as above. The mapping $\omega_{\mathcal{G}_{\mathsf{Z},{\rm sqadb}}}$ is a smooth cocycle. Moreover, the $\Aut(\mathsf{G})$-action on $\mathcal{G}_{\mathsf{Z},{\rm adb}}$ lifts to a $C([0,1]^2)$-linear strongly continuous $\Aut(\mathsf{G})$-action on $C^*(\mathcal{G}_{\mathsf{Z},{\rm sqadb}},\omega_{\mathcal{G}_{\mathsf{Z},{\rm sqadb}}})$ following the constructions proceeding Theorem \ref{groupoiddescropofig}.
\end{proposition}

\begin{proof}
To prove that $\omega_{\mathcal{G}_{\mathsf{Z},{\rm sqadb}}}$ is a smooth cocycle, we note that it is by construction a cocycle pointwise so it suffices to prove that 
$$\psi^*\omega_{\mathcal{G}_{\mathsf{Z},{\rm sqadb}}}:\mathfrak{g}/\mathfrak{z}\times\mathfrak{g}/\mathfrak{z}\times \mathfrak{z}^*\times [0,1]\to U(1),$$
is smooth. We compute that 
$$\psi^*\omega_{\mathcal{G}_{\mathsf{Z},{\rm sqadb}}}(X,Y,\xi,t,s)=\begin{cases}
\e^{i s\xi[X,Y]},\; &t=0,\\
e^{-is (\delta_{t^{-1}}^*\xi).(\omega_\mathsf{Z}(\e^{\delta_tX},\e^{\delta_tY}))},\; &t>0.
\end{cases}$$
We $\rho_\mathfrak{g}(X,Y)=\log(\e^X\e^Y)$, which is a polynomial map $\mathfrak{g}\times \mathfrak{g}\to \mathfrak{g}$ since $\mathfrak{g}$ is nilpotent. Let $j:\mathfrak{g}/\mathfrak{z}\to \mathfrak{g}$ denote the inclusion defined from the Jordan-Hölder basis that we use to define $\omega_\mathsf{Z}$. A short computation allows us to write 
\begin{align*}
\omega_\mathsf{Z}(\e^{\delta_tX},\e^{\delta_tY})=&\rho_\mathfrak{g}(\delta_tj(X),\rho_\mathfrak{g}(\delta_tY,-j\rho_{\mathfrak{g}/\mathfrak{z}}(\delta_tX,\delta_tY)))=\\
=&-[\delta_tX,\delta_tY]+P(X,Y,t)=-\delta_t[X,Y]+P(X,Y,t),
\end{align*}
for a polynomial $P(X,Y,t)$ such that $\delta_{t^{-1}}P(X,Y,t)$ vanishes to first order at $t=0$. The last equality follows from the Baker-Campbell-Hausdorff formula. We conclude that 
\begin{align*}
\psi^*\omega_{\mathcal{G}_{\mathsf{Z},{\rm sqadb}}}(X,Y,\xi,t,s)=&
\begin{cases}
\e^{i s\xi[X,Y]},\; &t=0\\
e^{is (\delta_{t^{-1}}^*\xi).(\delta_t[X,Y]-P(X,Y,t))},\; &t>0.
\end{cases}=\\
=&\begin{cases}
\e^{i s\xi[X,Y]},\; &t=0\\
e^{is \xi.([X,Y]-\delta_{t^{-1}}P(X,Y,t))},\; &t>0.
\end{cases},
\end{align*}
and that $\psi^*\omega_{\mathcal{G}_{\mathsf{Z},{\rm sqadb}}}$ is smooth.

The proof that the $\Aut(\mathsf{G})$-action lifts goes as in Theorem \ref{groupoiddescropofig}.
\end{proof}

We introduce the notation $\Delta\subseteq [0,1]\times [0,1]$ for the diagonal. When appropriate, we identify $\Delta$ with $[0,1]$.

\begin{proposition}
\label{contracllemem}
Let $\mathsf{G}$ be a simply connected nilpotent Lie group and $\mathcal{G}_{\mathsf{Z},{\rm adb}}$ and $\omega_{\mathcal{G}_{\mathsf{Z},{\rm adb}}}$ be constructed as above. 
\begin{itemize}
\item The groupoid homomorphism 
$$\mathcal{G}_{\mathsf{Z},{\rm adb}}\times (0,1]\to \mathcal{G}_{\mathsf{Z},{\rm adb}}\times (0,1],$$
defined by 
$$\begin{cases}
(X , Y, \xi,0,s) \mapsto(X , Y, s^{-1}\xi,0,s),\; &t=0,\\
(g , h, \xi,t,s) \mapsto (g , h, s^{-1}\xi,t,s),\; &t>0.
\end{cases}$$
induces an $\Aut(\mathsf{G})$-equivariant $C([0,1]^2)$-linear isomorphism of $C^*$-algebras 
$$C^*(\mathcal{G}_{\mathsf{Z},{\rm adb}}\times(0,1], \omega_{\mathcal{G}_{\mathsf{Z},{\rm sqadb}}})\to C^*(\mathcal{G}_{\mathsf{Z},{\rm adb}}, \omega_{\mathcal{G}_{\mathsf{Z},{\rm sqadb}}}|_{s=1})\otimes C_0(0,1].$$
\item The Fourier transform in the central direction from Theorem \ref{groupoiddescropofig} defines an $\Aut(\mathsf{G})$-equivariant $C[0,1]$-linear isomorphism of $C^*$-algebras 
$$C^*(\mathsf{G}_{{\rm adb}})\to C^*(\mathcal{G}_{\mathsf{Z},{\rm sqadb}}|_\Delta, \omega_{\mathcal{G}_{\mathsf{Z},{\rm sqadb}}}|_\Delta).$$
\end{itemize}
\end{proposition}

We omit the proof as it follows from standard computations and the proof of Theorem \ref{groupoiddescropofig}. For the next proposition, we recall that if $\Xi\to \Gamma$ denotes the total space of flat orbits there is an explicit vector bundle trivialization $\Xi\cong \Gamma\times \mathfrak{z}^\perp$ by Proposition \ref{totalspaceofflatorbits} that is readily verified to be $\Aut(\mathsf{G})$-equivariant. We conclude that there is an $\Aut(\mathsf{G})$-equivariant groupoid isomorphism 
$$\mathcal{G}_{\mathsf{Z},{\rm adb}}|_{\Gamma\times \{0\}}\cong \Xi^*,$$
where the vector bundle $\Xi^*\to \Gamma$ is considered as a groupoid in the fibrewise addition. We consider $\Xi^*$ as an open subgroupoid of the Lie groupoid $\mathfrak{g}/\mathfrak{z}\times \mathfrak{z}^*\rightrightarrows \mathfrak{z}^*$. The Lie groupoid carries a unitary $2$-cocycle 
\begin{equation}
\label{kirildongropodidi}
\omega(X,Y,\xi):=\e^{i\xi[X,Y]}.
\end{equation}
Note that $\mathfrak{g}/\mathfrak{z}\times \mathfrak{z}^*=\mathcal{G}_{\mathsf{Z},{\rm adb}}|_{s=0}$ and $\omega=\omega_{\mathcal{G}_{\mathsf{Z},{\rm adb}}}|_{t=0,s=1}$. 
We write $\Delta_0$ for the diagonal in $[0,1)^2$. The following proposition follows from the notational discussion above, Theorem \ref{restrictqotuegroudp} and Proposition \ref{groupoiddescropofig}.

\begin{proposition}
\label{plentyofsesforg}
Let $\mathsf{G}$ be a simply connected nilpotent Lie group admitting flat orbits and $\mathcal{G}_{\mathsf{Z},{\rm adb}}$ and $\omega_{\mathcal{G}_{\mathsf{Z},{\rm adb}}}$ be constructed as above. Then there are $\Aut(\mathsf{G})$-equivariant commuting diagrams with exact rows:
\tiny
\[
\begin{CD}
0  @>>> \Sigma I_\mathsf{G}@>>> C^*(\mathcal{G}_{\mathsf{Z},{\rm sqadb}}|_{\Gamma\times \Delta_0}, \omega_{\mathcal{G}_{\mathsf{Z},{\rm sqadb}}})@>{\mathrm{ev}_0}>>  C_0(\Xi)@>>>0 \\
@. @VVV  @VVV @ VVV@.\\
0  @>>>  \Sigma C^*(\mathsf{G}) @>>>C^*(\mathcal{G}_{\mathsf{Z},{\rm sqadb}}|_{ \Delta_0}, \omega_{\mathcal{G}_{\mathsf{Z},{\rm sqadb}}})@>{\mathrm{ev}_0}>>  C_0(\mathfrak{g}^*)@>>>0 \\
\end{CD}, \]
\normalsize
and 
\tiny
\[
\begin{CD}
0  @>>> \Sigma I_\mathsf{G}@>>> C^*(\mathcal{G}_{\mathsf{Z},{\rm sqadb}}|_{\Gamma\times [0,1)\times \{1\}}, \omega_{\mathcal{G}_{\mathsf{Z},{\rm sqadb}}})@>{\mathrm{ev}_0}>>  C^*(\Xi^*,\omega)@>>>0 \\
@. @VVV  @VVV @ VVV@.\\
0  @>>>  \Sigma C^*(\mathsf{G}) @>>>C^*(\mathcal{G}_{\mathsf{Z},{\rm sqadb}}|_{[0,1)\times \{1\}}, \omega_{\mathcal{G}_{\mathsf{Z},{\rm sqadb}}})@>{\mathrm{ev}_0}>>  C^*(\mathfrak{g}/\mathfrak{z}\times \mathfrak{z}^*,\omega)@>>>0 \\
\end{CD}, \]
\normalsize
and 
\tiny
\[
\begin{CD}
0  @>>> \Sigma C^*(\Xi_P^*,\omega)@>>> C^*(\mathcal{G}_{\mathsf{Z},{\rm sqadb}}|_{\Gamma\times \{0\}\times [0,1)}, \omega_{\mathcal{G}_{\mathsf{Z},{\rm sqadb}}})@>{\mathrm{ev}_0}>>  C_0(\Xi)@>>>0 \\
@. @VVV  @VVV @ VVV@.\\
0  @>>>  \Sigma C^*(\mathfrak{g}/\mathfrak{z}\times \mathfrak{z}^*,\omega) @>>>C^*(\mathcal{G}_{\mathsf{Z},{\rm sqadb}}|_{\{0\}\times[0,1) }, \omega_{\mathcal{G}_{\mathsf{Z},{\rm sqadb}}})@>{\mathrm{ev}_0}>>  C_0(\mathfrak{g})@>>>0 \\
\end{CD}, \]
\normalsize
The vertical maps in the three diagrams above are all ideal inclusions.
\end{proposition}

\begin{definition}
\label{autgododlelendn}
We define the $KK$-morphisms
\begin{itemize}
\item $\psi\in KK^{\Aut(\mathsf{G})}_0(C_0(\mathfrak{g}^*),C^*(\mathsf{G}))$;
\item $\psi_I\in KK^{\Aut(\mathsf{G})}_0(C_0(\Xi),I_\mathsf{G})$;
\item $\psi_\omega\in KK^{\Aut(\mathsf{G})}_0(C^*(\mathfrak{g},\omega),C^*(\mathsf{G}))$;
\item $\psi_{I,\omega}\in KK^{\Aut(\mathsf{G})}_0(C^*(\Xi^*,\omega),I_\mathsf{G})$;
\item $\psi_{\Xi}\in KK^{\Aut(\mathsf{G})}_0(C_0(\Xi),C^*(\Xi^*,\omega))$;
\end{itemize}
as the boundary mappings from the short exact sequences appearing in Proposition \ref{plentyofsesforg}.
\end{definition}

\begin{remark}
\label{psivsdefomem}
We remark that a standard argument shows that 
$$\psi=[\mathrm{ev}_0]^{-1}\otimes_{C^*(\mathcal{G}_{\mathsf{Z},{\rm sqadb}}|_{\Delta}, \omega_{\mathcal{G}_{\mathsf{Z},{\rm sqadb}}})}[\mathrm{ev}_1],$$ 
where the evaluation maps are
\begin{align*}
\mathrm{ev}_0&:C^*(\mathcal{G}_{\mathsf{Z},{\rm sqadb}}|_{\Delta}, \omega_{\mathcal{G}_{\mathsf{Z},{\rm sqadb}}})\to C_0(\mathfrak{g}^*),\quad\mbox{and},\\
\mathrm{ev}_1&:C^*(\mathcal{G}_{\mathsf{Z},{\rm sqadb}}|_{\Delta}, \omega_{\mathcal{G}_{\mathsf{Z},{\rm sqadb}}})\to C^*(\mathsf{G}).
\end{align*}
Indeed, $[\mathrm{ev}_0]\in KK(C^*(\mathcal{G}_{\mathsf{Z},{\rm sqadb}}|_{\Delta}, \omega_{\mathcal{G}_{\mathsf{Z},{\rm sqadb}}}), C_0(\mathfrak{g}^*)$ is invertible as it is surjective and its kernel is contractible by Proposition \ref{contracllemem}. Analogous statements hold for $\psi_I$,  $\psi_\omega$, $\psi_{I,\omega}$ and $\psi_{I,\Xi}$ as well.
\end{remark}

\begin{remark}
Given a group $G$ and a principal $G$-bundle $P\to X$ one has an induction functor $P : KK^G \rightarrow KK^{X}$ such that for a $G-C^*-$algebra $A$
\[ P(A) := \{ f \in C(P, A) : f(p.g)=g^{-1}.f(p), \; (pG\mapsto \|f(p)\|_A)\in C_0(X) \}. \]
For the principal $\Aut(\mathsf{G})$-bundle $P\to M$ associated with a locally trivial bundle $\mathcal{G}$ of nilpotent Lie groups on $M$, then
\[ C^*(\mathbf{g}) = C_0(\mathbf{g}^*) = P(C_0(\mathfrak{g}^*)), \; C^*(\mathcal{G}) = P(C^*(\mathsf{G})), \; \mbox{and} \; C_0(\Xi_P)=P(C_0(\Xi)),\]
and so forth.
\end{remark}

\begin{theorem}
\label{nistorconnethomfortwistgroup}
Let $\mathcal{G}\to M$ be a locally trivial bundle of Lie groups of type $\mathsf{G}$ with frame bundle $P\to M$. We denote its Lie algebroid by $\mathbf{g}$. Then $\psi\in KK^{\Aut(\mathsf{G})}_0(C_0(\mathfrak{g}^*),C^*(\mathsf{G}))$ induces an element $\psi_P\in KK(C^*(\mathcal{G}),C_0(\mathbf{g}^*))$ such that 
\begin{itemize}
\item $\psi_P$ is a $KK$-isomorphism, in particular the Kasparov product with $\psi_{P}$ defines isomorphisms 
$$K_*(C^*(\mathcal{G}))\cong K^*(\mathbf{g}^*).$$
\item The element $\psi_P$ is natural with respect to invariant open inclusions and restriction to invariant closed subsets, and in particular if $\mathsf{G}$ admits flat orbits then there is a commuting diagram in $KK$
\[
\begin{tikzcd}
I_{\mathsf{G},P} \arrow[rr] &&C^*(\mathcal{G})&\\
&&\\
C_0(\Xi_P)\arrow[uu, "\psi_{P,I}"]  \arrow[rr] && C_0(\mathbf{g}^*)\arrow[uu,"\psi_P"]
\end{tikzcd}
\]
where the vertical morphisms are isomorphisms. Here, the element $\psi_{P,I}\in KK_0(C_0(\mathbf{g}^*|_{\Gamma_P}),I_{\mathsf{G},P})$ is defined from $\psi_I\in KK^{\Aut(\mathsf{G})}_0(C_0(\Xi),I_\mathsf{G})$.
\end{itemize}
\end{theorem}

\begin{proof}
By naturality of boundary mappings the second item of the theorem follows from the first item. To prove the first item, we follow the idea in \cite[Lemma 3]{nistorsolvabel}. By a Mayer-Vietoris argument, we can assume that $\mathcal{G}=M\times G\to M$ is the trivial bundle. If we pick a character $\chi:\mathfrak{g}\to \R$, the Packer-Raeburn stabilization trick and Connes' Thom isomorphism allow to reduce the dimension of $G$ and item 1 follows from induction on the dimension of $G$.
\end{proof}

From Theorem \ref{nistorconnethomfortwistgroup} and the naturality of pushforwards we can conclude the following. 

\begin{corollary}
\label{onsurjinteroffm}
Let $\mathcal{G}\to M$ be a locally trivial bundle of Lie groups over a compact manifold of type $\mathsf{G}$ with frame bundle $P\to M$. Assume that $\mathsf{G}$ admits flat orbits. 
Then the map $K_*(I_{\mathsf{G},P})\to K_*(C^*(\mathcal{G}))$ induced by the inclusion is surjective if and only if $(p_\Gamma)_!:K^*(\Gamma_P)\to K^*(M)$ is surjective.
\end{corollary}

\begin{proposition}
Let $\mathcal{G}\to M$ be a locally trivial bundle of Lie groups of type $\mathsf{G}$ with frame bundle $P\to M$. Then the following diagram in $KK$ commutes:
\[
\begin{tikzcd}
&I_{\mathsf{G},P} &\\
&&\\
C_0(\Xi_P) \arrow[uur,"\psi_{I,P}"]\arrow[rr, "\psi_{\Xi,P}"] && C^*(\Xi_P^*,\omega_P)\arrow[uul,"\psi_{\omega,I,P}"]
\end{tikzcd}
\]
\end{proposition}

\begin{proof}
The proof follows from naturality of boundary mappings and the fact that all involved maps comes from a deformation parametrized by $[0,1]^2$ as in Theorem \ref{plentyofsesforg}.
\end{proof}

\section{On surjectivity of $K_*(I_{\mathsf{G},P})\to K_*(C^*(\mathcal{G}))$}
\label{suhjrjrjed}

For our index theoretical considerations later on, it will be important to have surjectivity of the mapping $K_*(I_{G,P})\to K_*(C^*(\mathcal{G}))$ induced from inclusion. This shall allow us to transfer index computations, a priori depending on the entire spectrum, to the flat orbits. We have not found a general condition for this to hold, but consider a number of results that are verifiable in examples. Throughout this subsection we let $\mathcal{G}\to M$ denote a locally trivial bundle of Lie groups of type $\mathsf{G}$ over a compact manifold with frame bundle $P\to M$. We tacitly assume that $\mathsf{G}$ admits flat orbits.

Recall from Corollary \ref{onsurjinteroffm} that surjectivity of $K_*(I_{\mathsf{G},P})\to K_*(C^*(\mathcal{G}))$ is equivalent to surjectivity of $(p_\gamma)_!:K^*(\Gamma_P)\to K^*(M)$. By the naturality of push-forwards, surjectivity of $(p_\gamma)_!:K^*(\Gamma_P)\to K^*(M)$ is in turn equivalent to surjectivity of the map $K^*(\Xi_P^*)\to K^*(\mathcal{A}^*)$ induced by inclusion.

\begin{proposition}
\label{consecimplesur}
Assume that the fibre bundle $\Gamma_P\to M$ admits a continuous section. Then  $K_*(I_{\mathsf{G},P})\to K_*(C^*(\mathcal{G}))$ is surjective. 
\end{proposition}

\begin{proof}
If $\Gamma_P\to M$ admits a continuous section, we can compose it with the zero section of $\Xi_P\to \Gamma_P$ and obtain a global section $s:M\to \Xi_P$. Taking a tubular neighborhood the image of $s$ in $\mathcal{A}^*$, we obtain a ball bundle $Q\to M$ contained in $\Xi_P^*$. We arrive at a commuting diagram where all arrows are push-forwards induced from open inclusions 
\[
\begin{tikzcd}
&K^*(Q)\arrow[ddr]\arrow[ddl] &\\
&&\\
K^*(\Xi_P^*)\arrow[rr] && K^*(\mathcal{A}^*)
\end{tikzcd}
\]
The map $K^*(Q)\to K^*(\mathcal{A}^*)$ is bijective, so $K^*(\Xi_P^*)\to K^*(\mathcal{A}^*)$ is surjective.
\end{proof}

\begin{corollary}
\label{fixedsecimplesur}
Assume that there is a subgroup $\ghani\subseteq \mathsf{G}$ to which $P$ reduces and there is a $\ghani$-fixed point in $\Gamma$. Then  $K_*(I_{\mathsf{G},P})\to K_*(C^*(\mathcal{G}))$ is surjective. 
\end{corollary}

\begin{proof}
Let $P_\ghani$ denote the reduction of $P$ to $\ghani$ and $\xi_0\in \Gamma$ the fixed point of $\ghani$. The map $x\mapsto [(\tilde{x},\xi_0)]$, for some pre-image $\tilde{x}\in P_\ghani$ of $x$, is independent of choices and induces a continuous section to $P_\ghani\to M$. Therefore, $\Gamma_P\to M$ admits a continuous section. The result follows from Proposition \ref{consecimplesur}.
\end{proof}

\begin{corollary}
\label{surjeforcooronde}
Assume that $\mathsf{G}$ is a Carnot-Lie group with one-dimensional center and that the central subbundle $\mathcal{Z}\to M$ is orientable. Then $K_*(I_{\mathsf{G},P})\to K_*(C^*(\mathcal{G}))$ is surjective.
\end{corollary}

\begin{proof}
Orientability of $\mathcal{Z}$ is equivalent to it being trivializable, so there exists a global section of the $\R^\times$-bundle $\Gamma_P=\mathcal{Z}^*\setminus \{0\}\to M$. The result follows from Proposition \ref{consecimplesur}.
\end{proof}

It follows from the next proposition that $K_*(I_{\mathsf{G},P})\to K_*(C^*(\mathcal{G}))$ is rationally surjective for all nilpotent Lie groups with one-dimensional center without any orientability assumption.

\begin{proposition}
\label{rationalsurje}
Assume that $\mathsf{G}$ satisfies that $\chi(\Gamma)\neq 0$ has non-vanishing Euler characteristic. Then $K_*(I_{\mathsf{G},P})\to K_*(C^*(\mathcal{G}))$ is rationally surjective. More precisly, $K_*(I_{\mathsf{G},P})\to K_*(C^*(\mathcal{G}))$ is surjective after inverting $\chi(\Gamma)$.
\end{proposition}

\begin{proof}
By the Riemann-Roch theorem, we have a commuting diagram involving Chern characters
\[
\begin{tikzcd}
K^*(\Gamma_P)\arrow[rr,"\pi_!"] \arrow[dd, "\mathrm{ch}"] &&K^*(M)\arrow[dd, "\mathrm{ch}"] &\\
&&\\
H^*_c(\Gamma_P,\mathbb{Q}) \arrow[rr] &&H^*(M,\mathbb{Q}),
\end{tikzcd}
\]
where the bottom arrow is the multiplication by the relative $\hat{A}$-genus composed with the push-forward $\pi_*:H^*_c(\Gamma_P,\mathbb{Q}) \to H^*(M,\mathbb{Q})$. Since multiplication by the relative $\hat{A}$-genus is invertible and the Chern character is a rational isomorphism, $\pi_!:K^*(\Gamma_P)\to K^*(M)$ is a rational isomorphism if and only if $\pi_*:H^*_c(\Gamma_P,\mathbb{Q}) \to H^*(M,\mathbb{Q})$ is a surjection. By Poincaré duality, this is in turn equivalent to the push-forward $\pi_* : H_{*}(\Gamma_M, \mathbb{Q}) \rightarrow H_{*}(M, \mathbb{Q})$ being a surjection. By \cite{Casson_Gottlieb} there exist a transfer map $t : H_{*}(M, \mathbb{Q}) \rightarrow H_{*}(\Gamma_P, \mathbb{Q})$ such that $\pi_* \circ t = \chi(\Gamma) 1$. Thus, if $\chi(\Gamma) \neq 0$ then $\pi_*$ is surjective.
\end{proof}

\section[Thom-Connes isomorphism and flat orbits]{The Thom-Connes isomorphism and the bundle of flat representations}
\label{sec:connesthomandadiaofofd}

The constructions with adiabatic deformations of Section \ref{subsec:nistorctsubsn} have favorable functorial properties and produce $K$-theoretical information, but for later purposes of computing with Carnot symbols we shall need to relate the Connes-Thom isomorphism to the bundle of flat representations from Section \ref{ctstructrifoeod}. Along the way, we will give a new proof of the fact that the ideal of flat orbits, also in a locally trivial bundle of nilpotent Lie groups, has vanishing Dixmier-Duoady class.

For these constructions, we shall need some further notation. Let $\mathsf{G}$ denote a nilpotent Lie group admitting flat orbits and $\mathcal{G}\to M$ a locally trivial bundle of nilpotent Lie groups with fibre $\mathsf{G}$. Recall the groupoid 
$$\mathcal{G}_{\mathsf{Z}}=\mathsf{G}/\mathsf{Z}\times \mathfrak{z}^*\rightrightarrows \mathfrak{z},$$ 
constructed in Section \ref{sec:groupoidforflatorbits}. We can analogously define $\mathcal{G}_{\mathcal{Z}}\rightrightarrows \mathcal{Z}^*$ as the pull back of $\mathcal{G}/\mathcal{Z}$ up to the dual central subbundle $\mathcal{Z}^*$. Analogously to Definition \ref{unitary2coclalddaadl}, after choosing a Riemannian metric for the Lie algebroid of $\mathcal{G}$, we have a unitary 2-cocycle $\omega_{\mathcal{G}_{\mathcal{Z}}}$ on $\mathcal{G}_{\mathcal{Z}}$ and a (relative to the choice of metric) canonical isomorphism 
$$C^*(\mathcal{G})\xrightarrow{\sim}C^*(\mathcal{G}_{\mathcal{Z}},\omega_{\mathcal{G}_{\mathcal{Z}}}),$$
implemented via the Fourier transform in the central direction as in Theorem \ref{groupoiddescropofig}. This isomorphism implements an isomorphism $I_{\mathsf{G},P}\cong C^*(\mathcal{G}_{\mathcal{Z}}|_{\Gamma_P},\omega_{\mathcal{G}_{\mathcal{Z}}})$. Similarly to the constructions in Section \ref{subsec:nistorctsubsn}, we can also define the adiabatic deformation $\mathcal{G}_{\mathcal{Z},{\rm adb}}\rightrightarrows \mathcal{Z}^*$ and the smooth adiabatic cocycle $\omega_{\mathcal{G}_{\mathcal{Z},{\rm adb}}} : (\mathcal{G}_{\mathcal{Z},{\rm adb}} \times[0,1])^{(2)} \rightarrow U(1)$. We shall now use the machinery of Section \ref{ctstructrifoeod} to describe the representation theory of $C^*(\mathcal{G}_{\mathcal{Z},{\rm adb}}|_{\Gamma_P},\omega_{\mathcal{G}_{\mathcal{Z}},{\rm adb}}^{s=1})$.

First, we consider the case of $M$ being a point and later we shall glue these construction together. For a Jordan-Hölder basis $\pmb{B}$ of $\mathfrak{g}$, we introduce the notation 
$$\mathcal{G}_{\mathsf{Z},{\rm adb},\pmb{B}}:=\mathcal{G}_{\mathsf{Z},{\rm adb}}|_{U_{\pmb{B}}\times [0,1]}.$$
Consider the subgroupoid $\mathcal{F}_{\pmb{B},{\rm adb}}\subseteq \mathcal{G}_{\mathsf{Z},{\rm adb},\pmb{B}}$ defined as follows. We let  $\mathcal{F}_{\pmb{B}}\rightrightarrows U_{\pmb{B}}$ be the Lie groupoid that integrates the quotient bundle $F_{V,\pmb{B}}/\mathfrak{z}\to U_{\pmb{B}}$ of the bundle of Vergne polarizations from Proposition \ref{vbfv}. The groupoid $\mathcal{F}_{\pmb{B}}$ is well defined by Proposition \ref{integrationofnilpotentalgebroid}. We let  $\mathcal{F}_{\pmb{B},{\rm adb}}\rightrightarrows U_{\pmb{B}}\times [0,1]$ be the adiabatic deformation of $\mathcal{B}$ constructed as in Definition \ref{adiaboatihomogenensosl}. 

We define the space $\mathpzc{H}^\infty_{\pmb{B},{\rm adb}, c}$ to consist of all $f\in C^\infty( \mathcal{G}_{\mathsf{Z},{\rm adb},\pmb{B}})$ such that
$$f(gh)=f(g)\omega_{\mathcal{G}_{\mathsf{Z},{\rm adb}}}^{s=1}(g,h)^{-1}, \quad\mbox{for}\quad (g,h)\in (\mathcal{G}_{\mathsf{Z},{\rm adb}}\times \mathcal{F}_{\mathsf{Z},{\rm adb}})\cap \mathcal{G}_{\mathsf{Z},{\rm adb}}^{(2)},$$ 
and that $f$  is compactly supported in the $U_{\pmb{B}}$-direction and of Schwarz class modulo $\mathcal{F}_{\pmb{B},{\rm adb}}$ in the $\mathsf{G}/\mathsf{Z}$-direction. On $\mathpzc{H}^\infty_{V,\epsilon,c}$, there is a $C_0(U_{\pmb{B}}\times [0,1])$-valued inner product (as on page \pageref{thehilbmdmoslsl}) and we let $\mathpzc{H}_{\pmb{B},{\rm adb}}$ denote its completion as a $C_0(U_{\pmb{B}}\times [0,1])$-Hilbert $C^*$-module.

Following the constructions of Section \ref{subsec:reptheory}, we make the following observation. We implicitly use the isomorphisms of Proposition \ref{contracllemem}.

\begin{proposition}
\label{contrsucuc}
The $C_0(U_{\pmb{B}}\times [0,1])$-Hilbert $C^*$-module $\mathpzc{H}_{\pmb{B},{\rm adb}}$ 
carries a natural $C[0,1]$-linear bijective representation 
$$C^*(\mathcal{G}_{\mathsf{Z},{\rm adb},\pmb{B}}, \omega_{\mathcal{G}_{\mathsf{Z},{\rm adb}}}^{s=1})\to \mathbb{K}_{C_0(U_{\pmb{B}}\times [0,1]))}(\mathpzc{H}_{\pmb{B},{\rm adb}}),$$
such that 
\begin{itemize}
\item There exists a bundle of Hilbert spaces $\mathcal{H}_{\pmb{B},{\rm adb}}\to U_{\pmb{B}}\times [0,1]$ such that 
$$\mathpzc{H}_{\pmb{B},{\rm adb}}=C_0(U_{\pmb{B}},\mathcal{H}_{\pmb{B},{\rm adb}}).$$
\item $\mathpzc{H}_{\pmb{B},{\rm adb}}|_{(0,1]}= C_0(U_{\pmb{B}}\times (0,1],\mathcal{H}_{\pmb{B},V})$ as $(C_0((0,1],I_{\mathsf{G},\pmb{B}}),C_0(U_{\pmb{B}}\times (0,1]))$-Hilbert $C^*$-modules. 
\item $\mathpzc{H}_{\pmb{B},{\rm adb}}|_{t=0}= C_0(U_{\pmb{B}},L^2(\mathfrak{g}/F_{V,\pmb{B}}))$ as $(C^*((\mathfrak{g}/\mathfrak{z})\times U_{\pmb{B}},\omega),C_0(U_{\pmb{B}}))$-Hilbert $C^*$-modules (using the notation $\omega$ as in Equation \eqref{kirildongropodidi}), where the action of $C^*((\mathfrak{g}/\mathfrak{z})\times U_{\pmb{B}},\omega)$ on the bundle of Hilbert spaces $L^2(\mathfrak{g}/F_{V,\pmb{B}})$ is given by twisted convolution.
\end{itemize}
\end{proposition}

For two Jordan-Hölder basis $\pmb{B}$ and $\pmb{B}'$, we can (following Lemma \ref{lionsintertwiners} and Section \ref{ctstructrifoeod}) define the adiabatic Lion intertwiner 
\begin{equation}
\label{adiabativcckckdodo}
\mathfrak{L}_{\pmb{B},\pmb{B}',{\rm adb}}^{(0)}:\mathpzc{H}_{\pmb{B},{\rm adb}}|_{U_{\pmb{B}}\cap U_{\pmb{B}'}}\to \mathpzc{H}_{\pmb{B},{\rm adb}}|_{U_{\pmb{B}}\cap U_{\pmb{B}'}},
\end{equation}
by extending the operation $\mathfrak{L}_{\pmb{B},\pmb{B}',{\rm adb}}^{(0)}:\mathpzc{H}^\infty_{\pmb{B}',{\rm adb}, c}\to \mathpzc{H}^\infty_{\pmb{B},{\rm adb}}$, defined by
$$\mathfrak{L}_{\pmb{B},\pmb{B}',{\rm adb}}^{(0)}f(gh):=\int_{(d^{-1}(g)\cap\mathcal{F}_{V,\pmb{B}}) /(d^{-1}(g)\cap\mathcal{F}_{V,\pmb{B}}\cap \mathcal{F}_{V,\pmb{B}'})}f(gh)\rd h,$$
by unitarity. We remark that at $t=0$, we have deformed away all group structure while retaining the cocycle in the central direction so we can consider the adiabatic Lion intertwiner at fixed $\xi$ and $t=0$ as the Lion intertwiner for the polarizations that $F_{V,\pmb{B}}(\xi)$ and $F_{V,\pmb{B}'}(\xi)$ define for the real Heisenberg algebra structure $(\mathfrak{g}/\mathfrak{z})_{\R\omega_\xi}$ (for notation and further context, see Example \ref{heisebendnedcd}). In particular, a computation as in Proposition \ref{descriptionofddinvariant} and using Lemma \ref{lionsintertwiners} allow us to deduce the following:

\begin{proposition}
\label{lioninadidiaodoado}
The adiabatic Lion intertwiner from Equation \eqref{adiabativcckckdodo} intertwines the action of $C^*(\mathcal{G}_{\mathsf{Z},{\rm adb}}, \omega_{\mathcal{G}_{\mathsf{Z},{\rm adb}}}^{s=1})$ on $\mathpzc{H}_{\pmb{B}',{\rm adb}}$ with that on $\mathpzc{H}_{\pmb{B},{\rm adb}}$. Moreover, on triple intersections $U_{\pmb{B}}\cap U_{\pmb{B}'}\cap U_{\pmb{B}''}$ the Lion intertwiners satisfy the following property:
$$\mathfrak{L}_{\pmb{B},\pmb{B}',{\rm adb}}^{(0)}\mathfrak{L}_{\pmb{B}',\pmb{B}'',{\rm adb}}^{(0)}\mathfrak{L}_{\pmb{B}'',\pmb{B},{\rm adb}}^{(0)}\equiv\e^{\frac{\pi i}{4}\mathrm{Mas}(F_{V,\pmb{B}},F_{V,\pmb{B}'},F_{V,\pmb{B}''})}.$$
\end{proposition}

This will allow to make several conclusions for the structure of $C^*(\mathcal{G}_{\mathcal{Z},{\rm adb}}|_{\Gamma_P},\omega_{\mathcal{G}_{\mathcal{Z}},{\rm adb}}^{s=1})$. We start by describing what happens at $t=0$. For this, and also later purposes, we need a lemma. 

\begin{lemma}
\label{lem:desciriddtriviala}
Suppose that $A$ is a continuous trace algebra with spectrum $X$ and trivial Dixmier-Duoady invariant. Assume for $j=1,2$ that $\mathcal{H}_j\to X$ is a bundle of Hilbert spaces and $\pi_j:A\to C_0(X,\mathbb{K}(\mathcal{H}_j))$ are $C_0(X)$-linear $*$-isomorphisms. Then there is a line bundle $L_{\mathcal{H}_1,\mathcal{H}_2}\to X$, uniquely determined up to isomorphism, such that 
$$\mathcal{H}_1\cong L_{\mathcal{H}_1,\mathcal{H}_2}\otimes \mathcal{H}_2,$$
via a unitary intertwining $\pi_1$ with $\pi_2$.
\end{lemma}

\begin{proof}
By Kuiper's theorem we can without restriction assume that $\mathcal{H}_1=X\times \ell^2(\mathbb{Z})$ is trivial, that $A=C_0(X,\mathbb{K}(\ell^2(\mathbb{Z})))$ and that $\pi_1$ is the identity map. Pick a cover $\mathfrak{U}=(U_j)_j$ of $X$ over which $\mathcal{H}_2$ and $\pi_2$ trivialises. In other words, there are unitary isomorphisms $L_j:\mathcal{H}_2|_{U_j}\cong U_j\times \ell^2(\Z)$ and $L_ja L_j^*=\pi_2(a)$ for $a\in C_0(U_j,\mathbb{K}(\ell^2(\mathbb{Z})))$. We note that for $a\in C_0(U_j\cap U_k,\mathbb{K}(\ell^2(\mathbb{Z})))$, $L_j^*L_ka (L_j ^*L_k)^*=a$ so $u_{jk}=L_j^*L_k$ defines a $U(1)$-valued Cech $1$-cochain. It is direct from the construction that $(u_{jk})_{j,k}$ is a $U(1)$-valued Cech $1$-cocycle, and it defines a line bundle $L_{\mathcal{H}_1,\mathcal{H}_2}\to X$. The collection $(L_j)_j$ glues together to the required isomorphism $\mathcal{H}_1\cong L_{\mathcal{H}_1,\mathcal{H}_2}\otimes \mathcal{H}_2$.
\end{proof}

For locally trivial bundle of nilpotent Lie groups with fibre $\mathsf{G}$ and frame bundle $P$, we define $\Xi_P$ as above, and consider the vector bundle $\Xi_P^*\to \Gamma_P$ as a groupoid. We consider the fibrewise Kirillov form $\omega$ as a unitary $2$-cocycle $\omega:(\Xi_P^*)^{(2)}\to U(1)$. Choose a Riemannian metric $g_\Xi$ on $\Xi_P$. The Riemannian metric and the Kirillov form induces a complex structure on $\Xi_P$, in particular $\Xi_P\to \Gamma_P$ inherits a spin$^c$-structure. Following the construction of the Fock space bundle in Example \ref{freestep2example}, we arrive at a bundle of Fock spaces $\mathpzc{F}_\Xi\to \Gamma_P$ carrying an $\omega$-twisted representation of $\Xi_P^*$. As a bundle, $\mathpzc{F}_\Xi\to \Gamma_P$ is defined as 
$$\mathpzc{F}_\Xi:=\bigoplus_{k=0}^\infty \Xi_P^{\otimes_\C^s k}.$$
The twisted action of $\Xi_P^*$ on $\mathpzc{F}_\Xi$ is defined as in Equation \eqref{twisededrep} by identifying $\mathpzc{F}_\Xi$ with the subbundle of fiberwise holomorphic sections of a weighted $L^2$-bundle defined from $\Xi_P^*\to \Gamma_P$ and the Riemannian metric. The representation theory of the twisted groupoid $(\Xi_P^*,\omega)$ is computed from that of the locally trivial bundle of Heisenberg groups $\Xi_P^*\rtimes_\omega \R\to \Gamma_P$. Standard results on the Bargman-Fock representations, see for instance \cite{follandphasespace}, implies the following.

\begin{proposition}
\label{ismomomafonaofjan}
Let $\mathcal{G}\to M$ be a locally trivial bundle of nilpotent Lie groups and $g_\Xi$ a choice of metric on its Lie algebroid. The bundle of Fock spaces 
$$\mathpzc{F}_\Xi:=\bigoplus_{k=0}^\infty \Xi_P^{\otimes_\C^s k}\to \Gamma_P,$$ 
and the $\omega$-twisted action of $\Xi_P^*$ on $\mathpzc{F}_\Xi$ defined as in Equation \eqref{twisededrep} defines a $*$-isomorphism 
$$C^*(\Xi_P^*,\omega)\cong C_0(\Gamma_P;\mathbb{K}(\mathpzc{F}_\Xi)).$$
In particular, the $C^*$-algebra $C^*(\Xi_P^*,\omega)$ is a continuous trace algebra with spectrum $\Gamma_P$ and vanishing Dixmier-Duoady invariant and if $\mathcal{H}'\to \Gamma_P$ is a bundle of Hilbert spaces such that there is an isomorphism 
$$C^*(\Xi_P^*,\omega)\cong C_0(\Gamma_P;\mathbb{K}(\mathcal{H}')),$$
then there exists a line bundle $L_{\mathcal{H}'}\to \Gamma_P$ and an isomorphism of bundles of Hilbert spaces $\mathcal{H}'\otimes L\cong \mathpzc{F}_\Xi$.
\end{proposition}

The uniqueness of the Fock bundle up to line bundles follows from  Lemma \ref{lem:desciriddtriviala}. Next we consider a result that will play an important role in comparing the action of the Fock bundle on $K$–theory to Nistor's Connes-Thom isomorphism.

\begin{theorem}
\label{compisjsdjwitundendd}
Let $\mathcal{G}\to M$ be a locally trivial bundle of nilpotent Lie groups and $g_\Xi$ a choice of metric on its Lie algebroid. The class $[\mathpzc{F}_\Xi]\in KK_0(C^*(\Xi_P^*,\omega),C_0(\Gamma_P))$ defined using the isomorphism of Proposition \ref{ismomomafonaofjan}, fits into a commuting diagram 
\[
\begin{tikzcd}
C_0(\Xi_P) \arrow[ddr,"{[\slashed{D}_\Xi]}"]\arrow[rr, "\psi_{\Xi,P}"] && C^*(\Xi_P^*,\omega_P)\arrow[ddl,"{[\mathpzc{F}_\Xi]}"]\\
&&\\
&C_0(\Gamma_P)&
\end{tikzcd}
\]
where $\psi_{\Xi,P}\in KK_0(C^*(\Xi_P^*,\omega),C_0(\Xi_P))$ is the deformation morphism induced from $\psi_\Xi$ defined in Definition \ref{autgododlelendn} and $[\slashed{D}_\Xi]\in KK_0(C_0(\Xi_P),C_0(\Gamma_P))$ is the dual Thom class, i.e. the fibre-wise spin$^c$-Dirac operator on $\Xi_P\to \Gamma_P$ defined from the complex structure induced from $\omega$ and $g_\Xi$.
\end{theorem}

In the proof of this theorem, we will use unbounded $KK$-theory. The reader can find more details on these methods in for instance \cite{kkbor,GM,meslandbeast,meslandrencompt}.

\begin{proof}
Let $\tau_{\Xi_P^*}\in KK_0(C_0(\Gamma_P),C_0(\Xi))$ denote the Thom class, it has an unbounded $KK$-representative defined from $(C_0(\Xi_P;S),T)$ where $S\to \Xi_P$ is the complex spinor bundle and $T$ is given by Clifford multiplication by the fibrewise complex coordinate. A computation with the unbounded Kasparov product shows that the $KK$-element 
$$\tau_{\Xi_P^*}\otimes_{C_0(\Xi_P)}\psi_{\Xi,P} \in KK_0(C_0(\Gamma_P),C^*(\Xi_P^*,\omega)),$$
is represented by the unbounded $(C_0(\Gamma_P),C^*(\Xi_P^*,\omega))$-Kasparov cycle $(C^*(\Xi_P^*,\omega;S),\mathcal{T})$ where $C^*(\Xi_P^*,\omega;S)$ is the completion of $\mathcal{S}(\Xi_P^*;S)$ in the $C^*(\Xi_P^*,\omega)$-valued inner product
\[\langle f, g\rangle (v) = \int_{(\Xi_P^*)_{p(v)}} \omega(v+w,w) \langle f(w), g(v+w)\rangle_S\rd w, \quad f,g\in \mathcal{S}(\Xi_P^*;S),\]
and the operator $\mathcal{T}$ is the regular self-adjoint operator on $C^*(\Xi_P^*,\omega;S)$ defined from the closure of $T$ acting on the dense subspace $\mathcal{S}(\Xi_P;S)\subseteq C^*(\Xi_P^*,\omega;S)$.

The $KK$-element 
$$\tau_{\Xi_P^*}\otimes_{C_0(\Xi_P)}\psi_{\Xi,P}\otimes_{C^*(\Xi_P^*,\omega)}[\mathpzc{F}_\Xi]\in KK_0(C_0(\Gamma_P),C_0(\Gamma_P)),$$
is therefore represented by the unbounded $KK$-cycle 
$$(\mathpzc{F}_{\Xi, S},\mathfrak{T})=(C^*(\Xi_P^*,\omega;S)\otimes \mathpzc{F}_\Xi,\mathcal{T}\otimes 1_{\mathpzc{F}_\Xi}).$$ 
The $C_0(\Gamma_P)$-module $\mathpzc{F}_{\Xi, S}$ identifies with the Fock space of spinor valued sections and $\mathfrak{T}$ with the fibrewise Toeplitz operator defined from Clifford multiplication by the fibrewise complex coordinate. 

Let us prove that 
\begin{equation}
\label{aoidnaodna}
[(\mathpzc{F}_{\Xi, S},\mathfrak{T})]=[(C_0(\Gamma_P),0)]=1_{C_0(\Gamma_P)}.
\end{equation}
One can prove Equation \eqref{aoidnaodna} by considering the corresponding unbounded cycle $(\mathpzc{F}_0,\mathfrak{T}_0)$ localized to a point. That is, $\mathpzc{F}_0=\bigoplus_{k=0}^\infty V^{\otimes_\C^{\rm sym}k}\otimes S_0$, for a complex spinor space $S_0$ and $V=\mathfrak{g}/\mathfrak{z}$, and $\mathfrak{T}_0$ is Clifford multiplication by the complex coordinate. The unbounded cycle $(\mathpzc{F}_0,\mathfrak{T}_0)$ is $U(d)$-equivariant. The operator $\mathfrak{T}_0$ is by standard considerations for Toeplitz operators surjective and has a one-dimensional kernel, on which $U(d)$ acts trivially. Therefore, as classes in $KK_0^{U(d)}(\C,\C)$ we have that 
\begin{equation}
\label{aoidnaodna23}
[(\mathpzc{F}_0,\mathfrak{T}_0)]=\epsilon_{U(d)},
\end{equation}
where $\epsilon_{U(d)}$ is the trivial $U(d)$-representation. Since $\Xi$ has a complex structure, the lift of $P$ to $\Gamma$ has a reduction to a principal $U(d)$-bundle $P_U\to \Gamma_P$. By going to local coordinates we see that $[(\mathpzc{F}_{\Xi, S},\mathfrak{T})]$ is the image of $[(\mathpzc{F}_0,\mathfrak{T}_0)]$ under the induction functor
$$KK_0^{U(d)}(\C,\C)\to KK_0^{U(d)}(C_0(P_U),C_0(P_U))\to KK_0(C_0(\Gamma_P),C_0(\Gamma_P)).$$
Therefore, Equation \eqref{aoidnaodna} follows from Equation \eqref{aoidnaodna23}.

We conclude that $\tau_{\Xi_P^*}\otimes_{C_0(\Xi_P)}\psi_{\Xi,P}\otimes_{C^*(\Xi_P^*,\omega)}[\mathpzc{F}_\Xi]=1_{C_0(\Gamma_P)}$, so $\psi_{\Xi,P}\otimes_{C^*(\Xi_P^*,\omega)}[\mathpzc{F}_\Xi]$ is a right inverse to the Thom isomorphism and since the Thom isomorphism is an isomorphism, uniqueness of inverses gives that $\psi_{\Xi,P}\otimes_{C^*(\Xi_P^*,\omega)}[\mathpzc{F}_\Xi]$ is the inverse and therefore $\psi_{\Xi,P}\otimes_{C^*(\Xi_P^*,\omega)}[\mathpzc{F}_\Xi]=[\slashed{D}_\Xi]$.
\end{proof}

\begin{theorem}
\label{trivialdldaaddo}
Let $\mathcal{G}\to M$ be a locally trivial bundle of nilpotent Lie groups with fibre $\mathsf{G}$ and frame bundle $P$. Assume that $\mathsf{G}$ admits flat orbits. Then $I_{\mathsf{G},P}$ is a continuous trace algebra with spectrum $\Gamma_P$ and vanishing Dixmier-Duoady invariant. 

In particular, there exists a bundle of Hilbert spaces $\mathcal{H}\to \Gamma_P$ and a $C_0(\Gamma_P)$-linear $*$-isomorphism $\pi_{\musFlat}:I_{\mathsf{G},P}\to C_0(\Gamma_P,\mathbb{K}(\mathcal{H}))$. This choice is unique up to isomorphism of the Hilbert space bundle and tensoring by a line bundle. 
\end{theorem}

\begin{proof}
We construct a cover $\mathfrak{U}$ of $\Gamma_P$ by covering $M$ by open subsets $(U_j)_{j}$ over which $\mathcal{G}$ trivializes and setting $\mathfrak{U}$ to be the image of $(U_j\times U_{\pmb{B}})_{j,\pmb{B}}$ in said trivializations. Over each $U_j\times U_{\pmb{B}}$ we can write 
$$I_{\mathsf{G},P}|_{U_j\times U_{\pmb{B}}}\cong C_0(U_j\times U_{\pmb{B}},\mathbb{K}(\mathcal{H}_{\pmb{B},{\rm adb}}|_{t=1})),$$ 
and the transitition cocycle lifts to a cochain of adiabatic Lion intertwiners as in Proposition \ref{lioninadidiaodoado} restricted to $t=1$. The same construction applies ad verbatim to $C^*(\Xi_P^*,\omega)$ with adiabatic Lion intertwiners restricted to $t=0$. By Proposition \ref{lioninadidiaodoado},  $\delta_{\rm DD}(I_{\mathsf{G},P})$ and $\delta_{\rm DD}(C^*(\Xi_P^*,\omega))$ are represented by the same $U(1)$-cocycle. But $\delta_{\rm DD}(C^*(\Xi_P^*,\omega))=0$ by Proposition \ref{ismomomafonaofjan}.
\end{proof}

\begin{proposition}
\label{metaplecticlinecorrection}
Let $\mathcal{G}\to M$ be a locally trivial bundle of nilpotent Lie groups $\mathsf{G}$ that admits flat orbits. Let $g_\Xi$ be a fixed metric on the Lie algebroid of $\mathcal{G}$. Assume that $\mathcal{H}\to \Gamma_P$ is a bundle of Hilbert spaces and $\pi_{\musFlat}:I_{\mathsf{G},P}\to C_0(\Gamma_P,\mathbb{K}(\mathcal{H}))$ a $C_0(\Gamma_P)$-linear $*$-isomorphism. Then there exists a uniquely determined line bundle 
$$\mathfrak{M}(\mathcal{H})\to \Gamma_P,$$ 
that we call the metaplectic correction bundle, satisfying that $\mathcal{H}\otimes \mathfrak{M}(\mathcal{H})$ extends to a Hilbert space bundle $\mathcal{H}_{\rm mod}\to \Gamma_P\times [0,1]$ and $\pi_{\musFlat}$ extends to a $C_0(\Gamma_P\times [0,1])$-linear $*$-isomorphism
$$\pi_{\rm mod}:C^*(\mathcal{G}_{\mathcal{Z},{\rm adb}}|_{\Gamma_P}, \omega_{\mathcal{G}_{\mathsf{Z},{\rm adb}}}|_{s=1})\to C_0(\Gamma_P\times [0,1],\mathbb{K}(\mathcal{H}_{\rm mod})),$$
and satisfying 
$$\mathcal{H}_{\rm mod}|_{t=0}\cong \mathpzc{F}_\Xi,$$
as $(C^*(\Xi_P^*,\omega),C_0(\Gamma_P))$-Hilbert $C^*$-modules.
\end{proposition}

\begin{proof}
By the argument in Theorem \ref{trivialdldaaddo}, $C^*(\mathcal{G}_{\mathcal{Z},{\rm adb}}|_{\Gamma_P}, \omega_{\mathcal{G}_{\mathsf{Z},{\rm adb}}}|_{s=1})$ is a continuous trace algebra with spectrum $\Gamma_P\times [0,1]$ and trivial Dixmier-Duoady class. The proposition now follows from Lemma \ref{lem:desciriddtriviala}.
\end{proof}

\begin{remark}
\label{undiqoqdoflat}
In light of Lemma \ref{lem:desciriddtriviala}, Proposition \ref{metaplecticlinecorrection} implies that for any two bundles of Hilbert spaces $\mathcal{H}_j\to \Gamma_P$ with equipped with $C_0(\Gamma_P)$-linear $*$-isomorphisms $\pi_j:I_{\mathsf{G},P}\to C_0(\Gamma_P,\mathbb{K}(\mathcal{H}))$, $j=1,2$, there is an isomorphism of Hilbert space bundles 
$$\mathcal{H}_1\otimes \mathfrak{M}(\mathcal{H}_1)\cong \mathcal{H}_2\otimes \mathfrak{M}(\mathcal{H}_2),$$
compatible with the action of $I_{\mathsf{G},P}$. 

The reader could note that $\mathfrak{M}(\mathcal{H})\to \Gamma_P$ depends on $\pi_{\musFlat}$. The dependence on $\pi_{\musFlat}$ can be seen from the proof of Lemma \ref{lem:desciriddtriviala}.
\end{remark}

Let us describe the metaplectic correction line bundle for a particular choice of bundle of flat orbit representations.

\begin{lemma}
\label{superimportanttechnicallemma}
Let $\mathcal{G}\to M$ be a locally trivial bundle of nilpotent Lie groups with fibre $\mathsf{G}$ and frame bundle $P$. Assume that $\mathsf{G}$ admits flat orbits. Consider the bundle of flat orbit representations
$$\mathcal{H}\to \Gamma_P,$$
defined from gluing together along the corrected Lion intertwiners defined in small enough local trivializations as in Proposition \ref{descriptionofddinvariant} and \ref{groupoidftincentrals1}. Then it holds that  
$$c_1(\mathfrak{M}(\mathcal{H}))-c_1(\det(\Xi^*_P)) \in \mathrm{im}(H^1(\Gamma_P,\Z/8)\to H^2(\Gamma_P,\Z)),$$
is in the image of the Bockstein map associated with multiplication by $8$. In particular, the image of $c_1(\mathfrak{M}(\mathcal{H}))$ in $H^2(\Gamma_P,\R)$ coincides with that of $c_1(\det(\Xi^*_P))$. 
\end{lemma}

\begin{proof}
The reader should note that $\Xi_P$ carries an essentially unique complex structure adapted to the Kirillov form, and $\det(\Xi^*_P)\equiv\wedge_\C^d\Xi^*_P\to \Gamma_P$ is a well defined line bundle (for $d=\mathrm{rk}(\Xi_P)=\mathrm{codim}(\mathfrak{z})/2$). To prove the lemma we must show that there are transition functions for $\mathfrak{M}(\mathcal{H})$ and $\det(\Xi^*_P)$ that coincides up to powers of $\e^{\frac{\pi i}{4}}$. We can construct $\mathfrak{M}(\mathcal{H})$ as the line bundle $L_{\mathcal{H}'}$ from Proposition \ref{ismomomafonaofjan} where $\mathcal{H}'\to \Gamma_P$ is the $\omega$-twisted representation of $\Xi_P^*$ constructed from gluing together along the corrected Lion intertwiners defined in small enough local trivializations as in Proposition \ref{descriptionofddinvariant}. We can therefore construct  $\mathfrak{M}(\mathcal{H})$ explicitly from the $U(d)$-frame bundle on $\Xi_P$ and the associated corrected Lion transforms (cf. Proposition \ref{lionsintertwinerswitheta}).

Fix a transition $U(d)$-valued Cech cocycle $(g_{ij})_{i,j}$ for $\Xi_P$. By the structure of the metaplectic representation of $U(d)$  on the Fock bundle (see \cite[Chapter 4]{follandphasespace}), the Fock bundle is constructed from the Lion intertwiner corrected by transition functions of the form $(\pm \sqrt{\det(g_{ij}^*)})_{i,j}$. By Remark \ref{commentaboutpairs},  the bundle $\mathcal{H}'$ is constructed from the Lion intertwiner corrected by transition functions of the form $(\e^{\frac{\pi i}{4}k_{0,i,j}}\sqrt{\det(g_{ij}^*)})_{i,j}$ for some $k_{0,i,j}\in \Z$. The transition functions of $\mathfrak{M}(\mathcal{H})$ therefore take the form $(\e^{\frac{\pi i}{4}k_{i,j}}\det(g_{ij}^*))_{i,j}$ for some $k_{i,j}\in \Z$. The cocycle $(\det(g_{ij}^*))_{i,j}$ is a transition cocycle for $\det(\Xi^*_P)$ so the lemma follows.
\end{proof}

\begin{proposition}
\label{commsquaoadodal}
Let $\mathcal{G}\to M$ be a locally trivial bundle of nilpotent Lie groups $\mathsf{G}$ that admits flat orbits. Let $g_\Xi$ be a fixed metric on the Lie algebroid of $\mathcal{G}$. Assume that $\mathcal{H}\to \Gamma_P$ is a bundle of Hilbert spaces and $\pi_{\musFlat}:I_{\mathsf{G},P}\to C_0(\Gamma_P,\mathbb{K}(\mathcal{H}))$ a $C_0(\Gamma_P)$-linear $*$-isomorphism. 
The class $[\mathcal{H}]\in KK_0(I_{\mathsf{G},P},C_0(\Gamma_P))$ fits into a commuting diagram 
\[
\begin{tikzcd}
C_0(\Xi_P) \arrow[dd,"{[\slashed{D}_\Xi]}"]\arrow[rr, "\psi_{I,P}"] && I_{\mathsf{G},P}\arrow[dd,"{[\mathcal{H}]}"]\\
&&\\
C_0(\Gamma_P)\arrow[rr,"{[\mathfrak{M}(\mathcal{H})^*]}"]&&C_0(\Gamma_P)
\end{tikzcd}
\]
where $\psi_{I,P}\in KK_0(I_{\mathsf{G},P},C_0(\Xi_P))$ is the deformation morphism induced from $\psi_I$ defined in Definition \ref{autgododlelendn} and $[\slashed{D}_\Xi]\in KK_0(C_0(\Xi_P),C_0(\Gamma_P))$ is the dual Thom class.
\end{proposition}

\begin{proof}
The proposition follows from Theorem \ref{compisjsdjwitundendd} and Proposition \ref{metaplecticlinecorrection}.
\end{proof}

We summarize our constructions in the following theorem:

\begin{theorem}
\label{maincomputationforisg}
Let $\mathcal{G}\to M$ be a locally trivial bundle of nilpotent Lie groups $\mathsf{G}$ that admits flat orbits. Let $g_\Xi$ be a fixed metric on the Lie algebroid of $\mathcal{G}$. Assume that $\mathcal{H}\to \Gamma_P$ is a bundle of Hilbert spaces and $\pi_{\musFlat}:I_{\mathsf{G},P}\to C_0(\Gamma_P,\mathbb{K}(\mathcal{H}))$ a $C_0(\Gamma_P)$-linear $*$-isomorphism. The element $[\mathcal{H}]\in KK(I_{G,P},C_0(\Gamma_P))$, the Thom isomorphism $\tau_{\Xi_P}\in KK(C_0(\Gamma_P),C_0(\Xi_P))$ (defined from the complex structure induced from the metric), and the metaplectic correction bundle $\mathfrak{M}(\mathcal{H})$ (see Proposition \ref{metaplecticlinecorrection}) fits into a commuting diagram in $KK$:
\[
\begin{tikzcd}
C_0(\Gamma_P)\arrow[dd,"{[\mathfrak{M}(\mathcal{H})]}"]&&I_{G,P} \arrow[ll, "{[\mathcal{H}]}"] \arrow[rr, "\subseteq"] && C^*(\mathcal{G}) \\
&&\\
C_0(\Gamma_P)\arrow[rr,"\tau_{\Xi_P}"]&&C_0(\Xi_P)\arrow[rr, "\subseteq"]\arrow[uu,"\psi_{I,P}"] &&C_0(\mathbf{g}^*)\arrow[uu,"\psi_{P}"] \\
\end{tikzcd}
\]
\end{theorem}

\begin{proof}
The right square commutes by Theorem \ref{nistorconnethomfortwistgroup}. The left triangle commutes by Proposition \ref{commsquaoadodal} using the fact that the inverse of the Thom class is the fibreswise Dirac operator. 
\end{proof}

An immediate consequence of Theorem \ref{maincomputationforisg} is the following. 

\begin{corollary}
Under the same assumptions and notations as in Theorem \ref{maincomputationforisg}, and $j:\Gamma_P\hookrightarrow \mathbf{g}^*$ denoting the inclusion, the following diagram commutes
\[
\begin{tikzcd}
I_{\mathsf{G},P} \arrow[rr, "\subseteq"]\arrow[dd,"{[\mathcal{H}\otimes \mathfrak{M}(\mathcal{H})]}"] &&C^*(\mathcal{G})\arrow[dd,"\psi_P^{-1}"]\\
&&\\
C_0(\Gamma_P) \arrow[rr, "j_!"] && C_0(\mathbf{g}^*),
\end{tikzcd}
\]
and the vertical maps are $KK$-isomorphisms. Both horizontal maps are surjective on $K$-theory if and only if one of them are.
\end{corollary}

\part{Carnot manifolds and associated groupoids}
\label{sec:carnotmfds}

\begin{center}
{\bf Introduction to part}
\end{center}

At last, we have reached the point in the work where we see some geometry. The underlying geometric object of study in this monograph is Carnot manifolds. We shall at first explore their general properties and then explore some prototypical examples. Until that point, the reader is encouraged to keep contact manifolds and filtrations in mind. This section contains no novel results, but is included for context. The contents of this part is organized into the three sections:
\begin{itemize}
\item Section \ref{subsec:carnot} where we recall the geometry of Carnot manifolds. The main object we make use of is the osculating Lie groupoid $T_HX\to X$.  An important notion we introduce here is that of $\pmb{F}$-regularity, when the osculating Lie groupoid has (up to isomorphism) constant fibre with flat coadjoint orbits. 
\item Section \ref{sec:examcarn} where the reader can find several examples of Carnot manifolds. For instance, contact manifolds, polycontact manifolds, pluricontact manifolds, equiregular differential systems, and parabolic geometries. 
\item Section \ref{subsec:parabolictanget} where we recall the construction of the parabolic tangent groupoid from \cite{vanErp_Yunckentangent}. Here we also recall the construction of the adiabatic parabolic tangent groupoid and equipp $\Xi_X$ and $\Gamma_X$ with spin$^c$-structures.
\end{itemize}

\section{Carnot manifolds}
\label{subsec:carnot}

\begin{definition}
A Carnot manifold $(X, \mathcal{F})$ is a smooth manifold $X$ equipped with a
filtration $\mathcal{F}$ of the tangent bundle $TX$ by smooth subbundles
\begin{equation}
\label{filterineindodod}
\mathcal{F} : 0 = T^{0}X \subset T^{-1}X \subset T^{-2}X \ldots \subset T^{-r+1}X\subset T^{-r}X = TX, 
\end{equation}
such that all inclusions are strict and for any vector fields $Y \in C^\infty(X,T^iX)$, $Y' \in C^\infty(X,T^jX)$ we have that $[Y, Y'] \in C^\infty(X,T^{i+j}X)$. When the filtration is understood from context, we simply write $X$ for the Carnot manifold. We call $r$ the depth of the filtration.

A morphism $\phi:(X_1, \mathcal{F}_1)\to (X_2, \mathcal{F}_2)$ of two Carnot manifolds is a smooth mapping $\phi:X_1\to X_2$ such that $D\phi$ preserves the filtration. An isomorphism of Carnot manifolds is a morphism $\phi:(X_1, \mathcal{F}_1)\to (X_2, \mathcal{F}_2)$ such that $\phi:X_1\to X_2$ is a diffeomorphism.
\end{definition}

\begin{remark}
In \cite{Morimoto,sadeghhigson,vanErp_Yunckentangent,vanErp_Yuncken}, a Carnot manifold is called a filtered manifold. In \cite{melinoldpreprint}, a filtration as in \eqref{filterineindodod} appears with a slightly different grading and is there called a Lie filtration. We have chosen the term Carnot manifold in accordance with \cite{choiponge,Dave_Haller1,Dave_Haller2,mohsen2}.
\end{remark}

Given a Carnot manifold $(X, \mathcal{F})$, define $\mathsf{gr}_i(TX) :=T^iX / T^{i+1}X$ and let $q_i : T^i X \rightarrow \mathsf{gr}_i(T X)$ denote the natural quotient map. We consider the graded vector bundle:
\[ \mathsf{gr}(TX) := \bigoplus_{i = -r}^{-1} \mathsf{gr}_i(TX) = \bigoplus_{i = -r}^{-1} T^iX / T^{i+1}X. \]
For every $i \in 1,\ldots,r$ one has an exact sequence
\[ 0 \rightarrow T^{i+1}X \rightarrow T^i X \rightarrow \mathsf{gr}_i(TX) \rightarrow 0, \]
and a choice of splitting of this sequence gives rise to a vector bundle isomorphism $T^i X \simeq T^{i+1} X \oplus \mathsf{gr}_i X$. Thus there exists a non-canonical vector bundle isomorphism $TX \simeq \mathsf{gr}(TX)$. 

Using the property that the Lie bracket on vector fields respect the filtration, one equips $\mathsf{gr}(TX)$ with a non-trivial graded Lie algebroid structure, which differs from the standard Lie algebroid structure of $TX$. By definition, a Carnot manifold $(X, \mathcal{F})$ induces a Lie algebra filtation of $C^\infty(X,TX)$ given by 
\begin{align}
\label{liealgfiltofvecodod}
0 = C^\infty(X,T^0X) &\subset C^\infty(X,T^{-1}X) \subset \ldots\\
\nonumber
\ldots &\subset C^\infty(X,T^{-r}X)= C^\infty(X,TX). 
\end{align}
We note that for $Y \in C^\infty(X,T^iX)$, $Y' \in C^\infty(X,T^jX)$ and $f, g \in C^\infty(X)$ one has
\[ [fY, gY'] = fg [Y,Y'] + f(Yg)Y' - g(Y'f)Y = fg[Y,Y'] + \Gamma^\infty(T^{i+j+1}X). \]
In particular, the Lie bracket on vector fields induces a $C^\infty(X)$-linear map
\begin{equation} 
\label{eq:1}
\mathcal{L}:C^\infty(X,\mathsf{gr}(TX))\times C^\infty(X,\mathsf{gr}(TX))\to C^\infty(X,\mathsf{gr}(TX)),\\
\end{equation}
defined on $Y\in C^\infty(X,\mathsf{gr}_i(TX))=C^\infty(X,T^iX) / C^\infty(X,T^{i+1}X)$ and $Y'\in C^\infty(X,\mathsf{gr}_j(TX))=C^\infty(X,T^jX) / C^\infty(X,T^{j+1}X)$ as 
$$\mathcal{L}(Y,Y'):=[Y,Y']\!\!\!\mod C^\infty(X,T^{i+j+1}X)\in C^\infty(X,\mathsf{gr}_{i+j}(TX)).$$
The $C^\infty(X)$-linearity of $\mathcal{L}$ allow us to identify $\mathcal{L}$ with an antisymmetric graded morphism 
$$\mathcal{L}:\mathsf{gr}(TX)\wedge \mathsf{gr}(TX)\to \mathsf{gr}(TX).$$
We call $\mathcal{L}$ the Levi bracket of $(X,\mathcal{F})$. The Levi bracket at $m$,
\[ \mathcal{L}_m : \mathsf{gr}(T_m X) \times \mathsf{gr}(T_m X) \rightarrow \mathsf{gr}(T_m X), \]
endows $\mathsf{gr}(T_m X)=\bigoplus_{i = -r}^{-1} \mathsf{gr}_i(T_mX)$ with the structure of a graded nilpotent Lie algebra of step length $\leq r$ -- the depth of the filtration. This discussion implies the following proposition.

\begin{proposition}
Let $(X,\mathcal{F})$ be a Carnot manifold with Levi bracket $\mathcal{L}$ on the assocated graded vector bundle $\mathsf{gr}(TX)$. The collection $(\mathsf{gr}(TX), \mathcal{L}, 0)$ defines a Lie algebroid over $X$ satisfying the assumptions of Proposition \ref{integrationofnilpotentalgebroid}.
\end{proposition}

\begin{remark}
\label{singularcarnotremark}
We note that in the spirit of \cite{androskand} and Equation \eqref{liealgfiltofvecodod}, we can define a singular Carnot structure as a Lie algebra filtration 
\[ \{0\} = \mathpzc{E}^0 \subset \mathpzc{E}^{-1} \subset \mathpzc{E}^{-2} \subset \ldots \subset \mathpzc{E}^{-r}= C^\infty(X,TX), \]
where each $\mathpzc{E}^j\subseteq C^\infty(X,TX)$ is a locally finitely generated $C^\infty(X)$-submodule such that $[\mathpzc{E}^i,\mathpzc{E}^j]\subseteq \mathpzc{E}^{i+j}$. A singular Carnot structure for which each $\mathpzc{E}^j$ locally is projective induces an ordinary Carnot structure. In the preprint \cite{androerp}, the machinery for Carnot manifolds is extended to singular Carnot structures.
\end{remark}

\begin{definition}
\label{osculatingdef}
Let $(X,\mathcal{F})$ be a Carnot manifold. 
\begin{itemize}
\item The osculating Lie algebroid $\mathfrak{t}_HX$ of $(X, \mathcal{F})$ is the graded Lie algebroid $(\mathsf{gr}(TX), \mathcal{L}, 0)$.
\item The osculating Lie groupoid $T_HX\to X$ of $(X, \mathcal{F})$ is the Lie groupoid integrated from  $\mathfrak{t}_HX$ as in Proposition \ref{integrationofnilpotentalgebroid}.
\end{itemize}

We say that $(X,\mathcal{F})$ is a \emph{regular Carnot manifold} if the osculating Lie groupoid $T_H X\to X$ is a locally trivial bundle of Carnot-Lie groups. If $T_HX\to X$ is of type $\mathfrak{g}$, for some graded nilpotent Lie algebra $\mathfrak{g}$, we say that $(X,\mathcal{F})$ is regular of type $\mathfrak{g}$. 
\end{definition}

\begin{proposition}
\label{carnotframbeldld}
Let $(X,\mathcal{F})$ be a regular Carnot manifold of type $\mathfrak{g}$. Then there exists a principal $\Aut_{\rm gr}(\mathfrak{g})$-bundle $P_X\to X$ such that 
$$\mathfrak{t}_HX\cong P_X\times_{\Aut_{\rm gr}(\mathfrak{g})} \mathfrak{g},$$ 
as graded Lie algebroids and 
$$T_HX\cong P_X\times_{\Aut_{\rm gr}(\mathfrak{g})} G,$$ 
as locally trivial bundles of Carnot-Lie group.
\end{proposition}

This proposition follows from Proposition \ref{integrationofloctrivialnilpotentalgebroid}. We call $P_X\to X$ the Carnot frame bundle of the regular Carnot manifold $(X,\mathcal{F})$. We write $Z_H X\subseteq T_HX$ for the central subgroupoid, it is canonically isomorphic to $P_X\times_{\Aut(\mathsf{G})}Z=P_X\times_{\Aut(\mathfrak{g})}\mathfrak{z}$ via the isomorphism of Proposition \ref{carnotframbeldld}. If $(X,\mathcal{F})$ is regular of type $\mathfrak{g}$, and $\mathfrak{g}$ admits flat coadjoint orbits we form the fibre bundle of flat orbits over $X$ as
$$p_\Gamma:\Gamma_X:=P_X\times_{\Aut(\mathsf{G})}\Gamma\to X.$$ 
We can view $\Gamma_X$ as an open subset of the dual bundle $(Z_HX)^*\to X$.

\begin{definition}
\label{osculatingdef2}
Let $(X,\mathcal{F})$ be a Carnot manifold. We say that $(X,\mathcal{F})$ is an \emph{$\pmb{F}$-regular Carnot manifold} if it is regular of type $\mathfrak{g}$ and
\begin{enumerate} 
\item the type $\mathfrak{g}$ admits flat coadjoint orbits; and 
\item the wrong way map $(p_\Gamma)_!:K^*(\Gamma_X)\to K^*(X)$ associated with the projection $p_\Gamma$ is a surjection.
\end{enumerate}
\end{definition}

More context for Condition (2) of Definition \ref{osculatingdef2} can be found in Section \ref{suhjrjrjed}. We note that the authors are unaware of any examples of  Carnot manifolds regular of a type $\mathfrak{g}$ admitting flat coadjoint orbits that are not $\pmb{F}$-regular.

We define the groupoid $\mathcal{G}_{Z,\mathcal{F}}$ to be the pull back of $T_HX/Z_HX$ up to $Z_HX^*$. By choosing a splitting of the inclusion $Z_HX\subseteq T_HX$, we can as in Section \ref{subsec:loctrivalandnda} and \ref{sec:connesthomandadiaofofd} define a $U(1)$-valued $2$-cocycle $\omega_\mathcal{F}$ on $\mathcal{G}_{Z,\mathcal{F}}$. Following the same arguments as in Section \ref{subsec:loctrivalandnda}, \ref{subsec:nistorctsubsn} and \ref{sec:connesthomandadiaofofd}, we arrive at the next proposition. 

\begin{proposition}
Let $(X,\mathcal{F})$ be a regular Carnot manifold of type $\mathfrak{g}$ with frame bundle $P_X\to X$. Then the groupoid $C^*$-algebra $C^*(T_HX)$ admits canonical isomorphisms 
$$C^*(T_HX)\cong C(X;P_X\times_{\Aut(\mathsf{G})}C^*(\mathsf{G}))\cong C^*(\mathcal{G}_{Z,\mathcal{F}},\omega_\mathcal{F}).$$

If $(X,\mathcal{F})$ is $\pmb{F}$-regular, then the ideal $I_X\subseteq C^*(T_HX)$ of elements vanishing outside the flat representations is well defined and admits canonical isomorphisms 
$$I_X\cong C(X;P_X\times_{\Aut(\mathsf{G})}I_{\mathsf{G}})\cong C^*(\mathcal{G}_{Z,\mathcal{F}}|_{\Gamma_X},\omega_\mathcal{F}).$$
Moreover, $I_X$ is a continuous trace algebra with spectrum $\Gamma_X$ and vanishing Dixmier-Duoady invariant such that the map 
$$K_*(I_X)\to K_*(C^*(T_HX)),$$
induced by the inclusion, is a surjection.
\end{proposition}

\begin{remark}
For any Carnot manifold $(X,\mathcal{F})$ with the central subgroupoid $Z_HX\subseteq T_HX$ being a well defined vector bundle on $X$, there is a canonical isomorphism $C^*(T_HX)\cong C^*(\mathcal{G}_{Z,\mathcal{F}},\omega_\mathcal{F})$. Moreover, if $\Gamma_X\subseteq Z_HX^*$ denotes the open subset of flat orbits, the ideal $I_X\subseteq C^*(T_HX)$ of elements vanishing outside the flat representations is well defined and admits a canonical isomorphism $I_X\cong  C^*(\mathcal{G}_{Z,\mathcal{F}}|_{\Gamma_X},\omega_\mathcal{F})$. If $\Gamma_X\neq \emptyset$, we can from abstract principles deduce that $I_X$ is a continuous trace algebra with spectrum $\Gamma_X$. It would be interesting to further explore the global structure of $I_X$ for general Carnot manifolds.
\end{remark}

Using results of Morimoto \cite{Morimoto}, the property of being regular can be verified on a pointwise basis.

\begin{theorem}[Morimoto, Chapter 3 in \cite{Morimoto}]
\label{morimorohospintise9}
Let $(X,\mathcal{F})$ be a Carnot manifold and $\mathfrak{g}$ a graded nilpotent Lie algebra. Then $(X,\mathcal{F})$ is regular of type $\mathfrak{g}$ if and only if for every $m\in X$, there exists a graded Lie algebra isomorphism $\mathsf{gr}(T_m X)\cong \mathfrak{g}$.
\end{theorem}

We remark that regular Carnot manifolds (of type $\mathfrak{g}$) need not be locally isomorphic to $\mathsf{G}$. Let us formalize this latter notion and recall a result from Morimoto \cite{Morimoto} ensuring local isomorphisms with $\mathsf{G}$.

\begin{definition}
Let $(X,\mathcal{F})$ be a Carnot manifold and $\mathfrak{g}$ a graded nilpotent Lie algebra. Set $G:=\mathrm{exp}(\mathfrak{g})$. We say that $(X,\mathcal{F})$ has the Darboux property (with respect to $\mathfrak{g}$) if for any $x\in X$, the following structures exist:
\begin{enumerate}
\item there is an open neighborhood $U \subset X$ of $x$, and
\item there is an open neighborhood $U_0\subseteq G$ of the identity, equipped with the Carnot structure induced from the filtering of $\mathfrak{g}$ defined from its grading, and 
\item there is an isomorphism of Carnot structures $f:U\to U_0$.
\end{enumerate}
\end{definition}

We note that if a Carnot manifold has the Darboux property (with respect to $\mathfrak{g}$) then it is regular of type $\mathfrak{g}$. The converse is true under additional cohomological assumptions on $\mathfrak{g}$.

\begin{theorem}[Generalized Darboux theorem]
\label{gendarboux}
Let $(X, \mathcal{F})$ be a regular analytic Carnot manifold of type $\mathfrak{g}$. Then $(X, \mathcal{F})$ has the Darboux property if $H^2(\mathfrak{g})_r = 0$, for all $r \geq 1$.
\end{theorem}

This theorem can be found in \cite[Corollary 3.6.1]{Morimoto}. The structure of coordinates on Carnot manifolds was further studied in \cite{choipongeprivI,choipongeprivII}.

\begin{remark}
The reader can find numerous graded Lie algebras $\mathfrak{g}$ for which $H^2(\mathfrak{g})_r = 0$, for all $r \geq 1$, in \cite{MorimotoTrans}. Most notably, the Heisenberg group satisfies $H^2(\mathfrak{g})_r = 0$, for all $r \geq 1$ and in this case Theorem \ref{gendarboux} reproduces the well known Darboux theorem.
\end{remark}

\subsection{Regularity of depth $2$ Carnot structures}

Let us give some further details in the case of a Carnot manifold of depth $2$. In other words, consider a manifold $X$ equipped with a filtration of the tangent bundle
$$0\subseteq H\subseteq TX.$$
Here $H\to X$ is an arbitrary subbundle of $TX$. The associated graded bundle is $\mathsf{gr}(TX):=H\oplus TX/H$ where $H$ has degree $-1$ and $TX/H$ has degree $-2$. The Levi form can be identified with a morphism
$$\mathcal{L}:H\wedge H\to TX/H.$$
Let $H^\perp\subseteq T^*X$ denote the subbundle of annihilators of $H$. If we pick a Riemannian metric on $X$, we have a vector bundle isomorphism $TX/H\cong H^\perp$, $\nu\mapsto \nu^\#$. The mapping $\mathcal{L}$ is determined by a linear mapping 
$$\omega:TX/H\to \mathfrak{so}(H),$$
related to $\mathcal{L}$ from the identity
$$\langle \omega(\nu)X,Y\rangle_H=\nu^\#(\mathcal{L}(X,Y)).$$
It follows from Proposition \ref{charofspeepoed} that for each $m\in X$, the step $2$ nilpotent Lie algebra $\mathsf{gr}(T_mX)$ is of type $(\mathrm{rk}(\omega_m),n-\mathrm{rk}(\omega_m))$ and that there exists a Lie algebra isomorphism 
$$\mathsf{gr}(T_mX)\cong (H_m\oplus \ker\omega_m)_{\omega_m(T_mX/H_m)}.$$
In general, $\ker\omega_m$ and $\omega_m(T_mX/H_m)$ can vary quite wildly with $m$ (see for instance Example \ref{boundaryofpseudoconvexz}). The following result shows that regularity is equivalent to global control of these distributions.

\begin{proposition}
\label{step2regularitychar}
Consider a Carnot manifold $X$ of step length $2$ defined from a subbundle $H\subseteq TX$. Let $V_{-1}$ denote the fibre of $H$ and $P_H\to X$ the $GL(V_{-1})$-frame bundle of $H$. Then $X$ is regular if and only if there is a subspace $W\subseteq \mathfrak{so}(V_{-1})$ and for 
\begin{equation}
\label{definienfwow}
H_W:=\{g\in GL(V_{-1}): g^tWg\subseteq W\}\subseteq GL(V_{-1}),
\end{equation}
there is an $H_W$-reduction $P_{H_W}\to X$ of $P_H$ such that 
$$\omega(TX/H)\cong P_{H_W}\times_{H_W}W.$$
If $X$ is regular, it is $\pmb{F}$-regular if and only if $W\cap GL(V_{-1})\neq \emptyset$ and the pushforward defined from $P_{H_W}\times_{H_W}(W\cap GL(V_{-1}))\to X$ is surjective in $K$-theory.
\end{proposition}

\begin{proof}
Assume first that $X$ is regular of type $\mathfrak{g}=V_W$ for an inner product space $V$ and a subspace $W\subseteq \mathfrak{so}(V)$ (see Proposition \ref{charofspeepoed}). It follows from construction of the osculating Lie algebroid that we can decompose $V=V_{-1}\oplus V_{-2}$ where $V_j$ has degree $j$ and that $W$ acts trivially on $V_{-2}$, and can be identified with a subspace of $\mathfrak{so}(V_{-1})$. As in Proposition \ref{autoforstep2decoss} we have that 
$$\Aut_{\rm gr}(\mathfrak{g})\cong \Hom(V_{-2},W)\rtimes (H_W\oplus GL(V_{-2})).$$ 
The induced action of $\Aut_{\rm gr}(\mathfrak{g})$ on $W$ factors over the $H_W$-action on $W$ defined by $\nu\mapsto g^t\nu g$. Let $P_X\to X$ denote the $\Aut_{\rm gr}(\mathfrak{g})$-principal bundle from Proposition \ref{carnotframbeldld}. Note that $H\cong P_X\times_{\Aut_{\rm gr}(\mathfrak{g})}V_{-1}$. It follows from Proposition \ref{charofspeepoed} that 
$$\omega(TX/H)\cong P_X\times_{\Aut_{\rm gr}(\mathfrak{g})}W,$$ 
as a sub-bundle of 
$$\mathfrak{so}(H)\cong P_X\times_{\Aut_{\rm gr}(\mathfrak{g})}\mathfrak{so}(V_{-1}).$$ 
The obvious homomorphism $\Aut_{\rm gr}(\mathfrak{g})\to H_W$ reduces $P_X$ to a principal $H_W$-bundle. 

Conversely, assume that $\omega(TX/H)\subseteq \mathfrak{so}(H)$ is a sub-bundle of the form specified in the proposition. Since $P_{H_W}$ is a reduction of $P_H$, $H\cong P_{H_W}\times_{H_W}V_{-1}$. Clearly, $\ker(\omega)$ will then form a sub-bundle of $TX/H$ and $TX/H\cong \ker(\omega)\oplus \omega(TX/H)$. Denote the typical fibre of $\ker(\omega)$ by $V_{-2}$ and $P_{\ker\omega}\to X$ the principal $GL(V_{-2})$-bundle of frames of $V_{-2}$. We can form the $H_W\oplus GL(V_{-2})$-bundle $P_{H_W}\times_X P_{\ker\omega}\to X$. Form the graded step $2$ nilpotent Lie algebra $\mathfrak{g}:=V_W$, where $V:=V_{-1}\oplus V_{-2}$. Since $H_W\oplus GL(V_{-2})\hookrightarrow \Aut_{\rm gr}(\mathfrak{g})$ (by the argument in the preceeding paragraph), we can form $P_X:=(P_{H_W}\times_X P_{\ker\omega})_{H_W\oplus GL(V_{-2})}\Aut_{\rm gr}(\mathfrak{g})$. From our construction we conclude that $\mathsf{gr}(TX)\cong P_X\times_{\Aut_{\rm gr}(\mathfrak{g})}\mathfrak{g}$ as bundles of graded Lie algebras so $X$ is regular of type $\mathfrak{g}$.
\end{proof}

\section{Examples of Carnot manifolds}
\label{sec:examcarn}

There is a large range of examples of Carnot manifolds. Let us give some examples that showcases the large range of behaviours that can be found in different Carnot manifolds. 

\begin{example}[Trivial filtration] 
\label{ex:trivialdcndododa}
Any smooth manifold $X$ admits a trivial filtration with depth $1$:
\[ \mathcal{F} : 0 = TX^0 \subset TX^{-1} = TX. \] 
We note that $\mathsf{gr}(TX)=TX$, with constant grading $-1$, as graded vector bundles and $\mathcal{L}=0$. Thus $\mathfrak{t}_HX=TX$ equipped with the zero Lie bracket and zero anchor map. It is clear that $(X, \mathcal{F})$ is a regular Carnot manifold of type $\mathbb{R}^n$ (with constant grading $-1$) and $P_X$ coincides with the standard frame bundle of $TX\to X$.
\end{example}

\begin{example}[Regular foliations]
\label{regfoliationexala}
A regular folation on a smooth manifold $X$ gives rise to a filtration of length $r = 2$. A regular foliation is characterized by an involutive subbundle $H\subseteq TX$ (the leafwise tangent space), i.e. $H$ satisfies $[Y, Y'] \in C^\infty(X,H)$ for any $Y, Y' \in C^\infty(X,H)$. For an involutive subbundle $H\subseteq TX$ we define the filtration
\[ \mathcal{F} : 0 = TX^0 \subset T^{-1}X:=H \subset TX^{-2} = TX.\]
For this filtration, $\mathsf{gr}(TX)=H\oplus TX/H$, where $H$ has degree $-1$ and $TX/H$ has degree $-2$, and $\mathcal{L}=0$. Thus $\mathfrak{t}_HX=H\oplus TX/H$ equipped with the zero Lie bracket and zero anchor map. It is clear that $(X, \mathcal{F})$ is a regular Carnot manifold of type $\mathbb{R}^m[-1] \oplus \mathbb{R}^{n - m}[-2]$ for $m:=\mathrm{rk}(H)$ and $n:=\dim(X)$. The Carnot frame bundle $P_X$ is a $GL(m)\oplus GL(n-m)$-reduction of the standard $GL(n)$-frame bundle of $TX\to X$.
    
A Carnot structure induced from a foliation has the Darboux property because regular foliations admit foliation charts (see for instance \cite{candelconlon}). The reader should take note that foliation charts exemplify the fact that the local filtered isomorphisms implementing the Carboux property can in general not be chosen as graded isomorphisms, e.g. when changing between two foliation charts the transitition matrix acting on the fibres of $TX\cong H\oplus TX/H$ will not necessarily be block diagonal but just lower diagonal.  

More generally, if we on a manifold have a filtration 
\[ \mathcal{F} : 0 = T^{0}X \subset T^{-1}X \subset T^{-2}X \ldots \subset T^{-r+1}X\subset T^{-r}X = TX, \]
such that for vector fields $Y \in C^\infty(X,T^iX)$, $Y' \in C^\infty(X,T^jX)$ satisfy that $[Y, Y'] \in C^\infty(X,T^{i+j+1}X)$, then by construction $\mathcal{L}=0$. Therefore, we can identify $\mathfrak{t}_HX=TX$ with the grading on $TX$ induced from the filtration and equipped with the zero bracket and the zero anchor map. The associated Carnot manifold is regular of type $\oplus_{j=-r}^{-1}\mathbb{R}^{m_j}[j]$ for $m_j:=\mathrm{rk}(T^{j}X/T^{j+1}X)$. The (slightly trivial) Carnot-Lie group $\oplus_{j=-r}^{-1}\mathbb{R}^{m_j}[j]$  is described in Subsection \ref{step1example}. The Carnot frame bundle $P_X$ is a $\oplus_{j=-r}^{-1} GL(m_j)$-reduction of the standard $GL(n)$-frame bundle of $TX\to X$. We note that in this example, the step length of each fibre of $\mathfrak{t}_HX$ is $1$ despite the fact that the depth of the filtration can be arbitrary large.
\end{example}

\begin{example}[Standard regular Carnot manifold of type $\mathfrak{g}$]
\label{standardcarnot}
Consider a filtered nilpotent Lie algebra 
\begin{equation}
\label{filteredliealgebra}
0 = \mathfrak{g}^{0} \subset \mathfrak{g}^{-1} \subset \mathfrak{g}^{-2} \ldots \subset \mathfrak{g}^{-r+1}\subset \mathfrak{g}^{-r} = \mathfrak{g}. 
\end{equation}
Let $\mathsf{G}$ be the associated simply connected nilpotent Lie group. Define a filtration $\mathcal{F}_\mathsf{G}$ of $T\mathsf{G}$ by identifying $T^i\mathsf{G} := \mathsf{G} \times \mathfrak{g}^i$ with a subbundle of $T \mathsf{G} \simeq \mathsf{G} \times \mathfrak{g}$. Then $(\mathsf{G}, \mathcal{F}_\mathsf{G})$ is a Carnot manifold. The Carnot manifold $(\mathsf{G}, \mathcal{F}_\mathsf{G})$ is regular but a short inspection of the construction of the osculating Lie algebroid shows that it is of type $\mathfrak{g}_{\rm gr}$, where $\mathfrak{g}_{\rm gr}$ is the graded Lie algebra associated with the filtration \eqref{filteredliealgebra}. In general, there is no Lie algebra isomorphism $\mathfrak{g}_{\rm gr}\cong \mathfrak{g}$ (for instance,  $\mathfrak{g}\not\cong \mathfrak{g}_{\rm gr}$ when $\mathfrak{g}$ does not admit a grading, cf. Remark \ref{lackofgradingremark}). 

If $\mathfrak{g}$ is graded, and the filtration is defined from $\mathfrak{g}^i = \bigoplus_{k\leq i} \mathfrak{g}_k$, then $\mathsf{G}$ is a Carnot-Lie group and we call $( \mathsf{G},\mathcal{F}_\mathsf{G})$ the standard regular Carnot manifold of type $\mathfrak{g}$. If $\mathsf{G}$ is a Carnot-Lie group and we use the filtering associated with the grading, it is readily verified that the standard regular Carnot manifold $(\mathsf{G},\mathcal{F}_\mathsf{G})$ has the Darboux property. As such, the generalized Darboux theorem \ref{gendarboux} gives sufficient but not necessarily necessary conditions for the Darboux property. We note that if we start from a simply connected Lie group $\mathsf{G}$ and a filtration of its Lie algebra, then $( \mathsf{G},\mathcal{F}_\mathsf{G})$ is isomorphic to the standard regular Carnot manifold of type $\mathfrak{g}_{\rm gr}$ given by $(\mathsf{G}_{\rm gr},\mathcal{F}_{ \mathsf{G}_{\rm gr}})$ where $\mathsf{G}_{\rm gr}:=\mathrm{exp}(\mathfrak{g}_{\rm gr})$.
\end{example}

\begin{example}[Quotients of nilpotent Lie groups by cocompact lattices]
\label{standardcarnotmodlattice}
Let $\mathsf{G}$ be a nilpotent Lie group with Lie algebra $\mathfrak{g}$. We say that a subgroup $\Gamma$ in $\mathsf{G}$ is a cocompact lattice if $\Gamma$ is discrete and $X := \Gamma \setminus \mathsf{G}$ is compact. Existence of a cocompact lattice in $\mathsf{G}$ is by \cite[Theorem 5.1.8]{Corwin_Greenleaf} equivalent to the Lie algebra $\mathfrak{g}$ admitting a rational structure. i.e. there is a rational Lie subalgebra $\mathfrak{g}_{\mathbb{Q}}$ such that $\mathfrak{g} \simeq \mathfrak{g}_{\mathbb{Q}} \otimes_\mathbb{Q} \mathbb{R}$. In fact, the cocompact lattice and the rational structure $\mathfrak{g}_{\mathbb{Q}}$ can be choosen so that $\log \Gamma \subset \mathfrak{g}_{\mathbb{Q}}$. 

Since $\Gamma$ is discrete, $TX \simeq T\mathsf{G}/\Gamma \simeq X  \times \mathfrak{g}$. Therefore, any Lie algebra filtration of  $\mathfrak{g}$ (as in Example \ref{standardcarnot}) produces a Carnot structure $\mathcal{F}$ on the closed manifold $X$. By similar arguments as in Example \ref{standardcarnot}, $(X,\mathcal{F})$ is regular of type $\mathfrak{g}_{\rm gr}$ and satisfies the Darboux property. We also note that the $9$-dimensional nilpotent Lie algebra constructed in \cite{dyer70} (cf. discussion in Remark \ref{lackofgradingremark}) admits no grading but admits cocompact lattices. 
\end{example}

\begin{example}[Contact manifolds]
\label{contactexample}
A contact structure on a manifold $X$ is a length $2$ filtration 
$$0\subseteq H\subseteq TX,$$
such that $H$ has codimension $1$ and $\mathcal{L}:H\wedge H\to TX/H$ is non-degenerate in all points. It follows that $H$ has even rank, say $2n$, and that $X$ is $2n+1$-dimensional. It is clear that for each $m\in X$, $\mathsf{gr}(TX)\cong \mathfrak{h}_n$ as graded Lie algebras. Therefore, Theorem \ref{morimorohospintise9} implies that a contact structure induces a regular Carnot structure on $X$ of type $\mathfrak{h}_n$. By the classical Darboux theorem, a contact manifold has the Darboux property. Another way to prove that a contact structure is regular, is to pick a metric on $X$, and since $U(n)\rtimes \Z/2$ coincides with the group of orthogonal matrices preserving a symplectic form up to sign, we can reduce the structure group of $H$ to $U(n)\rtimes \Z/2$. Therefore, Proposition \ref{step2regularitychar} implies that a contact structure is regular of type $\mathfrak{h}_n$ by taking $W$ to be the span of the standard symplectic form. We note that for a contact manifold, the bundle of flat orbits $p_\Gamma:\Gamma_X\to X$ is an $\R^\times$-bundle over $X$. 

There are different methods to produce global bundles of flat representations for the osculating groupoid of a contact manifold. The bundle of flat representations on $\Gamma$ from Example \ref{heisebendnedcd} induces up to $\Gamma_X$. Another, slightly more convenient, method is to choose a metric on $H$ and construct the Fock space bundle $\mathpzc{F}_X\to \Gamma_X$ as in Example \ref{freestep2example} or equivalently as in Proposition \ref{ismomomafonaofjan}. Note that for a contact manifold, $\mathcal{G}_{Z,\mathcal{F}}|_{\Gamma_X}=\Xi^*_X=p_\Gamma^*H$ and the Levi form induces a symplectic form on $p_\Gamma^*H$.  This method produces a Hilbert space bundle canonically isomorphic to $\bigoplus_{k=0}^\infty p_\Gamma^*H^{\otimes^{* \rm sym}_\C n}$ -- the bosonic Fock space bundle constructed from symmetric tensor powers of $p_\Gamma^*H^*$ with the complex structure induced from the metric and the symplectic form on $p_\Gamma^*H$. By construction, there is an induced isomorphism $I_X\cong C_0(\Gamma_X,\mathbb{K}(\mathpzc{F}_X))$ and the corresponding metaplectic correction bundle $\mathfrak{M}(\mathpzc{F})\to \Gamma_X$ is trivial (cf. Proposition \ref{metaplecticlinecorrection}). The Fock space bundle construction on a contact manifold was heavily used in Baum-van Erp's proof of the index theorem for contact manifolds \cite{baumvanerp}.

There are further conditions one can impose on contact structures. We say that a contact structure is co-orientable if $TX/H$ is orientable (i.e. trivializable), or equivalently that the structure group of $H$ can be reduced to $U(n)$, or equivalently there exists a one-form $\theta\in C^\infty(X,T^*X)$ with $H=\ker\theta$. In this case, the Cartan formula shows that the Levi form can be identified with $\rd \theta$ under the trivialization $\theta:TX/H\xrightarrow{\sim} X\times \R$. A contact structure is co-orientable if and only if $\Gamma_X\to X$ trivializes as an $\R^\times$-bundle -- hence the constructions in the preceeding paragraph descends to $X$ for co-orientable contact structures. In particular, a Riemannian metric and the Levi form gives rise to a complex structure on $H$ for co-orientable contact structures. 
\end{example}

\begin{example}[Boundaries of pseudo-convex domains]
\label{boundaryofpseudoconvexz}
Assume that $X=\partial \Omega$ is the boundary of a smooth domain $\Omega$ in a complex manifold $Z$. Write $\Omega=\{z\in Z: \phi(x)<0\}$ for some $\phi\in C^\infty(Z,\R)$ with $\rd\phi|_{X}$ nowehere vanishing. We can identify $TX$ with the subbundle of $TZ|_X$ defined as the annihilator of $\rd\phi|_X$, i.e. as the real tangential vectors to $X$ in $Z$. Let us assume that the covector field $\rd^c\phi:=i(\partial-\bar{\partial})\phi$ is nowhere vanishing and therefore defines a sub-bundle $H:=\ker(\rd^c\phi|_{TX})$ of $TX$. The sub-bundle $H$ consists of vectors that are real tangential to $X$ in $Z$ but normal in the complex sense. The filtration $0\subseteq H\subseteq TX$, gives rise to a Carnot structure on $X$ that encodes several interesting features studied in complex analysis. The Levi bracket $\mathcal{L}$ for this filtration is under $\rd^c\phi:TX/H\xrightarrow{\sim}X\times \R$ identified with the  restriction of the Levi form $\rd\rd^c\phi$ to $X$. Recall the terminology that $\phi$ is (strictly) pluri-subharmonic in $x\in Z$ if the Levi form $\rd\rd^c\phi|_{T_xX}$ is (strictly) positive.

From the discussion in Example \ref{contactexample}, $X$ is a contact manifold if and only if $\pm \phi$ is strictly pluri-subharmonic near $X$. In the case that $\Omega=\{z\in Z: \phi(x)<0\}$ for some $\phi\in C^\infty(Z,\R)$ which is strictly pluri-subharmonic near $X$, one says that $\Omega$ is strictly pseudo-convex. Therefore, the boundary of a strictly pseudo-convex domain is a regular Carnot manifold of type $\mathfrak{h}_n$. In complex analysis, strict pseudo-convexity ensures several good features when solving the $\bar{\partial}$-problem. In particular, these can be used to construct the Szegö projection in the Carnot calculus (see \cite[Chapter III.6]{taylorncom}) of the filtration, for more details see Subsubsection \ref{ex:szegoexampe} below. 

In the case that $\Omega=\{z\in Z: \phi(x)<0\}$ for some $\phi\in C^\infty(Z,\R)$ which is pluri-subharmonic near $M$, but not strictly pluri-subharmonic, one says that $\Omega$ is weakly pseudo-convex. Let us consider a highly concrete example and take $\phi(z_0,z'):=|z_1|^4+|z_2|^2-1$ on $Z=\C^{n+m}$. Then $X=\{(z_1,z_2)\in \C^n\times \C^m: |z_1|^4+|z_2|^2=1\}$ and in this case $\rd\rd^c\phi(z_1,z_2)=12i\sum_{j=1}^n|z_{1,j}|^2\rd z_{1,j}\wedge \rd\bar{z}_{1,j}+i\sum_{j=1}^m\rd z_{2,j}\wedge \rd\bar{z}'_{2,j}$. In particular, for each $x=(z_1,z_2)\in M$ we have an isomorphism of graded Lie algebras
\begin{align*}
\mathsf{gr}(T_xM)\cong &
\mathfrak{h}_{k(x)}\oplus \R^{l_1(x)}[-1]\oplus \R^{l_2(x)}[-2], \\
&\mbox{where}\quad 
\begin{cases}
l_1(x):=2\#\{j: z_{1,j}=0\},\\
{}\\
l_2(x):=
{\begin{cases}
1, \; &\mbox{if $n=1$ and $z_1=0$},\\
0, \; &\mbox{otherwise},
\end{cases}}\\
{}\\
k(x):=n+m-1-l_1(z)/2.
\end{cases}
\end{align*}
This gives an example (one of the few in this work) of a non-regular Carnot structure on $S^{2n+2m-1}\cong X$.
\end{example}

We shall now explore two further special cases of step $2$-filtrations of manifold that in the literature goes under the names polycontact structure and pluricontact structure.

\begin{example}[Polycontact manifolds]
\label{polycontactexama}
A polycontact structure on a manifold $X$ is a length $2$ filtration 
$$0\subseteq H\subseteq TX,$$
such that the Levi bracket $\mathcal{L}:H\wedge H\to TX/H$, for any $m\in X$ and $\nu\in H^\perp_m\cong (T_mX/H_m)^*$, satisfies that $\nu\circ\mathcal{L}_m$ is a non-degenerate $2$-form on $H_m$. Further details, references and examples can be found in \cite{vanerpszego}. A subbundle $H\subseteq TX$ defines a polycontact structure if and only if $H^\perp\setminus X\subseteq T^*X$ is a symplectic subspace if and only if for any $\theta\in C^\infty(X,H^\perp)$, $\rd \theta|_H$ is non-degenerate in all points $m$ with $\theta(m)\neq 0$. The regularity of a polycontact structure is characterized by Proposition \ref{step2regularitychar}. 

We note that $H$ defines a polycontact structure if and only if for each $x\in X$, $\mathfrak{g}_x:=\mathsf{gr}(T_xX)$ is a step $2$ nilpotent Lie algebra satisfying $\mathfrak{z}_x=C(\mathfrak{g}_x)= T_xH/H_x$ and its set of flat orbits takes the form $\Gamma_x=\mathfrak{z}_x^*\setminus \{0\}=H_x^\perp\setminus \{0\}$. 

An example of a polycontact structure arises as follows. Assume that $X$ is a compact complex manifold and $H\subseteq TX$ is a complex subbundle of complex codimension $1$ such that the the Levi bracket $H\times H\to TX/H$ is non-degenerate. The Levi bracket is complex linear, so it is readily seen that the associated Carnot manifold is polycontact. The osculating Lie group in each fibre is a complex Heisenberg group, as studied in Example \ref{complexheisenbergexample}, so the associated Carnot manifold is regular of type $\mathfrak{h}_n\otimes \C$ by Theorem \ref{morimorohospintise9} where $2n=\dim_\C(X)-1$. The structure group for this Carnot manifold can be reduced to the group 
$$\ghani:=\{A\in GL(n,\C): A\omega A^T\in \C\omega\},$$
where $n$ denotes the rank of $H$ and $\omega$ is the standard complex symplectic form. Write $\det_\omega:\ghani\to \C^\times$ for the homomorphism determined by $A\omega A^T=\det_\omega(A)\omega$. Up to isomorphism, we can identify $TX/H$ with the complex line bundle $L_X:=P_\ghani\times_{\det_\omega}\C\to X$. Since $H$ is a complex symplectic bundle, we can reduce its structure group to $Sp(2n,\C)$, and $L_X$ is trivial. 
\end{example}
    
\begin{example}[Pluricontact manifolds] 
\label{pluricontactexamokkoa}
A pluricontact structure on a manifold $X$ is a length $2$ filtration 
$$0\subseteq H\subseteq TX,$$
such that the Levi bracket $\mathcal{L}:H\wedge H\to TX/H$ satisfies that 
$$\Gamma_H:=\{\nu\in H^\perp: \nu\circ \mathcal{L}_{p(\nu)} \text{ is nondegenerate on $H_{p(\nu)}$} \}\subseteq H^\perp,$$
has nonempty intersection with each fibre of $H^\perp$ over $X$. Here $p:H^\perp\to X$ denotes the projection mapping. We call $\Gamma_H$ the nondegeneracy locus of $H$. We say that $X$ is a pluricontact manifold of rank $m$ and codimension $l$ for $m:=\mathrm{rk}(H)/2$ and $l:=\dim(X)-\mathrm{rk}(H)$. We also let $(X,H)$ denote the associated Carnot manifold. We also consider the degeneracy variety of $H$ as the subset of bundle of projective spaces $P(H^\perp)\to X$
\[ V_H = \{ [\nu] \in P(H^\perp) : \nu \circ \mathcal{L}_H \text{ is degenerate } \}. \]

We note that $H$ defines a pluricontact structure if and only if for each $x\in X$, $\mathfrak{g}_x:=\mathsf{gr}(T_xX)$  is a step $2$ nilpotent Lie algebra satisfying  $\mathfrak{z}_x=C(\mathfrak{g}_x)= T_xX/H_x$ and its set of flat orbits is a  non-empty set that takes the form $\Gamma_x=\Gamma_{H,x}:=\{\nu\in H^\perp_x: \nu\circ \mathcal{L}_{x} \text{ is nondegenerate on $H_{x}$} \}$. In particular, the regular pluricontact manifolds has the type being a step $2$ nilpotent Lie algebra such that its centre is homogeneous and admitting flat orbits. 

For a regular pluricontact manifold, one can construct a global bundle of flat representations for the osculating groupoid as in Example \ref{contactexample}. Assume that $X$ is a regular pluricontact manifold of type $\mathfrak{g}=\mathfrak{g}_{-1}\oplus \mathfrak{g}_{-2}$. Let $\Gamma\subseteq \mathfrak{g}_{-2}^*$ denote the set of flat orbits of $\mathfrak{g}$ and $\Xi\to \Gamma$ the vector bundle whose total space is the union of all flat orbits. We note that $\Xi_X\cong p_\Gamma^*H^*$ as vector bundles on $\Gamma_X$. The Kirillov form induces a symplectic form $\omega$ on $p_\Gamma^*H$, and indeed $\omega_\xi=\xi\circ \mathcal{L}$ for the Levi form $\mathcal{L}$. In particular, a choice of metric induces an adapted complex structure on $(p_\Gamma^*H,\omega)$. We can now construct the Fock space bundle $\mathpzc{F}_X\to \Gamma_X$ as in Example \ref{freestep2example} or equivalently as in Proposition \ref{ismomomafonaofjan}. Note that for a pluricontact manifold, $\mathcal{G}_{Z,\mathcal{F}}|_{\Gamma_X}=\Xi^*_X=p_\Gamma^*H$. As in Example \ref{contactexample}, $\mathpzc{F}_X\cong \bigoplus_{k=0}^\infty p_\Gamma^*H^{\otimes^{* \rm sym}_\C n}$ and there is an induced isomorphism $I_X\cong C_0(\Gamma_X,\mathbb{K}(\mathpzc{F}_X))$. The metaplectic correction bundle $\mathfrak{M}(\mathpzc{F})\to \Gamma_X$ is again trivial (cf. Proposition \ref{metaplecticlinecorrection}).

Let us comment on some special cases. If $(X,H)$ has codimension $1$, i.e. $H$ has codimension $1$, then $(X, H)$ is a contact manifold if and only if it is a pluri-contact manifold. In this case $\Gamma_H =H^\perp\setminus X$ is an $\R^\times$-bundle over $X$. More generally, any polycontact manifold is a pluricontact manifold with $\Gamma_H =H^\perp\setminus X$. If $(X_1, H_1)$, $(X_2, H_2)$ are two pluricontact manifolds then $(X_1 \times X_2, H_1 \oplus H_2)$ is also a pluricontact manifold with $\Gamma_{H_1 \oplus H_2} = \Gamma_{H_1} \times \Gamma_{H_2}$.
\end{example}

\begin{example}[Equiregular differential systems] 
\label{ex:regudlaladoad}
Consider a smooth manifold $X$ with a bracket generating distribution, i.e. a locally finitely generated $C^\infty(X)$-submodule $\mathpzc{E}^{-1}\subseteq C^\infty(X,TX)$ generating $C^\infty(X,TX)$ as a Lie algebra in a finite number of steps, say $r$. We arrive at a singular Carnot structure (in the sense of Remark \ref{singularcarnotremark}):
\[ \{0\} = \mathpzc{E}^0 \subset \mathpzc{E}^{-1} \subset \mathpzc{E}^{-2} \subset \ldots \subset \mathpzc{E}^{-r}= C^\infty(X,TX), \]
where $\mathpzc{E}^{-j}$ is defined inductively by
\[ \mathpzc{E}^{-j} = \mathpzc{E}^{-j+1} + [\mathpzc{E}^{-j+1}, \mathpzc{E}^{-1}]. \]
Bracket generating distributions plays an important role in Hörmander's sum-of-squares theorem \cite{horsumsquares} in that a sum of squares of a collection of vector fields from $\mathpzc{E}^{-1}$, where the collection spans $\mathpzc{E}^{-1}$ in all points and is locally finite, is hypoelliptic.

If each $\mathpzc{E}^{-j}$ is represented by a sub-bundle $T^{-j}X\subseteq TX$, we have a Carnot structure. We call such a Carnot structure an equiregular differential system, as in \cite{Morimoto,MorimotoDiff}. Carnot structures of this type are called equiregular Carnot-Caratheodory manifolds in \cite{gromovsub} whereas the Carnot-Caratheodory manifolds in \cite{gromovsub} will in our terminology be a singular Carnot structure (as in Remark \ref{singularcarnotremark}) generated by a locally finitely generated submodule $\mathpzc{E}^{-1}\subseteq C^\infty(X,TX)$. Let us comment on some special cases.
\begin{itemize}
\item[a)] Contact and pluricontact manifolds are examples of step 2 equiregular differential systems.
\item[b)] Let $X$ be a smooth 4-manifold. An \textit{Engel structure} on $X$ \cite{EngelStructure} is a smooth rank 2 subbundle $T^{-1} X \subset TX$ such that Lie brackets of sections of $T^{-1} X$ generate a rank 3 subbundle $T^{-2} X$ that together with triple brackets of section of $T^{-1} X$ generate $TX$. Hence, $X$ is a Carnot manifold
\[ 0  = T^0 X \subset T^{-1} X \subset T^{-2} X \subset T^{-3} X = TX. \]
Moreover, $X$ is regular of type $\mathfrak{g}$, where $\mathfrak{g}$ is the Lie algebra from Example \ref{engelalgebra}. According to a result of Vogel \cite{Vogel}, every closed parallelizable smooth 4-manifold admits an Engel structure. While Example \ref{engelalgebra} shows that Engel structures are not $\pmb{F}$-regular, the top coarse strata provides an analogous structure to which the results of this work extend.
    \end{itemize}
\end{example}

\begin{example}[Higher order contact manifolds] 
For a fibre bundle $p : M \rightarrow N$, we consider the manifold $X:=J^k(M,N)$ of $k$-jets of cross-sections of $p$. For local coordinates $(x^1, \ldots, x^n)$ on $N$ and local coordinates $(x^1, \ldots, x^n, y_1, \ldots, y^m)$ on $M$, $(x^1, \ldots, x^n, \ldots, p_\alpha^i, \ldots)$ gives a local coordinate system of $J^k(M, N)$ for $p_\alpha^i = \frac{\partial^{|\alpha|} y^i}{\partial x^\alpha}$ with $\alpha = (\alpha_1, \ldots, \alpha_n)$, $|\alpha| \leq k$. We define the subbundles  $D^{-j}\subseteq TJ^k(M,N)$, for $j \geq 1$, by setting
\[ D^{-j}_x:=\left\{Y\in T_xJ^k(M;N): 
\begin{matrix}
\rd p_\alpha^i (Y)=&\!\!\!\!\!\!\!\!\!\!\!\!\sum_{l=1}^n p^i_{(\alpha_1, \ldots, \alpha_l+1, \ldots, \alpha_n)} \rd x^l(Y)\quad\mbox{for} \\ 
&\qquad\qquad\qquad i = 1,\ldots,n, \ |\alpha| \leq k - j\end{matrix}\right\}, \]
for $x\in J^k(M,N)$. By \cite[Example 3), page 10]{Morimoto}, $D^{-j}$ are well-defined subbundles such that $D^{-j} = D^{-j+1} + [D^{-j+1}, D^{-1}]$ and $D^j = TJ^k(M,N)$ for $j \leq -k - 1$. We thus have an equiregular differential system on $X=J^k(M,N)$ generated by $D^{-1}$ that gives a Carnot structure of length $k+1$. As remarked in \cite{Morimoto}, if $\dim M=\dim N + 1$ then $X$ is a contact manifold having $D^{-1}$ as its contact structure. By the results of \cite[Chapter 5]{MorimotoTrans}, the Carnot structure of $X$ is regular and satisfies the assumptions of the Generalized Darboux theorem (see Theorem \ref{gendarboux} above). Higher order contact structures make their appearance in Vinogradov's theory for diffietes \cite{vinocoh} describing spaces of formal solutions of partial differential equations in terms of methods of algebraic geometry.
\end{example}

\begin{example}[Parabolic geometries] 
\label{parapara}
Parabolic geometry is an extension of Cartan's work in differential geometry. Such a geometry is locally modelled on spaces of the form $G/P$ where $G$ is a semisimple Lie group and $P \subset G$ is a parabolic subgroup. The differential geometry is encoded in a Cartan connection with an associated curvature that detects local deviations from the model $G/P$. A standard reference for parabolic geometry is \cite{cap}. Our interest in parabolic geometry comes from that on such geometries there are associated graded differential complexes, curved Bernstein-Gelfand-Gelfand (BGG) sequences \cite{morecap}, that are naturally studied by means of Carnot calculus \cite{Dave_Haller1}.

For the definition of a parabolic geometry, we need some further terminology. For more details, see \cite{cap, morecap,Dave_Haller1}. For $k\in \mathbb{N}$, a $|k|$-grading on a Lie algebra $\mathfrak{g}$ is a $\Z$-grading supported in $\{-k,\ldots, k\}$, i.e. a decomposition $\mathfrak{g} = \mathfrak{g}_{-k} \oplus \ldots \oplus \mathfrak{g}_k$, such that
    \begin{enumerate}
        \item $[\mathfrak{g}_i, \mathfrak{g}_j] \subset \mathfrak{g}_{i+j}$, for $i,j=-k,\ldots, k$,
        \item $\mathfrak{g}_{-} := \mathfrak{g}_{-k} \oplus \ldots \oplus \mathfrak{g}_{-1}$ is generated as a Lie algebra by $\mathfrak{g}_{-1}$,
        \item $\mathfrak{g}_{-k} \neq \{0\}$ and $\mathfrak{g}_{k} \neq \{ 0 \}$.
    \end{enumerate}
The filtration associated with the grading is defined as $\mathfrak{g}^i = \mathfrak{g}_{i} \oplus \ldots \oplus \mathfrak{g}_k$. For a Lie group $G$ with Lie algebra $\mathfrak{g}$ and a parabolic subgroup $P$, we tacitly assume that the parabolic subgroup $P$ is related to the $|k|$-grading by $\mathfrak{g}^0 = \mathfrak{p}$. We use the notation $\mathfrak{p}_+=\mathfrak{g}^1  $. This notation is motivated by the discussion in \cite[Item 2.2]{morecap}, eminating in that for complex semisimple Lie groups a choice of $|k|$-grading is equivalent to a choice of parabolic subgroup $P\subseteq G$ in a semisimple Lie group, indeed there is a Cartan algebra and a set of positive roots inducing a canonical $|k|$-grading of $\mathfrak{g}$ with $\mathfrak{g}^0 = \mathfrak{p}$ and $\mathfrak{g}^1 = \mathfrak{p}_+$ is the sum of positive weight spaces (see \cite{morecap}). We note that $\mathfrak{g}_-\cong \mathfrak{g}/\mathfrak{p}$ as graded Lie algebras. Moreover, the adjoint action defines a homomorphism $P\to \Aut_{\rm gr}(\mathfrak{g}_-)$ whose kernel contains $\mathrm{exp}(\mathfrak{p}_+)$. We set $G_0:=P/\mathrm{exp}(\mathfrak{p}_+)$. 
    
A parabolic geometry of type $(G, P)$ is a pair $(X,\omega)$ where $X$ is a smooth manifold and $\omega \in \Sigma^1(V, \mathfrak{g})$ is a Cartan connection on a principal $P$-bundle $V\to X$. Recall that a Cartan connection is an element $\omega \in \Sigma^1(V, \mathfrak{g})$ such that
\begin{enumerate}
\item $Ad(h)(R_h^* \omega) = \omega, \ h \in P$;
\item $\omega(X_\xi) = \xi, \ \xi \in \frak{p}$;
\item $\omega_p : T_p V  \rightarrow \frak{g}$ is a linear isomorphism for $p \in V$.
\end{enumerate}
    
A parabolic geometry on $X$ induces a filtration on $TX$ as follows. We can define the $G_0$-principal bundle $P_{X,0}:=V/\mathrm{exp}(\mathfrak{p}_+)\to X$ and define the locally trivial bundle of graded Lie algebras $P_{X,0}\times_{G_0}\mathfrak{g}_-\to X$. The axioms of a parabolic geometry ensures that the Cartan connection $\omega$ induces a canonical isomorphism of vector bundle $TX \simeq P_{X,0}\times_{G_0}\mathfrak{g}_-$, and we can define $T^jX\subseteq TX$ as the image of $P_{X,0}\times_{G_0}\mathfrak{g}^j_-$, where $\mathfrak{g}^j_-:=\mathfrak{g}^j/\mathfrak{p}\subseteq \mathfrak{g}_-$. If this filtration defines a Carnot structure on $X$ such that the isomorphism $TX \simeq P_{X,0}\times_{G_0}\mathfrak{g}_-$ induces an isomorphism of Lie algebroids $\mathsf{gr}(TX)\cong P_{X,0}\times_{G_0}\mathfrak{g}_-$, we say that the parabolic geometry $(X,\omega)$ is regular. It follows from the construction that a regular parabolic geometry of type $(G,P)$ induces a regular Carnot structure of type $\mathfrak{g}_-$, indeed its Carnot frame bundle is given by $P_{X,0}\times_{G_0}\Aut_{\rm gr}(\mathfrak{g}_-)$. A regular parabolic geometry on $X$ induces an $\pmb{F}$-regular Carnot structure if and only if $\mathfrak{g}_-$ admits flat orbits and the standard surjectivity result holds.

To give an example, consider the flag manifold $X:=SL(n,\R)/B$ for $SL(n,\R)$ obtained as the quotient by its standard Borel group. The space $X$ is best understood from the Iwasawa decomposition of $SL(n,\R)$, e.g. one can write $X=SO(n)/M$ for the finite group of diagonal isometries. The space $X$ is a parabolic geometry of type $(SL(n,\R),B)$. The type of the associated regular Carnot structure on $X:=SL(n,\R)/B$ is that of upper triangular, real, unipotent $n\times n$-matrices considered in Example \ref{uppertriangularexample}.

The definition of a regular parabolic geometry was formulated slightly different in \cite{morecap}, but by \cite[Corollary 3.1.8]{cap} the notions are equivalent. There is even an equivalence of categories between regular Carnot manifolds of type $\mathfrak{g}_-$ and normal regular parabolic geometries, see \cite{cap}. 

Let us comment on some special cases.
\begin{itemize}
\item[a)] \textit{Generic rank two distributions in dimension five}, \cite{Cartan1}, \cite{Cartan2}, \cite{Cartan3}, \cite{Cartan4}. Let $X$ be a smooth 5-manifold with a smooth rank $2$ subbundle $T^{-1} X \subset TX$ of Cartan type such that Lie brackets of sections of $T^{-1} X$ generate a rank $3$ distribution $T^{-2} X$ and triple brackets of section of $T^{-1} X$ generate $TX$. The filtration
    \[ \{ 0 \} = T^0 X \subset T^1 X \subset T^2 X \subset T^3 X = TX, \]
defines a Carnot structure on $X$. Moreover, $X$ is a regular normal parabolic geometry of type $(G_2, P)$, where $P$ is a maximal parabolic subgroup corresponding to the shorter simple root. Topological obstructions in the orientable case were given in \cite[Theorem 1]{GenericDave}. 
\item[b)] \textit{Generic rank three distribution in dimension six} Let $X$ be a smooth 6-manifold with a smooth rank 3 subbundle $T^{-1} X \subset TX$ such that $T^{-2} X = T^{-1} X + [T^{-1} X, T^{-1} X]$ is equal to $TX$. In other words, every point $x \in X$ has a neighborhood $U$ on which there exist vector fields $X_1, X_2, X_3$ that are sections of $T^{-1} X$ over $U$, are everywhere linearly independent on $U$, and have the property that the six vector fields
    \[ X_1 , X_2 , X_3 , [X_2, X_3] , [X_3, X_1] , [X_1, X_2] \]
    are everywhere linearly indepdendent on $U$, see \cite{Generic6}. In this case, $X$ is a regular normal parabolic geometry of type $(SO(4,3), H)$, where $H$ is a stabilizer of a null three-plane in $\mathbb{R}^{4,3}$.
\item[c)] In \cite{cap} the reader can find a plethora of other parabolic geometries such as \textit{conformal Riemannian geometries}, \textit{projective} and \textit{almost quaternionic geometries}, and a \textit{real version of CR-structures}, see \cite{cap}. For more detials and differential operators in the conformal setting, see \cite{juhligen}.
\end{itemize}
\end{example}

\section{The parabolic tangent groupoid}
\label{subsec:parabolictanget}

The parabolic tangent groupoid of a Carnot manifold was constructed in \cite{choiponge,vanErp_Yunckentangent}, see also \cite{sadeghhigson,pongegroupoid}. As mentioned in Example \ref{ex:tangentgroupoidladladld}, it is a variation of Connes' tangent groupoid that incorporates the anisotropy of the osculating Lie groupoid. The construction of the parabolic tangent groupoid was in \cite{mohsen1} further studied as deformation groupoids, cf Example \ref{ex:defomfodmomedm}. We will briefly recall its construction in the next theorem from which we afterwards draw some consequences, the details are referred to \cite{vanErp_Yunckentangent}. 

For a Carnot manifold $X$, there is a filtered isomorphism $\mathfrak{t}_HX\cong TX$. The constructions in this subsections are independent of the choice of filtered isomorphism $\mathfrak{t}_HX\cong TX$. Note that $\mathfrak{t}_HX=T_HX$ as fibre bundles. We will for notational brevity identify $T_HX= TX$ as fibre bundles in this subsection and consider the dilation action on $T_HX$ as an $\R_+$-action on $TX$.

\begin{theorem}[Theorem 2 in \cite{vanErp_Yunckentangent}]
Let $X$ be a Carnot manifold with osculating Lie groupoid $T_HX\to X$. The parabolic tangent groupoid 
$$\mathbb{T}_HX:=T_HX\times \{0\}\dot{\cup} X\times X\times (0,\infty)\rightrightarrows X\times [0,\infty),$$
is a Lie groupoid in the canonical smooth structure defined from declaring the map
$$\psi:TX\times [0,\infty)\to \mathbb{T}_HX, \quad 
\psi(x,v,t):=
\begin{cases}
(\mathrm{exp}^\nabla(\delta_t(v)),x,t), \; &t>0,\\
(x,v,0),\; &t=0,
\end{cases}$$
to be smooth. Here $\nabla$ is an $\R^\times$-equivariant connection on $TX$ and $\mathrm{exp}^\nabla:TX\to X$ the associated exponential.  
\end{theorem}

We note that the Lie algebroid of $\mathbb{T}_HX$ can as a vector bundle over $X\times [0,\infty)$ be identified with $TX\times [0,\infty)$. Based on the ideas of Section \ref{subsec:nistorctsubsn} we construct a double deformation of a Carnot manifold as follows.

\begin{definition}
Let $X$ be a Carnot manifold. The Lie groupoid
$$\mathbb{A}_HX:=(TX\times [0,\infty)\times \{0\})\dot{\cup}(\mathbb{T}_HX\times (0,\infty))\rightrightarrows X\times [0,\infty)\times [0,\infty),$$
constructed as the adiabatic deformation of $\mathbb{T}_HX$, will be called the adiabatic parabolic tangent groupoid of $X$.
\end{definition}

We note that the adiabatic parabolic tangent groupoid $\mathbb{A}_HX$ satisfies that 
\begin{align*}
\mathbb{A}_H X|_{[0,\infty)\times \{s\}} &=\mathbb{T}X , \ \text{if  $s > 0$}, \\
\mathbb{A}_H X|_{\{t\}\times [0,\infty)} &=\mathbb{T}_HX , \ \text{if  $t > 0$},\\
\mathbb{A}_H X|_{\{0\}\times [0,\infty)} &=(T_HX)_{\rm adb},
\end{align*}
where $(T_HX)_{\rm adb}$ denotes Nistor's adiabatic deformation as in Section \ref{subsec:nistorctsubsn}. The adiabatic parabolic groupoid fibers over $(t,s)\in [0,\infty)\times [0,\infty)$ as follows:
\begin{align*}
( \mathbb{A}_H X)_{t,s}  &= X \times X, \ \text{if $s ,t > 0$}, \\
( \mathbb{A}_H X)_{0,s}  &= T_HX, \ \text{if $s > 0$}, \\
( \mathbb{A}_H X)_{t,0}  &= T X, \ \text{if $t\geq 0$}.
\end{align*}
In the next section on the Carnot calculus, we shall also make use of the zoom action. There, as well as in \cite{vanErp_Yuncken}, the zoom action is mainly used on $\mathbb{T}_HX$ but we define it also on $\mathbb{A}_HX$ for the sake of completeness.

\begin{definition}
The zoom action of $\R_+$ on $\mathbb{T}_HX$ is defined for $\lambda>0$ by 
$$\begin{cases} 
\alpha_\lambda(x,g,0):=(x,\delta_\lambda(g),0), \; & (x,g)\in T_HX,\\
\alpha_\lambda(x,y,t):=(x,y,\lambda^{-1}t), \; & (x,y,t)\in X\times X\times (0,\infty).
\end{cases}$$
The zoom action of $\R_+$ on $\mathbb{A}_HX$ is defined for $\lambda>0$ by 
$$\begin{cases} 
\alpha_\lambda(x,g,0,s):=(x,\delta_\lambda(g),0,\lambda^{-1}s), \; & (x,g,0,s)\in T_HX\times \{0\}\times(0,\infty),\\
\alpha_\lambda(x,g,t,0):=(x,\delta_\lambda(g),\lambda^{-1}t,0), \; & (x,g,t,0)\in T_HX\times [0,\infty)\times \{0\},\\
\alpha_\lambda(x,y,t,s):=(x,y,\lambda^{-1}t,\lambda^{-1}s), \; & (x,y,ts)\in X\times X\times (0,\infty)\times (0,\infty).
\end{cases}$$
\end{definition}

The tangent groupoid induces a short exact sequence
$$0\to C_0(0,\infty)\otimes \mathbb{K}(L^2(X))\to C^*(\mathbb{T}_HX)\xrightarrow{\mathrm{ev}_0} C^*(T_HX)\to 0,$$
using that $C^*(X\times X)\cong \mathbb{K}(L^2(X))$. We let $\ind_{T_HX}\in KK_0(C^*(T_HX),\C)$ denote the associated $KK$-class. Similarly, we define $\ind_{TX}\in KK_0(C_0(T^*X),\C)$ denote the class associated with Connes' tangent groupoid. The following theorem follows from a diagram chase with the adiabatic parabolic tangent groupoid $\mathbb{A}_HX$. We write $\psi_X\in KK_0(C_0(T^*X),C^*(T_HX)))$ for the $KK$-morphism defined from Nistor's adiabatic deformation (if the Carnot manifold is regular, $\psi_X$ is induced from $\psi\in KK^{\Aut(\mathsf{G})}_0(C_0(\mathfrak{g}^*),C^*(\mathsf{G}))$ in Definition \ref{autgododlelendn}).

\begin{theorem}
\label{outertringaleldldle}
Let $X$ be a Carnot manifold. Then the following diagram commutes
\[
\begin{tikzcd}
C_0(T^*X) \arrow[ddr,"\ind_{TX}"]\arrow[rr, "\psi_{X}"] && C^*(T_HX)\arrow[ddl,"\ind_{T_HX}"]\\
&&\\
&\C&
\end{tikzcd}
\]
In particular, we have that 
$$\ind_{T_HX}(x)=\int_{T^*X}\ch(\psi_X^{-1}[x])\wedge \mathrm{Td}(T^*X), \quad x\in K_*(C^*(T_HX)).$$
\end{theorem}

In the case of regular Carnot manifolds, we can describe $\ind_{T_HX}$ even more precisely on the ideal of flat orbits using Theorem \ref{maincomputationforisg}. First we need a spin$^c$-structure on the space of flat orbits.

\begin{proposition}
\label{spincongamma}
Let $X$ be an $\pmb{F}$-regular Carnot manifold of type $\mathfrak{g}$. Write $\Xi_X\to \Gamma_X$ for the total space of flat orbits in $\mathfrak{t}_HX^*$ as a bundle over the space of flat orbits $\Gamma_X$ and fix a metrix $g_\Xi$ on $\Xi_X$. Then the following holds:
\begin{itemize}
\item Associated with $g_\Xi$ there is a complex structure on the vector bundle $\Xi_X\to \Gamma_X$.
\item The open inclusion $\Xi_X\subseteq T^*X$ and the almost complex structure on $T^*X$ induces an almost complex structure on $\Xi_X$ as a manifold.
\item The manifold $\Gamma_X$ carries a unique spin$^c$-structure making the projection map $\Xi_X\to \Gamma_X$ spin$^c$-oriented.
\end{itemize}
\end{proposition}

The result follows from the 2/3-property for spin$^c$-vector bundles. We remark that since $\Xi$ has even dimension and $\Xi_X\to \Gamma_X$ has even rank, the spin$^c$-manifold $\Gamma_X$ has even dimension.

\begin{theorem}
\label{somecomewithe}
Let $X$ be a regular Carnot manifold of type $\mathfrak{g}$ such that $\mathsf{G}$ admits flat orbits and fix a metrix $g_\Xi$ on $\Xi_X$. Assume that $\mathcal{H}\to \Gamma_X$ is a bundle of Hilbert spaces and $\pi_{\musFlat}:I_{X}\to C_0(\Gamma_X,\mathbb{K}(\mathcal{H}))$ a $C_0(\Gamma_X)$-linear $*$-isomorphism. Consider the $KK$-morphism
$$\chi_X:=[\mathcal{H}\otimes \mathfrak{M}(\mathcal{H})]\otimes_{C_0(\Gamma_X)} [\Gamma_X]\in KK(I_{X},\C),$$
where $[\Gamma_X]\in KK_0(C_0(\Gamma_X),\C)$ is the fundamental class of $\Gamma_X$ associated with the spin$^c$-structure induced from $\Xi_X$ and $\mathfrak{M}(\mathcal{H})$ is the metaplectic correction bundle (see Proposition \ref{metaplecticlinecorrection}). 

The morphism $\chi_X$ fits into a commuting diagram in $KK$:
\[
\begin{tikzcd}
C_0(T^*X) \arrow[ddddrrr, "\ind_{TX}"] \arrow[rrrrrr, "\psi"] &&&&& &C^*(T_HX)\arrow[ddddlll,"\ind_{T_HX}"] \\
&&C_0(\Xi_X)\arrow[rr, "\psi_I"]\arrow[dddr]\arrow[ull,"\subseteq"]&&I_X\arrow[dddl,"\chi_X"]\arrow[urr,"\subseteq"]&& \\
&&&\\
&&&\\
&&&\C&
\end{tikzcd}
\]
\end{theorem}

\begin{proof}
The top square of the diagram commutes by Theorem \ref{nistorconnethomfortwistgroup} and the outer triangle commutes by Theorem \ref{outertringaleldldle}. Let $j:I_X\to C^*(T_HX)$ and $j_\Xi:C_0(\Xi_X)\to C_0(T^*X)$ denote the inclusions. We need to prove that $j\otimes_{I_X}\ind_{T_HX}=\chi_X$. Functoriality of wrong way maps implies that 
$$j_\Xi\otimes_{C_0(\Xi)}\ind_{TX}=[\Xi_X],$$
where $[\Xi_X]\in KK_0(C_0(\Xi),\C)$ is the fundamental class of $\Xi_X$. Theorem \ref{maincomputationforisg} and the functoriality of wrong way maps shows that 
$$\psi_I\otimes_{I_X}\chi_X=[\Xi_X].$$ 
Using that $\psi$ and $\psi_I$ are isomorphisms (see Theorem \ref{nistorconnethomfortwistgroup}) the equality $j\otimes_{I_X}\ind_{T_HX}=\chi_X$ follows.
\end{proof}

\part{$H$-elliptic operators on Carnot manifolds}
\label{part:pseudod}

\begin{center}
{\bf Introduction to part}
\end{center}

The operators we consider in this monograph are not elliptic in the ordinary sense, but do satisfy an ellipticity condition in the pseudodifferential calculus of \cite{vanErp_Yuncken} which is better suited for Carnot manifolds. In the literature there are several approaches to pseudodifferential calculus on ``nilpotent geometries'', for instance \cite{bealsgreiner,christgelleretal,cum89,dynin75,dynin76,melroseeptein,eskeewertthesis,eskeewertpaper,fischruzh,melinoldpreprint,pongemonograph,taylorncom}.
These apporaches all originated in the ideas of Folland-Stein \cite{follandstein} who, using Darboux' theorem, modeled the boundary of a strictly pseudoconvex domain (cf. Example \ref{boundaryofpseudoconvexz}) locally as a Heisenberg group and treated the Szegö projection thereon using ordinary pseudodifferential methods upon defining a principal symbol as a convolution operator on the Heisenberg group in each fibre. We take our stance in recent years' progress by van Erp-Yuncken \cite{vanErp_Yuncken} and Dave-Haller \cite{Dave_Haller1, Dave_Haller2} that provides a coherent picture of the Carnot calculus on general Carnot manifolds. Let us give an overview of the sections in this part:

\begin{itemize}
\item In Section \ref{pseudcalcalc}, we recall the Carnot calculus of Melin \cite{melinoldpreprint}, following modern developments by van Erp-Yuncken \cite{vanErp_Yuncken}. We also provide a range of different examples and recall the graded calculus of Dave-Haller \cite{Dave_Haller1}. 
\item In Section \ref{subsec:hellleleldlda} we set up the represented symbol calculus following Dave-Haller \cite{Dave_Haller1, Dave_Haller2} and revisit the examples. The Rockland condition from \cite{Dave_Haller1}, and standard techniques, readily allow us to extends standard Fredholm theory for elliptic operators to $H$-elliptic operators on a general Carnot manifold. 
\item In Section \ref{secondnaaoaoac} we refine analytic properties of the symbol calculus of $H$-elliptic operators. For $\pmb{F}$-regular manifolds, we study analytic properties of the symbol of an $H$-elliptic operators as a multiplier of the $C_0(\Gamma_X)$-Hilbert $C^*$-module of flat orbit representations, also characterizing $H$-ellipticity in terms of invertibility as a multiplier on the flat orbit representations. 
\item In Section \ref{sec:ktheoomon} we utilize the technical tools of Section \ref{secondnaaoaoac} to study the $K$-theoretical invariants in $K_*(C^*(T_HX))$ and $K^*(\Gamma_X)$. We implement the two step process of lifting $[\sigma_H(D)]\in K_*(C^*(T_HX))$ to an elliptic complex on $\Gamma_X$: first we need to localize the support of $[\sigma_H(D)]\in K_*(C^*(T_HX))$ to $\Gamma_X$, i.e. construct a pre-image under $K_*(I_X)\to  K_*(C^*(T_HX))$, and then we need to approximate by a finite rank elliptic complex, i.e. implement the Morita equivalence $K_*(I_X)\cong K^*(\Gamma_X)$. After proving that this is possible in the abstract, we restrict to polycontact manifolds where we carry out these computations in the examples coming from $\Delta_H+\gamma T$-operators and projections of Hermite type.
\end{itemize}

\section{The pseudodifferential calculus of van Erp-Yuncken}
\label{pseudcalcalc}

In this subsection we will review the pseudodifferential calculus of van Erp-Yuncken \cite{vanErp_Yuncken}. We will follow their notation, but for convenience we incorporate vector bundles from the start. Consider a Lie groupoid $\mathcal{G}$ with unit space $\mathcal{G}^{(0)}$ and complex vector bundles $E_1,E_2\to \mathcal{G}^{(0)}$. When considering function spaces on $\mathcal{G}$ we consider them a left/right module of relevant function spaces on $\mathcal{G}^{(0)}$ via the range/source mapping. Following \cite{vanErp_Yuncken}, we define the space of $r$-fibred distributions $\mathcal{E}_r'(\mathcal{G};E_1,E_2)$ consists of continuous left $C^\infty(\mathcal{G}^{(0)})$-linear mappings 
$$u:C^\infty(\mathcal{G},s^*E_1)\to C^\infty(\mathcal{G}^{(0)},E_2).$$
For notational convenience, we write $u(\varphi)(x)=\int_{\mathcal{G}_x} u(\gamma) \varphi(\gamma)$ for $u\in \mathcal{E}_r'(\mathcal{G};E_1,E_2)$ and $\varphi\in C^\infty_c(\mathcal{G},s^*E_1)$.

If the unit space of $\mathcal{G}$ is compact, the Schwarz kernel theorem and the Haar system of $\mathcal{G}$ identifies $\mathcal{E}_r'(\mathcal{G};E_1,E_2)$ with a subspace of $\mathcal{E}'(\mathcal{G},r^*E_2\otimes (s^*E_1)^*\otimes |\Lambda_r|)$ (here $|\Lambda_r|$ denotes the density bundle of $\ker (Dr)$). If the unit space of $\mathcal{G}$ is non-compact,  $\mathcal{E}_r'(\mathcal{G};E_1,E_2)$ can be identified with a space of properly supported elements in $\mathcal{D}'(\mathcal{G},r^*E_2\otimes (s^*E_1)^*\otimes |\Lambda_r|)$. By the same token, we can always  identify the space of properly supported smooth functions $C^\infty_p(\mathcal{G},r^*E_2\otimes (s^*E_1)^*\otimes |\Lambda_r|)$ with its preimage in $\mathcal{E}_r'(\mathcal{G};E_1,E_2)$.

Convolution defines an operation $\mathcal{E}_r'(\mathcal{G};E_1,E_2)\times C^\infty_c(\mathcal{G},s^*E_1)\to C^\infty_c(\mathcal{G},s^*E_2)$ by 
\[u \ast \varphi(\gamma) = \int_{\mathcal{G}_{r(\gamma)}} u(\gamma\eta^{-1}) \varphi(\eta), \quad u\in \mathcal{E}_r'(\mathcal{G};E_1,E_2), \; \varphi\in C^\infty_c(\mathcal{G},s^*E_1).\]
Using duality, convolution defines an operation 
$$\mathcal{E}_r'(\mathcal{G};E_1,E_2)\times \mathcal{E}_r'(\mathcal{G};E_2,E_3)\to \mathcal{E}_r'(\mathcal{G};E_1,E_3),$$
for any three vector bundles $E_1,E_2,E_3\to \mathcal{G}^{(0)}$. Similar ideas produces an adjoint operation 
$$\mathcal{E}_r'(\mathcal{G};E_1,E_2)\to \mathcal{E}_r'(\mathcal{G};E_2,E_1), \quad u\mapsto u^*(\gamma):=u(\gamma^{-1})^*.$$

We now restrict to a closed Carnot manifold $X$ and let $T_HX\to X$ denote the osculating groupoid associated with the filtering (see Definition \ref{osculatingdef}). The pseudodifferential calculus makes heavy use of van Erp-Yuncken's parabolic tangent groupoid $\mathbb{T}_HX\to X\times [0,\infty)$ reviewed above in Section \ref{subsec:parabolictanget}. For two complex vector bundles $E_1,E_2\to X$ on the compact Carnot manifold $X$, we write 
$$\mathcal{E}_r'(\mathbb{T}_HX;E_1,E_2):=\mathcal{E}_r'(\mathbb{T}_HX;E_1\times [0,\infty),E_2\times [0,\infty)).$$ 
The projection mapping $\mathbb{T}_HX\to T_HX$ is proper in a neighborhood of $T_HX\times \{0\}\subseteq \mathbb{T}_HX$ so $C^\infty(X\times [0,\infty))$-linearity ensures that there is a well defined continuous mapping 
$$\mathrm{ev}_0:\mathcal{E}_r'(\mathbb{T}_HX;E_1,E_2)\to \mathcal{E}_r'(T_HX;E_1,E_2),$$
that respects convolution. Moreover, for $t>0$, the projection mapping $\mathbb{T}_HX\to X\times X\times \{t\}$ is proper in a neighborhood of $X\times X\times \{t\}\subseteq \mathbb{T}_HX$ so $C^\infty(X\times [0,\infty))$-linearity ensures that there is a well defined continuous induced mapping 
$$\mathrm{ev}_t:\mathcal{E}_r'(\mathbb{T}_HX;E_1,E_2)\to \mathcal{E}_r'(X\times X;E_1,E_2),$$
that respects convolution. We note that the Schwarz kernel theorem ensures that 
$$\mathcal{E}'_r(X\times X; E_1,E_2)\cong C^\infty(X,E_2)\hat{\otimes}\mathcal{E}'(X,(E_1)^*\otimes |\Lambda_X|),$$ 
is the space of all continuous linear mappings $C^\infty(X,E_1)\to C^\infty(X,E_2)$.

The action of $\mathbb{R}_+$ on $\mathbb{T}_HX$ defines a continuous action on $\mathcal{E}_r'(\mathbb{T}_HX;E_1,E_2)$. Following \cite[Definition 18]{vanErp_Yuncken} we define for $m\in \mathbb{C}$ the space 
$$\pmb{\Psi}^m_H(X;E_1,E_2)\subseteq\mathcal{E}_r'(\mathbb{T}_HX;E_1,E_2),$$ 
to consist of those $\pmb{P}$ for which 
$$\lambda_*\pmb{P}-\lambda^m\pmb{P}\in C^\infty_p(\mathbb{T}_HX,r^*E_2\otimes (s^*E_1)^*\otimes |\Lambda_r|),$$
for all $\lambda>0$. We note that for $\pmb{P}\in \pmb{\Psi}^m_H(X;E_1,E_2)$ and $t,\lambda>0$, we have that 
\begin{equation}
\label{comparofeval}
\mathrm{ev}_{\lambda t}\pmb{P}-\lambda^m \mathrm{ev}_t\pmb{P}\in C^\infty(X\times X,E_2\otimes(E_1)^*\otimes |\Lambda_r|).
\end{equation}
We can conclude the next proposition from homogeneity and using the fact that $C^\infty_p(\mathcal{G},r^*E_2\otimes (s^*E_1)^*\otimes |\Lambda_r|)$ is an ideal in the space of $r$-fibred distribution with respect to convolution.

\begin{proposition}
\label{productsinfatthm}
The convolution product on the space of $r$-fibred distributions restricts to a well defined $C^\infty[0,\infty)$-linear composition product 
$$\pmb{\Psi}^m_H(X;E_1,E_2)\times \pmb{\Psi}^{m'}_H(X;E_2,E_3)\to \pmb{\Psi}^{m+m'}_H(X;E_1,E_3),$$
for any three vector bundles $E_1,E_2,E_3\to X$ and the adjoint restricts to a well defined $C^\infty[0,\infty)$-antilinear adjoint operation 
$$\pmb{\Psi}^m_H(X;E_1,E_2)\to \pmb{\Psi}^{m}_H(X;E_2,E_1).$$
\end{proposition}

We  define $\Psi_H^m(X;E_1,E_2)\subseteq \mathcal{E}'_r(X\times X; E_1,E_2)$ as the image of $\pmb{\Psi}^m_H(X;E_1,E_2)$ under $\mathrm{ev}_1$. For $\pmb{P}\in\pmb{\Psi}^m_H(X;E_1,E_2)$, Equation \eqref{comparofeval} ensures that $\mathrm{ev}_0(\pmb{P})\in \mathcal{E}'_r(T_HX;E_1,E_2)$ is determined modulo $C^\infty_c(T_HX;E_1,E_2)$ by $\mathrm{ev}_1(\pmb{P})$. Based on this observation, we are in a position to define the symbol algebra and the principal symbol mapping.

\begin{definition}
For a closed Carnot manifold $X$, complex vector bundles $E_1,E_2\to X$ and $m\in \mathbb{C}$, define 
$$\Sigma^m_H(X;E_1,E_2):=\{K\in \mathcal{E}'_r(T_HX;E_1,E_2)/C^\infty_c(T_HX;E_1,E_2): \, \lambda_*K=\lambda^m K\}.$$
For $P\in \Psi_H^m(X;E_1,E_2)$, we define $\sigma_m(P)\in \Sigma^m_H(X;E_1,E_2)$ as the class of $\mathrm{ev}_0(\pmb{P})$ for any $\pmb{P}\in \pmb{\Psi}^m_H(X;E_1,E_2)$ with $\mathrm{ev}_1(\pmb{P})=P$. 
\end{definition}

The convolution product on $\pmb{\Psi}^m_H$ from Proposition \ref{productsinfatthm} restricts to a well defined composition product 
$$\Psi^m_H(X;E_1,E_2)\times \Psi^{m'}_H(X;E_2,E_3)\to \Psi^{m+m'}_H(X;E_1,E_3),$$
and induces a symbol product
$$\Sigma^m_H(X;E_1,E_2)\times \Sigma^{m'}_H(X;E_2,E_3)\to \Sigma^{m+m'}_H(X;E_1,E_3),$$
for any three vector bundles $E_1,E_2,E_3\to X$. The adjoint also restricts to a well defined order preserving adjoint operation on $\Psi^m_H$ and $\Sigma^m_H$. The reader should note that the products and adjoints on $\Psi^m_H$ and $\Sigma^m_H$ are defined using the space of distributions $\pmb{\Psi}^m_H$ on the tangent groupoid in order to ensure that $\Psi^m_H$ and $\Sigma^m_H$ fit into a calculus, but in practice the operations are defined completely in terms of the groupoid structure of $X\times X$ and $T_HX$ for $\Psi^m_H$ and $\Sigma^m_H$, respectively. The next proposition is a direct consequence of \cite[Proposition 37]{vanErp_Yuncken}. We introduce the notation 
\begin{align}
\label{tidlemsossom}
\tilde{\Sigma}^m_H&(X;E_1,E_2):= \\
\nonumber
:=&\{P\in \mathcal{E}'_r(T_HX;E_1,E_2):\lambda_*P-\lambda^mP\in C^\infty_c(T_HX,\Hom(E_1,E_2)\otimes |\Lambda|)\}.
\end{align}

\begin{proposition}
\label{embeddingboldpsi}
Multiplication by $t$ defines an embedding 
$$\pmb{\Psi}^{m-1}_H(X;E_1,E_2)\hookrightarrow \pmb{\Psi}^m_H(X; E_1,E_2).$$
Moreover, evaluation at $t=0$ defines an isomorphism 
\begin{align*}
\pmb{\Psi}^m_H(X; &E_1,E_2)/t\pmb{\Psi}^{m-1}_H(X;E_1,E_2)\cong\tilde{\Sigma}^m_H(X;E_1,E_2).
\end{align*}
\end{proposition}

\begin{remark}
There is a substantial body of literature on pseudodifferential calculus on groupoids, see for instance \cite{debordskandpseudo,debordskandext,debordskandblowup,debordlescure,nistweinxu}. The calculus $\Psi^m_H$ of van Erp-Yuncken \cite{vanErp_Yuncken} is not a pseudodifferential calculus on a groupoid in the classical sense. Indeed, the groupoid on which the calculus $\Psi^m_H$ is defined is the pair groupoid $X\times X$ whose classical groupoid calculus would be the classical pseudodifferential calculus on $X$.
\end{remark}

\begin{example}
\label{compaoaoadoadl}
A natural source of examples in the Carnot calculus is differential operators. A Carnot manifold carries an associated filtration on the algebra of differential operators. Let $E$ and $F$ be smooth vector bundles over $X$. Choose a linear connection $\nabla$ on $E$. For $m\in \mathbb{N}$, we define $\mathcal{DO}_H^m(X;E_1,E_2)$ as the linear space spanned by operators of the form $a \nabla_{X_k} \ldots \nabla_{X_1}$, where $X_i \in C^\infty(X;T^{p_i} X)$, $a \in C^\infty(X;\Hom(E_1,E_2))$ and $p_1 + \ldots + p_k \leq m$. The space $\mathcal{DO}_H^m(X;E_1,E_2)$ is independent of the choice of connection on $E$ since it can be described as local operators $A:C^\infty(X;E_1)\to C^\infty(X;E_2)$ that in local coordinates and local trivializations can be written as a sum of terms  $a X_k \ldots X_1$ where $X_i \in \Gamma^\infty(T^{p_i} X)$, $a \in C^\infty(X;\Hom(E_1,E_2))$ for $p_1 + \ldots + p_k \leq m$. 

By \cite[Example 24]{vanErp_Yuncken}, it holds that  $\mathcal{DO}_H^m(X;E_1,E_2)\subseteq \Psi_H^m(X;E_1,E_2)$. Let us briefly describe why that is. For a $P\in \mathcal{DO}_H^m(X;E_1,E_2)$, we can identify it with its Schwartz kernel 
$$P(x,y):=P^t\delta_x(y),$$
and to it associate its principal cosymbol $P_0\in \mathcal{E}'_r(T_HX; E_1,E_2)$ defined as the fibrewise Schwartz kernel (of convolution type) of the leading terms in $P$ frozen in points on $X$. It was proven in \cite[Example 24]{vanErp_Yuncken} that the element
$$\pmb{P}:=\begin{cases} 
t^m P, \; &t>0,\\
P_0, \; &t=0, \end{cases}$$
belongs to $\pmb{\Psi}^m_H(X;E_1,E_2)$. Therefore, $P=\mathrm{ev}_1(\pmb{P})\in \Psi_H^m(X;E_1,E_2)$ and its principal symbol is $\sigma_m(P)=P_0\mod C^\infty_c(T_HX;E_1,E_2)$.

If $X$ is a regular Carnot manifold, we can describe the principal symbol of a  differential operator in another way. Let $\mathfrak{g}$ be the type of $X$, i.e. the graded Lie algebra $\mathfrak{g}$ such that $\mathfrak{g}\cong \mathsf{gr}[TX]_x$ as Lie algebras for all $x\in X$. For $D \in \mathcal{DO}_H^m(X;E_1,E_2)$ its principal symbol at each $x \in X$ can be identified with an element
\[ \sigma_x^m(D) \in \mathcal{U}_m(\mathfrak{g}) \otimes \mathrm{Hom}(E_{1,x}, E_{2,x}), \]
where $\mathcal{U}_m(\mathfrak{g})$ is the subspace of elements of degree $m$ the universal enveloping algebra of a graded Lie algebra $\mathfrak{g}$, which is spanned by elements of degree $\leq m$. Note that $\mathcal{U}_m(\mathfrak{g}) \otimes \mathrm{Hom}(E_{1,x}, E_{2,x})$ is naturally identified with a subspace of $\Sigma^m_H(X;E_1,E_2)_x$. The principal symbol on differential operators is characterized by the following properties, see \cite[Subsection 2.2]{Dave_Haller1}:
\begin{enumerate}
\item $\sigma_x^{m+m'}(D_1D_2) = \sigma_x^m(D_1) \sigma_x^{m'}(D_2)$, for all $x\in X$, $D_1\in \mathcal{DO}_H^m(X;E_2,E_3)$, and $D_2\in \mathcal{DO}_H^{m'}(X;E_1,E_2)$.
\item $\sigma^m_x(\nabla_Y) = Y_x \otimes id_E$, for all  $x\in X$ and $Y\in C^\infty(X;T^mX)$.
\item $\sigma^0_x(a) = a(x)$ for all $x\in X$ and $a\in C^\infty(X,\mathrm{Hom}(E_1,E_2))$.
\end{enumerate}
For computational purposes, the next proposition comes in handy for constructing examples.
\end{example}

\begin{proposition}
\label{surjecofsymbosl}
Let $X$ be a regular Carnot manifold and define the bundle $\mathcal{U}_m(\mathfrak{t}_HX)\to X$ as the bundle of elements of degree $m$ the universal enveloping algebra of the osculating Lie algebroid, i.e. $\mathcal{U}_m(\mathfrak{t}_HX):=P_X\times_{\Aut_{\rm gr}(\mathfrak{g})}\mathcal{U}_m(\mathfrak{g})$ where $\mathfrak{g}$ is the type of $X$ and $P_X\to X$ the graded frame bundle. Then for any two vector bundles $E_1,E_2\to X$, the symbol mapping restricts to a surjection
$$\mathcal{DO}_H^m(X;E_1,E_2)\to C^\infty(X,\mathcal{U}_m(\mathfrak{t}_HX)\otimes \Hom(E_1,E_2))).$$
\end{proposition}

We use the notation 
$$\dim_h(X):=\sum_{j=1}^r j\mathrm{rk}(T^{-j}X/T^{-j+1}X),$$
for the homogeneous dimension of a Carnot manifold $X$. We also introduce the notation 
$$|\cdot|:T_HX\to \R_+,$$
for the composition of the the logarithm $T_HX\to \mathfrak{t}_HX$ and a homogeneous norm on $\mathfrak{t}_HX$ which is smooth outside the zero section, cf. \cite[Section 3]{choiponge} or \cite[Section 3]{Dave_Haller1}. The construction of $|\cdot|$ involves a choice, e.g. from fixing a choice of metrics on $T^{-j}X/T^{-j-1}X$, $j=1,\ldots, r$, but for our purposes the precise choice will not matter. 

We define the subspace $P^m_H(T_HX,E_1,E_2)\subseteq C^\infty(T_HX;\Hom(E_1,E_2)\otimes |\Lambda|)$ as the space of all smooth density valued $p$ that are fibrewise a homogeneous polynomial with $\lambda_*p=\lambda^mp$. We remark that $P^m_H(T_HX,E_1,E_2)=0$ unless $-m-\dim_h(X)\in \N$. For a compact Carnot manifold $X$, we define $\mathcal{S}(T_HX,E_1,E_2)\subseteq C^\infty(T_HX;\Hom(E_1,E_2)\otimes |\Lambda|)$ as the space of fibrewise Schwartz functions and $\mathcal{S}'(T_HX,E_1,E_2)\subseteq\mathcal{D}'(T_HX;\Hom(E_1,E_2)\otimes |\Lambda|)$ for its topological dual. We define 
$$\mathcal{S}'_r(T_HX,E_1,E_2):=\left\{u\in \mathcal{S}'(T_HX,E_1,E_2): 
\begin{matrix} 
\mathrm{singsupp}(u)\subseteq X, \\ 
\varphi u\in \mathcal{E}'_r(T_HX,E_1,E_2), \; \forall \varphi\in C^\infty_c(T_HX)
\end{matrix}
\right\}.$$
In other words, $\mathcal{S}'_r(T_HX,E_1,E_2)$ is the space of distributions tempered along the fibre, smoothly varrying along the base and smooth outside the zero section.

\begin{proposition}[\cite{Dave_Haller1,vanErp_Yuncken}]
\label{strucuturaldescofsymboslalld}
Let $X$ be a compact Carnot manifold, $E_1,E_2,E_3\to X$ complex vector bundles and $m,m'\in \mathbb{C}$. 
\begin{enumerate}[i)]
\item For $A\in \Psi^m_H(X;E_2,E_3)$ and $A'\in \Psi^{m'}_H(X;E_1,E_2)$ it holds that 
$$AA'\in \Psi^{m+m'}_H(X;E_1,E_3)$$ 
and 
$$\sigma_{m+m'}(AA')=\sigma_m(A)\sigma_{m'}(A').$$
Moreover, $A^*\in \Psi^m_H(X;E_3,E_2)$ and 
$$\sigma_m(A^*)=\sigma_m(A)^*.$$
\item The principal symbol mapping fits into a short exact sequence 
$$0\to \Psi^{m-1}_H(X;E_1,E_2)\to \Psi^m_H(X; E_1,E_s)\xrightarrow{\sigma_m} \Sigma^m_H(X;E_1,E_2)\to 0,$$
where the first mapping is induced from Proposition \ref{embeddingboldpsi} (cf. \cite[Corollary 38]{vanErp_Yuncken}). That is, $\Psi^{m-1}_H(X;E_1,E_2)=\ker(\sigma_m)$ and $\sigma_m$ is surjective. 
\item Any $a\in \Sigma^m_H(X;E_1,E_2)$ can be represented by an $a_0\in  \mathcal{E}'_r(T_HX;E_1,E_2)$ that near the zero section $X\subseteq T_HX$ has the form 
\begin{equation}
\label{decompkspskpd}
a_0=k_\infty+k_m+p_m\log|\cdot|,
\end{equation}
where 
\begin{itemize}
\item $k_\infty\in C^\infty(T_HX,\Hom(E_1,E_2)\otimes |\Lambda|)$
\item $k_m\in \mathcal{S}'_r(T_HX;E_1,E_2)$ satisfies that $\lambda_*k=\lambda^m k$;
\item $p_m\in P^m_H(T_HX,E_1,E_2)$.
\end{itemize}
If $-m-\dim_h(X)\notin \N$, the terms $k_m$ and $p_m$ in \eqref{decompkspskpd} are uniquely determined by $a\in \Sigma^m_H(X;E_1,E_2)$. If $-m-\dim_h(X)\in \N$, the term $p_m$ \eqref{decompkspskpd} is uniquely determined by $a$, and $k_m$ uniquely determined by $a$ up to $P^m_H(X,E_1,E_2)$. In particular, we can identify
\begin{align}
\label{symbolreps}
\Sigma^m_H(X;&E_1,E_2)=\\
\nonumber
=&\{K\in \mathcal{S}'_r(T_HX;E_1,E_2)/P^m_H(T_HX;E_1,E_2): \, \lambda_*K=\lambda^m K,\}.
\end{align}
\end{enumerate}
\end{proposition}

\begin{proof}
Item i) follows from Proposition \ref{productsinfatthm}. Item ii) is stated as \cite[Corollary 38]{vanErp_Yuncken}. Item iii) is contained in Lemma 3.8 of \cite[Lemma 3.8]{Dave_Haller1} using that homogeneous distributions are tempered.
\end{proof}

\begin{remark}
Let us make an observation concerning the failure of uniqueness of the decomposition \eqref{decompkspskpd} when $-m-\dim_h(X)\in \N$ by comparing to the case of a trivial filtration, so $\dim_h(X)=\dim(X)$. The principal symbol (of order $m$) in van Erp-Yuncken's set up is for a trivial filtration dual -- via the Fourier transform -- to the classical principal symbol which is a smooth function on $T^*X\setminus X$ homogeneous of degree $m$ (see \cite[Chapter XVIII]{horIII}). Indeed, the difference between van Erp-Yuncken's principal symbol and the classical symbol is that the former is a leading term of the distributional kernel near the diagonal while the latter is roughly speaking the leading term at infinity of the kernel's Fourier transform transversally to the diagonal. 

The failure of uniqueness in the decomposition \eqref{decompkspskpd} stems from the fact that a smooth function on $T^*X\setminus X$ which is homogeneous of degree $m$ extends uniquely to distribution on $T^*X$ if and only if $-m-\dim(X)\notin \N$ (cf. \cite[Chapter III.2]{horI}). Indeed, for $-m-\dim(X)\in \N$, the set of extensions is an affine space modelled on derivatives of order $m$ in the fibre direction of the delta function on $X\subseteq T^*X$, this space is under the Fourier transform dual to the space $P^m_H(TX;E_1,E_2)$ consisting of fibrewise polynomials. The lift constructed in Proposition \ref{strucuturaldescofsymboslalld} is analogous to defining an extension of the classical principal symbol to a distribution on $T^*X$. For the purpose of pseudodifferential calculus, the precise choice of extension of the classical symbol (or lift as in Proposition \ref{strucuturaldescofsymboslalld} for the case for Carnot operators) is an irrelevant artifact of the calculus. 
\end{remark}

The next result is a consequence of Proposition \ref{strucuturaldescofsymboslalld} and will be used later on for defining $KK$-cycles from Carnot operators. We first note that we can embed $C^\infty(X)\subseteq  \pmb{\Psi}^0_H(X; E)$ as multiplication operators for any vector bundle $E\to X$, similarly we can also embed $C^\infty(X)\subseteq  \Psi^0_H(X; E)$ and $C^\infty(X)\subseteq  \Sigma_H^0(X; E)$. By the definition of the product in $\Sigma_H^0(X; E)$, as a convolution product on a groupoid where the source and range coincides, we have that $C^\infty(X)\subseteq  \Sigma_H^0(X; E)$ is central and that $C^\infty(X)$ acts as central multipliers of $\Sigma_H^m(X; E_1,E_2)$ for any $m$ and any two vector bundles $E_1,E_2\to X$.

\begin{corollary}
\label{commwithfunticododo}
Let $X$ be a Carnot manifold and $E_1,E_2\to X$ complex vector bundles. For $\mathbb{D}\in \pmb{\Psi}^m_H(X; E_1,E_2)$ and $f\in C^\infty(X)$ it holds that 
\begin{align*}
[\mathbb{D},f]\in t\pmb{\Psi}^{m-1}_H(X; &E_1,E_2)\\
&+C^\infty_p(\mathbb{T}_HX,r^*E_2\otimes (s^*E_1)^*\otimes |\Lambda_r|).
\end{align*}
In particular, for any $D\in \Psi^m_H(X; E_1,E_2)$ and $f\in C^\infty(X)$ it holds that $[D,f]\in \Psi^{m-1}_H(X; E_1,E_2)$.
\end{corollary}

\begin{proof}
Since $C^\infty(X)$ acts as central multipliers of $\Sigma_H^m(X; E)$ for any $m$, Proposition \ref{strucuturaldescofsymboslalld} implies that the principal symbol of $[\mathbb{D},f]$ vanishes. The corollary now follows from Proposition \ref{embeddingboldpsi}.
\end{proof}

\subsection{Examples of Carnot operators}
\label{examplesubsechoperaoro}

We give some examples of Carnot operators. As will become apparent later on, we are in some of the examples mainly interested in properties of the principal symbol so in light of Proposition \ref{surjecofsymbosl} it will in these cases suffice to describe the operator on a Carnot-Lie group in order to generate Carnot operators on Carnot manifolds.  \\

\subsubsection{Baum-van Erp type operators}
\label{bveops1}
Consider a Carnot-Lie group $\mathfrak{g}$ of depth $r$ equipped with a graded basis. We write the basis as $\{\{X_{j,k}\}_{k=1}^{\dim \mathfrak{g}_{-j}}\}_{j=1}^r$, where $\{X_{j,k}\}_{k=1}^{\dim \mathfrak{g}_j}$ is a basis for $\mathfrak{g}_{-j}$ for $j=1,\ldots, r$. For any two finite-dimensional complex vector spaces $V,W$ and a collection $\gamma=(\gamma_{j,k})_{j,k}\subseteq \Hom(V,W)$ and an $m\in \N_+$ such that $j|m$ when $\gamma_{j,k}\neq 0$, we define the operator
$$D_\gamma=\sum_{j=1}^r\sum_{k=1}^{\dim \mathfrak{g}_j} \gamma_{j,k}X_{j,k}^{m/j}\in \mathcal{U}_m(\mathfrak{g})\otimes \Hom(V,W).$$
Under the additional condition of $D_\gamma$ being $H$-elliptic, see more below in Section \ref{subsec:hellleleldlda}, we call an operator of the form $D_\gamma$ a Baum-van Erp type operator. 

If $X$ is a Carnot manifold, $E_1,E_2\to X$ are vector bundles and $D_\gamma\in \mathcal{DO}_H^m(X;E_1,E_2)$ is such that for all $x\in X$, the principal symbol $\sigma_x(D_\gamma)\in \mathcal{U}_m(\mathfrak{t}_HX_x)\otimes \Hom(E_{1,x},E_{2,x})$ is a Baum-van Erp type operator, then we say that $D_\gamma\in \mathcal{DO}_H^m(X;E_1,E_2)$ is a Baum-van Erp type operator. 

For equiregular differential systems, see more in Example \ref{ex:regudlaladoad}, there is an important special case -- the sub-Laplacian. We make the following definition:

\begin{definition}
Let $X$ be a compact Carnot manifold arising from a regular differential system with a fixed volume density. An operator $\Delta_H\in \mathcal{DO}_H^2(X;\C)$ is called a sub-Laplacian if it is formally self-adjoint (with respect to the volume density) and 
$$\sigma^2_x(\Delta_H)=-\sum_{j=1}^N \sigma_1(X_j)^2,$$
for some set of elements $X_1,\ldots,X_N\in T^{-1}X_x$ spanning $T^{-1}X_x/T^{-2}X_x$.
\end{definition}

Hörmander's sum of squares theorem \cite{horsumsquares} ensures that a sub-Laplacian is hypoelliptic. It is readily verified that the sub-Laplacian $\Delta_H$ is a Baum-van Erp type operator. The choice of the term Baum-van Erp type operator in this monograph is motivated by Baum-van Erp's beautiful solution to the index problem for the Baum-van Erp type operator $\Delta_H+\gamma Z$ on a coorientable contact manifold \cite{baumvanerp} (here $Z$ denotes the Reeb field). 

Let us provide a direct construction of a sub-Laplacian on any compact Carnot manifold $X$ arising from a regular differential system. We fix a volume density on $X$, and for a vector field $X$ we let $X^*$ denote its formal adjoint with respect to the volume density.  If $X_1,\ldots, X_N\in C^\infty(X,T^{-1}X)$ is a collection of vector fields spanning $T^{-1}X/T^{-2}X$ in all points, the associated sub-Laplacian is defined by
$$\Delta_H:=\sum_{j=1}^N X_j^*X_j\in \mathcal{DO}_H^2(X;\C).$$
Existence of such a collection of vector fields is ensured by compactness of $X$. The principal symbol of $\Delta_H$ is independent of volume density, and in fact $\sigma_2(\Delta_H)=-\sum_{j=1}^N \sigma_1(X_j)^2$. \\

\subsubsection{Dirac type operators}
\label{ex:Diracex1}
Again we consider a construction at the level of a Carnot-Lie group $\mathsf{G}$ of depth $r$. Pick a basis $(X_j)_{j=1}^{\dim(\mathfrak{g}_{-1})}$ of $\mathfrak{g}_{-1}$ and for some finite-dimensional Hilbert space $V$, a collection of self-adjoint $(\gamma_j)_{j=1}^{\dim(\mathfrak{g}_{-1})}\subseteq \End(V)$. Consider the operator:
$$\slashed{D}_H=\sum_{j=1}^{\dim \mathfrak{g}_r} i\gamma_j X_{j}.$$
If $(\gamma_j)_{j=1}^{\dim(\mathfrak{g}_{-1})}$ satisfies the Clifford relations $\gamma_j\gamma_k+\gamma_k\gamma_j=2\delta_{jk}$, we say that $D_H$ is a Dirac type operator. Operators of this type are, as we shall see below in Example \ref{ex:Diracex2}, rarely $H$-elliptic.\\

\subsubsection{A positive even order Carnot differential operator}
\label{posoofosos1}

Let $X$ be a compact Carnot manifold of depth $r$ and $E\to X$ a hermitean vector bundle. Later on in the monograph, we shall need to make use of an auxiliary operator $\mathfrak{D}$ in the calculus that admits complex powers in the Carnot calculus. The example consists of a positive even order Carnot differential operator and is introduced here; below in Example \ref{posoofosos2} we show that the $\mathfrak{D}$ we construct in this example produces a hypoelliptic operator so by \cite[Theorem 2]{Dave_Haller2}, $\mathfrak{D}$ will fit the bill of admitting complex powers in the calculus. 

Pick a connection $\nabla$ on $E$ and a volume density on $X$. We also pick a collection of vector fields $\{\{X_{j,k}\}_{k=1}^{N_j}\}_{j=1}^r$ on $X$ such that $\{X_{j,k}\}_{k=1}^{N_j}\subseteq C^\infty(X,T^{-j}X)$ and its images in $C^\infty(X,T^{-j}X/T^{-j+1}X)$ spans $T^{-j}X/T^{-j+1}X$ in all points. Existence of such a collection of vector fields is ensured by compactness of $X$. We define the differential operator
$$\mathfrak{D}:=\sum_{j=1}^r \sum_{k=1}^{N_j} (\nabla_{X_{j,k}}^*)^{r!/j}\nabla_{X_{j,k}}^{r!/j}\in \mathcal{DO}_H^{2(r!)}(X;E).$$
The $*$ denotes the formal adjoint in the hermitean metric on $E$ and the volume density on $X$. We have for any $x\in X$ that 
$$\sigma^{m}_{x}(\mathfrak{D}):=\sum_{j=1}^r \sum_{k=1}^{N_j} (-X_{j,k,x})^{r!/j}(X_{j,k,x})^{r!/j}\otimes 1_{E_x}\in \mathcal{U}_{2(r!)}(\mathfrak{t}_HX_x)\otimes \End(E_x).$$
At the symbol level, operators similar to $\mathfrak{D}$ were studied by Mantiouou-Ruzhansky \cite{Mantoiu_Ruzhansky}.

At this stage, we remark that in specific examples there are often better suited positive even order Carnot differential operators than $\mathfrak{D}$, e.g. of substantially lower order, even when requiring hypoellipticity. For instance, the sub-Laplacian can be used to construct an example of order $2$ if the Carnot structure on $X$ comes from a regular differential system.\\

\subsubsection{Bernstein-Gelfand-Gelfand complex}
\label{ex:BBBBBBGGGGGG1}
In Example \ref{parapara} above we saw how regular parabolic geometries gave rise to regular Carnot manifolds, following \cite{cap}. The easiest example being the ``flat'' local model $G/P$, for $G$ a semi-simple Lie group and $P$ a parabolic subgroup coming from a $|k|$-grading. On parabolic geometries there are certain complexes of differential operators, the so called Bernstein-Gelfand-Gelfand (BGG) complexes. The BGG-complexes on parabolic geometries were introduced in \cite{morecap} and further studied and placed in the Carnot calculus by Dave-Haller \cite{GenericDave,Dave_Haller1}. The origin of the BGG-complex is in representation theory \cite{bggoriginal}, where Bernstein-Gelfand-Gelfand constructed a BGG-complex as a method of resolving finite-dimensional representations of $G$ by Verma modules for $P$ being the Borel subgroup. For more details on how BGG-complexes relate to other notions, e.g. Kostant's Borel-Bott-Weil theorem \cite{kostant61} or Rumin complexes, we refer the reader to \cite{Dave_Haller1}. In fact, we refer the entire construction of general BGG-complexes to \cite{Dave_Haller2} as it is quite technical. We shall in this work focus on the abstract setting coming out of \cite{Dave_Haller2}. As we discuss below in Section \ref{sec:grafadpknapdnarock}, the associated index problem is still open but due to its relation to Julg's approach \cite{julgcr,julgsp} to the Baum-Connes conjecture for $Sp(n,1)$, we believe it to be of general interest.  

To describe the differential complexes, that act between graded vector bundles, we need some further terminology. A vector bundle $\pmb{E}\to X$ is said to be graded if it is equipped with a vector bundle decomposition $\pmb{E}=\oplus_{k\in \R}\pmb{E}[k]$. The space of graded Carnot operators of order $m$ between two graded vector bundles $\pmb{E}$ and $\pmb{F}$ is defined as 
$$\Psi^m_{H,{\rm gr}}(X;\pmb{E},\pmb{F}):=\bigoplus_{j,k}\Psi_H^{m+k-j}(X;\pmb{E}[k],\pmb{F}[j]).$$
We consider an element of $T=(T_{j,k})_{j,k}\in \Psi^m_{H,{\rm gr}}(X;\pmb{E},\pmb{F})$ as a matrix, implementing an inclusion
$$\Psi^m_{H,{\rm gr}}(X;\pmb{E},\pmb{F})\subseteq \Psi^*_{H}(X;\pmb{E},\pmb{F}),$$
into the ungraded calculus. The matrix product defines a composition
$$\Psi^m_{H,{\rm gr}}(X;\pmb{E},\pmb{F})\times \Psi^{m'}_{H,{\rm gr}}(X;\pmb{F},\pmb{G})\to \Psi^{m+m'}_{H,{\rm gr}}(X;\pmb{E},\pmb{G}),$$
for any three graded vector bundles $\pmb{E},\pmb{F},\pmb{G}\to X$ and $m,m'\in \R$.
Similarly, we define the graded symbol algebra as
$$\Sigma^m_{H,{\rm gr}}(X;\pmb{E},\pmb{F}):=\bigoplus_{j,k}\Sigma_H^{m+k-j}(X;\pmb{E}[k],\pmb{F}[j]),$$
and the graded symbol mapping 
$$\sigma_{H,{\rm gr}}^m: \Psi^m_{H,{\rm gr}}(X;\pmb{E},\pmb{F})\to \Sigma^m_{H,{\rm gr}}(X;\pmb{E},\pmb{F}),\quad (T_{j,k})_{j,k}\mapsto (\sigma_H^{m+k-j}T_{j,k})_{j,k}.$$
The constructions of Subsection \ref{pseudcalcalc} extends ad verbatim to the graded case. We remark that graded pseudodifferential operators plays a central role in the theory of boundary value problems under the name of Douglis-Nirenberg calculus, see for instance \cite{agmon, grubb77}, see also \cite[Chapter XIX.5]{horIII} and \cite{bandgoffsar}. 

The structure arising from BGG-complexes, and the graded complexes appearing in \cite{Dave_Haller1}, fits into the following definition that we refine below in Definition \ref{rocklandseqinad}. 

\begin{definition}
\label{hcomplexdefdo}
Consider a sequence 
$$0\to C^\infty(X;\pmb{E}_1)\xrightarrow{D_1}C^\infty(X;\pmb{E}_2)\xrightarrow{D_2}\cdots \xrightarrow{D_{N-1}}C^\infty(X;\pmb{E}_{N})\xrightarrow{D_{N}} C^\infty(X;\pmb{E}_{N+1})\to 0,$$
where $\pmb{E}_1,\ldots, \pmb{E}_{N+1}\to X$ are graded vector bundles and, for some $m_1,\ldots, m_N$, we have that $D_j\in \Psi^{m_j}_{H,{\rm gr}}(X;\pmb{E}_j,\pmb{E}_{j+1})$. 
\begin{itemize}
\item If $D_{j+1}D_j=0$ for all $j$, we say that the sequence is a graded $H$-complex.
\item If $\sigma_H^{m_{j+1}}(D_{j+1})\sigma_H^{m_j}(D_j)=0$ for all $j$, we say that the sequence is a graded $H$-almost-complex.
\end{itemize}
We say that the sequence is of order $\pmb{m}=(m_1,\ldots, m_N)$. 
\end{definition}

\begin{example}
\label{examplefromderhamadma}
An instructive example of a graded $H$-complex comes from the de Rham complex on a Carnot manifold. If $X$ has a Carnot structure, we can view $\pmb{E}_j:=\wedge^jT^*X\otimes \C$ as graded vector bundles upon choosing a splitting $TX\cong \mathsf{gr}(TX)$. Note that $T^*X$ is graded as a dual vector space. Therefore, the exterior differential $\rd:C^\infty(X;\pmb{E}_j)\to C^\infty(X;\pmb{E}_{j+1})$ is of graded order $0$ in the Carnot calculus. The de Rham sequence 
$$0\to C^\infty(X;\pmb{E}_1)\xrightarrow{\rd }C^\infty(X;\pmb{E}_2)\xrightarrow{\rd }\cdots \xrightarrow{\rd }C^\infty(X;\pmb{E}_{N})\xrightarrow{\rd } C^\infty(X;\pmb{E}_{N+1})\to 0,$$
is a graded $H$-complex of order $\pmb{0}=(0,\ldots, 0)$.

More generally, if $\pmb{E}\to X$ is a graded vector bundle equipped with a connection $\nabla_{\pmb{E}}$ of graded order zero and curvature of degree $\leq -1$, then with $\pmb{E}_j:=\wedge^jT^*X\otimes \pmb{E}$, we have a sequence 
$$0\to C^\infty(X;\pmb{E}_1)\xrightarrow{\nabla_{\pmb{E}} }C^\infty(X;\pmb{E}_2)\xrightarrow{\nabla_{\pmb{E}} }\cdots \xrightarrow{\nabla_{\pmb{E}} }C^\infty(X;\pmb{E}_{N})\xrightarrow{\nabla_{\pmb{E}} } C^\infty(X;\pmb{E}_{N+1})\to 0,$$
which is a graded $H$-almost-complex of order $\pmb{0}=(0,\ldots, 0)$.

Associated with such a graded $H$-almost-complex of differential operators defined from a graded connection, Dave-Haller \cite{Dave_Haller1} associates a new BGG-like sequence 
$$0\to C^\infty(X;\mathcal{H}_1)\xrightarrow{D_1}C^\infty(X;\mathcal{H}_2)\xrightarrow{D_2}\cdots \xrightarrow{D_{N-1}}C^\infty(X;\mathcal{H}_{N})\xrightarrow{D_{N}} C^\infty(X;\mathcal{H}_{N+1})\to 0,$$
where the bundles are defined from Lie algebra cohomology of the fibre, producing a new graded $H$-almost-complex. This construction  extends the curved BGG-complex of Cap-Slovak-Souček \cite{morecap}. 

An example of curved BGG-complexes that has played a particular role in the study of the Baum-Connes conjecture for rank $1$ semisimple Lie groups is the Rumin complex \cite{julgsun1,julgcr,julgsp}, compare to \cite{yunckensl3} for the higher rank case. In the special case of a contact manifold, a construction of the Rumin complex can be found in \cite{julgsun1} and for the construction on a quaternionic contact manifold see \cite{julgcr,julgsp}.
\end{example}

\subsubsection{Szegö projections}
\label{ex:szegoexampe}
A well studied class of operators in complex analysis is Szegö projections. For a domain in a complex manifold, the Szegö projection is a projection in a space of functions on the boundary onto the subspace of functions extending holomorphically to the interior. As mentioned in Example \ref{boundaryofpseudoconvexz}, if $X=\partial\Omega$ is a compact boundary of a strictly pseudoconvex domain, viewed as a Carnot manifold in its contact structure, the Szegö projection is a projection $P_{S,\Omega}\in \Psi_H^0(X,\C)$ by \cite[Chapter III.6]{taylorncom} assuming the technical condition that the $\bar{\partial}_b$-operator on $X$ has closed range on $L^2$. The closed range condition on the $\bar{\partial}_b$-operator is equivalent to $X$ being $CR$-embeddable in $\C^N$ (for a large $N$, see more in \cite{kohnduke}) and is automatic if $\dim(X)\geq 5$ by a theorem of Boutet de Monvel \cite{bdmembedd}. 

More generally, \cite[Chapter 6]{melroseeptein} considered generalized Szegö projections on a co-orientable contact manifold $X$. We return to study generalized Szegö projections further below in Example \ref{ex:szegoexampe2}. 

Let us end this example by making some remarks on the Szegö projection for the boundary $X$ of a weakly pseudo-convex domains. The weak pseudo-convexity produces a depth $2$ Carnot structure on $X$, and it is regular of type $\mathbb{H}_n$ if and only if $X$ is strictly pseudoconvex. The situation for solving the $\bar{\partial}$-problem is more complicated in the weakly pseudoconvex case, and a well studied problem by the school of Kohn (for an overview, see \cite{dangelokohn}). For weakly pseudo-convex domains of finite type (see for instance \cite{dangelokohn,hsiaosavale}), the situation is better. A weakly pseudo-convex domain $\Omega=\{z:\phi(z)<0\}$ is of finite type if the sub-bundle $H:=\ker(\rd^c\phi|_{TX})\subseteq TX$ is bracket generating, i.e. the Carnot structure $H$ defines can be refined to a singular Carnot structure (in the sense of Remark \ref{singularcarnotremark}). defined by $\mathpzc{E}^{-1}=C^\infty(X,H)$ and $\mathpzc{E}^{-j-1}=\mathpzc{E}^{-j}+[\mathpzc{E}^{-j},C^\infty(X,H)]$. Compare to the explicit example in Example \ref{boundaryofpseudoconvexz}. For instance, in a weakly pseudoconvex domain of finite type the $\bar{\partial}$-problem comes with well understood sub-elliptic estimates and the Szegö projection preserves $C^\infty(X)$. For weakly pseudo-convex domains that are not of finite type, there are examples where sub-elliptic estimates fail, the Szegö projection does not preserve the space of real analytic functions (see  \cite{christsze}) and the closely related Bergman projection does not preserve the space of smooth functions (see \cite{christberg}), indicating that  the Carnot calculus is not suited for studying the Szegö projection of a pseudoconvex domain which is not of finite type. This fact indicates that Carnot manifolds and their Carnot calculus do not form suitable tools for dealing with weakly pseudo-convex domains that are not of finite type. It poses an interesting problem to study if the Szegö projection of a weakly pseudoconvex domain of finite type belongs to its Carnot calculus or perhaps a Carnot calculus adapted to singular Carnot manifolds, such as the calculus of \cite{androerp}. The results of \cite{charpdup,hsiaosavale} indicate that it might be the case.

\section{$H$-ellipticity and the Rockland condition}
\label{subsec:hellleleldlda}

The symbol algebra in the Carnot calculus is in general a noncommutative algebra. The standard way to deal with the issues arising from such complications is to study symbols as (unbounded) multipliers on suitable domains in the groupoid $C^*$-algebra of the osculating Lie groupoid. As such, ellipticity properties in the calculus can be verified fibrewise in all representations of that fibre via generalizations of the Rockland theorem. We give an overview of relevant results following \cite{Dave_Haller1}, see also \cite{eskeewertthesis,eskeewertpaper}. 

We first introduce some notations. Let $\mathsf{G}$ denote a nilpotent Lie group. Recall that a nilpotent Lie group carries a unique polynomial structure determined by the exponential map. Let $\mathcal{S}(\mathsf{G})$ denote the Schwartz space of $\mathsf{G}$. We write $\mathcal{S}_0(\mathsf{G})\subseteq \mathcal{S}(\mathsf{G})$ for the subspace of functions $f\in \mathcal{S}(\mathsf{G})$ such that 
$$\int_\mathsf{G} p(g)f(g)\rd g=0,$$
for all polynomials $p$ on $\mathsf{G}$.

For a unitary representation $\pi : \mathsf{G} \rightarrow \mathcal{U}(\mathcal{H}_\pi)$, we write 
$$\mathcal{S}(\pi):=\pi(\mathcal{S}(\mathsf{G}))\mathcal{H}_\pi,\quad \mbox{and}\quad \mathcal{S}_0(\pi):=\pi(\mathcal{S}_0(\mathsf{G}))\mathcal{H}_\pi.$$
We note that the inclusion $\mathcal{S}(\pi)\subseteq \mathcal{H}_\pi$ is dense in the norm topology. Moreover, the orthogonal complement of $\mathcal{S}_0(\pi)\subseteq \mathcal{H}_\pi$ is the closed subspace of all $\mathsf{G}$-invariant vectors in $\mathcal{H}_\pi$; this follows from the fact that $\mathcal{S}_0(\mathsf{G})$ is dense in the kernel of the trivial representation $C^*(\mathsf{G})\to \C$. In particular, the inclusion $\mathcal{S}_0(\pi)\subseteq\mathcal{H}_\pi$ is dense as soon as $\mathcal{H}_\pi$ contains no invariant vectors. We can identify $\mathcal{S}(\pi)$ with the closed subspace of the Frechet space $\mathcal{S}(\mathsf{G},\mathcal{H})$ consisting of the functions $f(g):=\pi(g)v$ for $v\in \mathcal{H}$. Similarly, $\mathcal{S}_0(\pi)$ can be identified with a closed subspace of the Frechet space $\mathcal{S}_0(\mathsf{G},\mathcal{H})$. Unless otherwise stated, we topologixe $\mathcal{S}_0(\pi)$ and $\mathcal{S}(\pi)$ as Frechet spaces using the inclusions $\mathcal{S}_0(\pi)\subseteq \mathcal{S}_0(\mathsf{G},\mathcal{H})$ and $\mathcal{S}(\pi)\subseteq \mathcal{S}(\mathsf{G},\mathcal{H})$, respectively. We write $\mathcal{L}(V,W)$ for the space of continuous linear operators between Frechet spaces and $\mathcal{L}(V):=\mathcal{L}(V,V)$ (we reserve the notation $\mathbb{B}$ for the space of bounded operators between Hilbert spaces).

Before describing the general construction, let us digest a bit on the Lie algebra case. The following result is well known, and more or less follows from definition. 

\begin{proposition}
\label{pitoleiielalad}
Let $\pi : \mathsf{G} \rightarrow \mathcal{U}(\mathcal{H}_\pi)$ be a unitary representation, and define 
\[\pi : \mathcal{U}(\mathfrak{g}) \rightarrow\mathcal{L}(\mathcal{S}(\pi)), \ \pi(X)v := \frac{d}{dt} |_{t = 0} \pi(\exp(tX))v\quad\mbox{for}\; X\in \mathfrak{g}.\]
The mapping $\pi$ is a well defined $*$-homomorphism on $\mathcal{U}(\mathfrak{g})$ when equipping $\mathcal{U}(\mathfrak{g})$ with the $*$-operation such that $X^*=-X$ for $Z\in \mathfrak{g}$ and $\mathcal{L}(\mathcal{S}(\pi))$ with the partially defined $*$-operation defined from formal adjoints on $\mathcal{S}(\pi)$.
\end{proposition}

The reader should beware that the Hilbert space adjoint of elements in $\pi(\mathcal{U}(\mathfrak{g}))$ is generally different from its formal adjoint on $\mathcal{S}(\pi)$, see more in Remark \ref{regularidremak}. We return to study how Proposition \ref{pitoleiielalad} extends to differential operators on a Carnot manifold (where $\mathfrak{g}=\mathfrak{t}_HX_x$ for $x\in X$) below in Example \ref{connectingliealgtorepsmssm}.

Let us now extend the construction of Proposition \ref{pitoleiielalad} to the symbol algebra on a Carnot manifold. In light of Proposition \ref{strucuturaldescofsymboslalld}.iii, to define a map on the symbol algebra we should first define it on the space of fibrewise tempered distributions. For a point $x\in X$, We write 
$$\Sigma^m_{H}(x;E_{1,x},E_{2,x})\subseteq \mathcal{S}'(T_HX_x;E_{1,x},E_{2,x})/P^m_H(T_HX_x;E_{1,x},E_{2,x}),$$
for the homogeneous elements of degree $m$. Restriction to a point induces a mapping 
$$\sigma_x^m:\Sigma^m_H(X;E_1,E_2)\to \Sigma^m_{H}(x;E_{1,x},E_{2,x}),$$
that respects products and adjoints. For an $x\in X$ and a unitary representation $\pi : T_HX_x \rightarrow \mathcal{U}(\mathcal{H}_\pi)$, we define 
\begin{align*}
\pi : &\mathcal{S}'(T_HX_x;E_{1,x},E_{2,x})\rightarrow\mathcal{L}(\mathcal{S}_0(\pi)\otimes E_{1,x},\mathcal{S}_0(\pi)\otimes E_{2,x}),\\
&  \pi(k)(\pi(f)v) :=\pi(k*f)v\quad\mbox{for}\; f\in \mathcal{S}_0(T_HX_x;E_{1,x}), \; v\in \mathcal{H}_\pi.
\end{align*}
By the definition of $\mathcal{S}_0$, $\pi(a)=0$ if $a\in P^m_H(T_HX;E_1,E_2)$ so $\pi$ descends to a morphism 
$$ \pi : \Sigma^m_H(X;E_1,E_2)\rightarrow\mathcal{L}(\mathcal{S}_0(\pi)\otimes E_{1,x},\mathcal{S}_0(\pi)\otimes E_{2,x}).$$
Next, we define the represented symbol of a Carnot operator. It will be considered as a family of operators parametrized by the set $\widehat{T_HX}$ of all irreducible unitary representations of the osculating Lie groupoid.

\begin{definition}
For $D\in \Psi_H^m(X;E_1,E_2)$, $x\in X$ and a unitary representation $\pi$ of $T_HX_x$, we define the represented symbol (in $(x,\pi)$) as
\[\sigma_{(x,\pi)}^m(D):=\pi(\sigma_x^m(D))\in \mathcal{L}(\mathcal{S}_0(\pi)\otimes E_{1,x},\mathcal{S}_0(\pi)\otimes E_{2,x}).\] 
\end{definition}

\begin{proposition}
\label{generalpitoleiielalad}
Let $X$ be a Carnot manifold, $x\in X$ and $\pi$ a unitary representation of $T_HX_x$.
The represented symbol satisfies the following. For $D\in \Psi_H^m(X;E_2,E_3)$ and $D'\in \Psi_H^{m'}(X;E_1,E_2)$ it holds that 
\begin{align}
\label{adjoadoanad}
\sigma_{x,\pi}^m(D^*)&=\sigma_{x,\pi}^m(D)^*;\\
\label{prodododoadodap}
\sigma_{x,\pi}^{m+m'}(DD')&=\sigma_{x,\pi}^m(D)\sigma_{x,\pi}^{m'}(D');
\end{align}
as operators on $\mathcal{S}_0(\pi)$ where the adjoint is formally defined in the inner product on $\mathcal{H}_\pi$. Moreover, if $m\leq 0$ and $\pi$ does not weakly contain the trivial representation, then $\sigma_{x,\pi}^m(D)$ extends to a bounded linear operator 
$$\sigma_{x,\pi}^m(D):\mathcal{H}_\pi\otimes E_{2,x}\to \mathcal{H}_\pi\otimes E_{3,x}.$$
\end{proposition}

\begin{proof}
The formulas \eqref{adjoadoanad} and \eqref{prodododoadodap} are readily verified from the fact that the representation is defined in terms of convolution. 

For $m=0$ in the last statement, we recall \cite[Theorem 8.18]{eskeewertthesis}, extending \cite[Theorem 1]{knappsteininter}, ensuring that $\sigma^0_x(D)$ acts as a bounded operator $L^2(T_HX_x,E_{1,x})\to L^2(T_HX_x,E_{2,x})$. Boundedness of $\sigma_{x,\pi}^m(D)$ follows. For $m<0$ in the last statement, we use \cite[Theorem 2]{Dave_Haller2} (reviewed below in Theorem \ref{cpxpoweroieoso}). By taking a strictly positive even order differential operator $\mathfrak{D}$ (see for instance Example \ref{posoofosos2} below) with $\mathfrak{D}^s\in \Psi_H^{s\tilde{m}}(X;E_2)$ for some $\tilde{m}\in \N$ and all $s\in \R$, we have that $D=D_0\mathfrak{D}^{m/\tilde{m}}$ with $D_0=D_0\mathfrak{D}^{-m/\tilde{m}}\in \Psi_H^0(X;E_2,E_3)$ by Proposition \ref{strucuturaldescofsymboslalld}. By \eqref{prodododoadodap}, $\sigma_{x,\pi}^{m}(D)=\sigma_{x,\pi}^0(D_0)\sigma_{x,\pi}^{m}(\mathfrak{D}^{m/\tilde{m}})$. The first factor $\sigma_{x,\pi}^0(D_0)$ has a bounded extension by the preceding paragraph. By using \cite[Theorem 2]{Dave_Haller2}, we see that $\sigma_{x,\pi}^{m}(\mathfrak{D})$ is invertible and therefore the second factor has a bounded extension being a negative power of an invertible operator.
\end{proof}

\begin{definition}
An element $D \in  \Psi_H^m(X;E_1,E_2)$ satisfies the Rockland condition if $\sigma_{(x,\pi)}^m(D)$
is injective for every $x \in X$ and every non-trivial irreducible unitary representation $\pi \in \widehat{T_HX}_x$. 

An element $D \in  \Psi_H^m(X;E_1,E_2)$ is $H$-elliptic if $\sigma_{x,\pi}^m(D)$
is bijective for every $x \in X$ and every non-trivial irreducible unitary representation $\pi \in \widehat{T_HX_x}$. If $\mathbb{D} \in  \pmb{\Psi}_H^m(X;E_1,E_2)$ lifts an $H$-elliptic $D \in  \Psi_H^m(X;E_1,E_2)$, we also say that $\mathbb{D}$ is $H$-elliptic.

Similarly, if $\pmb{E}_1,\pmb{E}_2\to X$ are graded vector bundles (cf. Subsubsection \ref{ex:BBBBBBGGGGGG1}), we say that $D \in  \Psi_{H,{\rm gr}}^m(X;\pmb{E}_1,\pmb{E}_2)$ satisfies the graded Rockland condition/is graded $H$-elliptic if the represented graded symbols $\sigma_{(x,\pi)}^m(D)$
is injective/bijective for every $x \in X$ and every non-trivial irreducible unitary representation $\pi \in \widehat{T_HX}_x$.
\end{definition}

The term $H$-elliptic is justified below in Theorem \ref{hellipticmeanshelliptic}. We give the following variation of an important result concerning the Rockland condition from \cite{Dave_Haller1} which extends a result from \cite{christgelleretal}.

\begin{lemma}
\label{charrockslslalem}
Let $X$ be a Carnot manifold, $E_1,E_2\to X$ two vector bundles and $D \in  \Psi_H^m(X;E_1,E_2)$. Then the following are equivalent:
\begin{enumerate}[i)]
\item $D$ satisfies the Rockland condition.
\item There exists a $b\in \Sigma^{-m}_H(X;E_2,E_1)$ such that 
$$b\sigma^m(D)=1.$$
\item There exists an $R\in \Psi^{-m}_H(X;E_2,E_1)$ such that 
$$RD-1\in \Psi^{-1}_H(X;E_1).$$
\end{enumerate}
Furthermore, if $m=0$, the conditions above are equivalent to that for any $x\in X$ and any non-trivial irreducible unitary representation $\pi$ of $T_HX_x$, the bounded linear operator 
$$\sigma_{x,\pi}^m(D):\mathcal{H}_\pi\otimes E_{1,x}\to \mathcal{H}_\pi\otimes E_{2,x},$$
is an injection of Hilbert spaces, in which case it also has closed range.
\end{lemma}

\begin{proof}
The implication i)$\Rightarrow$ii) is found in \cite[Lemma 3.11]{Dave_Haller1} and the converse is immediate. The equivalence of item ii) and iii) follows from Proposition \ref{strucuturaldescofsymboslalld}. The last statement follows from Proposition \ref{generalpitoleiielalad}.
\end{proof}

The contents of Lemma \ref{charrockslslalem} also extends to the graded setting. For clarity, we formulate the next two results in the graded setting.

\begin{lemma}
\label{charHELLrockslslalem}
Let $X$ be a Carnot manifold, $\pmb{E}_1,\pmb{E}_2\to X$ two hermitean graded vector bundles and $D \in  \Psi_{H,{\rm gr}}^m(X;\pmb{E}_1,\pmb{E}_2)$. Then the following are equivalent:
\begin{enumerate}[i)]
\item $D$ is graded $H$-elliptic.
\item $\sigma^m_{H,{\rm gr}}(D)\in \Sigma^m_{H,{\rm gr}}(X;\pmb{E}_1,\pmb{E}_2)$ has an inverse $b\in  \Sigma^{-m}_{H,{\rm gr}}(X;\pmb{E}_2,\pmb{E}_1)$.
\item $D$ and $D^*$ satisfy the graded Rockland condition.
\item The operator 
$$\tilde{D}:=\begin{pmatrix} 0& D\\ D^*&0\end{pmatrix},$$
satisfies the graded Rockland condition.
\end{enumerate}
Furthermore, if $m=0$ and $E_1=\pmb{E}_1$ and $E_2=\pmb{E}_2$ are trivially graded, the conditions above are equivalent to that for any $x\in X$ and any non-trivial irreducible unitary representation $\pi$ of $T_HX_x$, the bounded linear operator 
$$\sigma_{x,\pi}^m(D):\mathcal{H}_\pi\otimes E_{1,x}\to \mathcal{H}_\pi\otimes E_{2,x},$$
is an isomorphism of Hilbert spaces.
\end{lemma}

\begin{proof}
It is clear that iii)$\Leftrightarrow$iv). If iii) holds, there are $b\in \Sigma^{-m}_{H,{\rm gr}}(X;\pmb{E}_2,\pmb{E}_1)$ and $b'\in\Sigma^{-m}_{H,{\rm gr}}(X;\pmb{E}_2,\pmb{E}_1)$ such that $b\sigma^m(D)=1$ and $b'\sigma^m(D^*)=1$ by Lemma \ref{charrockslslalem}. It follows from Proposition \ref{strucuturaldescofsymboslalld} that $b\sigma^m(D)=1$ and $\sigma^m(D)(b')^*=1$, so uniqueness of inverses shows that $b=(b')^*\in \Sigma^{-m}_{H,{\rm gr}}(X;\pmb{E}_2,\pmb{E}_1)$ is an inverse to $\sigma^m(D)$. Therefore we have the implication iii)$\Rightarrow$ ii). The implication  ii)$\Rightarrow$ i) is immediate. Penultimately, if i) holds, then $D$ by definition satisfies the Rockland condition and surjectivity of $\pi(\sigma^m(D))$ for all $\pi$ ensures injectivity of $\pi(\sigma^m(D^*))$ for all $\pi$ (using Proposition \ref{strucuturaldescofsymboslalld}) so $D^*$ satisfies the Rockland condition. Therefore  i)$\Rightarrow$ iii). Finally, the last condition stated for $m=0$ immediately implies item iv) and follows from from item ii) using Proposition \ref{generalpitoleiielalad}.
\end{proof}

The next theorem summarizes the key features of Rockland- and $H$-elliptic operators from \cite[Section 3]{Dave_Haller2}. We formulate it in the graded case which is proven in the same way as the ungraded case. Starting from Lemma \ref{charrockslslalem} and \ref{charHELLrockslslalem}, the theorem is readily deduced from asymptotic completeness of the Carnot calculus. 

\begin{theorem}
\label{hellipticmeanshelliptic}
Let $X$ be a closed Carnot manifold of homogeneous dimension $n$ equipped with a volume density, $\pmb{E}_1,\pmb{E}_2\to X$ graded hermitean vector bundles and $m\in \mathbb{C}$. Assume that $D\in \Psi^m_{H,{\rm gr}}(X; \pmb{E}_1,\pmb{E}_2)$.Then it holds that:
\begin{enumerate}
\item If $D$ satisfies the Rockland condition, then there is a left parametrix $R\in \Psi^{-m}_{H,{\rm gr}}(X; \pmb{E}_2,\pmb{E}_1)$, i.e. $RD-1$ is smoothing. In particular, $D$ is hypoelliptic and $\ker(D)\subseteq C^\infty(X;E_1)$ is finite-dimensional.
\item If $D$ is $H$-elliptic, there is a parametrix $R\in \Psi^{-m}_{H,{\rm gr}}(X; \pmb{E}_2,\pmb{E}_1)$, i.e. $RD-1$ and $DR-1$ are smoothing. In particular, $D$ and $D^*$ are hypoelliptic and $\ker(D)\subseteq C^\infty(X;E_1)$ and $\ker(D^*)\subseteq C^\infty(X;E_2)$ are finite-dimensional.
\end{enumerate}
In particular, for any $D\in \Psi^m_{H,{\rm gr}}(X; \pmb{E}_1,\pmb{E}_2)$ the following are equivalent:
\begin{itemize}
\item $D$ is $H$-elliptic. 
\item There is an $R_0\in \Psi^{-m}_{H,{\rm gr}}(X; \pmb{E}_2,\pmb{E}_1)$ with $R_0D-1\in\Psi^{-1}_{H,{\rm gr}}(X; \pmb{E}_1)$ and $DR_0-1 \in\Psi^{-1}_{H,{\rm gr}}(X; \pmb{E}_2)$,
\item There is a parametrix $R\in \Psi^{-m}_{H,{\rm gr}}(X; \pmb{E}_2,\pmb{E}_1)$, i.e. $RD-1\in\Psi^{-\infty}(X; \pmb{E}_1)$ and $DR-1 \in\Psi^{-\infty}(X; \pmb{E}_2)$,
\end{itemize}
\end{theorem}

\begin{remark}
In the recent work \cite{androerp}, the Rockland condition was proven to be equivalent to \emph{maximal hypoellipticity}, proving a conjecture of Helffer-Nourigat. The results of \cite{androerp} hold in the context of the singular Carnot structures considered in Remark \ref{singularcarnotremark}. 
\end{remark}

\begin{example}
\label{connectingliealgtorepsmssm}
To tie together Proposition \ref{pitoleiielalad} with Proposition \ref{generalpitoleiielalad}, let us make a computation with the represented symbol of a differential operator. Assume that $X$ is a Carnot manifold and $D\in \mathcal{DO}_H^{m}(X;E_1,E_2)$. To simplify notation, set $n_j:=\mathrm{rk}(T^{-j}X/T^{-j+1}X)$ and $n:=\dim(X)=\sum_j n_j$. Fix an $x\in X$ and set $\mathsf{G}:=T_HX_x$. In a neighbourhood $U$ of $x$, we pick a basis $\{\{X_{j,k}\}_{k=1}^{n_j}\}_{j=1}^r$ of $TX$ such that $\{X_{j,k}\}_{k=1}^{n_j}\subseteq C^\infty(X,T^{-j}X)$ and its images in $C^\infty(X,T^{-j}X/T^{-j+1}X)$ spans $T^{-j}X/T^{-j+1}X$ in all points. For a multiindex $\alpha=(\alpha_1,\ldots, \alpha_r)\in \N^n=\prod_j \N^{n_j}$ we write $|\alpha|_H:=\sum_j j |\alpha_j|$ and 
$$X^\alpha=X_{1,1}^{\alpha_{1,1}}X_{1,2}^{\alpha_{1,2}}\cdots X_{r,n_r-1}^{\alpha_{r,n_r-1}}X_{r,n_r}^{\alpha_{r,n_r}}.$$
We can assume the neighbourhood $U$ is small enough so that $E_1$ and $E_2$ trivializes, set $k_1:=\mathrm{rk}(E_1)$ and $k_2:=\mathrm{rk}(E_2)$. Therefore, in $U$ we can write 
$$D=\sum_{|\alpha|_H\leq m} a_\alpha X^\alpha,$$
for some $(a_\alpha)_{|\alpha|_H\leq m}\subseteq C^\infty(U;\Hom(\C^{k_1},\C^{k_2}))$. As in Example \ref{compaoaoadoadl}, $\sigma^m(D)$ is given by the convolution kernel $K\in \mathcal{S}'_r(T_HX;E_1,E_2)$ that over $U$ is determined by $\sum_{|\alpha|_H= m} a_\alpha X^\alpha$ as acting only in the fibre direction. In light of Proposition \ref{pitoleiielalad}, we compute that for a unitary representation $\pi$ we have that 
$$\sigma^m_{(x,\pi)}(D)=\sum_{|\alpha|_H= m} a_\alpha(x) \pi(X^\alpha_x)\in \mathcal{L}(\mathcal{S}_0(\pi)\otimes \C^{k_1},\mathcal{S}_0(\pi)\otimes \C^{k_2}),$$
where by construction $\{\{X_{j,k,x}\}_{k=1}^{n_j}\}_{j=1}^r$ is a homogeneous basis of the Lie algebra $\mathfrak{g}=\mathfrak{t}_HX_x$ of $\mathsf{G}:=T_HX_x$.

\end{example}

\subsection{Revisiting the examples of Subsection \ref{examplesubsechoperaoro}}

We now return to study $H$-ellipticity in the specific examples appearing above in Subsection \ref{examplesubsechoperaoro}.

\begin{example}[Baum-van Erp type operators]
\label{bveops2}
We return to study Baum-van Erp type operators, and to give more precise descriptions of $H$-ellipticity of the formal expressions appearing in Subsubsection \ref{bveops1}. In light of  the terminology of Example \ref{connectingliealgtorepsmssm}, a Baum-van Erp type operator is an $H$-elliptic differential operator $D_\gamma\in \mathcal{DO}_H^{m}(X;E_1,E_2)$ such that each point in $X$ has a neighbourhood $U$ with
\begin{equation}
\label{dgammaamdladladl}
D_\gamma=\sum_{j=1}^r\sum_{k=1}^{n_j} \gamma_{j,k}\nabla_{X_{j,k}}^{m/j}+ \mathcal{DO}_H^{m-1}(U;E_1,E_2),
\end{equation}
for some local basis $\{\{X_{j,k}\}_{k=1}^{n_j}\}_{j=1}^r$ of $TX$ as in Example \ref{connectingliealgtorepsmssm}, a connection $\nabla$ on $E_1$ and a collection $\gamma=(\gamma_{j,k})_{j,k}\subseteq C^\infty(U;\Hom(E_1,E_2))$ such that $j|m$ when $\gamma_{j,k}\neq 0$. It would be quite arduous to in general characterize when a differential operator $D_\gamma$ satisfying the condition \eqref{dgammaamdladladl} is $H$-elliptic. We consider the following special case.

Assume that $X$ has depth $r=2$ and is equipped with a Riemannian metric. For notational simplicity, write 
$$H:=T^{-1}X.$$
Let $E\to X$ be a hermitean vector bundle. Consider a differential operator as in \eqref{dgammaamdladladl} of the form 
$$D_\gamma=\Delta_H+T_\gamma,$$
where $\Delta_H$ is the sub-laplacian associated with the Riemannian metric and $T_\gamma$ is a first order differential operator (first order in the classical sense). Write $\mathfrak{t}_{H}X=\mathfrak{t}_{H,-1}X\oplus \mathfrak{t}_{H,-2}X=H\oplus \mathfrak{t}_{H,-2}X$ and identify 
$$\mathfrak{t}_{H,-2}X^*=(TX/T^{-1}X)^*=H^\perp.$$
Define the fibrewise linear polynomial 
$$\gamma:=\sigma^1(T_\gamma)|_{H^\perp}\in C^\infty(H^\perp,\End(E)).$$ 
A short argument shows that $\sigma^2_H(D_\gamma)$ is determined by the metric on $H$ and $\gamma$. For notational simplicity, we write a graded orthonormal local basis for $TX$ as $X_1,\ldots, X_{n_1}$, $Z_1,\ldots Z_{n_2}$ where $X_1,\ldots, X_{n_1}$ is an orthonormal basis of $H$ and $Z_1,\ldots Z_{n_2}$ induces an orthonormal basis for $TX/H$. The order $2$ differential operator $D_\gamma\in \mathcal{DO}_H^{2}(X;E)$ will in each coordinate chart look like
\begin{equation}
\label{secondoroddbbamddevep}
D_\gamma=-\sum_{j=1}^{n_1} X_j^2+\sum_{l=1}^{n_2} \gamma_lZ_l+\mathcal{DO}_H^{1}(X;E),
\end{equation}
for the collection $(\gamma_l)_{l=1}^{n_2}$ of smooth $\End(E)$-valued sections determined by 
$$\gamma=\sigma^1\left(\sum_{l=1}^{n_2} \gamma_lZ_l\right)\bigg|_{H^\perp}.$$ 
In other words, up to lower order terms, the term $\sum_{l=1}^{n_2} \gamma_lZ_l$ is determined by the restriction of its ordinary symbol to cotangent vectors annihilating $H=T^{-1}X$. Let us give result characterizing those $\gamma\in C^\infty(\mathfrak{t}_{H,-2}X^*,\End(E))$ for which $D_\gamma$ is $H$-elliptic and as such defines a Baum-van Erp type operator.

For the flat orbits, we introduce some further notations. Write $\Gamma_X$ for the set of flat orbits of the osculating Lie groupoid, and $p_\Gamma:\Gamma_X\to X$ for the projection. If $X$ is $\pmb{F}$-regular, $\Gamma_X$ is a fibre bundle but in general it is just a topological space (possibly empty). The vector bundle $\Xi_X=p_\Gamma^*H\to \Gamma_X$ carries a symplectic structure defined from the Kirillov form $\omega$. Relative to a metric on $H$, we can associate a new metric $g_\omega$ on $\Xi_X$ and a complex structure $J_\omega$ on $\Xi_X$. Using the auxiliary metric, we can identify $g_\omega$ with a symmetric complex linear endomorphism of $\Xi_X$ and $\omega$ with an anti-symmetric complex antilinear endomorphism of $\Xi_X$. The two are related by $g_\omega=|\omega|=\sqrt{-\omega^2}$. We define the endomorphism 
\begin{equation}
\label{skgoadoma}
s_k(g_\omega):\Xi_X^{\otimes_\C^{\rm sym}k}\to \Xi_X^{\otimes_\C^{\rm sym}k},
\end{equation}
on the $k$:th symmetric tensor power, from 
$$s_k(g_\omega)(v_1\otimes_\C^{\rm sym}\cdots \otimes_\C^{\rm sym}v_k):=\sum_{j=1}^kv_1\otimes_\C^{\rm sym}\cdots\otimes_\C^{\rm sym}g_\omega v_j \otimes_\C^{\rm sym}\cdots\otimes_\C^{\rm sym}v_k.$$

\begin{theorem}
\label{charhellfofodo}
Let $X$ be a Carnot manifold of depth $2$ defined from a sub-bundle $H=T^{-1}X$ with a Riemannian metric $g$ and $D_\gamma\in \mathcal{DO}_H^{2}(X;E)$ as in Equation \eqref{secondoroddbbamddevep}. Let $\mathcal{L}:H\wedge H\to TX/H$ denote the Levi bracket and $\omega_\xi:=\xi\circ\mathcal{L}$ the associated antisymmetric $2$-form on $H$. Then $D_\gamma$ is $H$-elliptic if and only if for any $x\in X$, and $\xi\in \mathfrak{t}_{H,-2}X^*_x\setminus \{0\}=H_x^{\perp}\setminus \{0\}$ the following holds
\begin{itemize}
\item If $\omega_\xi$ has rank $k$, for some $k<n_1$, then the following spectral condition holds
$$\mathrm{Spec}(\gamma(\xi))\cap \left\{\lambda\in \R: \;|\lambda|\geq\mathrm{Tr}(|\omega_\xi|)/2\right\}=\emptyset,$$ 
where $|\omega_\xi|$ is defined from polar decomposition of $\omega_\xi$ relative to the metric on $H$.
\item If $\omega_\xi$ has full rank (i.e. is a symplectic form on $H$), then 
\begin{align}
\label{rankconditions}
s_k(g_\omega(\xi))+\frac{\mathrm{Tr}(|\omega_\xi|)}{2}+\gamma(\xi):\Xi_X^{\otimes_\C^{\rm sym}k}\otimes E_x\to \Xi_X^{\otimes_\C^{\rm sym}k}\otimes E_x,
\end{align}
also have the full ranks for all $k\in \N$.
\end{itemize}
\end{theorem}

The proof of the proposition is in spirit to reduce to rays in the representation space and there use that the sub-Laplacian on the Heisenberg group has a well understood represented symbol. The characterization of Theorem \ref{charhellfofodo} is by \cite[Theorem 1.1]{horcharnice} equivalent to hypoellipticity with loss of one derivative. 

\begin{proof}
Fix an $x\in X$. We write $\mathfrak{g}:=\mathfrak{t}_{H}X_x$ and $\mathsf{G}:=T_HX_x$. We need to characterize when $\sigma^2_{(x,\pi)}(D_\gamma)$ is bijective for a non-trivial representations $\pi\in \hat{\mathsf{G}}$. Write $\pi_\xi$ for a unitary representation associated with $\xi\in \mathfrak{g}^*\setminus \{0\}$ via the orbit method. Note that $\mathfrak{g}_{-2}$ is a subspace of the centre of $\mathfrak{g}$. If $\xi\in \mathfrak{g}_{-2}^\perp\setminus \{0\}$ then $\pi_\xi$ is a character and so $\sigma^2_{(x,\pi_\xi)}(D_\gamma)=\sum_j (\xi(X_j))^2>0$ is bijective. We can therefore assume that $\xi|_{\mathfrak{g}_{-2}}\neq 0$. 

If $\omega_\xi=\xi\circ \mathcal{L}=0$, then $\mathfrak{g}$ is abelian and 
$$\sigma^2_{(x,\pi_\xi)}(D_\gamma)=\sum_j (\xi(X_j))^2+\gamma(\xi).$$
Using that $\gamma$ is linear in $\xi$, we see that $\sigma^2_{(x,\pi_\xi)}(D_\gamma)$ is bijective for all $\xi$ if and only if $\lambda+\gamma(\xi)$ is invertible for all $\lambda\geq0$ and all $\xi\neq 0$. Similarly, $\sigma^2_{(x,\pi_\xi)}(D_\gamma^*)$ is bijective for all $\xi$ if and only if $\lambda-\gamma(\xi)$ is invertible for all $\lambda\geq0$ and all $\xi\neq 0$. 

For the remainder of the proof, consider a $\xi$ with $\xi|_{\mathfrak{g}_{-2}}\neq 0$ and $\omega_\xi\neq 0$. Consider the Lie algebra $\mathfrak{g}'_\xi:=(\mathfrak{g}_{-1})_{\R\omega_\xi}$ (cf. Subsection \ref{step2subsectionexamples} for notation) and grade $\mathfrak{g}'_\xi$ by declaring the natural quotient map $\mathfrak{g}\to \mathfrak{g}'_\xi$ to be graded. Define the model operator 
$$L_{\gamma,\xi}:=-\sum_{j=1}^{n_1} X_j^2-i\gamma(\xi) Z\in \mathcal{U}_2(\mathfrak{g}')\otimes \End(E_x),$$
where $Z$ denotes the canonical basis for $\mathfrak{g}'_{-2}$. Note that $L_{\gamma,\xi}^*=L_{-\gamma^*,\xi}$. It is clear that $\sigma^2_{(x,\pi_\xi)}(D_\gamma)$ and $\sigma^2_{(x,\pi_\xi)}(D_\gamma^*)$ are injective if and only if $\pi(L_{\gamma,\xi})$ and $\pi(L_{-\gamma^*,\xi})$ are injective for the non-abelian representations $\pi$ of $\mathfrak{g}'_\xi$ defined from $\xi$. Pick a complex structure on $\mathfrak{g}_{-1}$ compatible with $\omega_\xi$ and the Riemannian metric. Using the associated polarization of $\mathfrak{g}'$, we can write $\pi(L_{\gamma,\xi})$ as the densely defined operator
$$\pi(L_{\gamma,\xi})=H_{d,\omega_\xi}+\lambda_0(\xi)+\gamma(\xi),$$
on $L^2(\R^d,E_x)$ where $2d$ is the rank of $\omega_\xi$, $H_{d,\omega_\xi}$ is the harmonic oscillator defined from the metric and the symplectic form on $\R^{2d}\cong \mathfrak{g}_{-1}/\mathrm{Ann}(\omega_\xi)$ and $\lambda_0(\xi)$ is a homogeneous degree two polynomial in the projection of $\xi$ onto $\mathrm{Ann}(\omega_\xi)$. 

The bottom of the spectrum of $H_{d,\omega_\xi}$ is $\mathrm{Tr}(|\omega_\xi|)/2$. If $k<n_1$, i.e. $\omega_\xi$ is degenerate, using that $\gamma$ is linear in $\xi$ and $\lambda_0$ is quadratic, we see that $\pi(L_{\gamma,\xi})$ is injective for all $\xi$ if and only if $\lambda+\gamma(\xi)$ is invertible for all $\lambda\geq\mathrm{Tr}(|\omega_\xi|)/2$ for all $\xi$. Therefore $\pi(L_{\gamma,\xi})$ and $\pi(L_{-\gamma^*,\xi})$ are injective for all $\xi$ if and only if $\lambda+\gamma(\xi)$ is invertible for all real $\lambda$ with $|\lambda|\geq\mathrm{Tr}(|\omega_\xi|)/2$. 

If $\omega_\xi$ has full rank, so $2d=\dim(\mathfrak{g}_{-1})$, $\pi(L_{\gamma,\xi})=H_{d,\omega_\xi}+\gamma(\xi)$. Therefore, injectivity of $\pi(L_{\gamma,\xi})$ and $\pi(L_{-\gamma^*,\xi})$ is equivalent to the condition \eqref{rankconditions} by the standard construction with creation/annihilation operators and regularity theory in the Shubin calculus on $\R^d$ (cf. \cite[Chapter IV]{shubinbook}).
\end{proof}

\begin{remark}
Let $D_\gamma$ be as in Theorem \ref{charhellfofodo} and assume that $D_\gamma$ is $H$-elliptic. It follows from Theorem \ref{charhellfofodo} that if the Levi bracket of a step $2$ Carnot manifold is non-zero but degenerate in all points, there is a smooth homotopy from $D_\gamma$ to a positive differential operator through $H$-elliptic operators. The structure of the characterization in Theorem \ref{charhellfofodo} indicates that the index theory of operators of the form $D_\gamma$ is found in the flat orbits, i.e. when the form $\omega_\xi$ induced from the Kirillov form is non-degenerate after quotienting out the center so the represented symbol of $D_\gamma$ has discrete spectrum.
\end{remark}

\end{example}

\begin{example}[Dirac type operators]
\label{ex:Diracex2}
Consider a Dirac type operator $\slashed{D}_H=\sum_{j=1}^{\dim \mathfrak{g}_r} i\gamma_j X_{j}$, for a basis $X_1,\ldots X_{\dim(\mathfrak{g}_{-1})}$ and Clifford matrices $\gamma_1,\ldots, \gamma_{\dim(\mathfrak{g}_{-1})}$, as in Subsubsection \ref{ex:Diracex1}. By the argument in \cite[Example 1.6.3]{mohsen2}, the operator $\slashed{D}_H$ is not $H$-elliptic unless we are dealing with a step $1$ group so $\mathfrak{g}=\mathfrak{g}_{-1}$. 

Let us give more details describing the failure of $H$-ellipticity in the special case of the three dimensional Heisenberg group $\mathbb{H}_1$. Write $X,Y,Z$ for the standard basis such that $[X,Y]=Z$ graded as $\mathfrak{h}_{1,-1}=\R X+\R Y$ and $\mathfrak{h}_{1,-2}=\R Z$. In a suitable choice of Clifford matrices, 
$$\slashed{D}_H=
\begin{pmatrix}
0& X+iY\\
-X+iY&0 \end{pmatrix}.$$
The operator $\slashed{D}_H$ is formally self-adjoint, so by Lemma \ref{charHELLrockslslalem} the operator $\slashed{D}_H$ is $H$-elliptic if and only if it satisfies the Rockland condition, which readily is seen to be equivalent to $\slashed{D}_H^2$ satisfying the Rockland condition. We compute that 
$$\slashed{D}_H^2=
\begin{pmatrix}
-X^2-Y^2+iZ&0\\
0&-X^2-Y^2-iZ \end{pmatrix}.$$
So in the Schrödinger representations parametrized by $\hbar\in \mathbb{R}^\times$, defined from $\xi_\hbar\in \mathfrak{h}_1^*$ given by $\xi_\hbar(xX+yY+zZ):=\hbar z$, it holds that 
$$\pi_\hbar(\slashed{D}_H^2)=\begin{pmatrix}
H_\hbar-\hbar&0\\
0&H_\hbar+\hbar \end{pmatrix},$$
where $H_\hbar$ is the harmonic oscillator whose spectrum is $(2\N+1)|\hbar|$, so injectivity of $\pi_\hbar(\slashed{D}_H^2)$ fails in the upper left corner for $\hbar>0$ and in the lower right corner for $\hbar<0$.

A potential fix for this problem is to correct $\slashed{D}_H$ by the failure of the Rockland condition. As we shall see in the Example \ref{ex:szegoexampe2} below, there is a self-adjoint order $0$ Carnot operator $P$ whose represented symbol $p_\hbar:=\pi_\hbar(P)$ is the projection onto the ground state of $H_\hbar$ for $\hbar\neq 0$, respectively, so for $\epsilon>0$ small enough the operator 
$$\slashed{D}_{\epsilon,H}=
\begin{pmatrix}
0& X+iY+\epsilon\sqrt{-X^2-Y^2}P\\
-X+iY+\epsilon\sqrt{-X^2-Y^2}P&0 \end{pmatrix},$$
is $H$-elliptic. Here $\sqrt{-X^2-Y^2}$ can made sense of using results of \cite{Dave_Haller2} on complex powers in the Carnot calculus, or \cite{pongemonograph}, reviewed below as Theorem \ref{cpxpoweroieoso}. These type of constructions were considered more generally in \cite{hasselmannthesis}. In combination with commutator estimates, these ideas were applied in \cite{gimpgoffcomm} to realize the Carnot-Caratheodory distance on contact manifolds as a Gromov-Hausdorff limit of a sequence Connes' spectral distances associated with a family of spectral triples defined from $H$-elliptic operators.
\end{example}

\begin{example}[A positive even order Carnot differential operator]
\label{posoofosos2}
Consider a hermitean vector bundle $E\to X$ over a compact Carnot manifold $X$ with a volume density and the differential operator 
$$\mathfrak{D}:=\sum_{j=1}^r \sum_{k=1}^{N_j} (\nabla_{X_{j,k}}^*)^{r!/j}\nabla_{X_{j,k}}^{r!/j}\in \mathcal{DO}_H^{2(r!)}(X;E),$$
from Subsubsection \ref{posoofosos1}. Recall that the vector fields span the osculating Lie groupoid in all points, and are choose so that for $x\in X$ 
$$\sigma^{2(r!)}_x(\mathfrak{D})=\sum_{j=1}^r \sum_{k=1}^{N_j} (-1)^{r!/j} X_{j,k,x}^{2(r!)/j}\in \mathcal{U}_{2(r!)}(\mathfrak{t}_HX_x)\otimes \End(E_x).$$
Let us prove that $\mathfrak{D}$ is $H$-elliptic. Since $\mathfrak{D}$ is formally self-adjoint, it suffices to prove the Rockland condition by Lemma \ref{charHELLrockslslalem}. Fix $x\in X$ and choose a non-trivial representation $\pi$ of $T_HX_x$. Take the smallest $r_0\leq r$ such that $\pi$ vanishes on the image of $T_xX/T^{-r_0}_xX$ in $\mathfrak{t}_HX_x$. After re-ordering, we can assume that $\pi(X_{r_0,1})$ acts as a non-zero imaginary scalar $i\lambda$. In particular, 
$$\sigma^{2(r!)}_{(x,\pi)}(\mathfrak{D})=\sum_{j=1}^{r_0-1} \sum_{k=1}^{N_j} (-1)^{r!/j} X_{j,k,x}^{2(r!)/j}+\sum_{k=2}^{N_j} (-1)^{r!/r_0} X_{r_0,k,x}^{2(r!)/r_0}+\lambda^{2(r!)/r_0}.$$
We have for any $\varphi\in \mathcal{S}_0(\pi)$ that 
\begin{align*}
\langle \sigma^{2(r!)}_{(x,\pi)}(\mathfrak{D})\varphi,\varphi\rangle_{\mathcal{H}_\pi}=&\sum_{j=1}^{r_0-1} \sum_{k=1}^{N_j} \|X_{j,k,x}^{(r!)/j}\varphi\|_{\mathcal{H}_\pi}^2+\sum_{k=2}^{N_j} \|X_{r_0,k,x}^{(r!)/r_0}\varphi\|_{\mathcal{H}_\pi}^2+\lambda^{2(r!)/r_0}\|\varphi\|_{\mathcal{H}_\pi}^2\geq\\
\geq & \lambda^{2(r!)/r_0}\|\varphi\|_{\mathcal{H}_\pi}^2.
\end{align*}
Since $\lambda^{2(r!)/r_0}>0$, $\sigma^{2(r!)}_{(x,\pi)}(\mathfrak{D})$ is injective. Therefore $\mathfrak{D}$ is $H$-elliptic. We conclude the following upon adding a positive constant to $\mathfrak{D}$. 

\begin{proposition}
\label{existenceofposiskdd}
Let $X$ be a compact Carnot manifold of depth $r$ with a fixed volume density and $E\to X$ a hermitean vector bundle. Then for any $m$ with $r!|m$, there exists a positive even order $H$-elliptic differential operator $\mathfrak{D}\in \mathcal{DO}_H^{2m}(X;E)$ such that 
$$\langle \mathfrak{D}\varphi,\varphi\rangle_{L^2(X;E)}\geq \|\varphi\|_{L^2(X;E)}^2, \;\mbox{ for all $\varphi\in C^\infty(X;E)$}.$$
\end{proposition}

\end{example}

\begin{example}[BGG]
\label{ex:BBBBBBGGGGGG2}
In Subsubsection \ref{ex:BBBBBBGGGGGG1} above we discussed the BGG-complex and its generalizations in \cite{morecap,Dave_Haller1}. 
These notions were abstracted in Definition \ref{hcomplexdefdo} to that of a graded $H$-(almost-)complex. We now consider the corresponding notion of $H$-ellipticity. 

\begin{definition}
\label{rocklandseqinad}
Consider a graded $H$-almost-complex 
\begin{align*}
0\to C^\infty(X;\pmb{E}_1)\xrightarrow{D_1}&C^\infty(X;\pmb{E}_2)\xrightarrow{D_2}\cdots\\
&\cdots \xrightarrow{D_{N-1}}C^\infty(X;\pmb{E}_{N})\xrightarrow{D_{N}} C^\infty(X;\pmb{E}_{N+1})\to 0,
\end{align*}
of order $\pmb{m}=(m_1,\ldots, m_N)$. We say that this is a \emph{graded Rockland sequence} if for any $x\in X$ and non-trivial irreducible unitary representation $\pi$ of $T_HX_x$, we have that
$$\mathrm{im}(\sigma_{(x,\pi)}^{m_{j}}(D_{j}))\subseteq \ker(\sigma_{(x,\pi)}^{m_{j+1}}(D_{j+1})),$$ 
is dense for all $j$. Here we consider the represented symbols as operators $\mathcal{S}_0(\mathcal{H}_\pi)\otimes \pmb{E}_{j,x}\to \mathcal{S}_0(\mathcal{H}_\pi)\otimes \pmb{E}_{j+1,x}$.
\end{definition}

It was proven in \cite{Dave_Haller1} that the two sequences appearing in Example \ref{examplefromderhamadma} above are graded Rockland sequences. Therefore (curved) BGG-complexes are examples of graded Rockland sequences. 
\end{example}

\begin{example}[Szegö projections]
\label{ex:szegoexampe2}
Let us describe the Szegö projections discussed in Subsubsection \ref{ex:szegoexampe} further. Albeit the Szegö projection was originally defined on boundaries of domains in complex manifolds, and the existence in that case is a bit subtle, one can consider more general situations.

For instance, \cite[Chapter 6]{melroseeptein} considered generalized Szegö projections on a co-orientable contact manifold $X$. The co-orientability ensures the existence of complex structures on $H$, cf. Example \ref{contactexample}. A generalized Szegö projection at level $N$ is a projection $P\in \Psi_H^0(X,\C)$ such that in some complex structure on $H$ one has that $\sigma_0(P)\in \Sigma^0_H(X,\C)\subseteq \mathcal{E}'_r(T_HX)$ is represented in the Fock space bundle $\mathpzc{F}_X\cong \bigoplus_{k=0}^\infty p_\Gamma^*H^{*\otimes_\C^{\rm sym}k}$ as the orthogonal projection onto 
$$(\R_-\times X\times \{0\})\dot{\cup} \left(\R_+\times \bigoplus_{k=0}^N H^{*\otimes_\C^{\rm sym}k}\right)\subseteq \mathpzc{F}_X,$$ 
where we have trivialized $\Gamma_X\cong \R^\times \times X$. Examples of generalized Szegö projections at level $N$ were considered in \cite[Chapter 15.3]{bdmguille} and further studied in \cite{engliszhanghigher} in terms of solution spaces to a higher Cauchy-Riemann equation 
$$\underbrace{\bar{\partial}_{T^*X^{\otimes N}}\circ\cdots \circ\bar{\partial}_{T^*X}\circ\bar{\partial}}_{N+1 \mbox{ times}}u=0,$$
in $\Omega$ which is a strictly pseudoconvex domain with boundary $X$. 

We remark that an alternative description of a generalized Szegö projection is as projections $P\in \Psi_H^0(X,\C)$ whose symbol is represented in a fibre of $\mathpzc{F}_X\cong \bigoplus_{k=0}^\infty p_\Gamma^*H^{*\otimes_\C^{\rm sym}}$ over a point $(x,\xi)\in \Gamma_X$ as a finite rank spectral projection of the image of $\sigma_2(\Delta_H)$ in the representation, for the sub-Laplacian $\Delta_H$ defined from the contact structure, the complex structure and a choice of volume density. Based on this observation one could on a Carnot manifold $X$ and for an $H$-elliptic operator $D$, talk about generalized Szegö projectors (with respect to $D$) as projections $P\in \Psi_H^0(X,\C)$ whose symbol commutes with $\sigma^m(D_\gamma)$ and as a multiplier of $C^*(T_HX)$ in each fibrewise representation of the osculating Lie groupoid is a finite rank spectral projector for the represented image of $\sigma^m(D)$. 

For the sake of keeping the discussion concrete, we restrict the discussion to step $2$ Carnot manifolds. We use the notation $H=T^{-1}X$. Let us recall the following result of van Erp \cite{vanerpszego}. For a step $2$ Carnot manifold, we can define the ideal of Hermite operators $\mathcal{I}_H^m(X;\C)\subseteq \Psi^0_H(X;\C)$ as those operators whose symbol vanish on $H^\perp$. 

\begin{theorem}
\label{vanerpandszego}
Assume that $X$ is a step $2$ Carnot manifold. Then the following are equivalent:
\begin{enumerate}
\item There exists an idempotent $P\in \Psi^0_H(X;\C)$ with $\sigma_H^0(P)\neq 0,1$.
\item There exists an idempotent in $\Sigma^0_H(X;\C)$ which is neither $0$ nor $1$.
\item There exists an idempotent $P\in \mathcal{I}_H^m(X;\C)$ with $\sigma_H^0(P)\neq 0,1$.
\item $X$ is polycontact (see Example \ref{polycontactexama}).
\end{enumerate}
If $X$ is poly-contact, for any idempotent $P\in \Psi^0_H(X;\C)$ either $P\in \mathcal{I}_H^m(X;\C)$ or $1-P\in \mathcal{I}_H^m(X;\C)$ and any idempotent in $\Sigma^0_H(X;\C)$ lifts to an idempotent in $\Psi^0_H(X;\C)$. 
\end{theorem}

This theorem from \cite{vanerpszego} shows that in the step $2$ case, it is natural to assume the Carnot manifold to be polycontact. To give further examples, we restrict to regular polycontact manifolds. Let $X$ be a regular polycontact manifold of type $\mathfrak{g}=\mathfrak{g}_{-1}\oplus \mathfrak{g}_{-2}$ and let $H=T^{-1}X\subseteq TX$ denote the contact structure. We have that $\mathfrak{g}_{-2}=\mathfrak{z}$ is the centre and $\mathfrak{g}$ admits flat orbits since the Kirillov form $\omega_\xi$ is non-degenerate for all $\xi\in \mathfrak{z}$. See more in Example \ref{polycontactexama} and \ref{pluricontactexamokkoa}. We can choose a complex structure on $p_\Gamma^*H\to \Gamma_X$ adapted to the Kirillov form $\omega$ (defined by $\omega_\xi:=\xi\circ \mathcal{L}$ on $p_\Gamma^*H_\xi$). Moreover, the complex structure and the Kirillov form induces a Fock space bundle $\mathpzc{F}_X\to \Gamma_X$ (cf. Example \ref{pluricontactexamokkoa}) with $\mathpzc{F}_X\cong \bigoplus_{k=0}^\infty p_\Gamma^*H^{\otimes^{* \rm sym}_\C n}$. Given $N\geq 0$, consider the projection $p_{N}\in \mathcal{M}(I_X)$ defined the projection onto the direct summand 
$$\bigoplus_{k=0}^N p_\Gamma^*H^{\otimes^{* \rm sym}_\C n}\subseteq \mathpzc{F}_X.$$ 
If there exists a projection $P\in \Psi^0_H(X;\C)$ such that $\sigma^0_{(x,\pi)}(P)=\pi(p_{N,x})$ for any $x\in X$ and any flat representation $\pi$, then $P$ is a generalized Szegö projection; we say that such a generalized Szegö is of level $N$. While our construction of $p_{N}$ formally makes sense for more general pluricontact manifolds, it can fail that $p_{N,+}\in \Sigma^0_H(X;\C)$. For instance, $p_{N}\notin \Sigma^0_H(X;\C)$ when $\mathfrak{g}=\mathbb{H}_{n_1}\times \mathbb{H}_{n_2}$ and so when $X$ is a product of contact manifolds. \\

{\bf Szegö projections on co-oriented contact manifolds}
We describe the situation for co-oriented contact manifolds in full detail. Let $X$ be a $2n+1$-dimensional co-oriented contact manifold. We write coordinates in the Heisenberg group $\mathbb{H}_n$ as $(z,t)\in \C^n\times \R$ with the product 
$$(z,t)(w,s)=(z+w,s+t+\mathrm{Im}\langle z,w\rangle).$$ 
In this convention, the complex structure associated with the Kirillov form and the standard (real) metric on $\C^n$ is the standard complex structure on $\C^n$. In this case, a computation with the Fourier transform (cf. \cite[Chapter XII]{steinreal}) shows that a Szegö projection on $X$ has Carnot symbol $p_0\in \Sigma^0_H(X;\C)$ which in local coordinates takes the form of the $0$-homogeneous distributional density
$$p_0(z,t):=c\frac{|\rd z\rd t|}{(it+|z|^2)^{n+1}},$$
for a suitable dimensional constant $c$. As a homogeneous distribution, $p_0$ is uniquely determined outside $(z,t)=0$. The behaviour of $p_0$ at $(z,t)=0$ is defined from a limiting procedure in $t$ in the upper complex halfplane; this definition of $p_0$ at $(z,t)=0$ ensures that $\pi(p_0)=0$ for $\pi \in \Gamma\setminus \Gamma_+$. By homogeneity arguments, a generalized Szegö projection of level $N$ (in the sense of Epstein-Melrose \cite{melroseeptein}) has the Carnot symbol $p_N\in \Sigma^0_H(X;\C)$ of the form 
$$p_N(x,z,t):=\frac{q_N(t,z)|\rd z\rd t|}{(it+|z|^2)^{n+N+1}},$$
for a homogeneous polynomial $q_N$ of degree $N$. The reader should be aware that these are pointwise expressions in the sense that they implicitly depend on the isomorphism $T_HX_x\cong \mathbb{H}_n$.\\

{\bf Szegö projections on general contact manifolds}
Consider now a $2n+1$-dimensional contact manifold $X$, but we make no assumption on co-orientability. In particular, the bundle $H^\perp \to X$ need not be oriented and $H$ need not allow a complex structure. Also in this case, a computation with the Fourier transform as above shows that there is a Szegö projection on $X$, constructed using Theorem \ref{vanerpandszego}, from a Carnot symbol $p_{tw}\in \Sigma^0_H(X;\C)$ that in local coordinates takes the form of the $0$-homogeneous distributional density
$$p_{tw}(z,t):=c\frac{|\rd z\rd t|}{(it+|z|^2)^{n+1}}+c\frac{|\rd z\rd t|}{(-it+|z|^2)^{n+1}}.$$
for a suitable dimensional constant $c$. As a homogeneous distribution, $p_{tw}$ is uniquely determined outside $(z,t)=0$ and the behaviour of $p_{tw}$ at $(z,t)=0$ is defined from a suitable limiting procedure. A suitable limiting procedure ensures that $p_{tw}$ acts as the projection onto the ground state bundle $X\times \C=(p_\Gamma^*H)^{\otimes_\C^{\rm sym}0}\subseteq \mathpzc{F}_X$. Geometrically, a lift $P\in \Psi_H^0(X;\C)$ could modulo lower order terms for $X$ being the boundary of a strictly pseudoconvex domain be interpreted as the projection onto the direct sum of the spaces of functions with a holomorphic extension and anti-holomorphic extension, respectively, to the interior of the domain.\\

{\bf Szegö projections on regular Carnot manifolds of type $\mathbb{H}_{n,\C}$}
Consider the polycontact manifold of  type $\mathbb{H}_{n,\C}$ constructed as in Example \ref{polycontactexama} from a compact complex manifold $X$ and a complex subbundle  $H\subseteq TX$ of complex codimension $1$ such that the the Levi bracket $H\times H\to TX/H$ is non-degenerate. Write $2n+1=\dim_\C(X)$. We write coordinates on $\mathbb{H}_{n,\C}$ as $(z,t)\in \C^{2n}\times \C$. A computation with the Fourier transform (in the central direction) and homogeneity considerations shows that the that the projection $p_\C$ onto the ground state bundle $X\times \C=(p_\Gamma^*H)^{\otimes_\C^{\rm sym}0}\subseteq \mathpzc{F}_X$ is an element $p_\C\in \Sigma_H^0(X;\C)$ that in local coordinates takes the form 
$$p_\C(z_1,z_2,t):=\sum_{k=0}^nc_k|z|^2|t|^{2k}\frac{|\rd z\rd t|}{(|t|^2+|z|^4)^{n+k-3/2}},$$
for suitable dimensional constants $c_0,c_1,\ldots, c_n$. The geometric interpretation of the idempotent $p_\C$ is less clear, but at a symbolic level $p_\C$ is the projection onto the kernel of the symbol of an operator of the form 
$$\Box_H=(\bar{\partial}_H^*\bar{\partial}_H)^2-T^*T,$$
where $\bar{\partial}_H:C^\infty(X)\to C^\infty(X;(H^{0,1})^*)$ is the $\bar{\partial}$-operator restricted to $H$ and $T$ is a suitable vector field transversal to $H$. 
\end{example}

\subsection{Analytic properties of $H$-elliptic operators}
\label{section:analyticprops}

To study the analytic properties of $H$-elliptic operators with ease, we will describe their actions on a natural scale of Sobolev spaces on a Carnot manifold. Such Sobolev spaces has been studied in for instance \cite{Dave_Haller1}. We model the setup close to that of an ordinary manifold in \cite{shubinbook}. First, we recall the following result from \cite{Dave_Haller1}.

\begin{proposition}[Proposition 3.9 \cite{Dave_Haller1}]
\label{bodlaldaladldalpa}
Let $X$ be a compact Carnot manifold with a fixed volume density, $E_1,E_2\to X$ hermitean vector bundles and $m\in \R$. For $D \in \Psi_H^m(X;E_1,E_2)$ it holds that
\begin{enumerate}
    \item If $\mathfrak{R}(m) \leq 0$, $D$ is bounded in $L^2$-norm and extends uniquely to a bounded operator $D : L^2(X,E_1) \rightarrow L^2(X,E_2)$.
    \item If $\mathfrak{R}(m) < 0$, then we have a compact operator $D : L^2(X,E_1) \rightarrow L^2(X,E_2)$.
    \item If $\mathfrak{R}(m) < -\dim_h(X)$, then $D : L^2(X,E_1) \rightarrow L^2(X,E_2)$ is trace class.
\end{enumerate}
\end{proposition}

Let $X$ be a compact Carnot manifold equipped with a volume density, $E\to X$ a hermitean vector bundle and $s\in \R$. Following \cite{Dave_Haller1}, we define the Sobolev space
$$W^s_H(X;E):=\{f\in \mathcal{D}'(X;E): Af\in L^2(X;E)\quad\forall A\in \Psi^s_H(X;E)\}.$$
In \cite{Dave_Haller1}, this space was denoted by $H^s(E)$. Using similar constructions as in \cite[Chapter I]{shubinbook}, one can equip $W^s_H(X;E)$ with a Hilbert space structure (albeit a rather non-canonical one). For a graded hermitean vector bundle $\pmb{E}\to X$, and $s\in \R$, we define the graded Sobolev space
$$\pmb{W}^s_{H,{\rm gr}}(X;\pmb{E}):=\bigoplus_j W^{s-j}_H(X;\pmb{E}[j]).$$
By construction, the actions of Carnot pseudo-differential operators
$$\Psi_H^m(X;E_1,E_2)\times W^s_H(X;E_1)\to W^{s-m}_H(X;E_2),$$
and graded Carnot pseudo-differential operators
$$\Psi_{H,{\rm gr}}^m(X;\pmb{E}_1,\pmb{E}_2)\times \pmb{W}^s_{H,{\rm gr}}(X;\pmb{E}_1)\to \pmb{W}^{s-m}_{H,{\rm gr}}(X;\pmb{E}_2),$$
are continuous. We shall make use of another description of Sobolev spaces, based on complex powers. We recall the following result from \cite{Dave_Haller2}.

\begin{theorem}[{\cite[Theorem 2]{Dave_Haller2}}]
\label{cpxpoweroieoso}
Let $X$ be a compact Carnot manifold with a fixed volume density and $E\to X$ a hermitean vector bundle. Suppose that $\mathfrak{D}\in \mathcal{DO}_H^{2m}(X;E)$ is a positive even order $H$-elliptic differential operator.  
Then $\mathfrak{D}$ is essentially self-adjoint on $L^2(X;E)$ and for $z\in \C$ the complex power $\mathfrak{D}^{-z}$ is a Carnot pseudodifferential operator of order $-zm$.
\end{theorem}

The reader can find more details on Carnot pseudodifferential operator of complex order in \cite[Section 6]{Dave_Haller2}. Only real powers will be used in this monograph. We remark that the essential self-adjointness of $\mathfrak{D}$ is automatic since $\mathfrak{D}$ is positive in form sense on $C^\infty(X;E)$. Therefore, the existence of the complex powers $\mathfrak{D}^{-z}$ is ensured by the spectral theorem. The importance of Theorem \ref{cpxpoweroieoso} is that the complex powers can be constructed inside the Carnot calculus. The existence of operators $\mathfrak{D}$ satisfying the assumptions of Theorem \ref{cpxpoweroieoso} is ensured by Proposition \ref{existenceofposiskdd}.

\begin{proposition}
Let $X$ be a compact Carnot manifold with a fixed volume density, $E\to X$ a hermitean vector bundle and $\mathfrak{D}\in \mathcal{DO}_H^{2m}(X;E)$ as in Proposition \ref{existenceofposiskdd}. Let $s\in \R$.

Consider the subspace $\mathfrak{D}^{-s/2m}L^2(X;E)\subseteq \mathcal{D}'(X;E)$ viewed as a Hilbert space by declaring $\mathfrak{D}^{s/2m}:\mathfrak{D}^{-s/2m}L^2(X;E)\to L^2(X;E)$ to be a unitary. Then the identity map on $C^\infty(X;E)$ extends by continuity to a Banach space isomorphism $\mathfrak{D}^{-s/2m}L^2(X;E)\cong W^s_H(X;E)$.
\end{proposition}

Henceforth, we consider $W^s_H(X;E)$ as a Hilbert space in the inner product 
$$\langle f,g\rangle_{W^s_H(X;E)}:=\langle \mathfrak{D}^{s/2m}f,\mathfrak{D}^{s/2m}g\rangle_{L^2(X;E)},$$
for some fixed choice of $\mathfrak{D}$ as in Proposition \ref{existenceofposiskdd} of some order $m$

The next result is proven as in \cite[Chapter I]{shubinbook}. It is a statement that entails the locality of the Sobolev spaces $W^s_H(X;E)$, i.e. $u\in W^s_H(X;E)$ if and only if it holds in each coordinate chart.

\begin{proposition}
Let $X$ be a compact Carnot manifold with a fixed volume density, $E\to X$ a hermitean vector bundle and $s\in \R$. Then for $u\in \mathcal{D}'(X;E)$ it holds that $u\in W^s_H(X;E)$ if and only if for any $x_0\in X$, there exists a neighborhood $U$ of $x_0$ and a $\varphi\in C^\infty_c(U)$ with $\varphi(x_0)\neq 0$ and $\varphi u\in W^s_H(X;E)$.
\end{proposition}

We now come to a result showing that index theory of $H$-elliptic operators is on a solid analytic footing. This type of result is standard in light of the methods above and the ideas in \cite[Chapter I]{shubinbook}, we include the statement for clarity and convenience of the reader.

\begin{theorem}
\label{indexindepofs}
Let $X$ be a compact Carnot manifold with a fixed volume density, $E_1,E_2\to X$ hermitean vector bundles, $s\in \R$ and $D \in \Psi_H^m(X;E_1,E_2)$ an $H$-elliptic operator. Then the following holds:
\begin{enumerate}[i)]
\item The operator $D_s:W^s_H(X;E_1)\to W^{s-m}_H(X;E_2)$ defined from $D$ is a continuous Fredholm operator with $\ker(D_s)\subseteq C^\infty(X;E_1)$ independent of $s$ and there exists a finite-dimensional space $V\subseteq C^\infty(X;E_2)$ (independent of $s$) such that the restricted quotient $V\to W^{s-m}_H(X;E_2)/DW^{s}_H(X;E_1)$ is an isomorphism. In particular, the Fredholm index $\ind(D_s)$ is independent of $s$. 
\item If $m>0$, the densely defined operator $D_{L^2}:L^2(X;E_1)\dashrightarrow L^2(X;E_2)$ defined from $D$ with domain $\mathrm{dom}(D_{L^2}):=W^m_H(X;E_1)$ is Fredholm and $\ker(D_{L^2})=\ker(D_s)\subseteq C^\infty(X;E_1)$ for any $s$ and $\coker(D_{L^2})=\coker(D^*_m)$. In particular, 
$$\ind(D_{L^2})=\ind(D_s),$$
for any $s\in \R$.
\item The index of $D_s$ and $D_{L^2}$ (if $m>0$) only depends on $\sigma_H(D)\in \Sigma^m_H(X;E_1,E_2)$.
\end{enumerate}
\end{theorem}

\begin{remark}
The Fredholm theory of graded $H$-elliptic operator as operators on graded Sobolev spaces is the same as the ungraded theory. That is, if $s\in \R$ and $D \in \Psi_{H,{\rm gr}}^m(X;\pmb{E}_1,\pmb{E}_2)$ is a graded $H$-elliptic operator, then 
\begin{enumerate}
\item $D_s:\pmb{W}^s_{H,{\rm gr}}(X;\pmb{E}_1)\to \pmb{W}^{s-m}_{H,{\rm gr}}(X;\pmb{E}_2)$ defined from $D$ is a continuous Fredholm operator with $\ker(D_s)\subseteq C^\infty(X;\pmb{E}_1)$ independent of $s$ and there exists a finite-dimensional space $V\subseteq C^\infty(X;\pmb{E}_2)$ (independent of $s$) such that the restricted quotient $V\to \pmb{W}^{s-m}_{H,{\rm gr}}(X;\pmb{E}_2)/DW^{s}_H(X;\pmb{E}_1)$ is an isomorphism. 
\item The Fredholm index $\ind(D_s)$ is independent of $s$ and only depends on $\sigma_{H,{\rm gr}}^m(D)\in \Sigma^m_{H,{\rm gr}}(X;\pmb{E}_1,\pmb{E}_2)$. 
\end{enumerate}
\end{remark}

\section[The action of $H$-elliptic operators]{The action of $H$-elliptic operators on certain Hilbert $C^*$-modules}
\label{secondnaaoaoac}

To study the index theory of $H$-elliptic operators, we need to better understand the action of related operators on certain Hilbert $C^*$-modules.

For a Carnot manifold $X$ and vector bundles $E,E'\to X$ consider the space
$$\pmb{\mathpzc{E}}^\infty_c(X;E):=C^\infty_c(\mathbb{T}_HX; r^*(E)).$$
There are actions by convolution 
\begin{align}
\label{somedifferentaciotosns}
\pmb{\mathpzc{E}}^\infty_c(X;E)\times C^\infty_c(\mathbb{T}_HX)\to \pmb{\mathpzc{E}}^\infty_c(X;E),\\
\nonumber
\mathcal{E}'_r(\mathbb{T}_HX;E,E')\times \pmb{\mathpzc{E}}^\infty_c(X;E)\to \pmb{\mathpzc{E}}^\infty_c(X;E').
\end{align}
If $E$ is hermitean, the convolution product induces a $C^\infty_c(\mathbb{T}_HX)$-sesqui-linear inner product 
$$\langle \cdot,\cdot\rangle_{\pmb{\mathpzc{E}}}: \pmb{\mathpzc{E}}^\infty_c(X;E)\times \pmb{\mathpzc{E}}^\infty_c(X;E)\to C^\infty_c(\mathbb{T}_HX).$$
We let $\pmb{\mathpzc{E}}(X;E)$ denote the completion of $\pmb{\mathpzc{E}}^\infty_c(X;E)$ as a $C^*(\mathbb{T}_HX)$-Hilbert $C^*$-module. We note that $C_0[0,\infty)$ act as central multipliers of $C^*(\mathbb{T}_HX)$, indeed $C^*(\mathbb{T}_HX)$ is a $C_0[0,\infty)$-$C^*$-algebra. This fact induces a natural $C_0[0,\infty)$-structure on $\pmb{\mathpzc{E}}(X;E)$ and any $T\in \End_{C^*(\mathbb{T}_HX)}^*(\pmb{\mathpzc{E}}(X;E))$ is $C_0[0,\infty)$-linear. We remark that for $t>0$ and $\pi_X$ denoting the regular representation of $X\times X$ (with respect to a fixed volume density), the localization of $\pmb{\mathpzc{E}}(X;E)$ in the associated representation $\pi_{X,t}:C^*(\mathbb{T}_HX)\to \mathbb{K}(L^2(X))$ takes the form
$$\pmb{\mathpzc{E}}(X;E)\otimes _{\pi_{X,t}}L^2(X)=L^2(X;E).$$
Moreover, for $t=0$ and any $x\in X$, a unitary representation $\pi$ of $T_HX_x$ induces a $\pi:C^*(\mathbb{T}_HX)\to \mathbb{K}(\mathcal{H}_\pi)$ and the localization of $\pmb{\mathpzc{E}}(X;E)$ takes the form
$$\pmb{\mathpzc{E}}(X;E)\otimes _{\pi}\mathcal{H}_\pi=\mathcal{H}_\pi\otimes E_x.$$

Any $\mathbb{D}\in  \pmb{\Psi}^{m}_H(X;E_1,E_2)$ acts by left convolution as a densely defined $C^*(\mathbb{T}_HX)$-linear operator $\pmb{\mathpzc{E}}(X;E_1)\to \pmb{\mathpzc{E}}(X;E_2)$ with domain $\pmb{\mathpzc{E}}^\infty_c(X;E)$.

\begin{definition}
Let $X$ be a compact Carnot manifold and $E_1,E_2\to X$ two vector bundles. An adjointable morphism $T\in \Hom_{C^*(\mathbb{T}_HX)}^*(\pmb{\mathpzc{E}}(X;E_1),\pmb{\mathpzc{E}}(X;E_2))$ is said to be locally $C^*(\mathbb{T}_HX)$-compact if for any $\varphi\in C_0[0,\infty)$ it holds that 
$$\varphi T\in \mathbb{K}_{C^*(\mathbb{T}_HX)}(\pmb{\mathpzc{E}}(X;E_1),\pmb{\mathpzc{E}}(X;E_2)).$$
\end{definition}

\begin{lemma}
\label{actionofbolddd}
In the setting of the previous paragraph, with $m\leq 0$, any $\mathbb{D}\in  \pmb{\Psi}^{m}_H(X;E_1,E_2)$ is closable as a densely defined operator $\pmb{\mathpzc{E}}(X;E_1)\to \pmb{\mathpzc{E}}(X;E_2)$. The closure of $\mathbb{D}$, that we by an abuse of notation denote by $\mathbb{D}$, is bounded and adjointable. Moreover, if $m<0$, $\mathbb{D}$ is locally $C^*(\mathbb{T}_HX)$-compact. 
\end{lemma}

\begin{proof}
To prove that  $\mathbb{D}$ is closable with adjointable and bounded closure, it suffices to prove that $\mathbb{D}$ is norm-bounded; in this case the Hilbert module adjoint of $\mathbb{D}$ is the closure of the formal adjoint of $\mathbb{D}$. To show that $\mathbb{D}$ is bounded, we note that it suffices to prove that there is a uniform norm bound on the localizations $\mathbb{D}\otimes_\pi 1$ where $\pi$ ranges over all irreducible representations of $C^*(\mathbb{T}_HX)$. The set of unitary equivalence classes of irreducible representations of $C^*(\mathbb{T}_HX)$ can be written as $\widehat{T_HX}\cup (0,\infty)$ consisting of representations over $t=0$ and the representations for $t>0$ of the pair groupoid $X\times X$. Since $m\leq 0$, the scaling invariance and Proposition \ref{bodlaldaladldalpa} gives a uniform norm bound on $(0,\infty)$. The uniform norm bound $\widehat{T_HX}$ follows from that for each $x$, $\mathrm{ev}_{t=0}(\mathbb{D})_x$ can be viewed as a properly supported compactly based operator on $T_HX_x$ and Proposition \ref{bodlaldaladldalpa} again produces a uniform estimate on $\widehat{T_HX}$.

To prove that  $\mathbb{D}$ is locally $C^*(\mathbb{T}_HX)$-compact for $m<0$, we can upon replacing $\mathbb{D}$ with $(\mathbb{D}^*\mathbb{D})^N$ assume that $m<<0$ can be taken arbitrarily negative. We can therefore, by \cite[Proposition 43]{vanErp_Yunckentangent}, assume that $\mathbb{D}\in \pmb{\Psi}^m_H(X;E_1,E_2)\cap C^k(\mathbb{T}_HX,r^*E_2\otimes (s^*E_1)^*\otimes |\Lambda_r|)$ for a large $k>>0$. Now, standard approximation arguments show that we on compact subsets of $[0,\infty)$, can approximate $\mathbb{D}$ in norm by $C^*(\mathbb{T}_HX)$-compact operators so $\mathbb{D}$ is locally $C^*(\mathbb{T}_HX)$-compact.
\end{proof}

We choose a $\mathfrak{D}_E\in \mathcal{DO}_H^{2m}(X;E)$ as in Proposition \ref{existenceofposiskdd}. For $s\in \R$, let $\pmb{\mathfrak{D}}_E^s\in \pmb{\Psi}^{2ms}_H(X;E)$ denote a strictly positive lift of the complex power $\mathfrak{D}_E^s\in \Psi_H^{2ms}(X;E)$ (cf. Proposition \ref{cpxpoweroieoso}). For $s>0$ we define 
$$\pmb{\mathpzc{E}}^s(X;E):=\pmb{\mathfrak{D}}_E^{-s/2m}\pmb{\mathpzc{E}}(X;E).$$
The space $\pmb{\mathpzc{E}}^s(X;E)$ is clearly $C^*(\mathbb{T}_HX)$-invariant and becomes a $C^*(\mathbb{T}_HX)$-Hilbert $C^*$-module by declaring $\pmb{\mathfrak{D}}_E^{s/2m}:\pmb{\mathpzc{E}}^s(X;E)\to \pmb{\mathpzc{E}}(X;E)$ to be unitary. For $s<0$ we can formally define $\pmb{\mathpzc{E}}^s(X;E)$ in the same way. Equivalently, but more rigorously, one can define $\pmb{\mathpzc{E}}^s(X;E)$ for $s<0$ by duality using $C^\infty_c(\mathbb{T}_HX; s^*(E))\subseteq \pmb{\mathpzc{E}}(X;E)^*:=\mathbb{K}_{C^*(\mathbb{T}_HX)}(\pmb{\mathpzc{E}}(X;E),C^*(\mathbb{T}_HX))$ and interchange the role of the left and right actions in \eqref{somedifferentaciotosns}. The definition in the graded case is similar, for a graded hermitean vector bundle $\pmb{E}\to X$ we define the $C^*(\mathbb{T}_HX)$-Hilbert $C^*$-module 
$$\pmb{\mathpzc{E}}^s_{\rm gr}(X;\pmb{E}):=\bigoplus_j \pmb{\mathpzc{E}}^{s-j}(X;\pmb{E}[j]).$$
We note that for $t>0$, 
$$\pmb{\mathpzc{E}}^s_{\rm gr}(X;\pmb{E})\otimes_{\pi_{X,t}}L^2(X)=\pmb{W}^s_{H,{\rm gr}}(X;\pmb{E}).$$
The operator 
$$\Lambda_{\pmb{E},s}:=\bigoplus_j \pmb{\mathfrak{D}}_{\pmb{E}[j]}^{(s-j)/2m},$$
defines a unitary equivalence $\pmb{\mathpzc{E}}^s_{\rm gr}(X;\pmb{E})\cong \bigoplus_j \pmb{\mathpzc{E}}(X;\pmb{E}[j])$.

\begin{proposition}
\label{propusingorder}
Let $\mathbb{D}\in \pmb{\Psi}^{m}_{H,{\rm gr}}(X;\pmb{E}_1,\pmb{E}_2)$ and $s\geq t+m$. Then as a densely defined operator $\pmb{\mathpzc{E}}^s_{\rm gr}(X;\pmb{E}_1)\to \pmb{\mathpzc{E}}^t_{\rm gr}(X;\pmb{E}_2)$, $\mathbb{D}$ is closable and its closure is bounded and adjointable. Moreover, if $s>t+m$, $\mathbb{D}:\pmb{\mathpzc{E}}^s_{\rm gr}(X;\pmb{E}_1)\to \pmb{\mathpzc{E}}^t_{\rm gr}(X;\pmb{E}_2)$ is locally $C^*(\mathbb{T}_HX)$-compact. 
\end{proposition}

\begin{proof}
After replacing $\mathbb{D}\in \pmb{\Psi}^{m}_{H,{\rm gr}}(X;\pmb{E}_1,\pmb{E}_2)$ with $\Lambda_{\pmb{E},t}\mathbb{D}\Lambda_{\pmb{E},-s}\in \pmb{\Psi}^{m+t-s}_H(X;\pmb{E}_1,\pmb{E}_2)$, we have reduced the statement to the ungraded case for $t=s=0$ which was proved in Lemma \ref{actionofbolddd} .
\end{proof}

\begin{proposition}
\label{regualdldlalaoad}
Let $\mathbb{D}\in \pmb{\Psi}^{m}_{H,{\rm gr}}(X;\pmb{E}_1,\pmb{E}_2)$, $m>0$, be graded $H$-elliptic. Take $s\in \R$. Then the closure of $\mathbb{D}$ as a densely defined operator $\pmb{\mathpzc{E}}^s_{\rm gr}(X;\pmb{E}_1)\to \pmb{\mathpzc{E}}^s_{\rm gr}(X;\pmb{E}_2)$ is regular with domain $\pmb{\mathpzc{E}}^{s+m}_{\rm gr}(X;\pmb{E}_1)$. If $s=0$ and $\pmb{E}_1$ and $\pmb{E}_2$ are trivially graded, the adjoint of $\mathbb{D}$ is the closure of $\mathbb{D}^*$, and in this case, the operator 
$$\tilde{\mathbb{D}}:=\begin{pmatrix} 0& \mathbb{D}\\ \mathbb{D}^*&0\end{pmatrix}\in \pmb{\Psi}^{m}_H(X;E_1\oplus E_2),$$
is essentially selfadjoint on $\pmb{\mathpzc{E}}(X;E_1\oplus E_2)$ and its closure has locally $C^*(\mathbb{T}_HX)$-compact resolvent.
\end{proposition}

The reader can find further details on regularity of multipliers defined from elliptic elements in \cite{Pierrot}.

\begin{proof}
Using order reduction with $\Lambda_{\pmb{E},s}$ as in the proof of Proposition \ref{propusingorder}, we can reduce to the case that $s=0$ and $\pmb{E}_1$ and $\pmb{E}_2$ are trivially graded. We write $\mathbb{D}= \mathbb{T}\pmb{\mathfrak{D}}_{E_1}^{m}$ where $\mathbb{T}=\mathbb{D}\pmb{\mathfrak{D}}_{E_1}^{-m}\in \pmb{\Psi}^{0}_H(X;E_1,E_2)$. Since $\mathbb{D}$ is $H$-elliptic, so is $\mathbb{T}$. Let $\mathbb{R}\in \pmb{\Psi}^{0}_H(X;E_2,E_1)$ denote a lift of a parametrix of $\mathrm{ev}_{t=1}(\mathbb{T})$. We have that $1-\mathbb{R}\mathbb{T}\in \cap_{k} t^k  \pmb{\Psi}^{-k}_H(X;E_1)$ and $1-\mathbb{T}\mathbb{R}\in \cap_{k} t^k  \pmb{\Psi}^{-k}_H(X;E_2)$.  In partcular, $1-\mathbb{R}\mathbb{T}$ and $1-\mathbb{T}\mathbb{R}$ are locally compact and $0$-homogeneous modulo $C^*(\mathbb{T}_HX))$-compact operators. We conclude that for a sequence $(\varphi_j)_{j\in \N}\subseteq \pmb{\mathpzc{E}}^\infty_c(X;E_1)$, it holds that $\varphi_j\to \varphi$ in $\pmb{\mathpzc{E}}(X;E_1)$ and $D\varphi_j$ converges in $\pmb{\mathpzc{E}}(X;E_2)$ if and only if $\varphi_j\to \varphi$ in $\pmb{\mathpzc{E}}(X;E_1)$ and $\pmb{\mathfrak{D}}_{E_1}^{m}\varphi_j\to \pmb{\mathfrak{D}}_{E_1}^{m}\varphi$ in $\pmb{\mathpzc{E}}(X;E_2)$. Therefore, the closure of $\mathbb{D}$ is the extension to $\pmb{\mathpzc{E}}^m(X;E_1)$. By the same argument, the closure of $\mathbb{D}^*$ is the extension to $\pmb{\mathpzc{E}}^m(X;E_2)$. The same argument applies in any localization in a pure state of $C^*(\mathbb{T}_HX)$, and the regularity of $\mathbb{D}$ as well as the description of its adjoint follows from the local/global principle \cite{leschkaad2,Pierrot}. Again, the local/global principle ensures that $\tilde{\mathbb{D}}$ is essentially selfadjoint. Its closure has locally $C^*(\mathbb{T}_HX)$-compact resolvent since it factors over the inclusion $\pmb{\mathpzc{E}}^m(X;E_1\oplus E_2)\hookrightarrow \pmb{\mathpzc{E}}(X;E_1\oplus E_2)$.
\end{proof}

\begin{remark}
\label{regularidremak}
To ensure regularity, the condition on $H$-ellipticity in Proposition \ref{regualdldlalaoad} can in general not be dropped. Consider an element $a\in \mathcal{U}_m(\mathfrak{g})$ such that $a=a^*$ in $\mathcal{U}_m(\mathfrak{g})$ but for a unitary representation $\pi$ of $\mathsf{G}$, $\pi(a)$ fails to be essentially self-adjoint. For large classes of nilpotent Lie algebras $\mathfrak{g}$, the existence of such an $a$ follows from \cite[Section 3]{arnalnonsymm}. If $X$ is a Carnot manifold such that $\mathfrak{g}=\mathfrak{t}_HX_x$ for some $x$ and $D\in  \mathcal{DO}_H^{m}(X;\C)$ such that $\mathrm{ev}_{t=0}(\mathbb{D})|_x=a$, we claim that the closure of $\mathbb{D}$ as a densely defined operator on $\pmb{\mathpzc{E}}(X;\C)$ is not regular. Localizing to $t=0$ and $x$, we see that it suffices to prove that the closure of $\mathrm{ev}_{t=0}(\mathbb{D})|_x=a$ is not regular as an unbounded multiplier of $C^*(\mathsf{G})$, for $\mathsf{G}=T_HX_x$. By \cite[Lemma 1.25]{Pierrot}, the Gårding space of $C^*(\mathsf{G})$ is a core for any element of $\mathcal{U}_m(\mathfrak{g})$ and so $\overline{a}$ is a self-adjoint multiplier of $C^*(\mathsf{G})$. However, for the representation $\pi$ above $\pi(a)$ is not essentially selfadjoint so $\pi(\overline{a})^*\neq \pi(\overline{a}^*)$. By the local/global principle \cite{leschkaad2,Pierrot} the multiplier $\overline{a}$ is not regular. Again by the local/global principle \cite{leschkaad2,Pierrot}, the closure of $\mathbb{D}$ as a densely defined operator on $\pmb{\mathpzc{E}}(X;\C)$ is not regular.
\end{remark}

To make constructions with Rockland sequences easier we state a result characterizing the Rockland property in terms of the Hilbert modules $\pmb{\mathpzc{E}}$ defined above. The proof is omitted as it follows from the constructions of \cite[Chapter 5]{Dave_Haller1} and the propositions above.

\begin{theorem}
\label{odnoaindaonafofbnaon}
Consider a graded $H$-almost-complex 
$$0\to C^\infty(X;\pmb{E}_1)\xrightarrow{D_1}C^\infty(X;\pmb{E}_2)\xrightarrow{D_2}\cdots \xrightarrow{D_{N-1}}C^\infty(X;\pmb{E}_{N})\xrightarrow{D_{N}} C^\infty(X;\pmb{E}_{N+1})\to 0,$$
of order $\pmb{m}=(m_1,\ldots, m_N)$ as in Definition \ref{hcomplexdefdo}. The following are equivalent:
\begin{enumerate}
\item This sequence is a graded Rockland sequence.
\item For any $x\in X$ and non-trivial irreducible unitary representation $\pi$ of $T_HX_x$, we have that the sequence 
\begin{align*}
0\to &\pmb{\mathpzc{E}}^{s_1}_{\rm gr}(X;\pmb{E}_1)\otimes_\pi \mathcal{H}_\pi\xrightarrow{\sigma_{(x,\pi)}^{m_{1}}(D_{1})}\pmb{\mathpzc{E}}^{s_2}_{\rm gr}(X;\pmb{E}_2)\otimes_\pi \mathcal{H}_\pi\xrightarrow{\sigma_{(x,\pi)}^{m_{2}}(D_{2})}\cdots \\
&\cdots\xrightarrow{\sigma_{(x,\pi)}^{m_{N-1}}(D_{N-1})}\pmb{\mathpzc{E}}^{s_{N}}_{\rm gr}(X;\pmb{E}_{N})\otimes_\pi \mathcal{H}_\pi\xrightarrow{\sigma_{(x,\pi)}^{m_{N}}(D_{N})} \pmb{\mathpzc{E}}^{s_{N+1}}_{\rm gr}(X;\pmb{E}_{N+1})\otimes_\pi \mathcal{H}_\pi\to 0,
\end{align*}
is exact, where $s_1,\ldots, s_N,s_{N+1}$ are such that $s_j-s_{j+1}=m_j$ for any $j$.
\item There exists $B_j\in \Psi_{H,{\rm gr}}^{-m_j}(X;\pmb{E}_{j+1},\pmb{E}_j)$ such that 
$$\sigma_H^{-m_j}(B_j)\sigma_H^{m_j}(D_j)+\sigma_H^{m_{j-1}}(D_{j-1})\sigma_H^{-m_{j-1}}(D_{j-1})=1,$$
for any $j$.
\end{enumerate}

\end{theorem}

\begin{proposition}
Let $X$ be a closed Carnot manifold with a fixed volume density, $E_1,E_2\to X$ two hermitean vector bundles and $\mathbb{F}\in \pmb{\Psi}^{0}_H(X;E_1,E_2)$ be an $H$-elliptic element which is unitary up to $t\pmb{\Psi}^{-1}_H$. Consider 
$$\tilde{\mathbb{F}}:=\begin{pmatrix} 0& \mathbb{F}\\ \mathbb{F}^*&0\end{pmatrix}\in \pmb{\Psi}^{0}_H(X;E_1\oplus E_2),$$
as a self-adjoint adjointable operator on $\pmb{\mathpzc{E}}(X;E_1\oplus E_2)$. Then $(\pmb{\mathpzc{E}}(X;E_1\oplus E_2), \tilde{\mathbb{F}})$ is an even $C_0[0,\infty)$-linear Kasparov module for $(C_0([0,\infty)\times X),C^*(\mathbb{T}_HX))$.

If $E=E_1=E_2$ and $\mathbb{F}$ additionally to above is self-adjoint up to $t\pmb{\Psi}^{-1}_H$, then $(\pmb{\mathpzc{E}}(X;E), \mathbb{F})$ is an odd $C_0[0,\infty)$-linear Kasparov module for $(C_0([0,\infty)\times X),C^*(\mathbb{T}_HX))$.
\end{proposition}

We remark that $\mathbb{F}$ is unitary/self-adjoint up to $t\pmb{\Psi}^{-1}_H$ precisely when $\mathrm{ev}_{t=0}(\mathbb{F})$ is unitary/self-adjoint up to $C^\infty_c$, i.e. $\sigma_H(\mathrm{ev}_{t=1}(\mathbb{F}))$ is a unitary/self-adjoint in $\Sigma^0_H$. 

\begin{proof}
We only prove the statement about even Kasparov modules, the odd statement is proven similarly. We need to prove that for any $a\in C_c^\infty([0,\infty)\times X)$, $a(\tilde{\mathbb{F}}^*-\tilde{\mathbb{F}}), a(\tilde{\mathbb{F}}^2-1), [\tilde{\mathbb{F}},a]\in \mathbb{K}_{C^*(\mathbb{T}_HX)}(\pmb{\mathpzc{E}}(X;E_1\oplus E_2))$. It follows from our assumptions and Corollary \ref{commwithfunticododo} that $a(\tilde{\mathbb{F}}^*-\tilde{\mathbb{F}}), a(\tilde{\mathbb{F}}^2-1), [\tilde{\mathbb{F}},a]\in t\pmb{\Psi}^{-1}_H(X;E_1\oplus E_2)$. Since $C_c[0,\infty)t\pmb{\Psi}^{-1}_H(X;E_1\oplus E_2)\subseteq C_c[0,\infty)\pmb{\Psi}^{-1}_H(X;E_1\oplus E_2)$, the operators $a(\tilde{\mathbb{F}}^*-\tilde{\mathbb{F}}), a(\tilde{\mathbb{F}}^2-1), [\tilde{\mathbb{F}},a]$ are all locally compact by Lemma \ref{actionofbolddd}. Using $C_0[0,\infty)$-linearity, compactness follows after multiplying by a $\chi\in C_c[0,\infty)$ with $\chi a=a$. 
\end{proof}

To make some constructions easier, we shall make use of higher order unbounded Kasparov modules. For more details, see \cite[Subsection 1.5]{gmrtwisted}. The next corollary is of {\bf crucial importance for defining $KK$-classes} from $H$-elliptic operators.

\begin{corollary}
\label{clalslasmadadado}
Let $X$ be a closed Carnot manifold with a fixed volume density, $E_1,E_2\to X$ two hermitean vector bundles and $\mathbb{D}\in \pmb{\Psi}^{m}_H(X;E_1,E_2)$, $m>0$, be an $H$-elliptic element. Consider 
$$\tilde{\mathbb{D}}:=\begin{pmatrix} 0& \mathbb{D}\\ \mathbb{D}^*&0\end{pmatrix}\in \pmb{\Psi}^{m}_H(X;E_1\oplus E_2),$$
as a self-adjoint regular operator on $\pmb{\mathpzc{E}}(X;E_1\oplus E_2)$. Then $(C^\infty_c(X\times [0,\infty)),\pmb{\mathpzc{E}}(X;E_1\oplus E_2), \tilde{\mathbb{D}})$ is an even higher order unbounded $C_0[0,\infty)$-linear Kasparov module for $(C_0([0,\infty)\times X),C^*(\mathbb{T}_HX))$ of order $m$. In particular, for the element 
$$\tilde{\mathbb{F}}_{\mathbb{D}}\in \End_{C^*(\mathbb{T}_HX)}^*(\pmb{\mathpzc{E}}(X;E_1\oplus E_2)),$$ 
defined by
$$\tilde{\mathbb{F}}_{\mathbb{D}}:=\tilde{\mathbb{D}}(1+\tilde{\mathbb{D}}^2)^{-1/2}=\begin{pmatrix} 0& \mathbb{D}(1+\mathbb{D}^*\mathbb{D})^{-1/2}\\ \mathbb{D}^*(1+\mathbb{D}\mathbb{D}^*)^{-1/2}&0\end{pmatrix},$$
the pair $(\pmb{\mathpzc{E}}(X;E_1\oplus E_2), \tilde{\mathbb{F}}_{\mathbb{D}})$ is an even $C_0[0,\infty)$-linear Kasparov module for $(C_0([0,\infty)\times X),C^*(\mathbb{T}_HX))$.

If $E=E_1=E_2$ and $\mathbb{D}$ additionally to above is self-adjoint, then $(C^\infty_c(X\times [0,\infty)),\pmb{\mathpzc{E}}(X;E), \mathbb{D})$ is an odd higher order unbounded $C_0[0,\infty)$-linear Kasparov module for $(C_0([0,\infty)\times X),C^*(\mathbb{T}_HX))$ (of order $m$) and the pair $(\pmb{\mathpzc{E}}(X;E), \mathbb{D}(1+\mathbb{D}^2)^{-1/2})$ is an odd $C_0[0,\infty)$-linear Kasparov module for $(C_0([0,\infty)\times X),C^*(\mathbb{T}_HX))$.
\end{corollary}

\begin{proof}
As in the previous result, the even statement is proven similarly to the odd so we restrict to the latter case. It is clear that $C^\infty_c(X\times [0,\infty))$ preserves $\mathrm{dom}(\mathbb{D})$ and by Proposition \ref{regualdldlalaoad}, $\mathbb{D}$ has compact resolvent. It remains to prove that $[\mathbb{D},a]$ is $1/m$-bounded with respect to $\mathbb{D}$ (in the sense of \cite[Definition 1.21]{gmrtwisted}). In other words, we need to prove that $[\mathbb{D},a](1+\mathbb{D}^2)^{-(m-1)/2m}$ and $(1+\mathbb{D}^2)^{-(m-1)/2m}[\mathbb{D},a]$ are bounded. Using complex interpolation for the domains of the complex powers of the self-adjoint regular operator $\mathbb{D}$, cf. \cite[Theorem 3]{seeley71}, this follows from the fact that $[\mathbb{D},a]\in C_c^\infty[0,\infty)t\pmb{\Psi}^{m-1}_H(X;E)\subseteq C_c^\infty[0,\infty)\pmb{\Psi}^{m-1}_H(X;E)$ (see Corollary \ref{commwithfunticododo} ) and therefore $[\mathbb{D},a](1+\mathbb{D}^2)^{-(m-1)/2m}$ factors as bounded mappings
$$\pmb{\mathpzc{E}}(X;E)\xrightarrow{(1+\mathbb{D}^2)^{-(m-1)/2m}} \pmb{\mathpzc{E}}^{m-1}(X;E)\xrightarrow{[\mathbb{D},a]} \pmb{\mathpzc{E}}(X;E),$$
and $(1+\mathbb{D}^2)^{-(m-1)/2m}[\mathbb{D},a]$ factors as the bounded mappings
$$\pmb{\mathpzc{E}}(X;E)\xrightarrow{[\mathbb{D},a]} \pmb{\mathpzc{E}}^{-(m-1)}(X;E)\xrightarrow{(1+\mathbb{D}^2)^{-(m-1)/2m}}  \pmb{\mathpzc{E}}(X;E). \qquad\qedhere$$
\end{proof}

\subsection{Symbols and flat orbits}
\label{charpadldaladoeo}

For regular Carnot manifolds, the osculating Lie groupoid is a locally trivial bundle of Carnot-Lie groups and we can in this case give a more detailed description of the Carnot symbol calculus. We later on restrict to $\pmb{F}$-regular Carnot manifolds, and study $H$-ellipticity in terms of the flat orbits building on the idea that the flat orbits form a Zariski open dense subset of the representation space. 

\begin{definition}
Let $\mathsf{G}$ be a simply connected Carnot-Lie group and $P^m_{\mathsf{G}}\subseteq C^\infty(\mathsf{G},|\Lambda|)$ the space of polynomial densities on $\mathsf{G}$ homogeneous of degree $m$. Define the Frechet $*$-spaces
$$\Sigma^m_{\mathsf{G}}:=\{k\in \mathcal{S}'(\mathsf{G})\cap C^\infty(\mathsf{G}\setminus \{0\},|\Lambda|)/P^m_{\mathsf{G}}: \lambda_*k=\lambda^m k\},$$
and 
$$\tilde{\Sigma}^m_{\mathsf{G}}:=\{k\in \mathcal{E}'(\mathsf{G})\cap C^\infty(\mathsf{G}\setminus \{0\},|\Lambda|): \lambda_*k-\lambda^m k\in C^\infty_c(\mathsf{G})\}.$$
The spaces $\Sigma^m_{\mathsf{G}}$ and $\tilde{\Sigma}^m_{\mathsf{G}}$ are equipped with associative products defined from convolution
$$\tilde{\Sigma}^m_{\mathsf{G}}\times \tilde{\Sigma}^{m'}_{\mathsf{G}}\to \tilde{\Sigma}^{m+m'}_{\mathsf{G}},$$
$$\Sigma^m_{\mathsf{G}}\times \Sigma^{m'}_{\mathsf{G}}\to \Sigma^{m+m'}_{\mathsf{G}}.$$
\end{definition}

Using the ideas underlying Proposition \ref{strucuturaldescofsymboslalld}, we arrive at a short exact sequence of Frechet $*$-spaces
\begin{equation}
\label{shssoosdods}
0\to C^\infty_c(\mathsf{G},|\Lambda|)\to \tilde{\Sigma}^m_{\mathsf{G}}\xrightarrow{\rho_m} \Sigma^m_{\mathsf{G}}\to 0.
\end{equation}
The projection $ \tilde{\Sigma}^m_{\mathsf{G}}\to \Sigma^m_{\mathsf{G}}$ admits a linear splitting by choosing a cut off function around $0\in \mathsf{G}$. Moreover, 
$$\rho_{m+m'}(aa')=\rho_m(a)\rho_{m'}(a') \quad\mbox{and}\quad \rho_m(a)^*=\rho_m(a^*),$$
for $a\in \tilde{\Sigma}^m_{\mathsf{G}}$, $a'\in \tilde{\Sigma}^{m'}_{\mathsf{G}}$. Pushing forward along graded automorphisms defines continuous actions
$$\Aut_{\rm gr}(\mathsf{G})\times \tilde{\Sigma}^{m}_{\mathsf{G}}\to \tilde{\Sigma}^{m}_{\mathsf{G}},$$
$$\Aut_{\rm gr}(\mathsf{G})\times \Sigma^{m}_{\mathsf{G}}\to \Sigma^{m}_{\mathsf{G}},$$
compatible with products, adjoints and the quotient map $\rho_m$. The following observation is a direct consequence of local triviality of a regular Carnot manifold. Recall the definition of $\tilde{\Sigma}^m_H$ from Equation \eqref{tidlemsossom}.

\begin{proposition}
\label{locallaslsksks}
Let $X$ be a regular Carnot manifold of type $\mathfrak{g}$ with graded frame bundle $P_X\to X$. Define the locally trivial bundles of Frechet $*$-spaces 
$$\Sigma_X^m:=P_X\times_{\Aut_{\rm gr}(\mathsf{G})}\Sigma^m_{\mathsf{G}}\to X,$$
$$\tilde{\Sigma}_X^m:=P_X\times_{\Aut_{\rm gr}(\mathsf{G})}\tilde{\Sigma}^m_{\mathsf{G}}\to X.$$
Then there are natural isomorphisms
$$\Sigma^m_{H}(X;\C)\cong C^\infty(X;\Sigma^m_X),$$
$$\tilde{\Sigma}^m_{H}(X;\C)\cong C^\infty(X;\tilde{\Sigma}^m_X),$$
compatible with products, adjoints and the quotient map $\rho_m$.
\end{proposition}

\begin{remark}
The situation for symbols of operators acting between vector bundles is quite similar. Let $E_1,E_2\to X$ be complex vector bundles. Let $V_1$ and $V_2$ denote the fibres of $E_1$ and $E_2$, respectively. The frame bundle $P_j\to X$ of $E_j$ is a principal $GL(V_j)$-bundle (that can be reduced to $U(V_j)$ given a hermitean structure). Writing $P_{X,E_1,E_2}:=P_X\times_X P_1\times_X P_2$, we have that $\Sigma^m_{H}(X;E_1,E_2)$ coincides with the sections of the bundle 
$$\Sigma_X^m\otimes \Hom(E_1,E_2)=P_{X,E_1,E_2}\times_{\Aut_{\rm gr}(\mathsf{G})\times \mathrm{GL}(V)\times \mathrm{GL}(W)}(\Sigma^m_{\mathsf{G}}\otimes\Hom(V,W)).$$
\end{remark}

\begin{remark}
In light of the global picture from Proposition \ref{locallaslsksks}, we note that if $\mathbb{D}\in \pmb{\Psi}^m_H(X;E_1,E_2)$ then $\rho_m:\tilde{\Sigma}^m_{H}(X;E_1,E_2)\to \Sigma^m_{H}(X;E_1,E_2)$ fits together with the symbol mapping via the identity
$$\rho_m\circ \mathrm{ev}_{t=0}(\mathbb{D})=\sigma_H\circ \mathrm{ev}_{t=1}(\mathbb{D}).$$
\end{remark}

Let us now consider the situation when there are flat orbits. For the remainder of this section we consider an $\pmb{F}$-regular Carnot manifold of type $\mathfrak{g}$. Recall the notation $\Gamma$ for the set of flat orbits of $\mathsf{G}$ and $\Gamma_X:=P_X\times_{\Aut_{\rm gr}(\mathsf{G})}\Gamma$ for the bundle of flat orbits of $T_HX$. We fix a bundle of flat representations $\mathcal{H}\to \Gamma_X$ as in Theorem \ref{trivialdldaaddo}. Recall that there is a $C_0(\Gamma_X)$-linear $*$-isomorphism $\pi_{\musFlat}:I_{X}\to C_0(\Gamma_X,\mathbb{K}(\mathcal{H}))$ and the choice of Hilbert space bundle $\mathcal{H}$ is unique up to isomorphism and tensoring by a line bundle. We introduce the notation 
$$\mathpzc{H}:=C_0(\Gamma_X,\mathcal{H}),$$
for the associated $C_0(\Gamma_X)$-Hilbert $C^*$-module and remark that $I_X\cong \mathbb{K}_{C_0(\Gamma_X)}(\mathpzc{H})$. We write $\mathcal{S}_0(\mathcal{H})\to \Gamma_P$ for the locally trivial bundle of Frechet spaces (with fibre $\mathcal{S}_0(\R^d)$ for $d=\mathrm{codim}(\mathfrak{z})/2$) defined from $\mathcal{S}_0(\mathcal{H})_\xi=\mathcal{S}_0(\mathcal{H}_\xi)$. Indeed, by the considerations of Section \ref{subsecfinestrat}, $\mathcal{S}_0(\mathcal{H})$ is well defined. Since the trivial representation is not flat, the inclusion $\mathcal{S}_0(\mathcal{H})\subseteq \mathcal{H}$ is fibrewise norm dense. We similarly define the dense $C_0(\Gamma_X)$-submodule 
$$\mathcal{S}_0(\mathpzc{H}):=C_0(\Gamma_X,\mathcal{S}_0(\mathcal{H}))\subseteq \mathpzc{H}.$$
An operator $T$ on $\mathpzc{H}$ is said to be locally $C_0(\Gamma_X)$-compact if $\varphi T\in \mathbb{K}_{C_0(\Gamma_X)}(\mathpzc{H})$ for any $\varphi\in C_0(\Gamma_X)$.

\begin{proposition}
\label{nicaoaooacjad}
Let $X$ be a compact $\pmb{F}$-regular Carnot manifold. Any $a\in \Sigma_H^m(X;\C)$ defines a densely defined $C_0(\Gamma_X)$-linear semiregular operator
$$\pi_{\musFlat}(a):\mathpzc{H}\dashrightarrow \mathpzc{H},$$
defined on $\mathrm{dom}(\pi(a)):=\mathcal{S}_0(\mathpzc{H})$ by 
$$(\pi_{\musFlat}(a)v)_\xi:=\pi_{\musFlat,\xi}(a)v_\xi.$$
If $m\leq 0$, then $\pi_{\musFlat}(a)$ is bounded, adjointable and $\pi_{\musFlat}(a)^*=\overline{\pi_{\musFlat}(a^*)}$. In particular, $\pi_{\musFlat}$ defines a faithful $*$-homomorphism 
$$\pi_{\musFlat}:\Sigma^m_H(X,\C)\to \End_{C_0(\Gamma_X)}^*(\mathpzc{H}),$$
for any $m\leq 0$. If $m<0$, then $\pi_{\musFlat}(a)$ is locally $C_0(\Gamma_X)$-compact.
\end{proposition}

\begin{proof}
The operator $\pi_{\musFlat}(a)$ preserves $\mathcal{S}_0(\mathpzc{H})$ and on there, the formal adjoint of $\pi_{\musFlat}(a)$ is $\pi_{\musFlat}(a^*)$. Therefore $\pi_{\musFlat}(a)$ is semi-regular. It is clear that $\pi_{\musFlat}(ab)=\pi_{\musFlat}(a)\pi_{\musFlat}(b)$ on $\mathcal{S}_0(\mathpzc{H})$. 

Assume now that $m\leq 0$. By order reduction it suffices to consider $m=0$. If suffices to prove that $\pi_{\musFlat}(a)$ is bounded; it then follows from the properties in the preceding parapraph that $\pi_{\musFlat}(a)$ is adjointable and $\pi_{\musFlat}(a)^*=\overline{\pi_{\musFlat}(a^*)}$. As in the proof of Proposition \ref{generalpitoleiielalad}, we use \cite[Theorem 8.18]{eskeewertthesis} to conclude that there is a constant $C>0$ such that $a$ is norm-bounded $L^2(T_HX_x)\to L^2(T_HX_x)$ by $C$ for any $x$. Therefore, $a$ defines an element of $\mathcal{M}(C^*(T_HX_x))$ and will therefore be bounded on $\mathpzc{H}_x$ by $C$ for any $x$. 

Finally, for $m<0$, we note that $a$ induces a $C^*(T_HX)$-compact multiplier on $\pmb{\mathpzc{E}}(X;\C)\otimes_{\mathrm{ev}_{t=0}} C^*(T_HX)$ by Proposition \ref{actionofbolddd}. In particular, $a$ acts as an element of $C^*(T_HX)$ and the proposition follows from the fact that the $C^*(T_HX)$-action on $\mathpzc{H}$ is $C_0(\Gamma_X)$-locally compact. 
\end{proof}

\begin{proposition}
\label{nicaoaooacjadtilde}
Let $X$ be an $\pmb{F}$-regular Carnot manifold. Any $a\in \tilde{\Sigma}_H^m(X;\C)$ defines a densely defined $C_0(\Gamma_X)$-linear semiregular operator
$$\tilde{\pi}_{\musFlat}(a):\mathpzc{H}\dashrightarrow \mathpzc{H},$$
defined on $\mathrm{dom}(\tilde{\pi}_{\musFlat}(a)):=\mathcal{S}_0(\mathpzc{H})$ by 
$$(\tilde{\pi}_{\musFlat}(a)v)_\xi:=\pi_{\musFlat,\xi}(a)v_\xi.$$
If $m\leq 0$, then $\tilde{\pi}_{\musFlat}(a)$ is bounded, adjointable and $\tilde{\pi}_{\musFlat}(a)^*=\overline{\tilde{\pi}_{\musFlat}(a)}$. In particular, $\tilde{\pi}_{\musFlat}$ defines a faithful $*$-homomorphism 
$$\tilde{\pi}_{\musFlat}:\tilde{\Sigma}^m_H(X,\C)\to \End_{C_0(\Gamma_X)}^*(\mathpzc{H}),$$
for any $m\leq 0$. If $m<0$, then $\tilde{\pi}_{\musFlat}(a)$ is locally $C_0(\Gamma_X)$-compact.
\end{proposition}

The proof of Proposition \ref{nicaoaooacjadtilde} goes along the same lines as the proof of Proposition \ref{nicaoaooacjad} and is omitted. 

\begin{remark}
We remark that the representations of Proposition \ref{nicaoaooacjad} and \ref{nicaoaooacjadtilde} are not compatible with the quotient mapping $\rho_m$, i.e $\tilde{\pi}_{\musFlat}(a)\neq \pi_{\musFlat}(\rho_m(a))$ in general. Indeed, the situation should be understood in analogy with the short exact sequence 
$$0\to C_0(\R^n)\to C(\overline{B}_n)\to C(S^{n-1})\to 0,$$
defining the radial compactification of $\R^n$ (by identifying $\R^n$ with the open unit ball $B_n$). In this case, $C(\overline{B}_n)$ and $C(S^{n-1})$ both acts on $C_0(\R^n\setminus \{0\})$, the latter by extending homogeneously.
\end{remark}

\begin{definition}
Let $X$ be an $\pmb{F}$-regular Carnot manifold. We fix a $\mathfrak{D}\in \mathcal{DO}_H^{2m}(X;\C)$ as in Proposition \ref{existenceofposiskdd}. For $s \geq 0$ we define 
$$\mathcal{H}^s:=\sigma_H(\mathfrak{D}^{-s/2m})\mathcal{H}\to \Gamma_X,$$
viewed as a the Hilbert space bundle by declaring $\sigma_H(\mathfrak{D}^{-s/2m}):\mathcal{H}\to \mathcal{H}^s$ to be a unitary isomorphism. We let $\mathpzc{H}^s$ denote the associated $C_0(\Gamma_X)$-Hilbert $C^*$-module. For $s< 0$ we define 
$$\mathcal{H}^s:=\sigma_H(\mathfrak{D}^{-s/2m})\mathcal{H}\to \Gamma_X,$$
and the associated $C_0(\Gamma_X)$-Hilbert $C^*$-module $\mathpzc{H}^s$, by duality. We also introduce the notation 
$$\mathpzc{H}^s(E):=\mathpzc{H}^s\otimes_{C_0(\Gamma_X)}C_0(\Gamma_X;E),$$
for a vector bundle $E\to X$. If $\pmb{E}\to X$ is a hermitean graded vector bundle, we set 
$$\pmb{\mathpzc{H}}^s(\pmb{E}):=\bigoplus_j \mathpzc{H}^{s-j}(\pmb{E}[j]).$$
\end{definition}

\begin{proposition}
\label{somrelareonon}
Let $X$ be an $\pmb{F}$-regular Carnot manifold, $\pmb{E}\to X$ a graded hermitean vector bundle and $\pmb{\mathpzc{E}}^s_{\rm gr}(X;\pmb{E})$ the $C^*(\mathbb{T}_HX)$-Hilbert $C^*$-module from Section \ref{secondnaaoaoac}. Consider $\mathpzc{H}$ as a $(C^*(\mathbb{T}_HX),C_0(\Gamma_X))$-Hilbert $C^*$-module via restriction to $t=0$. Then there is a canonical isomorphism of $C_0(\Gamma_X)$-Hilbert $C^*$-modules
$$\pmb{\mathpzc{H}}^s(\pmb{E})\cong \pmb{\mathpzc{E}}^s_{\rm gr}(X;\pmb{E})\otimes_{C^*(\mathbb{T}_HX)}\mathpzc{H}.$$
Conversely, the localization of $\pmb{\mathpzc{E}}^s_{\rm gr}(X;\pmb{E})|_{t=0}$ to $\Gamma_X$ along the natural map $C_0(\Gamma_X)\to \mathcal{ZM}(C^*(T_HX))$, is canonically isomorphic to $\mathbb{K}_{C_0(\Gamma_X)}(\mathpzc{H},\pmb{\mathpzc{H}}^s(\pmb{E}))$ as right $I_X$-Hilbert $C^*$-modules.
\end{proposition}

\begin{proof}
The statement is readily verified in local coordinates and globalizes due to the regularity.
\end{proof}

\begin{proposition}
Let $X$ be a regular Carnot manifold admitting flat orbits and $\pmb{E}_1,\pmb{E}_2\to X$ two graded, hermitean vector bundles. Any $a\in \Sigma^m_{H,{\rm gr}}(X;\pmb{E}_1,\pmb{E}_2)$ defines a bounded densely defined $C_0(\Gamma_X)$-linear operator $\pmb{\mathpzc{H}}^s(\pmb{E}_1)\dashrightarrow \pmb{\mathpzc{H}}^{s-m}(\pmb{E}_2)$ defined on $\mathrm{dom}(\pi_{\musFlat,s}(a)):=\mathcal{S}_0(\pmb{\mathpzc{H}}^s(\pmb{E}_1))$ by 
$$(\pi(a)_{\musFlat,s}v)_\xi:=\pi_{\musFlat,s,\xi}(a)v_\xi.$$
The closure of $\pi_{\musFlat,s}(a)$ is an adjointable operator
$$\pi_{\musFlat,s}(a):\pmb{\mathpzc{H}}^s(\pmb{E}_1)\to \pmb{\mathpzc{H}}^{s-m}(\pmb{E}_2).$$ 
Moreover, we obtain a mapping
$$\pi_{\musFlat,s}:\Sigma_{H,{\rm gr}}^m(X;\pmb{E}_1,\pmb{E}_2)\to \Hom_{C_0(\Gamma_X)}^*(\pmb{\mathpzc{H}}^s(\pmb{E}_1), \pmb{\mathpzc{H}}^{s-m}(\pmb{E}_2)),$$
which respects products in the sense that for $a\in \Sigma_{H,{\rm gr}}^m(X;\pmb{E}_2,\pmb{E}_3)$ and $a'\in \Sigma_{H,{\rm gr}}^{m'}(X;\pmb{E}_1,\pmb{E}_2)$ we have that 
$$\pi_{\musFlat,s-m}(a')\pi_{\musFlat,s}(a)=\pi_{\musFlat,s}(a'a)\in \Hom_{C_0(\Gamma_X)}^*(\pmb{\mathpzc{H}}^s(\pmb{E}_1), \pmb{\mathpzc{H}}^{s-m-m'}(\pmb{E}_3).$$
\end{proposition}

\begin{proof}
Using order reductions with $\Lambda_{\pmb{E},s}$ as in the proof of Proposition \ref{propusingorder}, we can reduce to the case that $s=0$ and $\pmb{E}_1$ and $\pmb{E}_2$ are trivially graded. Further order reductions reduce to $m=0$. Now the Proposition follows from Proposition \ref{nicaoaooacjadtilde}.
\end{proof}

\subsection{Characterizing $H$-ellipticity in the flat orbits}

We now restrict to the transveral $\Gamma^\partial_X:=P_X\times_{\Aut_{\rm gr}(\mathsf{G})}\Gamma_\partial$ for the dilation action (see more in Proposition \ref{gammapartialdef}). By Proposition \ref{gammapartialdef}, $\Gamma^\partial_X$ is a smooth hypersurface in $\Gamma_X$, in general $\Gamma^\partial_X$ is not compact. We introduce the subscript $\partial$ on modules and bundles to indicate restriction to $\Gamma^\partial_X$. For instance: 
$$\mathpzc{H}^s_\partial (E):=\mathpzc{H}^s(E)\otimes_{C_0(\Gamma_X)}C_0(\Gamma^\partial_X),$$
where the tensor product is over the $*$-homomorphism $C_0(\Gamma_X)\to C_0(\Gamma^\partial_X)$ is defined from restriction along the inclusion $\Gamma^\partial_X\subseteq \Gamma_X$. 

\begin{definition}
Let $X$ be an $\pmb{F}$-regular Carnot manifold and $E_1,E_2\to X$ two vector bundles. We say that $a\in \Sigma_H^m(X;E_1,E_2)$ is \emph{a uniformly $H$-elliptic symbol on flat orbits} if for some $s\in \R$, $\pi_{\musFlat,s}(a):\mathpzc{H}^s_\partial(E_1)\to \mathpzc{H}^{s-m}_\partial(E_2)$ is invertible as a bounded adjointable $C_0(\Gamma_X)$-linear operator.
\end{definition}

\begin{theorem}
\label{unmomnomhell}
Let $X$ be an $\pmb{F}$-regular Carnot manifold, $E_1,E_2\to X$ two vector bundles and $D\in \Psi_H^m(X;E_1,E_2)$. Then the following are equivalent:
\begin{enumerate}[i)]
\item $D$ is $H$-elliptic.
\item $\sigma_H(D)$ is uniformly $H$-elliptic symbol on flat orbits.
\item $\pi_{\musFlat,s}(\sigma_H(D)):\mathpzc{H}^s_\partial(E_1)\to \mathpzc{H}^{s-m}_\partial(E_2)$ is invertible as a bounded adjointable $C_0(\Gamma_X^\partial)$-linear operator for all $s\in \R$.
\end{enumerate}
\end{theorem}

\begin{proof}
It is clear that iii)$\Rightarrow$ ii). It follows from Lemma \ref{charHELLrockslslalem} that i)$\Rightarrow$ iii). It remains to prove that  ii)$\Rightarrow$ i). After replacing $D$ with $\mathfrak{D}_{E_2}^{(s-m)/2m_2}D\mathfrak{D}_{E_1}^{s/2m_1}$ and doubling up the vector bundle, we can assume that $D\in \Psi_H^0(X;E)$ and that $\sigma_H(D):\mathpzc{H}_\partial(E)\to \mathpzc{H}_\partial(E)$ is invertible as a bounded adjointable $C_0(\Gamma_X)$-linear operator. It follows from Proposition \ref{nicaoaooacjad} that $\sigma_H(D)$ admits an inverse in the $C^*$-algebra closure of $\Sigma^0_H(X;E)$ inside $\mathrm{End}_{C_0(\Gamma_X)}^*(\mathpzc{H}(E))$. Since the flat orbits are dense in the representation space, i.e. $C^*(T_HX)\to \mathcal{M}(I_X)$ is faithful, the isomorphism $\mathrm{End}_{C_0(\Gamma_X)}^*(\mathpzc{H}(E))\cong \mathcal{M}(I_X)$ ensures that $\sigma_{(x,\pi)}^0(D)$ is invertible for all $x\in X$ and all non-trivial irreducible representations. As such, $D$ is $H$-elliptic by Lemma \ref{charHELLrockslslalem}.
\end{proof}

Let us give an example of how to apply Theorem \ref{unmomnomhell}. Recall the setting and notation of Theorem \ref{charhellfofodo}.

\begin{proposition}
\label{charhellfofodotwo}
Let $X$ be a compact regular polycontact manifold. Let $H=T^{-1}X$ denote the polycontact structure and assume it is equipped with a Riemannian metric $g$. Take $D_\gamma\in \mathcal{DO}_H^{2}(X;E)$ as in Equation \eqref{secondoroddbbamddevep}. Let $\omega$ denote the Kirillov form and take  $s_k(g_\omega)$ as in \eqref{skgoadoma}. Consider the family of endomorphisms
$$\gamma_k:=s_k(g_\omega)+\frac{\mathrm{Tr}(|\omega|)}{2}+\gamma:p_\Gamma^*H^{\otimes_\C^{\rm sym}k}\otimes E\to p_\Gamma^*H_X^{\otimes_\C^{\rm sym}k}\otimes E.$$

Then $D_\gamma$ is $H$-elliptic if and only if each $\gamma_k$ is an automorphism.
\end{proposition}

\begin{proof}
By Theorem \ref{unmomnomhell}, it suffices to prove uniform invertibility of $\sigma_H^2(D_\gamma)\sigma^2_H(\Delta_H)^{-1}$. Using the computations in the proof of Theorem \ref{charhellfofodo}, we see that $\sigma^2_H(\Delta_H)$ acts on $\mathpzc{F}_X$ as $s_k(g_\omega)+\frac{\mathrm{Tr}(|\omega|)}{2}$. A rather naive estimate shows that 
$$C_1(1+k)|\xi|\leq \left\|s_k(g_\omega)+\frac{\mathrm{Tr}(|\omega|)}{2}\right\|_{\End(p_\Gamma^*H_\xi^{\otimes_\C^{\rm sym}k})}\leq C_2(1+k)|\xi|.$$
We conclude that uniform invertibility of $\sigma_H^2(D_\gamma)\sigma^2_H(\Delta_H)^{-1}$ is equivalent to each $\gamma_k$ being an automorphism and there being a norm estimate $\|\gamma_k^{-1}\|_{L^\infty}\lesssim (1+k)^{-1}$. However, by compactness of $S(H^\perp)$ the norm estimate is immediate.
\end{proof}

\section{$K$-theoretical invariants of $H$-elliptic operators} 
\label{sec:ktheoomon}

In this section we shall define a number of different $K$-theoretical invariants associated with an $H$-elliptic operator. We shall rather freely use the relation between $K$-theory, $KK$-theory and Fredholm morphisms of Hilbert $C^*$-modules,  \cite{cumero,GM, higsonprimer, knudjen,kaspnov,meslandbeast,meslandrencompt}, and (higher order, unbounded) Kasparov modules. These invariants, and their internal relations, will be put to use below in Chapter \ref{dualityapapadoda}. For a hermitean vector bundle $E\to X$, we introduce the notation $C^*(T_HX;E)$ for the closure of $C^\infty_c(T_HX;r^*E)$ as a $C^*(T_HX)$-module, note that $C^*(T_HX;E)=\pmb{\mathpzc{E}}(X;E)|_{t=0}$.

\begin{definition}
\label{definilalald}
Let $X$ be a compact Carnot manifold with a fixed volume density, $E_1,E_2\to X$ hermitean vector bundles and $D\in \Psi^m_H(X;E_1,E_2)$ an $H$-elliptic operator with $m>0$. 
\begin{itemize}
\item The even $K$-homology class $[D]\in KK_0(C(X),\C)$ of $D$ is defined as the class of the even higher order spectral triple $(L^2(X;E_1\oplus E_2), \tilde{D})$, 
where 
$$\tilde{D}:=\begin{pmatrix} 0& D\\ D^*&0\end{pmatrix}.$$
\item The even $K$-theory class $[\sigma_H(D)]\in KK_0(\C, C^*(T_HX))$ is defined as the class of the even higher order unbounded Kasparov module $(C^*(T_HX;E_1\oplus E_2), a)$ where $a\in \tilde{\Sigma}^m_H(X;E_1\oplus E_2)$ is any lift of $\sigma_H(\tilde{D})$.
\item The even $KK$-class $[\mathbb{D}]\in KK_0^{[0,\infty)}(C_0([0,\infty)\times X), C^*(\mathbb{T}_HX))$ is defined as the class of the even higher order unbounded Kasparov module $(\pmb{\mathpzc{E}}(X;E_1\oplus E_2), \tilde{\mathbb{D}})$ for some $\mathbb{D}\in \pmb{\Psi}^m_H(X;E_1,E_2)$ lifting $D$.
\end{itemize}

If additionally $E:=E_1=E_2$ and $D=D^*$, 
\begin{itemize}
\item The odd $K$-homology class $[D]\in KK_1(C(X),\C)$ of $D$ is defined as the class of the odd higher order spectral triple $(L^2(X), D)$.
\item The odd $K$-theory class $[\sigma_H(D)]\in KK_1(\C, C^*(T_HX))$ is defined as the class of the odd higher order unbounded Kasparov module $(C^*(T_HX;E), a)$ where $a\in \tilde{\Sigma}^m_H(X;E)$ is any lift of $\sigma_H(D)$.

\item The odd $KK$-class $[\mathbb{D}]\in KK_1^{[0,\infty)}(C_0([0,\infty)\times X), C^*(\mathbb{T}_HX))$ is defined as the class of the odd higher order unbounded Kasparov module $(\pmb{\mathpzc{E}}(X;E), \mathbb{D})$ for some self-adjoint $\mathbb{D}\in \pmb{\Psi}^m_H(X;E)$ lifting $D$.
\end{itemize}
\end{definition}

\begin{remark}
The classes in Definition \ref{definilalald} are well defined due to Corollary \ref{clalslasmadadado}. Using Proposition \ref{embeddingboldpsi} and Lemma \ref{actionofbolddd}, it is readily verified that the classes  $[D]\in KK_*(C(X),\C)$ and $[\mathbb{D}]\in KK_*^{[0,\infty)}(C_0([0,\infty)\times X), C^*(\mathbb{T}_HX))$ are determined by the principal symbol $\sigma_H(D)$, i.e. $[D]=[D']$ and $[\mathbb{D}]=[\mathbb{D}']$ as soon as $\sigma_H(D)=\sigma_H(D')$. 
\end{remark}

\begin{remark}
For an $H$-elliptic operator $D\in \Psi^m_H(X;E_1,E_2)$ for $m\leq0$, the classes $[D]$, $[\sigma_H(D)]$ and $[\mathbb{D}]$ are defined similarly after increasing the order to being positive by conjugating by suitable powers of an invertible positive Carnot operator. It is readily verified that the choice of invertible positive Carnot operator does not affect the class.
\end{remark}

\begin{proposition}
\label{properldaldlaadlda}
Let $X$ be a compact Carnot manifold with a fixed volume density, and $D$ an $H$-elliptic operator in the Carnot calculus. The classes $[D]$, $[\sigma_H(D)]$ and $[\mathbb{D}]$ are related by
$$[D]=[\mathbb{D}]\otimes [\mathrm{ev}_{t=1}],\quad\mbox{and}\quad[\sigma_H(D)]=[i^*]\otimes [\mathbb{D}]\otimes [\mathrm{ev}_{t=0}],$$
where $i:\C\to C(X)$ is the inclusion of the unit.
\end{proposition}

\begin{proof}
The identity $[D]=[\mathbb{D}]\otimes [\mathrm{ev}_{t=1}]$ follows from the construction, cf. Corollary \ref{clalslasmadadado}.  The identity $[\sigma_H(D)]=[i^*]\otimes [\mathbb{D}]\otimes [\mathrm{ev}_{t=0}]$ follows from that $\mathrm{ev}_{t=0}(\mathbb{D})\in \tilde{\Sigma}^m_H$ is a lift of $\sigma_H(D)$.
\end{proof}

\begin{remark}
Let $X$ be a compact Carnot manifold with a fixed volume density, and $F\in \Psi^0_H(X;E)$ a self-adjoint $H$-elliptic operator in the Carnot calculus which is unitary up to $\Psi^{-1}_H(X;E)$. Following Proposition \ref{strucuturaldescofsymboslalld}, we have a short exact sequence of Frechet $*$-algebras
$$0\to C^\infty_c(T_HX,\End(E)\otimes |\Lambda|)\to \tilde{\Sigma}^0_{H}(X;E)\xrightarrow{\rho_m} \Sigma^0_{H}(X;E)\to 0.$$
Compare to Equation \eqref{shssoosdods}. Since $\sigma_H(F)$ is a self-adjoint unitary, the element 
$$p_F:=\frac{1+\sigma_H(F)}{2}\in \Sigma^0_{H}(X;E),$$
is a projection and defines an element $[p_F]\in K_0(\Sigma^0_{H}(X;E))$. Under the isomorphism $K_1(C^*(T_HX))\cong KK_1(\C,C^*(T_HX))$ we have the identity
$$[\sigma_H(F)]=j_*\partial [p_F]\in K_1(C^*(T_HX)),$$
where $\partial:K_0(\Sigma^0_{H}(X;E))\to K_1(C^\infty_c(T_HX,\End(E)\otimes |\Lambda|)))$ is the boundary mapping and $j:C^\infty_c(T_HX,\End(E)\otimes |\Lambda|)\to C^*(T_HX;\End(E))$ is the inclusion. 

Similar statements can be made for the even $K$-theory class of an $H$-elliptic operator $[\sigma_H(F)]\in K_0(C^*(T_HX))$ if  $F\in \Psi^0_H(X;E_1,E_2)$ is unitary up to $\Psi^{-1}_H$. More directly (and generally), following \cite{baumvanerp}, if $F\in \Psi^m_H(X;E_1,E_2)$ is an $H$-elliptic operator and $a\in \tilde{\Sigma}_H^m(X;E_1,E_2)$ lifts $\sigma_H(D)$ then $(C^\infty_c(T_HX;E_1),C^\infty_c(T_HX,E_2),a)$ is a cycle for algebraic relative $K$-theory class in $K_0^{\rm alg}(\tilde{\Sigma}_H^m(X),C^\infty_c(T_HX)$. Using excision in algebraic $K$-theory, we can identify 
$$K_0^{\rm alg}(\tilde{\Sigma}_H^m(X),C^\infty_c(T_HX)\cong K_0(C^\infty_c(T_HX)),$$ 
and under this identification
$$[\sigma_H(F)]=j_* [(C^\infty_c(T_HX;E_1),C^\infty_c(T_HX,E_2),a)]\in K_0(C^*(T_HX)).$$
\end{remark}

\subsection{Localization and approximation of the symbol class}

Let us specialize to the case that $X$ is a regular Carnot manifold admitting flat orbits. We follow the notation of the Subsection \ref{charpadldaladoeo}. We are interested in better describing the class $[\sigma_H(D)]$ and in particular in finding pre-images under the surjection $K_*(I_X)\to K_*(C^*(T_HX))$. We write 
$$\pmb{\mathpzc{E}}^s_{\rm gr}(X;\pmb{E})_{t=0}:=\pmb{\mathpzc{E}}^s_{\rm gr}(X;\pmb{E})\otimes _{\mathrm{ev}_{t=0}}C^*(T_HX),$$
where $\pmb{\mathpzc{E}}^s_{\rm gr}(X;\pmb{E})$ is as in Section \ref{secondnaaoaoac}. For $B$-Hilbert $C^*$-modules $\mathpzc{E}_1$ and $\mathpzc{E}_2$, write 
$$\mathcal{Q}_B(\mathpzc{E}_1,\mathpzc{E}_2):=\Hom_B^*(\mathpzc{E}_1,\mathpzc{E}_2)/\mathbb{K}_B(\mathpzc{E}_1,\mathpzc{E}_2),$$
for the corona algebra space. We also write 
\begin{equation}
\label{gammalambda}
\Gamma_{\geq \lambda_0,X}:=\cup_{\lambda\geq \lambda_0}\delta_\lambda(\Gamma^\partial_X),
\end{equation}
for $\lambda_0>0$. The set $\Gamma_{\geq \lambda_0,X}$ is closed both in $\Gamma_X$ and in $\widehat{T_HX}$.

\begin{lemma}
\label{tehcnilmultplleme}
Let $X$ be a compact $\pmb{F}$-regular Carnot manifold. Let $\pmb{E}_1,\pmb{E}_2\to X$ be graded hermitean vector bundles. Assume that $\tilde{a}\in \tilde{\Sigma}^m_{H,{\rm gr}}(X;\pmb{E}_1,\pmb{E}_2)$ is invertible modulo $C^\infty_c(T_HX;\Hom(E)\otimes |\Lambda|)$. Let 
\begin{align*}
q:\Hom_{C^*(T_HX)}^*(\pmb{\mathpzc{E}}^s_{\rm gr}&(X;\pmb{E}_1)_{t=0},\pmb{\mathpzc{E}}^{s-m}_{\rm gr}(X;\pmb{E}_2)_{t=0})\to \\
&\mathcal{Q}_{C^*(T_HX)}^*(\pmb{\mathpzc{E}}^s_{\rm gr}(X;\pmb{E}_1)_{t=0},\pmb{\mathpzc{E}}^{s-m}_{\rm gr}(X;\pmb{E}_2)_{t=0}),
\end{align*}
denote the quotient map.

Then, up to stabilizing by a trivial, graded vector bundle, there exists a self-adjoint element 
$$\hat{\sigma}\in \Hom_{C_0(\Gamma_X)}^*(\pmb{\mathpzc{H}}^s_{\rm gr}(\pmb{E}_1),\pmb{\mathpzc{H}}^{s-m}_{\rm gr}(\pmb{E}_2))$$
 such that 
\begin{enumerate}
\item $\hat{\sigma}$ has an inverse $\hat{\sigma}^{-1}\in \End_{C_0(\Gamma_X)}^*(\pmb{\mathpzc{H}}^{s-m}_{\rm gr}(\pmb{E}_2),\pmb{\mathpzc{H}}^{s}_{\rm gr}(\pmb{E}_1))$ modulo $C_0(\Gamma_X)$-compact operators. 
\item $\hat{\sigma}$ (and similarly for $\hat{\sigma}^{-1}$) is in the image of the injection
\begin{align*}
\Hom_{C^*(T_HX)}^*(\pmb{\mathpzc{E}}^s_{\rm gr} (X;\pmb{E}_1)_{t=0}&,\pmb{\mathpzc{E}}^{s-m}_{\rm gr}(X;\pmb{E}_2)_{t=0})\hookrightarrow\\
&\hookrightarrow\Hom_{C_0(\Gamma_X)}^*(\pmb{\mathpzc{H}}^s_{\rm gr}(\pmb{E}_1),\pmb{\mathpzc{H}}^{s-m}_{\rm gr}(\pmb{E}_2))
\end{align*}
(cf. Proposition \ref{somrelareonon}) and there is a norm-continuous homotopy
$$q(\tilde{a})\sim_h q(\hat{\sigma}),$$
in the set of invertibles in 
$$\mathcal{Q}_{C^*(T_HX)}^*(\pmb{\mathpzc{E}}^s_{\rm gr}(X;\pmb{E}_1)_{t=0},\pmb{\mathpzc{E}}^{s-m}_{\rm gr}(X;\pmb{E}_2)_{t=0})$$
\item For $\lambda_0>0$, then as operators on $\pmb{\mathpzc{H}}_{\geq \lambda_0,\partial}^s(\pmb{E}_1):=\pmb{\mathpzc{H}}^s_{\rm gr}(\pmb{E}_1)\otimes_{C_0(\Gamma_X)}C_0(\Gamma_{\geq \lambda_0,X})$, 
$$\hat{\sigma}=\rho_m(a),$$ 
and in particular $\hat{\sigma}$ is as an operator on $\pmb{\mathpzc{H}}^s_{\rm gr}(\pmb{E}_1)$ homogeneous of degree $m$ up to locally $C_0(\Gamma_X)$-compact operators.
\end{enumerate}
In particular, the index class $[\tilde{a}]\in K_0(C^*(T_HX))$, defined from the Kasparov cycle $(\pmb{\mathpzc{E}}^s_{\rm gr}(X;\pmb{E}_1)_{t=0},\pmb{\mathpzc{E}}^{s-m}_{\rm gr}(X;\pmb{E}_2)_{t=0}, \tilde{a})$, is the image of the index class $[\hat{\sigma}]\in K_0(I_X)$, defined from the Kasparov cycle $(\pmb{\mathpzc{H}}^s_{\rm gr}(\pmb{E}_1),\pmb{\mathpzc{H}}^{s-m}_{\rm gr}(\pmb{E}_2),\hat{\sigma})$, under the surjection 
$$K_0(C_0(\Gamma_X))\cong K_0(I_X)\to K_0(C^*(T_HX)),$$
where the first map is defined from the Morita equivalence $\mathpzc{H}$.
\end{lemma}

\begin{proof}
After order reduction, we can assume that $s=m=0$ and that all bundles are trivially graded. To avoid notational complications we only consider the case that the vector bundles are the trivial line bundle, the general case is treated similarly but with lengthier notation. Note that under these reductions, 
\begin{align*}
\End_{C^*(T_HX)}^*(\mathpzc{E}(X;\C)_{t=0})=&\mathcal{M}( C^*(T_HX)),\\
\mathbb{K}_{C^*(T_HX)}(\mathpzc{E}(X;\C)_{t=0})=&C^*(T_HX),\\
\End_{C_0(\Gamma_X)}^*(\mathpzc{H})=&\mathcal{M}(I_X),\\
\mathbb{K}_{C_0(\Gamma_X)}(\mathpzc{H})=&I_X.
\end{align*}
Consider the commuting diagram with exact rows
\tiny
\[
\begin{CD}
0  @>>> C^*(T_HX)@>>>\mathcal{M}( C^*(T_HX))@>>>   \mathcal{Q}( C^*(T_HX)@>>>0 \\
@. @VVV  @VVV @ VV=V@.\\
0  @>>>  C^*(T_HX)/I_X @>>>\mathcal{M}( C^*(T_HX))/I_X @> >> \mathcal{Q}( C^*(T_HX))@>>>0 \\
\end{CD}. \]
\normalsize
Since $j_*:K_*(I_X)\to K_*(C^*(T_HX))$ is surjective the map $K_*(C^*(T_HX))\to K_*(C^*(T_HX)/I_X)$, induced by the quotient map, is the zero map. By naturality of boundary mappings, the map $K_*(\mathcal{M}( C^*(T_HX))/I_X) \to K_*(\mathcal{Q}( C^*(T_HX)))$ is surjective. In particular, the class $[\tilde{a}]\in K_1(\mathcal{Q}( C^*(T_HX)))$ is in the image of $K_1(\mathcal{M}( C^*(T_HX))/I_X) \to K_1(\mathcal{Q}( C^*(T_HX))$. Therefore, there is, up to stabilization, an element $\tilde{\sigma}\in \mathcal{M}( C^*(T_HX))$ invertible modulo $I_X$ which is homotopic to $\tilde{a}$ modulo $C^*(T_HX)$ in the invertibles of $\mathcal{Q}(C^*(T_HX))$. Since the set $\Gamma_{\geq \lambda_0,X}\subseteq \widehat{T_HX}$ is closed, and $\tilde{a}$ is invertible modulo $C^*(T_HX)$, $\tilde{a}$ acts invertibly over $\Gamma_{\geq \lambda_0,X}$ for some $\lambda_0>0$. Standard considerations ensures that we can take $\tilde{\sigma}$ to act as $\tilde{a}$ on $\Gamma_{\geq \lambda_0,X}$. The proof is complete upon taking $\hat{\sigma}$ to be the image of $\tilde{\sigma}$ under $\mathcal{M}( C^*(T_HX))\to  \mathcal{M}( I_X)$.
\end{proof}

\begin{remark}
\label{isofrmomsisglalm}
Since $\tilde{\sigma}$ is unitary up to $I_X$, it is especially invertible in the trivial representation. In particular, if $F\in \Psi^m_H(X;E_1,E_2)$ is an $H$-elliptic operator, Lemma \ref{tehcnilmultplleme} induces a (stable) isomorphism of vector bundles 
$$\sigma_0:E_1\xrightarrow{\sim} E_2.$$
\end{remark}

\begin{lemma}
\label{tehcnilmultpllemeodd}
Let $X$ be a compact $\pmb{F}$-regular Carnot manifold, $E\to X$ a hermitean vector bundle and $m>0$. Assume that $\tilde{a}\in \tilde{\Sigma}^m_{H}(X;E)$ is invertible modulo $C^\infty_c(T_HX;\Hom(E)\otimes |\Lambda|)$ and formally self-adjoint.

Then, up to stabilizing by a trivial vector bundle, there exists a self-adjoint, regular, densely defined operator 
$$\hat{\sigma}:\mathpzc{H}(E)\dashrightarrow \mathpzc{H}(E),$$
with 
$$\Dom(\hat{\sigma})=\mathpzc{H}^m(E),$$
such that 
\begin{enumerate}
\item $\hat{\sigma}$, viewed as an adjointable map of $C_0(\Gamma_X)$-Hilbert $C^*$-modules $\mathpzc{H}^m(E)\to \mathpzc{H}(E)$, has an inverse $\hat{\sigma}^{-1}:\mathpzc{H}^m(E)\to\mathpzc{H}(E)$ modulo $C_0(\Gamma_X)$-compact operators. 
\item $\hat{\sigma}$ (and similarly for $\hat{\sigma}^{-1}$) is in the image of the injection
\begin{align*}
\Hom_{C^*(T_HX)}^*(\pmb{\mathpzc{E}}^m(X;E)_{t=0},\pmb{\mathpzc{E}}(X;E)_{t=0})\hookrightarrow\Hom_{C_0(\Gamma_X)}^*(\mathpzc{H}^m(E),\mathpzc{H}(E)),
\end{align*}
(cf. Proposition \ref{somrelareonon}) and there is a norm-continuous homotopy
$$q(\tilde{a})\sim_h q(\hat{\sigma}),$$
in the set of invertibles in 
$$\mathcal{Q}_{C^*(T_HX)}^*(\pmb{\mathpzc{E}}^m(X;E)_{t=0},\pmb{\mathpzc{E}}(X;E)_{t=0})$$
\item For $\lambda_0>0$, then as operators on $\mathpzc{H}_{\geq \lambda_0,\partial}^m(E):=\mathpzc{H}^m(E)\otimes_{C_0(\Gamma_X)}C_0(\Gamma_{\geq \lambda_0,X})$, 
$$\tilde{\sigma}=\rho_m(a),$$ 
and in particular $q(\tilde{\sigma})$ is as an operator on $\mathpzc{H}^m(E)$ homogeneous of degree $m$ up to locally $C_0(\Gamma_X)$-compact operators.
\end{enumerate}
In particular, the index class $[\tilde{a}]\in K_1(C^*(T_HX))$, defined from the unbounded Kasparov cycle $(C^*(T_HX;E)), \tilde{a})$, is the image of the index class $[\hat{\sigma}]\in K_1(C_0(\Gamma_X))$, defined from the unbounded Kasparov cycle $(\mathpzc{H}(E),\hat{\sigma})$, under the surjection 
$$K_1(C_0(\Gamma_X))\cong K_1(I_X)\to K_1(C^*(T_HX)),$$
where the first map is defined from the Morita equivalence $\mathpzc{H}$.
\end{lemma}

\begin{proof}
The method of proof of Lemma \ref{tehcnilmultplleme} produces a formally self-adjoint $\tilde{\sigma}\in \Hom_{C^*(T_HX)}^*(\pmb{\mathpzc{E}}^m(X;E)_{t=0},\pmb{\mathpzc{E}}(X;E)_{t=0})$ fulfilling item (1)-(3). Viewing the associated $\hat{\sigma}$ as a densely defined operator, the self-adjointness and regularity then follows from item (1).
\end{proof}

\subsection{Further computations with the symbol class in the even case} 
The following lemma will be used for explicit geometric solutions to examples of the index problem. We note that the lemma requires assumptions that do not apply to all $H$-elliptic operators, just as is the case with the explicit geometric solution in \cite{baumvanerp}. We shall make use of Atiyah's notion of an elliptic complex, see Definition \ref{elliptcimadomad} below.

\begin{lemma}
\label{redielaclaclad}
Let $X$ be a compact $\pmb{F}$-regular Carnot manifold. Consider an orthogonal decomposition 
$$\mathcal{H}=\bigoplus_{k=0}^\infty \mathcal{H}_k,$$
of the bundle of flat representations $\mathcal{H}\to \Gamma_X$, where each $\mathcal{H}_k\to \Gamma_X$ is a finite rank hermitean vector bundle with $C_0(\Gamma_X,\mathcal{H}_k)\subseteq \cap_{s>0}\mathpzc{H}^s$. Define $\mathcal{H}^N:=\bigoplus_{k=0}^N \mathcal{H}_k$ and let $P_N$ denote the associated projection. 

Assume that 
$$\hat{\sigma}\in \Hom_{C_0(\Gamma_X)}^*(\pmb{\mathpzc{H}}^s_{\rm gr}(\pmb{E}_1),\pmb{\mathpzc{H}}^{s-m}_{\rm gr}(\pmb{E}_2))$$
for some $N_0>0$ satisfies 
\begin{enumerate}
\item $\hat{\sigma}$ is a $C_0(\Gamma_X)$-linear adjointable mapping invertible modulo $C_0(\Gamma_X)$-compact operators;
\item $[P_N,\hat{\sigma}]\in  \mathbb{K}_{C_0(\Gamma_X)}(\pmb{\mathpzc{H}}^s_{\rm gr}(\pmb{E}_1),\pmb{\mathpzc{H}}^{s-m}_{\rm gr}(\pmb{E}_2))$ for $N>N_0$;
\item the class of the $(\C,C_0(\Gamma_X))$-Kasparov cycle 
$$((1-P_N)\pmb{\mathpzc{H}}^s_{\rm gr}(\pmb{E}_1),(1-P_N)\pmb{\mathpzc{H}}^{s-m}_{\rm gr}(\pmb{E}_2),(1-P_N)\hat{\sigma}(1-P_N)),$$
vanishes in $KK_0(\C,C_0(\Gamma_X))\cong K^0(\Gamma_X)$ for all $N>N_0$. 
\end{enumerate}
Then the vector bundle morphism
$$\hat{\sigma}_N:=(P_N\otimes 1_{E_2})\hat{\sigma} (P_N\otimes 1_{E_1}):\mathcal{H}^{N}\otimes \pmb{E}_1\to \mathcal{H}^{N}\otimes \pmb{E}_2,$$
is invertible outside a compact subset of $\Gamma_X$. The elliptic complex on $\Gamma_X$
$$(\mathcal{H}^{N}\otimes \pmb{E}_1,\mathcal{H}^{N}\otimes \pmb{E}_2,\hat{\sigma}_N).$$
defines a class $[(\mathcal{H}^{N}\otimes \pmb{E}_1,\mathcal{H}^{N}\otimes \pmb{E}_2,\hat{\sigma}_N)]\in K^0(\Gamma_X)$ that for $N>N_0$ is independent of $N$. Moreover, for $N>N_0$, it holds that 
$$[(\mathcal{H}^{N}\otimes \pmb{E}_1,\mathcal{H}^{N}\otimes \pmb{E}_2,\hat{\sigma}_N)]=[(\pmb{\mathpzc{H}}^s_{\rm gr}(\pmb{E}_1),\pmb{\mathpzc{H}}^{s-m}_{\rm gr}(\pmb{E}_2),\hat{\sigma})]\in K^0(\Gamma_X).$$
\end{lemma}

\begin{proof}
By assumption (2), we can assume that $P_N$ commutes with $\hat{\sigma}$. Therefore item (1) ensures that $\hat{\sigma}_N$ is invertible modulo $C_0(\Gamma_X)$-compact operators on the sections of a vector bundle. As such, $\hat{\sigma}_N$ is invertible outside a compact subset of $\Gamma_X$. Moreover, by item (3), we have that 
\begin{align*}
[(\pmb{\mathpzc{H}}^s_{\rm gr}(\pmb{E}_1),&\pmb{\mathpzc{H}}^{s-m}_{\rm gr}(\pmb{E}_2),\hat{\sigma})]=\\
=&[(\mathcal{H}^{N}\otimes \pmb{E}_1,\mathcal{H}^{N}\otimes \pmb{E}_2,\hat{\sigma}_N)]+\\
&+[((1-P_N)\pmb{\mathpzc{H}}^s_{\rm gr}(\pmb{E}_1),(1-P_N)\pmb{\mathpzc{H}}^{s-m}_{\rm gr}(\pmb{E}_2),(1-P_N)\hat{\sigma}(1-P_N))]=\\
=&[(\mathcal{H}^{N}\otimes \pmb{E}_1,\mathcal{H}^{N}\otimes \pmb{E}_2,\hat{\sigma}_N)],
\end{align*}
for $N>N_0$. 
\end{proof}

\begin{remark}
We note that the elliptic complex $(\mathcal{H}^{N}\otimes \pmb{E}_1,\mathcal{H}^{N}\otimes \pmb{E}_2,\hat{\sigma}_N)$ in Lemma \ref{redielaclaclad} depends on the choice of bundle of flat orbits $\mathcal{H}$; if we replace $\mathcal{H}$ by $\mathcal{H}\otimes L$ for a line bundle $L\to \Gamma_X$ then $(\mathcal{H}^{N}\otimes \pmb{E}_1,\mathcal{H}^{N}\otimes \pmb{E}_2,\hat{\sigma}_N)$ transforms to $(\mathcal{H}^{N}\otimes \pmb{E}_1,\mathcal{H}^{N}\otimes \pmb{E}_2,\hat{\sigma}_N)\otimes L$. On the other hand, the class $[(\mathcal{H}^{N}\otimes \pmb{E}_1,\mathcal{H}^{N}\otimes \pmb{E}_2,\hat{\sigma}_N))]\otimes [\mathcal{H}^*]\in K_0(I_X)$ remains unaltered by changing $\mathcal{H}$ to $\mathcal{H}\otimes L$ because $L\otimes L^*\cong 1$. 
\end{remark}

\begin{remark}
Lemma \ref{redielaclaclad} mimics constructions of \cite{baumvanerp} for contact manifolds. The constructions from \cite{baumvanerp} fits with decompositions of $\mathcal{H}$ as follows. If $\ghani\subseteq \Aut_{\rm gr}(\mathfrak{g})$ is a compact subgroup such that the graded frame bundle of $X$ reduces to $\ghani$, and
$[\zeta|_\ghani]=0\in H^1(\ghani,\check{H}^1(\Gamma,U(1)))$ and $[c_\ghani]=0\in H^2(\ghani,C(\Gamma,U(1)))$, then 
\begin{equation}
\label{pqdecomeomd}
\mathcal{H}=\bigoplus_{\chi\in \hat{\ghani}} \mathcal{H}_\chi,
\end{equation}
under the Peter-Weyl decomposition. Assuming that each irreducible $\ghani$-representation has finite multiplicity in $\mathcal{H}$, then we have a decomposition as in Lemma \ref{redielaclaclad}. For the case of a $2n+1$-dimensional co-oriented contact manifold considered in \cite{baumvanerp} the relevant structure group is $\ghani=U(n)$. 
\end{remark}

\begin{lemma}
\label{redielaclacladcompactgamma}
Let $X$ be a compact $\pmb{F}$-regular Carnot manifold with $\Gamma^\partial_X$ compact. Assume that $\pmb{E}\to X$ is a graded, hermitean vector bundle. Consider an orthogonal decomposition 
$$\mathcal{H}=\bigoplus_{k=0}^\infty \mathcal{H}_k,$$
as in Lemma \ref{redielaclaclad} with the additional assumption that the decomposition is dilation invariant. Assume that 
$$\hat{\sigma}\in \Hom_{C_0(\Gamma_X)}^*(\pmb{\mathpzc{H}}^s_{\rm gr}(\pmb{E}),\pmb{\mathpzc{H}}^{s-m}_{\rm gr}(\pmb{E}))$$
satisfies the assumptions of Lemma \ref{redielaclaclad} and additionally, there is a homogenous of degree $m$ and invertible
$$\sigma\in \Hom_{C_0(\Gamma_X)}^*(\pmb{\mathpzc{H}}^s_{\rm gr}(\pmb{E}),\pmb{\mathpzc{H}}^{s-m}_{\rm gr}(\pmb{E}))$$
such that, for some $\lambda_0>0$, then as operators on $\pmb{\mathpzc{H}}_{\geq \lambda_0,\partial}^s(\pmb{E}):=\pmb{\mathpzc{H}}^s_{\rm gr}(\pmb{E})\otimes_{C_0(\Gamma_X)}C_0(\Gamma_{\geq \lambda_0,X})$, 
$$\tilde{\sigma}=\sigma.$$ 

Then, under the composition of the natural mappings
$$K^0(\Gamma_X)\cong K^0(\Gamma_X^\partial\times \R)\to K^0(\Gamma_X^\partial \times S^1),$$
we have that 
$$[(\pmb{\mathpzc{H}}^s_{\rm gr}(\pmb{E}),\pmb{\mathpzc{H}}^{s-m}_{\rm gr}(\pmb{E}),\hat{\sigma})]\mapsto [E_{\sigma,N}]-[(\mathcal{H}^N_\partial\otimes \pmb{E})\times S^1],$$ 
where $N>N_0$ and the vector bundle $E_{\sigma,N}\to \Gamma_X^\partial \times S^1$ is obtained from clutching $(\mathcal{H}^N_\partial\otimes \pmb{E})\times [0,1]\to  \Gamma_X^\partial \times [0,1]$ along $P_N\sigma P_N$ at the boundary. 
\end{lemma}

\begin{proof}
Using the reduction of Lemma \ref{redielaclaclad}, this is a standard computation in $K$-theory. Compare to \cite{baumvanerp}.
\end{proof}

\begin{corollary}
\label{charhellfofodotwocordonk}
Let $X$ be a compact regular polycontact manifold with polycontact structure $H$, equipped with a Riemannian metric $g$. Assume that $D_\gamma\in \mathcal{DO}_H^{2}(X;E)$ is as in Equation \eqref{secondoroddbbamddevep} and is $H$-elliptic. For $N>>0$, the symbol class $[\sigma_H(D_\gamma)]\in K_0(C^*(T_HX))$ has a pre-image $x\in K^0(\Gamma_X)=K^0(H^\perp\setminus X)$ that under the natural mapping 
$$K^0(H^\perp\setminus X)\cong K^0(S(H^\perp) \times \R)\to K^0(S(H^\perp) \times S^1),$$  
is mapped to $[E_{\gamma,N}]-[(\oplus_{k\leq N}(p_\Gamma^*H)^{\otimes^{\rm sym}_\C k}\otimes E)\times S^1]$ where $E_{\gamma,N}\to S(H^\perp) \times S^1$ is obtained from clutching $(\oplus_{k\leq N}(p_\Gamma^*H)^{\otimes^{\rm sym}_\C k}\otimes E)\times [0,1]\to  S(H^\perp) \times [0,1]$ along $\oplus_{k\leq N}\gamma_k$ at the boundary (where $\gamma_k$ is as in Proposition \ref{charhellfofodotwo}).
\end{corollary}

\subsection{Further computations with the symbol class in the odd case} 

Let us consider a lemma that applies to the case of a self-adjoint $H$-elliptic operator where the symbol can be decomposed along an idempotent. Recall the notation 
$$\mathpzc{H}_{\geq \lambda_0,\partial}^s(E):=\mathpzc{H}^s(E)\otimes_{C_0(\Gamma_X)}C_0(\Gamma_{\geq \lambda_0,X}).$$

\begin{lemma}
\label{oddredielaclacladcompactgamma}
Let $X$ be a compact $\pmb{F}$-regular Carnot manifold with $\Gamma^\partial_X$ compact. Assume that $E\to X$ is a hermitean vector bundle. Consider a formally self-adjoint $\tilde{a}\in \tilde{\Sigma}^m_H(X;E)$, of order $m>0$ and invertible modulo $C^\infty_c(T_HX;E)$. Assume that there is an idempotent $p\in \mathbb{K}_{C_0(\Gamma_X^\partial)}(\mathpzc{H}_\partial(E))$ such that for some $\lambda_0>0$
\begin{enumerate}
\item $p$ extends by continuity and homogeneity to a bounded idempotent operator on $\mathpzc{H}^m(E)$;
\item $[p,a]\in \mathbb{K}_{C_0(\Gamma_X)}(\mathpzc{H}_{\geq \lambda_0,\partial}^m(E),\mathpzc{H}_{\geq \lambda_0,\partial}(E))$;
\item $pa$ and $-(1-p)a$ are positive in the sense of densely defined quadratic forms on $\mathpzc{H}_{\geq \lambda_0,\partial}(E)$.
\end{enumerate}

Then, the index class $[\tilde{a}]\in K_1(C^*(T_HX))$, defined from the unbounded Kasparov cycle $(C^*(T_HX;E)), \tilde{a})$, is the image of an element $x\in K^1(\Gamma_X)$ that under the Bott isomorphism
$$K^1(\Gamma_X)\cong K^0(\Gamma_X^\partial),$$
is mapped to the class $[p(\mathcal{H}_\partial(E))]$.
\end{lemma}

\begin{proof}
The index class $[(C^*(T_HX;E)), \tilde{a})]\in K_1(C^*(T_HX))$ is the image of the class of the projection 
$$\left[\frac{1}{\pi}\arctan(\tilde{a})+\frac{1}{2}\right]\in \mathcal{Q}(C^*(T_HX;E)),$$
under the boundary mapping $K_0(\mathcal{Q}(C^*(T_HX)))\to K_1(C^*(T_HX))$. The image under this boundary mapping is the class of the unitary 
\begin{equation}
\label{unininararokadoamim}
\e^{2\pi i\left(\frac{1}{\pi}\arctan(\tilde{a})+\frac{1}{2}\right)}.
\end{equation}

On the other hand, $[p(\mathcal{H}_\partial(E))]$ corresponds to the class $[p]\in K_0(I^\partial_X)$ under the isomorphism $K^0(\Gamma_X^\partial)\cong K_0(I^\partial_X)$ defined from the Morita equivalence $\mathcal{H}$. The image of $[p]\in K_0(I^\partial_X)$ under 
$$K_0(I^\partial_X)\cong K_1(I_X)\to K_1(C^*(T_HX)),$$
is given by the unitary $\e^{2\pi ix}$ for a lift $x$ of $p$ under $C^*(T_HX)\to I_X^\partial$. By our assumptions it follows that $\e^{2\pi ix}$ is homotopic inside the unitaries of $C^*(T_HX)$ to the unitary \eqref{unininararokadoamim}.
\end{proof}

\begin{corollary}
\label{charhellfofodszego}
Let $X$ be a compact regular polycontact manifold with polycontact structure $H$, equipped with a Riemannian metric $g$. Assume that $P\in \Psi_H^0(X;E)$ is an idempotent Hermite operator. The symbol class $[\sigma_H(2P-1)]\in K_1(C^*(T_HX))$ has a pre-image $x\in K^1(\Gamma_X)=K^1(H^\perp\setminus X)$ that under the Bott mapping 
$$K^1(H^\perp\setminus X)\cong K^0(S(H^\perp)),$$  
is mapped to the class of the vector bundle 
$$\sigma_H(P)(\oplus_{k}(p_\Gamma^*H)^{\otimes^{\rm sym}_\C k}\otimes E)\to S(H^\perp).$$
\end{corollary}

\part[$K$-homological dualities]{$K$-homological dualities and index theory on Carnot manifolds}
\label{dualityapapadoda}

\begin{center}
{\bf Introduction to part}
\end{center}

To study index theory of $H$-elliptic operators on Carnot manifolds, dualities in $KK$-theory and geometric $K$-homology will play a cardinal role. Much of the technical setup of the earlier parts in this work serves the purpose of describing these dualities in more detail. In the last part of this work we put the various pieces together. We assume that the reader is familiar with $KK$-theory \cite{cumero,higsonprimer, knudjen,kaspnov,meslandbeast,meslandrencompt}, geometric $K$-homology \cite{Baum_Douglas,BDbor,baumvanerpI,Dee1,Ravthesis} and the connections between the two \cite{baumoyono,baumvanerpII,DGI}. The goal of this part is to set the index problem for $H$-elliptic operators in the Carnot calculus in a $K$-homological context via the dualities discussed above in Theorem \ref{commtererar} and apply the $K$-theory computations from the previous Part \ref{part:pseudod}. This part is divided into the following four sections:
\begin{itemize}
\item In Section \ref{geomedadonaodn} we develop a variation of geometric $K$-homology that uses geometric cycles modelled on noncompact spin$^c$-manifolds.
\item In Section \ref{sec:ordinadadp} we study $K$-homological dualities.
\item In Section \ref{sec:indexingener} we conclude index theorems for $H$-elliptic operators from Section \ref{sec:ordinadadp}.
\item In Section \ref{sec:grafadpknapdnarock} we discuss the problem of solving the index problem for graded Rockland sequences.
\end{itemize}

\section{Geometric $K$-homology with coefficients in elliptic complexes}
\label{geomedadonaodn}

In order to construct geometric dualities, we shall first need to construct a slight modification of Baum-Douglas geometric $K$-homology \cite{Baum_Douglas,BDbor}. Fix a finite CW-complex $X$. A geometric cycle on $X$ is a triple $(M,E,f)$, where $M$ is a compact spin$^c$-manifold, $E\to M$ a complex vector bundle and $f:M\to X$ is a continuous map. We recall that Baum-Douglas constructs $K_*^{\rm geo}(X)$ as the group of equivalence classes of geometric cycles modulo a relation generated by the direct sum/disjoint union relation, bordism and vector bundle modification \cite{Baum_Douglas,BDbor}. By the results of \cite{Baum_Douglas,BDbor}, proven in more detail in \cite{baumoyono,baumvanerpII} there is a natural isomorphism 
$$\gamma:K_*^{\rm geo}(X)\to K_*^{\rm an}(X):=KK_*(C(X),\C),$$
that maps the class of the cycle $(M,E,f)$ to the $K$-homology class $f_*[D_E]$ where $[D_E]\in KK_*(C(M),\C)$ is the $K$-homology class of a spin$^c$-Dirac operator on $M$ twisted by $E$. 

As we saw in the preceding parts, the spin$^c$-manifold $\Gamma_X$ plays an important role when $X$ is a regular Carnot manifold with flat orbits. However, $\Gamma_X$ is not compact. For the specific example of contact manifolds, this problem was solved in \cite{baumvanerp} by using the dilation action to reduce to $\Gamma_X^\partial$ which is compact as it is a trivial two-fold cover of $X$ for a compact cooriented contact manifold. In general, the authors see no naturally associated compact manifold constructed from $\Gamma_X$. So we shall utilize a mild modification of Baum-Douglas' geometric $K$-homology modelled on non-compact manifolds with coefficients in elliptic complexes. As discussed in Remark \ref{commenonmodedl} above, this modification fits well with Baum-Douglas' solution to the index problem for elliptic operators. This non-compact analogue of Baum-Douglas' geometric $K$-homology relies on well known results, for instance Bunke's work \cite{bunkerelative}, and is in its own right not overly exciting. We develop it for the sake of simplifying the description of the geometric solution to the index problem for $H$-elliptic operators. A similar model for geometric $K$-homology can be extracted from \cite{emermeyer}.

Recall Atiyah's notion of an elliptic complex:

\begin{definition}
\label{elliptcimadomad}
Let $M$ be a topological space. An elliptic complex on $M$ is a triple $\xi=(E_1,E_2,\sigma)$ where $E_1,E_2\to M$ are complex vector bundles and $\sigma:E_1\to E_2$ is a vector bundle morphism which is an isomorphism outside a compact subset of $M$. An isomorphism of elliptic complexes $(E_1',E_2',\sigma')\cong (E_1,E_2,\sigma)$ consists of isomorphisms $E_1\cong E_1'$ and $E_2\cong E_2'$ compatible with the bundle maps $\sigma$ and $\sigma'$. 

We also introduce the following terminology.
\begin{itemize}
\item The support of $\xi$ or $\sigma$ is defined as
$$\mathrm{supp}(\xi)\equiv \mathrm{supp}(\sigma):=\{x\in M: \sigma(x)\;\mbox{is not an isomorphism}\}.$$
\item If an elliptic complex $(E_1,E_2,\sigma)$ has empty support $\mathrm{supp}(\sigma)=\emptyset$ we say that $(E_1,E_2,\sigma)$ is degenerate.
\item If $(E_1,E_2,\sigma)$ is an elliptic complex on $M\times [0,1]$, we say that the elliptic complexes $(E_1|_{M\times \{0\}},E_2|_{M\times \{0\}},\sigma|_{M\times \{0\}})$ and $(E_1|_{M\times \{1\}},E_2|_{M\times \{1\}},\sigma|_{M\times \{1\}})$ are homotopic. 
\item We say that two elliptic complexes $\xi'=(E_1',E_2',\sigma')$ and $\xi=(E_1,E_2,\sigma)$ are equivalent if there are degenerate elliptic complexes $\xi_0, \xi_1$ and a homotopy from $\xi\oplus \xi_0$ to $\xi'\oplus \xi_1$.
\end{itemize}
The set of equivalence classes of isomorphism classes of elliptic complexes is denoted by $\tilde{K}^0(M)$. 
\end{definition}

It is well known that $\tilde{K}^0(M)\cong K^0(M)\cong KK_0(\C,C_0(M))$ via natural isomorphisms if $M$ has finite covering dimension. Indeed, if $M$ is compact the isomorphism $\tilde{K}^0(M)\cong K^0(M)$ is implemented by mapping $(E_1,E_2,\sigma)\mapsto [E_1]-[E_2]$; its inverse maps a vector bundle $E\to M$ to the elliptic complex $(E,0,0)$. We note that the isomorphism $\tilde{K}^0(M)\cong KK_0(\C,C_0(M))$ indicates that is is possible to model $\tilde{K}^0(M)$ on triples $(\mathpzc{E}_1,\mathpzc{E}_2,\sigma)$ where $\mathpzc{E}_1$ and $\mathpzc{E}_2$ are $C_0(M)$-Hilbert $C^*$-modules and $\sigma$ is a Fredholm map which is invertible modulo $C_0(M)$-compact operators. We return to this last point below in a more specific context.

\begin{definition}[Cycles for $\tilde{K}_*^{\rm geo}(X)$]
Let $X$ be a finite CW-complex. A cycle for $\tilde{K}_*^{\rm geo}(X)$ (or a Baum-Douglas cycle with coefficients in elliptic complexes) is a triple $(M,\xi,f)$ where 
\begin{enumerate}
\item $M$ is a spin$^c$-manifold;
\item $\xi$ is an elliptic complex on $M$; and
\item $f:M\to X$ is a continuous map.
\end{enumerate}
Two cycles $(M,\xi,f)$ and $(M',\xi',f')$ for $\tilde{K}_*^{\rm geo}(X)$ are said to be isomorphic if there is a spin$^c$-preserving diffeomorphism $\varphi:M\to M'$ such that $\varphi^*\xi'=\xi$ and $f=f'\circ \varphi$.
\end{definition}

We shall only concern ourselves with $X$ being a finite CW-complex, but in principle one can modify the construction to non-compact $X$, but still locally compact. For instance, compactly supported cycles would be defined as above with the additional requirement that $f$ has compact range and locally finite cycles would be defined as above with the additional requirement that $f$ is proper on the support of $\xi$. The reader can find details on such variations in \cite{emermeyer}.

\begin{example}
\label{knasdpandpinas}
Let us give an example of a Baum-Douglas cycle with coefficients in elliptic complexes that is prototypical for how they shall be used in this work. Assume that $X$ is an $\pmb{F}$-regular Carnot manifold. Take a relative cycle $(p,q,u)$ for $K_0(I_X)$, i.e. $p,q\in \mathcal{M}(I_X)$ are projections and $u\in  \mathcal{U}(I_X)$ is a unitary muliplier with $upu^*-q\in I_X$. Assume for simplicity that $p$ and $q$ are $C_0(\Gamma_X)$-locally compact, i.e. $g p, gq\in I_X$ for any $g\in C_0(\Gamma_X)$. Take a bundle of flat representation $\mathcal{H}\to \Gamma_X$. We then consider the triple 
$$(p\mathcal{H},q\mathcal{H},qup).$$
Since  $p$ and $q$ are $C_0(\Gamma_X)$-locally compact, $p\mathcal{H},q\mathcal{H}\to \Gamma_X$ are complex vector bundles and since $upu^*-q\in I_X$ the morphism $qup:p\mathcal{H}\to q\mathcal{H}$ is an isomorphism outside a compact subset (where its inverse is $pu^*q$). 

The prototypical example of relevance for this work is the Baum-Douglas cycle with coefficients in elliptic complexes of the form 
$$(\Gamma_X,(p\mathcal{H},q\mathcal{H},qup),p_\Gamma),$$
where $p_\Gamma:\Gamma_X\to X$ is the projection. We equipp $\Gamma_X$ with the spin$^c$-structure from Proposition \ref{spincongamma}. Cycles of this form are constructed from $H$-elliptic operators by means of the methods of Section \ref{sec:ktheoomon}.
\end{example}

To construct a $K$-homology group, we need to construct an equivalence relation. We proceed as in \cite{Baum_Douglas,BDbor}. By an abuse of notation we shall mean isomorphism class of a geometric cycle when referring to a geometric cycle and by a set of geometric cycles we tacitly mean the corresponding set of isomorphism classes of geometric cycles.

\begin{definition}[The direct sum/disjoint union relation for $\tilde{K}_*^{\rm geo}(X)$]
Let $X$ be a finite CW-complex. For two cycles for $\tilde{K}_*^{\rm geo}(X)$ of the form $(M,\xi,f)$ and $(M,\xi',f)$ we declare that 
$$(M,\xi,f)\dot{\cup} (M,\xi',f)\sim_{didi} (M,\xi+\xi',f).$$
\end{definition}

\begin{definition}[The degenerate relation for $\tilde{K}_*^{\rm geo}(X)$]
Let $X$ be a finite CW-complex. A cycle $(M_0,\xi_0,f_0)$ for $\tilde{K}_*^{\rm geo}(X)$ is said to be degenerate if $\xi_0$ is degenerate. We say that two cycles $(M,\xi,f)$ and $(M',\xi',f')$ for $\tilde{K}_*^{\rm geo}(X)$ are degenerately equivalent if there are degenerate cycles $(M_0,\xi_0,f_0)$ and $(M_1,\xi_1,f_1)$ for $\tilde{K}_*^{\rm geo}(X)$ with 
$$(M,\xi,f)\dot{\cup}(M_0,\xi_0,f_0)=(M',\xi',f')\dot{\cup}(M_1,\xi_1,f_1).$$
\end{definition}

\begin{definition}[The bordism relation for $\tilde{K}_*^{\rm geo}(X)$]
Let $X$ be a finite CW-complex. A cycle with boundary for $\tilde{K}_*^{\rm geo}(X)$ is a triple $(W,\xi,f)$ where 
\begin{enumerate}
\item $W$ is a spin$^c$-manifold with boundary;
\item $\xi$ is an elliptic complex on $W$; and
\item $f:W\to X$ is a continuous map.
\end{enumerate}
The cycle $\partial(W,\xi,f):=(\partial W,\xi|_{\partial W},f|_{\partial W})$ is said to be nullbordant. 

We say that two cycles $(M,\xi,f)$ and $(M',\xi',f')$ for $\tilde{K}_*^{\rm geo}(X)$ are bordant if there is a cycle with boundary $(W,\xi,f)$ for $\tilde{K}_*^{\rm geo}(X)$ such that 
$$\partial(W,\xi,f)=(M,\xi,f)\dot{\cup} -(M',\xi',f').$$
We write this as 
$$(M,\xi,f)\sim_{\rm bor}(M',\xi',f').$$

The bordism relation is an equivalence relation.
\end{definition}

Before proceeding further we note the strength of the bordism relation in the following two lemmas.

\begin{lemma}
\label{bordismsotosupport}
Let $X$ be a finite CW-complex and $(M,\xi,f)$ a cycle for $\tilde{K}_*^{\rm geo}(X)$. Assume that $M_0\subseteq M$ is an open domain with smooth boundary such that $\mathrm{supp}(\xi)\subseteq M_0$. Then there is a bordism
\begin{equation}
\label{firstobovfodo}
(M,\xi,f)\sim_{\rm bor}(M_0,\xi|_{M_0},f|_{M_0}).
\end{equation}
\end{lemma}

\begin{proof}
We note that the cycle $(M_0,\xi|_{M_0},f|_{M_0})$ is well defined because $\mathrm{supp}(\xi)\subseteq M_0$. To construct the bordism, we take 
$$W:=M\times [0,1]\setminus ((M\setminus M_0)\times \{0\}).$$
Since $M\setminus M_0$ is closed in $M$, it is clear that $W$ is a spin$^c$-manifold with boundary. We tacitly abuse the notations and write $f$ for the composition of $f:M\to X$ with obvious maps such as the projection $W\to M$. The cycle $(W,(\xi\times [0,1])|_{W},f)$ is a cycle with boundary, and since 
$$\partial (W,(\xi\times [0,1])|_{W},f)=(-(M_0,\xi|_{M_0},f|_{M_0}))\dot{\cup}(M,\xi,f),$$
we conclude that there is a bordism as in Equation \eqref{firstobovfodo}.
\end{proof}

\begin{lemma}
\label{bordismsotorcompact}
Let $X$ be a finite CW-complex and $(M,\xi,f)$ a cycle for $\tilde{K}_*^{\rm geo}(X)$. Assume that $M_0\subseteq M$ is a pre-compact open domain with smooth boundary such that $\mathrm{supp}(\xi)\subseteq M_0$. Consider the spin$^c$-manifold obtained from the doubling construction
$$Z:=2M_0=\overline{M_0}\cup_{\partial M_0}(-\overline{M_0}).$$
Write $\xi=(E_1,E_2,\sigma)$ and define $E_\xi\to Z$ as the vector bundle 
$$E_1\cup_{\sigma|_{\partial M_0}}E_2\to Z,$$
obtained from clutching $E_1$ and $E_2$ along the isomorphism $\sigma|_{\partial M_0}$. Then $(M,\xi,f)$ and $(Z,(E_\xi,0,0),f)$ lie in the same equivalence class of the relation generated by the disjoint union/direct sum relation, degenerate equivalence and the bordism relation.
\end{lemma}

In particular,  a cycle with elliptic complexes $(M,\xi,f)$ is equivalent to a cycle of the form $(Z,(E_\xi,0,0),f)$ where $(Z,E_\xi,f)$ is an ordinary Baum-Douglas geometric cycle (with a compact spin$^c$-manifold and vector bundle coefficients).

\begin{proof}
We can extend $E_1,E_2\to M_0$ to vector bundles on $Z$ by gluing along the identity map on $\partial M_0$. By an abuse of notations we denote these also by $E_1,E_2\to Z$. By construction, $E_\xi|_{M_0}=E_1|_{M_0}$ and $E_\xi|_{Z\setminus M_0}=E_2|_{Z\setminus M_0}$. The morphism $\sigma$ glues together with the identity on $Z\setminus M_0$ to a vector bundle morphism 
$$\tilde{\sigma}:E_\xi\to E_2,$$
over $Z$. Note that the support of $\sigma$ in $M_0$ is unaltered by this extension procedure. 
Next, we prove that 
$$(Z,(E_\xi,0,0),f)\sim_{\rm bor} (Z,(E_\xi,0,0),f)\dot{\cup} (Z,(0,E_2,0),f).$$
This follows from the fact that $(Z,(0,E_2,0),f)=\partial (M_0\times [0,1],(0,E_2\times [0,1],0),f)$. 
We have a direct sum/disjoint union relation
$$(Z,(E_\xi,0,0),f)\dot{\cup} (Z,(0,E_2,0),f)\dot{\cup}(Z,(E_2,E_2,0),f)\sim_{didi} (Z,(E_\xi\oplus E_2,E_2\oplus E_2,0),f).$$
Since homotopy of elliptic complexes induces a bordism, and $Z$ is compact, we have a bordism 
$$(Z,(E_\xi\oplus E_2,E_2\oplus E_2,0),f)\sim_{\rm bor} \left(Z,(E_\xi\oplus E_2,E_2\oplus E_2,\tilde{\sigma}\oplus 1),f\right).$$
Consider the spin$^c$-manifold with boundary
$$W':=Z\times [0,1]\setminus ((Z\setminus M_0\times \{0\}).$$
The cycle with boundary 
$$\left(W',\left((E_\xi\oplus E_2,E_2\oplus E_2,\tilde{\sigma}\oplus 1)\times [0,1]\right)|_{W'},f\right)$$
implements a bordism
$$\left(Z,(E_\xi\oplus E_2,E_2\oplus E_2,\tilde{\sigma}\oplus 1),f\right)\sim_{\rm bor}(M_0,\xi|_{M_0}+(E_2,E_2,1),f),$$
because $\tilde{\sigma}|_{M_0}=\sigma$. Since $(E_2,E_2,1)$ is degenerate, we have a degenerate equivalence 
$$(M_0,\xi|_{M_0}+(E_2,E_2,1),f)\sim (M_0,\xi|_{M_0},f).$$

Combining these constructions with the bordism of Lemma \ref{bordismsotosupport}, we see that $(M,\xi,f)$ and $(Z,(E_\xi,0,0),f)$ lie in the same equivalence class of the relation generated by the disjoint union/direct sum relation, degenerate equivalence and the bordism relation.
\end{proof}

\begin{remark}
Lemma \ref{bordismsotorcompact} is akin to the ideas appearing in Baum-Douglas' geometric proof of the Atiyah-Singer index theorem, discussed in Remark \ref{commenonmodedl} above. Indeed, if $A:C^\infty(X;E_1)\to C^\infty(X;E_2)$ is an elliptic pseudodifferential operator on a compact manifold $X$, consider the spin$^c$-manifold $M:=T^*X$ and the elliptic complex $\xi_A=(p^*E_1,p^*E_2,\sigma(A))$. Here $p:M=T^*X\to X$ denotes the projection. If we take $M_0$ to be the ball bundle in $T^*X$, Lemma \ref{bordismsotorcompact} provides a cycle $(Z,E_{\xi_A},p\cup_{\partial M_0}p)$ bordant to $(T^*X,\xi_A,p)$ which coincides on the nose with Baum-Douglas' geometric cycle representing the index problem for elliptic pseudodifferential operators \cite[Part 5]{Baum_Douglas}. In situations where there is no natural place to clutch, one might argue that not clutching produces a more natural construction.
\end{remark}

\begin{definition}[Vector bundle modification for $\tilde{K}_*^{\rm geo}(X)$]
Let $X$ be a finite CW-complex. Consider a cycle $(M,\xi,f)$ for $\tilde{K}_*^{\rm geo}(X)$ and a spin$^c$-vector bundle $V\to M$ of even rank. Write $M^V:=S(V\oplus 1_\R)$ for the sphere bundle in some metric on $V$, where $1_\R\to M$ denotes the trivial line bundle. Let $\beta_V\to M^V$ denote the Bott bundle (cf. \cite{Baum_Douglas,BDbor}) and let $p_V:M^V\to M$ denote the projection. 

The vector bundle modification of $(M,\xi,f)$ along $V$ is the cycle for $\tilde{K}_*^{\rm geo}(X)$ defined by 
$$(M,\xi,f)^V:=(M^V,p_V^*\xi\otimes \beta_V,f\circ p_V).$$ 
\end{definition}

We can now define the geometric $K$-homology group with coefficients in elliptic complexes. 

\begin{definition}
Let $X$ be a finite CW-complex. Write $\sim_{\rm BD}$ for the equivalence relation on the set of isomorphism classes of cycles for $\tilde{K}_*^{\rm geo}(X)$ generated by 
\begin{itemize}
\item The disjoint union/direct sum relation
\item Degenerate equivalence
\item The bordism relation
\item Vector bundle modification
\end{itemize}
Let $\tilde{K}_*^{\rm geo}(X)$ denote the set of equivalence classes under $\sim_{\rm bor}$ on the isomorphism classes of cycles for $\tilde{K}_*^{\rm geo}(X)$.
\end{definition}

\begin{theorem}
\label{isomofopsmdapada}
Let $X$ be a finite CW-complex. The set $\tilde{K}_*^{\rm geo}(X)$ forms a $\Z/2$-graded abelian group under disjoint union graded by parity of dimension. Moreover, the map at the level of cycles 
\begin{equation}
\label{isomorpkfmaoadladlda}
(M,E,f)\mapsto (M,(E,0,0),f),
\end{equation}
defines an isomorphism of $\Z/2$-graded abelian groups
$$K_*^{\rm geo}(X)\cong \tilde{K}_*^{\rm geo}(X).$$
Its inverse is the mapping 
$$(M,\xi,f)\mapsto (Z,E_\xi,f),$$
where $(Z,E_\xi,f)$ is constructed as in Lemma \ref{bordismsotorcompact}.
\end{theorem}

\begin{proof}[Sketch of proof]
We only sketh the details in the proof of this theorem as the result follows from standard techniques. Indeed, almost by definition $-(M,\xi,f):=(-M,\xi,f)$ is an inverse modulo bordism under disjoint union so  $\tilde{K}_*^{\rm geo}(X)$ forms a $\Z/2$-graded abelian group. The fact that $K_*^{\rm geo}(X)\cong \tilde{K}_*^{\rm geo}(X)$ via the mapping \eqref{isomorpkfmaoadladlda} follows from that the map is surjective by Lemma \ref{bordismsotorcompact}. A modification of the proof of Lemma \ref{bordismsotorcompact} shows that the clutching construction $(M,\xi,f)\mapsto (Z,E_\xi,f)$ produces a left inverse so the map is injective. 
\end{proof}

\begin{definition}[Analytic assembly of cycles for $\tilde{K}_*^{\rm geo}(X)$]
\label{analalaslaslaslsalaslasl}
Let $(M,\xi,f)$ be a cycle for  $\tilde{K}_*^{\rm geo}(X)$. Write $\xi=(E_1,E_2,\sigma)$. We choose a complete metric on $M$ and consider a complete spin$^c$-Dirac operator $\slashed{D}_M$ on the complex spinor bundle $S_M\to M$. We assume that $E_1$ and $E_2$ are equipped with hermitean structures such that $\sigma$ is unitary outside a compact. A (hermitean) connection for $\xi$ is a pair of (hermitean) connections $(\nabla_{E_1},\nabla_{E_2})$ for $E_1$ and $E_2$, respectively, such that 
$$\nabla_{E_1}=\sigma^*\nabla_{E_2}\sigma,$$ 
outside a compact subset of $M$. Choose a positive function $\chi\in C^\infty_c(M)$ such that $\chi=1$ on the supports of $1-\sigma^*\sigma$ and $\nabla_{E_1}-\sigma^*\nabla_{E_2}\sigma$ and a $\lambda\geq 0$ such that 
$$\lambda>\sup\{\|1-\sigma(x)^*\sigma(x)\|, \|(\nabla_{E_1}-\sigma^*\nabla_{E_2}\sigma)(x)\|: \; x\in M\}.$$

If $(\nabla_{E_1},\nabla_{E_2})$ is a hermitean connection for $\xi=(E_1,E_2,\sigma)$, we define the analytic assembly $\tilde{\gamma}(M,\xi,f)\in K_*^{\rm an}(X)$ as the class of the $K$-cycle
\begin{equation}
\label{deifneifnedinadomlasmpasdm}
\left(\begin{matrix}L^2(M;S_M\otimes E_1)\\\oplus\\ L^2(M;-S_M\otimes E_2)\end{matrix},D_\xi(\lambda\chi+D_\xi^2)^{-1/2}\right)
\end{equation}
where $C(X)$ acts on the Hilbert space by pull back along $f$ and where
$$D_\xi:=\begin{pmatrix}
D_M\otimes \nabla_{E_1}& \sigma^*\\
\sigma&-D_M\otimes \nabla_{E_2}
\end{pmatrix}.$$
\end{definition}

The choice of $\lambda$ is justified from the proof of the following proposition.

\begin{proposition}
The $K$-cycle \eqref{deifneifnedinadomlasmpasdm} defining the analytic assembly $\tilde{\gamma}(M,\xi,f)\in K_*^{\rm an}(X)$ is well defined and the class $\tilde{\gamma}(M,\xi,f)$ does not depend on the choices involved (of $\chi$, $\lambda$, connections or metrics).
\end{proposition}

\begin{proof}
First we prove that $D_\xi$ is self-adjoint. For that purpose, we write 
\begin{equation}
\label{dxisllalasl}
D_\xi:=\begin{pmatrix}
D_M\otimes \nabla_{E_1}& 0\\
0&-D_M\otimes \nabla_{E_2}
\end{pmatrix}+\begin{pmatrix}
0& \sigma^*\\
\sigma&0
\end{pmatrix}.
\end{equation}
Since $D_M$ is a Dirac operator defined from a complete metric the first term is self-adjoint, and the second term is bounded and self-adjoint by construction. The two terms in \eqref{dxisllalasl} anti-commute up to a smooth, compactly supported vector bundle endomorphism, because $\nabla_{E_1}=\sigma^*\nabla_{E_2}\sigma$ outside a compact, so $D_\xi$ is self-adjoint. 

The decomposition \eqref{dxisllalasl} shows that 
$$D_\xi^2=\begin{pmatrix}
1+(D_M\otimes \nabla_{E_1})^2& 0\\
0&1+(D_M\otimes \nabla_{E_2})^2
\end{pmatrix}+V, $$
where 
$$V:=\begin{pmatrix}
1-\sigma^*\sigma&(D_M\otimes \nabla_{E_1})\sigma^*-\sigma^*(D_M\otimes \nabla_{E_2})\\
\sigma (D_M\otimes \nabla_{E_2}) -(D_M\otimes \nabla_{E_1})\sigma&1-\sigma\sigma^*
\end{pmatrix}$$
is a self-adjoint, smooth, compactly supported vector bundle endomorphism. In particular, with $\chi$ and $\lambda$ as in Definition \ref{analalaslaslaslsalaslasl} we can take $\epsilon>0$ so that 
$$\lambda-\epsilon\geq\sup\{\|1-\sigma(x)^*\sigma(x)\|, \|(\nabla_{E_1}-\sigma^*\nabla_{E_2}\sigma)(x)\|: \; x\in M\}.$$
and therefore 
$$\langle (D_\xi^2+\lambda \chi)\varphi,\varphi\rangle_{L^2}\geq \epsilon \langle \varphi,\varphi\rangle_{L^2},\quad \varphi\in C^\infty_c.$$
Therefore, $D_\xi^2+\lambda \chi$ is invertible and so $D_\xi$ is ``invertible at infinity'' in the terminology of \cite{bunkerelative}. It follows from \cite[Proposition 1.13]{bunkerelative} that the $K$-cycle \eqref{deifneifnedinadomlasmpasdm} is well defined. Moreover, the choices involved in the construction are all from pathconnected spaces that therefore does not affect the class in $K_*^{\rm an}(X)$.
\end{proof}

By combining \cite{bunkerelative} with Kucerovsky's theorem \cite{dantheman}, we in fact have the following result. We write $C_g(M)$ for the $C^*$-algebra of bounded continuous function defined in \cite{bunkerelative}. 

\begin{proposition}
\label{anaalalsaaskasprood}
Let $(M,\xi,f)$ be a cycle for $K_*^{\rm an}(X)$ with $\xi=(E_1,E_2,\sigma)$. Take a complete spin$^c$-Dirac operator $\slashed{D}_M$ on $M$ and write $[\slashed{D}_M]\in KK_*(C_0(M),\C)$ for its associated $K$-homology class. We also define the $KK$-class
$$[M,\xi]:=\left[\left(C_0(M,E_1\oplus E_2),\begin{pmatrix}0&\sigma^*\\\sigma&0\end{pmatrix}\right)\right]\in KK_0^{M}(C_{g}(M),C_0(M)).$$
Then the analytic assembly of $(M,\xi,f)$ can be written as a Kasparov product
$$\tilde{\gamma}(M,\xi,f)=f_*\left([M,\xi]\otimes_{C_0(M)}[\slashed{D}_M]\right)\in K_*^{\rm an}(X).$$
\end{proposition} 

\begin{lemma}
\label{anasdisso}
The analytic assembly of cycles with coefficients in the elliptic complexes respects the disjoint union/direct sum relation and degenerate equivalence.
\end{lemma}

\begin{proof}
It is immediate that the disjoint union/direct sum relation. As for respecting  degenerate equivalence, it suffices to prove that $\tilde{\gamma}(M,\xi,f)=0\in K_*(X)$ if $\xi$ is degenerate. Indeed, if $\xi=(E_1,E_2,\sigma)$ is degenerate we can assume that $\sigma$ is unitary and that $\nabla_1=\sigma^*\nabla_2\sigma$ on all of $M$. We define the symmetry 
$$F:=\begin{pmatrix}
0& i\sigma^*\\
-i\sigma&0
\end{pmatrix}.$$
The symmetry $F$ commutes with $C(X)$ and anti-commutes with $D_\xi$ and so the argument proving \cite[Theorem 1.14]{bunkerelative} shows that $\tilde{\gamma}(M,\xi,f)=0\in K_*(X)$. 
\end{proof}

\begin{lemma}
\label{anasborvb}
The analytic assembly of cycles with coefficients in the elliptic complexes respects bordism and vector bundle modification.
\end{lemma}

\begin{proof}
That analytic assembly respects vector bundle modification is proven in the same way as in \cite[Subsection 4.2]{kkbor}. More specifically, using an argument as in \cite[Subsection 4.2]{kkbor} we can write $\tilde{\gamma}((M,\xi,f)^V)=\tilde{\gamma}(M,\xi,f)+x$ where $x$ is represented by a $K$-cycle $(\mathcal{H},F)$ admitting a symmetry commuting with $C(X)$ and anti-commuting with $F$ up to compact operators and the argument proving \cite[Theorem 1.14]{bunkerelative} shows that $x=0\in K_*(X)$.

To show that analytic assembly respects bordism, we need to prove that  if $(W,\xi,f)$ is a cycle with boundary then $\tilde{\gamma}(\partial (W,\xi,f))=0$. Write $\xi=(E_1,E_2,\sigma)$. Write $C_g(W)$ and $C_g(\partial W)$ for the $C^*$-algebras of bounded continuous functions from \cite{bunkerelative}. The Dirac operator defines a $K$-homology class $[\slashed{D}_W]\in KK_*(C_0(W^\circ),\C)$, and similarly, the Dirac operator on the boundary produces a class $[\slashed{D}_{\partial W}]\in KK_*(C_0(\partial W),\C)$. It follows from the Baum-Douglas-Taylor theorem that 
$$\partial[\slashed{D}_W]=[\slashed{D}_{\partial W}],$$
under the boundary mapping induced from the short exact sequence 
\begin{align*}
0\to C_0(W^\circ)\to& C_0(\overline{W})\to C_0(\partial W))\to 0.
\end{align*}
Write $C_{g,0}(W^\circ):=\{f\in C_g(\overline{W}): f|_{\partial W}=0\}$. On the other hand, we have classes 
\begin{align*}
[W^\circ,\sigma]&:=\left[\left(C_0(W^\circ,E_1\oplus E_2),\begin{pmatrix}0&\sigma^*\\\sigma&0\end{pmatrix}\right)\right]\in KK_0^{W^\circ}(C_{g,0}(W^\circ),C_0(W^\circ));\\
[\partial{W},\sigma]&:=\left[\left(C_0(\partial W,E_1\oplus E_2),\begin{pmatrix}0&\sigma^*\\\sigma&0\end{pmatrix}\right)\right]\in KK_0^{\partial{W}}(C_g(\partial{W}),C_0(\partial{W})).
\end{align*}
Here we grade $E_1\oplus E_2$ with $E_1$ as even and $E_2$ as odd. Using $C_0(\overline{W})$-linearity, we have that 
$$[\partial{W},\sigma]\otimes_{C_0(\partial{W})}\partial=\partial \otimes_{C_{g,0}(W^\circ)}[W^\circ,\sigma],$$
where the second boundary mapping is that coming from the short exact sequence 
$$0\to C_{g,0}(W^\circ)\to C_g(\overline{W})\to C_g(\partial W)\to 0.$$
Let $\iota:\partial W\hookrightarrow \overline{W}$ denote the boundary inclusion. Using \cite{bunkerelative} and Kucerovsky's theorem \cite{dantheman}, we have that 
\begin{align*}
\tilde{\gamma}(\partial (W,\xi,f))=&(f\circ \iota)_*([\partial{W},\sigma]\otimes_{C_0(\partial{W})}[\slashed{D}_{\partial W}])=\\
=&(f\circ \iota)_*([\partial{W},\sigma]\otimes_{C_0(\partial{W})}\partial[\slashed{D}_{ W}])=\\
=&(f\circ \iota)_*\circ \partial([W^\circ,\sigma]\otimes_{C_0(W^\circ)}[\slashed{D}_{W}])=0,
\end{align*}
since $\iota_*\circ \partial =0$.

\end{proof}

The following observation shows that assembly of cycles with coefficients in elliptic complexes can by gluing together a compact manifold from the support of the elliptic complex, as in Lemma \ref{bordismsotorcompact}, can be reduced to assembly of an ordinary Baum-Douglas cycle. 

\begin{proposition}
\label{assememdldbordismsotorcompact}
Let $X$ be a finite CW-complex and $(M,\xi,f)$ a cycle for $\tilde{K}_*^{\rm geo}(X)$. Construct $(Z,E_\xi,f)$ as in Lemma \ref{bordismsotorcompact}. Then it holds that 
$$[\tilde{\gamma}(M,\xi,f)]=\gamma[(Z,E_\xi,f)] \in K_*^{\rm an}(X).$$
\end{proposition}

Proposition \ref{assememdldbordismsotorcompact} follows from combining Lemma \ref{bordismsotorcompact} with the Lemmas \ref{anasdisso} and \ref{anasborvb}.

\begin{theorem}
\label{isofromalnalnadlnd}
Let $X$ be a finite CW-complex. The analytic assembly map induces an isomorphism 
$$\tilde{\gamma}:\tilde{K}_*^{\rm geo}(X)\to K_*^{\rm an}(X).$$ 
of $\Z/2$-graded abelian groups. The analytic assembly map $\tilde{\gamma}$ fits into the commuting diagram 
\begin{equation}
\label{commdiagoamadomad}
\begin{tikzcd}
 K_*^{\rm geo}(X) \arrow[rr, "\cong"] \arrow[rd, "\gamma"] &  & \tilde{K}_*^{\rm geo}(X) \arrow[ld, "\tilde{\gamma}"] \\
 & K_*^{\rm an}(X) &
\end{tikzcd}
\end{equation}
with the isomorphism of Theorem \ref{isomofopsmdapada} and the analytic assembly map $\gamma:K_*^{\rm geo}(X)\to K_*^{\rm an}(X)$ of Baum-Douglas.
\end{theorem}

\begin{proof}
Firstly, it follows from Lemma \ref{anasdisso} and \ref{anasborvb} that $\tilde{\gamma}$ is well defined. Secondly, by definition, the diagram \eqref{commdiagoamadomad} commutes at the level of cycles up to an operator homotopy from $\lambda\chi$ to $1$ in the definition of analytic assembly. Therefore the diagram commutes. We deduce that $\tilde{\gamma}$ is an isomorphism from Theorem \ref{isomofopsmdapada} and the fact that $\gamma:K_*^{\rm geo}(X)\to K_*^{\rm an}(X)$ is an isomorphism (see \cite{baumoyono} or \cite{baumvanerpII}). 
\end{proof}

\begin{definition}[Chern characters on $\tilde{K}_*^{\rm geo}(X)$]
Consider a cycle $(M,\xi,f)$ for $\tilde{K}_*^{\rm geo}(X)$. Let $[M]\in H_*^{\rm lf}(M)$ denote the fundamental class in the localy finite homology, $\mathrm{Td}(M)\in H^{\rm ev}(M;\mathbb{Q})$ the Todd class and $\ch(\xi)\in H^{\rm ev}_c(M;\mathbb{Q})$ the Chern character in compactly supported cohomology of $[\xi]\in K^0(X)$. We define 
$$\ch(M,\xi,f):=f_*(\ch(\xi)\cap \mathrm{Td}(M)\cap [M])\in H_*(X;\mathbb{Q}).$$
\end{definition}

We note that if $(\nabla_{E_1},\nabla_{E_2})$ is a connection for $\xi=(E_1,E_2,\sigma)$ as in Definition \ref{analalaslaslaslsalaslasl}, then $\ch(\xi)$ is represented in the compactly supported de Rham cohomology group $H^{\rm ev}_{\rm c, dR}(M)$ by Chern-Weil theory as 
$$\ch(\xi)=\left[\mathrm{Tr}_{E_1}\left(\mathrm{e}^{-\frac{\nabla_{E_1}^2}{2\pi i}}\right)-\mathrm{Tr}_{E_2}\left(\mathrm{e}^{-\frac{\nabla_{E_2}^2}{2\pi i}}\right)\right].$$

\begin{theorem}
Let $X$ be a finite CW-complex. The Chern character induces a well defined mapping 
$$\mathrm{ch}:\tilde{K}_*^{\rm geo}(X)\to H_*(X;\mathbb{Q}).$$ 
of $\Z/2$-graded abelian groups. The Chern character is a rational isomorphism.
\end{theorem}

\begin{proof}[Sketch of proof]
It is clear from the definition that $\mathrm{ch}$ respects the disjoint union/direct sum relation and degenerate equivalence (since $\ch(\xi)=0\in H^{\rm ev}_c(M;\mathbb{Q})$ if $M$ is degenerate). A standard computation with characteristic classes shows that if $V\to M$ is an even rank spin$^c$-vector bundle then 
$$(p_V)_*(\ch[\beta_V]\cap \mathrm{Td}(M^V)\cap [M^V])=\mathrm{Td}(M)\cap [M],$$ 
so $\ch$ is unaltered by vector bundle modification. Finally, if $(W,\xi,f)$ is a cycle with boundary then 
\begin{align*}
\ch(\partial (W,\xi,f))=&(f\circ \iota)_*(\ch(\xi|_{\partial W})\cap \mathrm{Td}(\partial W)\cap [\partial W])=\\
=&(f\circ \partial)_*(\ch(\xi)\cap \mathrm{Td}(W)\cap [W])=0,
\end{align*}
where $\iota:\partial W\hookrightarrow W$ denotes the boundary inclusion. Therefore $\ch$ respects bordism. As such $\ch$ is well defined. It follows that $\ch$ is a rational isomorphism from Theorem \ref{isomofopsmdapada} and the fact that $\ch:K_*^{\rm geo}(X)\to H_*(X;\mathbb{Q})$ is a rational isomorphism.
\end{proof}

\begin{remark}
In \cite[Part 5]{Baum_Douglas}, Baum-Douglas introduced what is now called the Baum-Douglas index problem. In its most general form, it is stated as follows. Let $X$ be a finite CW-complex and $x\in K_*^{\rm an}(X)$ a $K$-homology class. How to construct a geometric cycle $(M,E,f)$ such that $\gamma[(M,E,f)]=x$? This problem was solved when $x$ was the $K$-homology class defined from an elliptic pseudodifferential operator by Baum-Douglas \cite{Baum_Douglas}, thus giving a geometric proof of Atiyah-Singer's index problem. The Baum-Douglas index problem played a pivotal role in Baum-van Erp's solution \cite{baumvanerp} to the index problem for $H$-elliptic operators on contact manifolds. This line of thought has also been put to use in the study of secondary invariants and Higson-Roe's analytic structure group \cite{DGIII}. 

In light of the theory developed in this section, we can formulate a non-compact version of Baum-Douglas index problem. If $X$ is a finite CW-complex and $x\in K_*^{\rm an}(X)$ is a $K$-homology class, how to construct a geometric cycle with coefficients in elliptic complexes $(M,\xi,f)$ such that $\tilde{\gamma}[(M,\xi,f)]=x$? The noncompact version of the Baum-Douglas index problem is weaker than the original Baum-Douglas index problem, but Lemma \ref{bordismsotorcompact} describes how to go from a non-compact solution to a compact one.
\end{remark}

\section{Dualities on Carnot manifolds}
\label{sec:ordinadadp}

\subsection{Poincaré duality -- geometric and analytic}
\label{ordinadinadianadp}

We first study Poincaré duality on a compact manifold $X$. Analytic Poincare duality is an isomorphism 
\[ \mathsf{PD}^{\rm an} : K^*(T^*X) \rightarrow K_*(X),\]
that in rough terms reconstructs an elliptic pseudodifferential operator from the index class of its principal symbol in $K^*(T^*X)$. The construction could be framed in terms of Kasparov's Poincaré duality or in terms of the tangent groupoid. We take the latter approach as a definition and state Kasparov's index theorem relating the two.

Let $\mathbb{T} X\rightrightarrows X\times [0,\infty)$ denote Connes' tangent groupoid, see for instance \cite{connesbook} or Example \ref{ex:tangentgroupoidladladld} above. The kernel of the restriction mapping $C^*(\mathbb{T} X|_{X\times [0,1]})) \to  C_0(T^*X)$ is isomorphic to the contractible $C^*$-algebra $C_0((0, 1],\mathbb{K}(L^2(X)))$. Therefore the induced restriction mapping
\[ e_0 : KK_0(C(X), C^*(\mathbb{T} X|_{X\times [0,1]})) \rightarrow KK_0(C(X), C_0(T^*X)) \]
is an isomorphism.

For a $C(X)$-$C^*$-algebra $A$, we note that, viewing $X$ as a trivial groupoid, there is a natural isomorphism $A \rtimes C(X) \simeq A$. Define the natural transformation:
\[ \alpha_A : KK(\mathbb{C}, A) \rightarrow KK^X(\mathbb{C}, A) \xrightarrow{\rtimes C(X)} KK^X(C(X), A) \xrightarrow{\mathsf{Forget}} KK(C(X), A).  \]

\begin{definition}[Analytic Poincare duality]
We define analytic Poincare duality
\[ \mathsf{PD}^{\rm an} : K^*(T^*X) \rightarrow K_*^{\rm an}(X),\]
as the mapping 
\[\mathsf{PD}^{\rm an}:= e_1 \circ e_0^{-1} \circ \alpha_{C_0(T^* X)}. \]
\end{definition}

\begin{proposition}
\label{chooseanop}
Assume that $D$ is an elliptic differential pseudodifferential operator $C^\infty(X;E_1)\to C^\infty(X;E_2)$ with principal symbol $\sigma(D)$. We let $\xi_D:=(p^*E_1,p^*E_2,\sigma(D))$ denote the associated elliptic complex on $T^*X$, defining its index class $[\sigma(D)]:=[\xi_D]\in K^0(T^*X)$. Then it holds that 
$$\mathsf{PD}^{\rm an}[\xi_D]=[D]\in K_0(X).$$
\end{proposition}

This proposition is well known. For a proof, one can apply Proposition \ref{properldaldlaadlda} to the trivial filtration of $X$. We remark that any elliptic complex $\xi$ on $T^*X$ is equivalent to one of the form $\xi_D$ for an elliptic pseudodifferential operator $D$. Therefore the property $\mathsf{PD}^{\rm an}[\xi_D]=[D]$ in Proposition \ref{chooseanop} characterizes Poincaré duality in degree $0$.

\begin{theorem}[Kasparov's index theorem]
Let $X$ be a compact manifold and let $[\bar{\partial}_{T^*X}]\in KK_0(C_0(T^*X),\C)$ denote the class of the Dolbeault-Dirac operator defined from the almost complex structure on $T^*X$. Then 
$$\mathsf{PD}^{\rm an}(x)=\alpha_{C_0(T^*X)}(x)\otimes_{C_0(T^*X)}[\bar{\partial}_{T^*X}], \quad x\in K^*(T^*X).$$
\end{theorem}

The reader can find Kasparov's index theorem as \cite[Theorem $5$]{kaspindex} or \cite[Section 12.3]{cumero}. Let us proceed by describing the geometric analogue of Poincaré duality.

\begin{definition}[Geometric Poincaré duality]
\label{geomeodmoemd}
We define geometric Poincare duality
\[ \mathsf{PD}^{\rm geo} : K^*(T^*X) \rightarrow \tilde{K}_*^{\rm geo}(X),\]
by setting 
$$\mathsf{PD}^{\rm geo}(\xi):=(T^*X,\xi,p),$$
for any elliptic complex $\xi$ on $T^*X$. Here $p:T^*X\to X$ denotes the projection map.
\end{definition}

\begin{remark}
As Definition \ref{geomeodmoemd} stands,  $\mathsf{PD}^{\rm geo}$ is only defined $K^0(T^*X) \rightarrow \tilde{K}_0^{\rm geo}(X)$. For the odd case, we use that $K^1(T^*X)=K^0(T^*X\times \R)$ as the definition of odd $K$-theory in topological $K$-theory, and 
\[ \mathsf{PD}^{\rm geo} : K^1(T^*X) \rightarrow \tilde{K}_1^{\rm geo}(X),\]
is defined by 
$$\mathsf{PD}^{\rm geo}(\xi):=(T^*X\times \R,\xi,p),$$
for any elliptic complex $\xi$ on $T^*X\times \R$. 
\end{remark}

It is clear that geometric Poincaré duality is well defined. Indeed, geometric Poincaré duality maps degenerate elliptic complexes to degenerate geometric cycles and homotopies of elliptic complexes to bordisms of geometric cycles.

\begin{remark}
Geometric Poincaré duality was in \cite{baumvanerp} defined using the ordinary Baum-Douglas model of geometric $K$-homology. Let us describe their construction and use Theorem \ref{isomofopsmdapada} to reconcile the two definitions.

For a compact manifold $X$, consider $\Sigma X = S(T^*X \times 1_\mathbb{R})$, i.e. the unit sphere bundle of $T^*X \times 1_\mathbb{R}$ in some metric. Since $\Sigma X$ is a hypersurface in the stably almost complex manifold $T^*X\oplus 1_\R$, $\Sigma X$ is a compact spin$^c$-manifold. We have that $\dim(\Sigma X)=2\dim(X)$ is even. Let $B^*X$ be the unit ball bundle of $T^*X$ and $S^*X$ the unit sphere bundle of $T^*X$. Then we can write
\[ \Sigma X = B^*X \cup_{S^*X} B^*X, \]
where we identify the first copy of $B^*X$ with the upper hemisphere and the second copy of $B^*X$ with the lower hemisphere. The role of $M_0$ in Lemma \ref{bordismsotorcompact} will be played by $B^*X$ so $\Sigma X=Z$.

Let $\xi=(E_1, E_2,\sigma)$ be a cycle of $K^0(T^*X)$. We note that since $T^*X$ is homotopy equivalent to $X$, the bundles $E_1$ and $E_2$ can be assumed to be pulled back from $X$. Therefore the dilation action on the fibres of $T^*X\to X$ lifts to $E_1$ and $E_2$. By rescaling, we can assume that the support of $\sigma$ is contained in the interior of $B^*X$. For $M_0=B^*X$ in Lemma \ref{bordismsotorcompact}, the bundle $E_\sigma$ over $\Sigma X=Z$ is exactly
\[ E_\xi = E_1 \cap_\sigma E_2. \]
The geometric Poincare duality map of \cite{baumvanerp} is
\[K^0(T^*X) \rightarrow K_0^{\rm geo}(X), \ \xi\mapsto(\Sigma X, E_\xi, p). \]
Therefore, the geometric Poincare duality map of \cite{baumvanerp} is compatible with the geometric Poincaré duality map of Definition \ref{geomeodmoemd} under the isomorphism of Theorem \ref{isomofopsmdapada}.
\end{remark}

\begin{theorem}
\label{bacvapapdcomo}
Let $X$ be a compact manifold. Then all the morphisms in the following diagram are isomorphisms and the diagram commutes:
\[
\begin{tikzcd}
& K^*(T^* X) \arrow[dl, "\mathsf{PD}^{\rm geo}"] \arrow[dr, "\mathsf{PD}^{\rm an}"] & \\
\tilde{K}_*^{\rm geo}(X) \arrow[rr, "\tilde{\gamma}"] & & K_*^{\rm an}(X)
\end{tikzcd}
\]
\end{theorem}

The reader can find a proof of this result in \cite{baumvanerpII}. The result in \cite{baumvanerpII} was stated in terms of $K_0^{\rm geo}(X)$ and ordinary geometric cycles, but the results in the two formulations are equivalent by Theorem \ref{isofromalnalnadlnd}.

\subsection{Poincaré duality and Carnot manifolds}

We now turn to studying Poincaré dualities on Carnot manifolds. We proceed as in Subsection \ref{ordinadinadianadp}. 

\begin{lemma}
The restriction mapping 
\[ e_0 : KK_0(C(X), C^*(\mathbb{T}_H X|_{X\times [0,1]})) \rightarrow KK_0(C(X), C^*(T_H X)) \]
is an isomorphism.
\end{lemma}

\begin{proof}
The kernel of $e_0$ is $C^*(\mathbb{T}_H X|_{X\times (0,1]})\cong C_0((0, \infty),\mathbb{K}(L^2(X)))$, which is contractible.
\end{proof}

\begin{definition}[Analytic Poincare${}_H$ duality]
We define analytic Poincare${}_H$ duality
\[ \mathsf{PD}^{\rm an}_H : K_*(C^*(T_H X)) \rightarrow K_*^{\rm an}(X),\]
as the mapping 
\[\mathsf{PD}^{\rm an}_H := e_1 \circ e_0^{-1} \circ \alpha_{C^*(T_H X)}. \]
\end{definition}

We note the following consequence of Proposition \ref{properldaldlaadlda}. Recall the definition of the symbol class $[\sigma_H(D)]\in K_0(C^*(T_HX))$ of an $H$-elliptic operator from Definition \ref{definilalald}.

\begin{proposition}
\label{chooseanopheisenberg}
Let $X$ be a compact Carnot manifold. Assume that $D$ is an $H$-elliptic differential pseudodifferential operator $C^\infty(X;E_1)\to C^\infty(X;E_2)$. Then it holds that 
$$\mathsf{PD}^{\rm an}_H[\sigma_H(D)]=[D]\in K_0(X).$$
\end{proposition}

To describe the geometric analogue of Poincaré duality on a Carnot manifold we need to assume $\pmb{F}$-regularity. The reader is encouraged to revisit Example \ref{knasdpandpinas} and the construction of the metaplectic correction $\mathfrak{M}(\mathcal{H})$ of a a bundle of flat representations $\mathcal{H}\to \Gamma_X$ in Section \ref{sec:connesthomandadiaofofd}.

\begin{definition}[Geometric Poincaré${}_H$ duality]
\label{geomeodmoemdheisenberg}
Let $X$ be an $\pmb{F}$-regular Carnot manifold and choose a bundle of flat representations $\mathcal{H}\to \Gamma_X$. We define geometric Poincaré${}_H$ duality
\[ \mathsf{PD}^{\rm geo} _H: K_*(C^*(T_HX)) \rightarrow \tilde{K}_*^{\rm geo}(X),\]
as follows. For $*=0$, we set 
$$\mathsf{PD}^{\rm geo}_H(x):=(\Gamma_X,(p\mathcal{H},q\mathcal{H},qup)\otimes \mathfrak{M}(\mathcal{H}),p_\Gamma),$$
where $x\in K_0(C^*(T_HX))$ has been lifted to a pre-image in $K_0(I_X)$ represented by a relative $K$-theory cycle $(p,q,u)$ as in Example \ref{knasdpandpinas}. Here $p_\Gamma:\Gamma_X\to X$ denotes the projection map.

For $*=1$, we set 
$$\mathsf{PD}^{\rm geo}_H(x):=(\Gamma_X^\partial,(p\mathcal{H},q\mathcal{H},qup)\otimes \mathfrak{M}(\mathcal{H}),p_\Gamma),$$
where $x\in K_1(C^*(T_HX))$ has been lifted to a pre-image in $K_0(I_X^\partial)$ represented by a relative $K$-theory cycle $(p,q,u)$ as in Example \ref{knasdpandpinas} but over $\Gamma_X^\partial$. 
\end{definition}

\begin{theorem}
\label{leftuppertriangle}
Let $X$ be a compact $\pmb{F}$-regular Carnot manifold. Then the geometric Poincaré${}_H$ duality $\mathsf{PD}^{\rm geo} _H: K_*(C^*(T_HX)) \rightarrow \tilde{K}_*^{\rm geo}(X)$ is a well defined map making the following diagram commutative
\[
\begin{tikzcd}
K_*(C^*(TX)) \arrow[rr, "\psi"] \arrow[rd, "\mathsf{PD}^{\rm geo}"] &  & K_*(C^*(T_H X))  \arrow[ld, "\mathsf{PD}^{\rm geo}_H"] \\
 & \tilde{K}_*^{\rm geo}(X) &
\end{tikzcd}
\]
\end{theorem}

\begin{proof}
Consider the map $\beta:K^*(\Gamma_X) \to \tilde{K}_*^{\rm geo}(X)$, $\xi\mapsto (\Gamma_X,\xi,p_\Gamma)$ and the following diagram:
\begin{equation}
\label{complicaatedadldaldmadl}
\begin{tikzcd}[column sep=large, row sep = large]
K_*(I_X) \arrow[rr, "j_*"] \arrow[d, "\otimes \mathcal{H}"] & & K_*(C^*(T_HX)) \arrow[d, "\psi"] \\
K^*(\Gamma_X) \arrow[dr, "{[\mathfrak{M}(\mathcal{H})]}"] & K^*(\Xi_X) \arrow[r, "\iota_\Xi^X!"]  & K^*(T^*X) \arrow[d, "\mathsf{PD}^{\rm geo}"] \\
&K^*(\Gamma_X) \arrow[u, "\tau_\Xi^X"] \arrow[r, "\beta"]  & K_*^{\rm geo}(X)
\end{tikzcd}
\end{equation}
By definition, we have that 
$$\mathsf{PD}^{\rm geo}_H\circ j_*=\beta([\mathcal{H}\otimes \mathfrak{M}(\mathcal{H})]\otimes \cdot).$$ 
Therefore, the theorem follows from the commutativity of the diagram \eqref{complicaatedadldaldmadl}. The bottom right square commutes due to the same general argument as in \cite[Lemma 2.7.4]{baumvanerp}. The upper (irregular) hexagon commutes by Theorem \ref{maincomputationforisg}.
\end{proof}

We can now conclude some computations of geometric Poincaré${}_H$ dual classes to symbol classes from the computations of Section \ref{sec:ktheoomon}. From Corollary \ref{charhellfofodotwocordonk} the following proposition follows. 

\begin{proposition}
\label{charhellfofodotwocordonkwithgeo}
Let $X$ be a compact regular polycontact manifold with polycontact structure $H$, equipped with a Riemannian metric $g$. Assume that $D_\gamma\in \mathcal{DO}_H^{2}(X;E)$ is as in Equation \eqref{secondoroddbbamddevep} and is $H$-elliptic. For $N>>0$, we have that 
$$\mathsf{PD}^{\rm geo} _H[\sigma_H(D_\gamma)]=\left[\left(S(H^\perp)\times S^1,E_{\gamma,N}, p\right)\right],$$
where $p:S(H^\perp)\times S^1\to X$ denotes the projection map and $E_{\gamma,N}\to S(H^\perp) \times S^1$ is obtained from clutching $(\oplus_{k\leq N}(p^*H)^{\otimes^{\rm sym}_\C k}\otimes E)\times [0,1]\to  S(H^\perp) \times [0,1]$ along $\oplus_{k\leq N}\gamma_k$ at the boundary (where $\gamma_k$ is as in Proposition \ref{charhellfofodotwo}).
\end{proposition}

From Corollary \ref{charhellfofodszego} the following proposition follows. 

\begin{proposition}
\label{charhellfofodszegowithgeo}
Let $X$ be a compact regular polycontact manifold with polycontact structure $H$, equipped with a Riemannian metric $g$. Assume that $P\in \Psi_H^0(X;E)$ is an idempotent Hermite operator. Then 
$$\mathsf{PD}^{\rm geo} _H[\sigma_H(2P-1)]=\left[\left(S(H^\perp),\sigma_H(P)(\oplus_{k}(p^*H)^{\otimes^{\rm sym}_\C k}\otimes E), p\right)\right],$$
where $p:S(H^\perp)\to X$ denotes the projection and $\sigma_H(P)(\oplus_{k}(p^*H)^{\otimes^{\rm sym}_\C k}\otimes E)\to S(H^\perp)$ is the vector bundle obtained from the range of the Carnot symbol of $P$. 
\end{proposition}

\subsection{The commuting tetraeder}

We now come to the main $K$-theoretical result of this part. It relates the different dualities from the last section to the map $\psi:K^*(T^*X)\to K_*(C^*(T_HX))$ constructed in Theorem \ref{nistorconnethomfortwistgroup}

\begin{theorem}
\label{commutingtetraederthm}
Let $X$ be a compact $\pmb{F}$-regular Carnot manifold. Then all morphism in the following diagram are isomorphisms and the diagram commutes
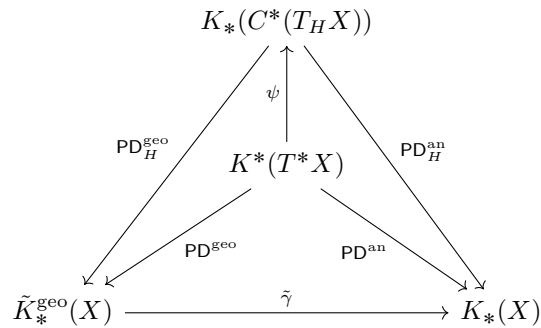
\[
\begin{tikzcd}
& K_*(C^*(T_H X)) \arrow{ddddl}[swap]{\mathsf{PD}_H^{\rm geo}}\arrow[ddddr, "\mathsf{PD}_H^{\rm an}"] & \\
&&\\
& K^*(T^* X) \arrow[ddl, "\mathsf{PD}^{\rm geo}"] \arrow{ddr}[swap]{\mathsf{PD}^{\rm an}} \arrow[uu, "\psi"] & \\
&&\\
\tilde{K}_*^{\rm geo}(X) \arrow[rr, "\tilde{\gamma}"] & & K_*(X)
\end{tikzcd}
\]
If $X$ is a Carnot manifold which is not necessarily $\pmb{F}$-regular, the diagram above with the upper left arrow removed only consists of isomorphisms and commutes.
\end{theorem}

The bottom triangle commutes by Theorem \ref{bacvapapdcomo} with all maps being isomorphisms, with no assumption or reference to Carnot structures. The left upper triangle commutes by Theorem \ref{leftuppertriangle} -- commutativity of this part of the diagram requires $\pmb{F}$-regularity. If the diagram in Theorem \ref{commutingtetraederthm} commutes, the fact that the bottom triangle commutes with all maps being isomorphisms and that $\psi$ is an isomorphism (see Theorem \ref{nistorconnethomfortwistgroup}) implies that all of the morphisms are isomorphisms. We shall below in Theorem \ref{rightuppertriangle} prove that the right upper triangle commutes with no assumption of regularity on the Carnot structure.

\begin{theorem}
\label{rightuppertriangle}
Let $X$ be a compact Carnot manifold. Then the following diagram commutes
\[
\begin{tikzcd}
& K_*(C^*(T_H X))\arrow[ddr, "\mathsf{PD}_H^{\rm an}"] & \\
&&\\
 K^*(T^* X) \arrow[rr, "\mathsf{PD}^{\rm an}"] \arrow[uur, "\psi"] &&K_*^{\rm an}(X) 
\end{tikzcd}
\]
\end{theorem}

Theorem \ref{rightuppertriangle} was proven for contact manifolds in \cite[Theorem 5.4.1]{baumvanerp}. Our proof follows a similar line of thought.

\begin{proof}
Recall the construction of the adiabatic parabolic tangent groupoid $\mathbb{A}_HX$ of $X$ from Section \ref{subsec:parabolictanget}.
The following diagram is commutative
\[
\begin{tikzcd}
C^*(\mathbb{A}_H X) \arrow[r] \arrow[d] & C^*(\mathbb{T} X) \arrow[d] \\
C^*(\mathbb{T}_H X) \arrow[r] & C^*(X \times X),
\end{tikzcd}
\]
where the maps are the appropriate evaluations at $1$.
Therefore, we have the following commutative diagram in $KK$
\[
\begin{tikzcd}
KK_*(C(X), T^*X) \arrow[r] \arrow[d] & KK_*(C(X), C^*(\mathbb{T} X)) \arrow[d] \arrow[r] & KK_*(C(X), C^*(TX)) \arrow[dl] \arrow[dd] \\
KK_*(C(X), C^*(\mathbb{T}_H X)) \arrow[r,] \arrow[d] & K_*(X) \\
KK_*(C(X), C^*(T_H X)) \arrow{ur}& & KK_*(C(X), C^*((T_HX)_{\rm adb})) \arrow[ll]
\end{tikzcd}
\]
Thus the following diagram commutes
\begin{equation}
\label{diagaondaon1}
\begin{tikzcd}
& KK_*(C(X), C^*(TX)) \arrow{dl}[swap]{\psi} \arrow[d] \\
KK_*(C(X), C^*(T_H X)) \arrow[r] & K_*(X)
\end{tikzcd}
\end{equation}
Since $\alpha$ is a natural transformation, the diagram commutes
\begin{equation}
\label{diagaondaon2}
\begin{tikzcd}
KK_*(\mathbb{C}, C^*(T_H X)) \arrow[d, "\alpha_{C^*(T_H X)}"] & KK_*(\mathbb{C}, C^*(TX)) \arrow[l, "\psi"] \arrow[d, "\alpha_{C^*(TX)}"] \\
KK_*(C(X), C^*(T_H X)) & KK_*(C(X), C^*(TX)) \arrow[l, "\psi"]
\end{tikzcd}
\end{equation}
Combining \eqref{diagaondaon1} with \eqref{diagaondaon2} we arrive at the conclusion of the theorem.
\end{proof}

\section{Index theorems for $H$-elliptic operators}
\label{sec:indexingener}

In this section we conclude abstract index formulas on a compact Carnot manifold from the commuting tetraeder in Theorem \ref{commutingtetraederthm}. We note that using the formalism of Section \ref{sec:ktheoomon}, the index theorems of this section extends ad verbatim to graded $H$-elliptic operators.

\begin{theorem}
\label{oenonoinon0on}
Let $X$ be a compact Carnot manifold. Assume that $D:C^\infty(X;E_1)\to C^\infty(X;E_2)$ is an $H$-elliptic pseudodifferential operator. Then the Baum-Douglas index problem for $[D]$ is solved by $\mathsf{PD}^{\rm geo}(\psi^{-1}[\sigma_H(D)])$.
\end{theorem}

\begin{proof}
By Proposition \ref{chooseanopheisenberg}, $[D]=\mathsf{PD}^{\rm an}_H[\sigma_H(D)]$ and by Theorem \ref{commutingtetraederthm}, 
$$\tilde{\gamma}(\mathsf{PD}^{\rm geo}(\psi^{-1}[\sigma_H(D)]))=\mathsf{PD}^{\rm an}_H[\sigma_H(D)].$$
\end{proof}

The next index formula which Theorem \ref{oenonoinon0on} implies has already been studied in the literature, see \cite{eskeewertthesis,eskeewertpaper,mohsen2}.

\begin{corollary}
Let $X$ be a compact Carnot manifold. Assume that $D:C^\infty(X;E_1)\to C^\infty(X;E_2)$ is an $H$-elliptic pseudodifferential operator. Then it holds that 
$$\ind(D)=\int_{T^*X} \ch(\psi^{-1}[\sigma_H(D)])\wedge \mathrm{Td}(X).$$
\end{corollary}

\begin{remark}
The reader can revisit Subsection \ref{section:analyticprops} for more details on the analytic interpretation of $\ind(D)$. Indeed, by Theorem \ref{indexindepofs}, the index $\ind(D)$ is the same when interpreting as the index of $D$ acting on Sobolev spaces, or as a closed operator on $L^2$ or as when acting on $C^\infty$.
\end{remark}

\begin{theorem}
\label{geomesolution}
Let $X$ be a compact $\pmb{F}$-regular Carnot manifold. Assume that $D:C^\infty(X;E_1)\to C^\infty(X;E_2)$ is an $H$-elliptic pseudodifferential operator. Then the Baum-Douglas index problem for $[D]$ is solved by $\mathsf{PD}^{\rm geo}_H([\sigma_H(D)])$.
\end{theorem}

\begin{proof}
By Proposition \ref{chooseanopheisenberg}, $[D]=\mathsf{PD}^{\rm an}_H[\sigma_H(D)]$ and by Theorem \ref{commutingtetraederthm}, 
$$\tilde{\gamma}(\mathsf{PD}^{\rm geo}_H([\sigma_H(D)]))=\mathsf{PD}^{\rm an}_H[\sigma_H(D)].$$
\end{proof}

From Theorem \ref{geomesolution}, and the definition of $\mathsf{PD}^{\rm geo}_H$ (see Definition \ref{geomeodmoemdheisenberg}), we conclude the following.

\begin{corollary}
\label{geosolforfregular}
Let $X$ be a compact $\pmb{F}$-regular Carnot manifold and choose a bundle of flat representations $\mathcal{H}\to \Gamma_X$. Assume that $D:C^\infty(X;E_1)\to C^\infty(X;E_2)$ is an $H$-elliptic pseudodifferential operator. Suppose that $[\sigma_H(D)]\in K_*(C^*(T_HX))$ is the image under $j:I_X\to C^*(T_HX)$ of a relative $K$-theory cycle $(p,q,u)$ for $I_X$ with $gp,gq\in I_X$ for all $g\in C_0(\Gamma_X)$.  Then the Baum-Douglas index problem for $[D]$ is solved by the geometric cycle with coefficients in the ellliptic complexes
$$(\Gamma_X,(p\mathcal{H},q\mathcal{H},qup)\otimes \mathfrak{M}(\mathcal{H}),p_\Gamma).$$
In this case,
$$\ind(D)=\int_{\Gamma_X}\ch[(p\mathcal{H},q\mathcal{H},qup)]\wedge \e^{c_1(\mathfrak{M}(\mathcal{H})}\wedge \mathrm{Td}(\Gamma_X).$$
\end{corollary}

The reader can return to Section \ref{sec:connesthomandadiaofofd} for further details on the construction of the bundle of flat representations $\mathcal{H}$ and its metaplectic correction line bundle $\mathfrak{M}(\mathcal{H})$. Further details on the symbol class $[\sigma_H(D)]\in K_*(C^*(T_HX))$ and the construction of pre-images under $j:I_X\to C^*(T_HX)$ as a relative $K$-theory cycle $(p,q,u)$ can be found in Section \ref{sec:ktheoomon}.

\begin{theorem}
Let $X$ be a compact $\pmb{F}$-regular Carnot manifold. Then there is a Chern character
$$\ch_H:K_*(C^*(T_HX))\to H^*_{\rm dR}(X),$$
defined from 
$$\ch_H(j_*x):=(p_\Gamma)_*\left(\ch(x\otimes \mathcal{H})\cup \e^{c_1(\mathfrak{M}(\mathcal{H}))+\frac12 c_1(\Gamma_X)}\right),$$
for $x\in K_*(I_X)$. For $y\in K_*(C^*(T_HX))$, we have that 
$$\ind_{T_HX}(y)=\int_X \ch_H(y)\wedge \hat{A}(X).$$
\end{theorem}

\begin{proof}
The map $\ch_H$ is well defined since $j_*:K_*(I_X)\to K_*(C^*(T_HX))$ is surjective and by the Riemann-Roch theorem (cf. \cite[Section 5]{careywang}) it satisfies $(\ch_H(y)\cup \hat{A}(X))\cap [X]=\ch_*^{\rm geo}(\mathsf{PD}^{\rm geo}_H(y))$ where $[X]$ denotes the fundamental class in homology and $\mathsf{PD}^{\rm geo}_H:K_*^{\rm geo}(X)\to H_*(X)$ denotes the geometric Chern character in homology. 
\end{proof}

The index theorems above hold in the generality of $\pmb{F}$-regular Carnot manifolds. Let us utilize the $K$-theory computations of Section \ref{sec:ktheoomon} in two concrete examples. 

\subsection{Example: A Boutet de Monvel type index theorem}

The following index theorem follows from Proposition \ref{charhellfofodszegowithgeo} and Theorem \ref{geomesolution}. It extends an index formula of \cite{bdmindex} from boundaries of strictly convex domains in complex manifolds, and the results of \cite{melroseeptein} from co-oriented contact manifolds, to polycontact manifolds.

\begin{theorem}
\label{bdmdmdam}
Let $X$ be a compact, regular, polycontact manifold with polycontact structure $H$. Assume that $P\in \Psi_H^0(X;E)$ is an idempotent Hermite operator. We write $[P]\in K_1^{\rm an}(X)$ for the $K$-homology class of $(L^2(X;E),2P-1)$. Let $p:S(H^\perp)\to X$ denote the projection. Set $\mathpzc{F}_X:=\bigoplus_k (p^*H)^{\otimes^{\rm sym}_\C k}$.

Then $(S(H^\perp),\sigma_H(P)(\mathpzc{F}_X\otimes E),p)$ solves the Baum-Douglas index problem for $[P]$. Moreover, for any $u\in C(X, GL_N(\C))$ the Toeplitz operator 
$$T_u:=(P\otimes 1_{\C^N})u(P\otimes 1_{\C^N}):PL^2(X;E\otimes \C^N)\to PL^2(X;E\otimes \C^N),$$
is Fredholm and its index is 
$$\ind(T_u)=\int_{S(H^\perp)}\ch[p^*u]\wedge \ch(\sigma_H(P)(\mathpzc{F}_X\otimes E))\wedge \mathrm{Td}(S(H^\perp)).$$
\end{theorem}

The reader should note that by Theorem \ref{vanerpandszego}, recalled from \cite{vanerpszego}, there is a plethora of idempotent Hermite operators on polycontact manifolds. 

\begin{example}
Suppose that $X$ is a compact, co-oriented contact manifold and $P\in \Psi^0_H(X)$ is a Szegö projection, i.e. that $\sigma^0_{(x,\pi)}(P)$ is the ground state projection if $\pi$ is a non-abelian representation with positive central character and $\sigma^0_{(x,\pi)}(P)=0$ otherwise. In this case, $\sigma_H(P)\mathpzc{F}_X$ is the trivial line bundle over the component $X\times \R_{>0}\subseteq \Gamma_X$. Theorem \ref{bdmdmdam} ensures that 
$$(\Gamma_X^{\partial},\sigma_H(P)(\mathpzc{F}_X),p)\sim_{\rm BD}(X,X\times \C,\mathrm{id}),$$ 
solves the Baum-Douglas index problem for $[P]\in K_1^{\rm an}(X)$. This recovers the classical Boutet de Monvel index theorem \cite{bdmindex} (see \cite{melroseeptein} for the general co-oriented contact case):
$$\ind(T_u)=\int_{X}\ch[u]\wedge \mathrm{Td}(X).$$
For a more direct proof, see \cite{baumvanerpbdmindex}. We also have the following generalization.
\end{example}

Recall the notion of a generalized Szegö projection at level $N$ from  \cite[Chapter 6]{melroseeptein}. See more in Example \ref{ex:szegoexampe2} above or further examples in \cite[Chapter 15.3]{bdmguille} or \cite{engliszhanghigher}.

\begin{corollary}
Let $X$ be a compact, co-oriented contact manifold and $P_N\in \Psi_H^0(X,\C)$ a generalized Szegö projection at level $N$. The index problem for $[P_N]$ is solved by the geometric cycle 
$$(\Gamma_X^{\partial},\sigma_H(P_N)(\mathpzc{F}_X),p)=(X,\oplus_{k=0}^N H^{\otimes^{\rm sym}_\C k},\mathrm{id}).$$ 
In particular, for any $u\in C(X, GL_q(\C))$ the Toeplitz operator 
$$T_u:=(P_N\otimes 1_{\C^q})u(P_N\otimes 1_{\C^q}):P_NL^2(X;\C^q)\to P_NL^2(X;\C^q),$$
is Fredholm and its index is 
$$\ind(T_u)=\sum_{k=0}^N\int_{X}\ch[u]\wedge \ch(H^{\otimes^{\rm sym}_\C k})\wedge \mathrm{Td}(X).$$
\end{corollary}

Let us solve the index problem for when $X$ is an arbitrary contact manifold, not necessarily co-oriented. Take a projection $P\in \Psi^0_H(X)$ with $\sigma^0_{(x,\pi)}(P)$ being the ground state projection for all non-abelian representations $\pi$. See more in Example \ref{ex:szegoexampe2} above. Also here we use the Fock bundle $\mathpzc{F}_X:=\oplus_{k=0}^\infty (p^*H)^{\otimes^{\rm sym}_\C k},\to \Gamma_X$, defined from the complex structure on $p^*H\to \Gamma_X$ induced from the Kirillov form.  In this case, $\sigma_H(P)\mathpzc{F}_X$ is the trivial line bundle on $\Gamma_X$.

\begin{corollary}
For a compact contact manifold $X$ and $P\in \Psi^0_H(X)$ a projection as in the preceding paragraph. The index problem for $[P]\in K_1^{\rm an}(X)$ is solved by the geometric cycle 
$$(\Gamma_X^{\partial},\sigma_H(P)(\mathpzc{F}_X),p)=(\Gamma_X^{\partial},\Gamma_X^{\partial}\times\C,p).$$ 
In particular, for any $u\in C(X, GL_q(\C))$ the Toeplitz operator 
$$T_u:=(P\otimes 1_{\C^q})u(P\otimes 1_{\C^q}):PL^2(X;\C^q)\to PL^2(X;\C^q),$$
is Fredholm and its index is 
$$\ind(T_u)=\int_{\Gamma_X^\partial}\ch[p^*u]\wedge  \mathrm{Td}(\Gamma_X^{\partial}).$$
Moreover, $[P]=0$ if $X$ is co-oriented.
\end{corollary}

\begin{proof}
Theorem \ref{bdmdmdam} provides the solution for the index problem and the index formula for Toeplitz operators is immediate. Concerning the vanishing in the co-oriented case, if $X$ is co-oriented, the co-orientation ensures that $\Gamma_X=\Gamma_X^+\dot{\cup}\Gamma_X^-$, where we can identify $\Gamma_X^\pm =X\times \R_\pm$. The complex structure on $p^*H|_{\Gamma_X^+}\to \Gamma_X^+$ descends to $H\to X$, and $X$ is equipped with the spin$^c$-structure induced from the stably almost complex structure (using $TX\cong H\oplus \R$). We note that $\Xi_X\to \Gamma_X$ can be identitfed with $p^*H\to \Gamma_X$, and so the spin$^c$-structure on $\Gamma_X$ is the one making the identification defined from the co-orientation 
$$\Gamma_X=(X\times \R_+)\dot{\cup} (\overline{X}\times \R_+),$$
spin$^c$-preserving. Therefore, the class $[P]\in K_1^{\rm an}(X)$ admits the geometric solution $(X,X\times \C,\mathrm{id})\dot{\cup}(\overline{X},X\times \C,\mathrm{id})$ which is nullbordant. In particular, $[P]=0\in K_1^{\rm an}(X)$ if $X$ is co-oriented.
\end{proof}

\subsection{Example: Baum-van Erp type operators for polycontact manifolds}

The following index theorem follows from Proposition \ref{charhellfofodotwocordonkwithgeo} and Theorem \ref{geomesolution}. It extends an index formula of \cite{baumvanerp} from co-oriented contact manifolds to polycontact manifolds.

\begin{theorem}
\label{charhellfofodotwocordonkwithgeoindex}
Let $X$ be a compact regular polycontact manifold with polycontact structure $H$, equipped with a Riemannian metric $g$. Assume that $D_\gamma\in \mathcal{DO}_H^{2}(X;E)$ is as in Equation \eqref{secondoroddbbamddevep} and is $H$-elliptic. For a sufficiently large $N>>0$, the index problem for $[D_\gamma]$ is solved by the geometric cycle
$$\left(S(H^\perp)\times S^1,E_{\gamma,N}, p\right),$$
where $p:S(H^\perp)\times S^1\to X$ denotes the projection map and $E_{\gamma,N}\to S(H^\perp) \times S^1$ is obtained from clutching $(\oplus_{k\leq N}(p^*H)^{\otimes^{\rm sym}_\C k}\otimes E)\times [0,1]\to  S(H^\perp) \times [0,1]$ along $\oplus_{k\leq N}\gamma_k$ at the boundary (where $\gamma_k$ is as in Proposition \ref{charhellfofodotwo}).
\end{theorem}

Recall that $H$-ellipticity of $D_\gamma$ was characterized in Theorem \ref{charhellfofodo}, see also \cite[Theorem 1.1]{horcharnice}.

\begin{example}
Let us compute an example with co-oriented contact manifolds in order to explain the two spin$^c$-structures in Baum-van Erp's index formula  \cite{baumvanerp} for operators of the form $D_\gamma$ on a co-oriented contact manifold. In this case, we have for any relative $K$-theory cycle $(p,q,u)$ as in Example \ref{knasdpandpinas} that 
\begin{align*}
(\Gamma_X,(p\mathcal{H},q\mathcal{H},qup)&\otimes \mathfrak{M}(\mathcal{H}),p)\sim_{\rm BD}\\
\sim_{\rm BD}&((0,\infty)\times X,(p\mathcal{H}_+,q\mathcal{H}_+,qup),p)+\\
&\qquad\qquad +((-\infty,0)\times X,(p\mathcal{H}_-,q\mathcal{H}_-,qup),p).
\end{align*}
The elliptic complexes $(p\mathcal{H}_\pm,q\mathcal{H}_\pm,qup)$ correspond to classes $[u_\pm]\in K^1(X)$ under the Bott isomorphism or to $K$-theory classes $[E_\pm]-N[1]$ in $K^0(X\times S^1)$. We can conclude that 
\begin{align*}
(\Gamma_X,(p\mathcal{H},q\mathcal{H},qup)&\otimes \mathfrak{M}(\mathcal{H}),p)\sim_{\rm BD}\\
\sim_{\rm BD} 
&(X\times S^1,E_+,p_1)+(-(X\times S^1),E_-,p_1),
\end{align*}
in $\tilde{K}_1^{\rm geo}(X)$, where $p_1:X\times S^1\to X$ denotes projection onto the first factor. We note here the additional sign that is coming from the difference in signs of orientation and spin$^c$-structure on $(0,\infty)$ versus that of $(-\infty,0)$. 
\end{example}

We can give an example of an index result for a differential operator $D_\gamma\in \mathcal{DO}_H^{2}(X;E)$, analogous to that in Theorem \ref{charhellfofodotwocordonkwithgeoindex}, for a four-dimensional manifold $X$ with an Engel structure, see Example \ref{engelalgebra} and \ref{ex:regudlaladoad}. If $X$ is a four-manifold, we say that a pair of vector fields $(Y_1,Y_2)$ generate an Engel structure if $Y_1,Y_2,[Y_1,Y_2], [Y_1,[Y_1,Y_2]]$ span $TX$ in all points. By results of Vogel \cite{Vogel}, existence of a pair generating an Engel structure is equivalent to $X$ being parallelizable. We equipp $X$ with the depth three Carnot structure where $Y_1$ and $Y_2$ are sections of $T^{-1}X$, and $[Y_1,Y_2]$ is a section of $T^{-2}X$. We conclude the following index result from work of Helffer \cite{helffercomp}. We use the notation 
$$S_0:=\{\lambda\in \C: \mathrm{Re}(\lambda)\neq 0\mbox{  or  } |\mathrm{Im}(\lambda)|<1/2\}.$$

\begin{theorem}
\label{indexforengel}
Let $X$ be a compact, four-dimensional Carnot manifold with the depth three Carnot structure defined from a pair $(Y_1,Y_2)$ generating an Engel structure. Assume that $E\to X$ is a hermitean vector bundle with connection $\nabla$ and $\gamma\in C^\infty(X;\End(E))$. Consider a differential operator $D_\gamma\in \mathcal{DO}_H^{2}(X;E)$ such that 
$$D_\gamma=\nabla_{Y_1}^*\nabla_{Y_1}+\nabla_{Y_2}^*\nabla_{Y_2}+\gamma \nabla_{[Y_1,Y_2]}+ \mathcal{DO}_H^{1}(X;E).$$ 

Then $D_\gamma$ is $H$-elliptic if and only if $\gamma-\lambda:E\to E$ is a vector bundle isomorphism for any $\lambda\in S_0$. In this case, it holds that 
$$\ind(D_\gamma)=0.$$
\end{theorem}

\begin{proof}
The $H$-ellipticity of $D_\gamma$ was characterized by Helffer \cite{helffercomp}. From the characterization of $H$-ellipticity and the fact that $S_0$ is invariant under rescaling by $t\in [0,1]$, we can construct a path $(D_{t\gamma})_{t\in [0,1]}$ with $D_{1\gamma}=D_\gamma$ and $D_{0\gamma}=\nabla_{Y_1}^*\nabla_{Y_1}+\nabla_{Y_2}^*\nabla_{Y_2}$. From self-adjointness of $D_{0\gamma}$, it is clear that 
$$\ind(D_\gamma)=\ind(D_{1,\gamma})=\ind(D_{0,\gamma})=0.$$
\end{proof}

\begin{remark}
\label{akfnaodns}
The results of \cite{helffercomp} holds in larger generality, and applies to operators of the form $D_\gamma=\sum_j \nabla_{Y_j}^2+\sum_{j,k} \gamma_{j,k} \nabla_{[Y_j,Y_k]}+ \mathcal{DO}_H^{1}(X;E)$ for suitable bracket generating vector fields $(Y_j)_j$. In particular, it is of interest to ask if $\ind(D_\gamma)=0$ for large classes of Carnot manifolds whose osculating Lie algebra has step length $>2$ in all points.
\end{remark}

\section[An outlook on graded Rockland sequences]{An outlook on graded Rockland sequences and their index theory}
\label{sec:grafadpknapdnarock}

By virtue of the BGG-complexes, there are several graded Rockland sequences of interest in geometry. See Subsubsection \ref{ex:BBBBBBGGGGGG1} and Example \ref{ex:BBBBBBGGGGGG2} above or \cite{morecap,GenericDave,Dave_Haller1}. While the (curved) BGG-complexes, at least for a flat connection, is constructed from a quasi-isomorphism of complexes its index is often computable. However, in light of usage of BGG-like complexes in the Baum-Connes conjecture \cite{julgsun1,julgcr,julgsp,yunckensl3}, it is of interest to understand the $K$-homological aspects of the associated classes. Here we present some relevant features deducible from abstract principles. Explicit results, both in the abstract and in examples, will be left to future work.

\subsection{The symbol class of a Rockland sequence}

To set up the symbol class, we shall use a model of $K$-theory akin to Atiyah's notion of an elliptic complex. The model is not new, but we have not been able to locate a reference. The reader can compare the model to Higson-Roe's notion of a Hilbert-Poincaré complex (see \cite{HRSur1}).

\begin{definition}
Let $B$ be a $C^*$-algebra and assume that
\begin{align*}
0\to\mathpzc{E}_1\xrightarrow{d_1}\mathpzc{E}_2\xrightarrow{d_2}\cdots \to \mathpzc{E}_N\xrightarrow{d_N}\mathpzc{E}_{N+1}\to 0,\\
0\leftarrow\mathpzc{E}_1\xleftarrow{b_1}\mathpzc{E}_2\xleftarrow{b_2}\cdots \leftarrow \mathpzc{E}_N\xleftarrow{b_N}\mathpzc{E}_{N+1}\leftarrow 0,
\end{align*}
are two sequences of $B$-Hilbert $C^*$-modules with adjointable maps between. We say that this data $(\mathpzc{E}_\bullet,d_\bullet,b_\bullet)$ is a $B$-Fredholm sequence of length $N+1$ if it holds that 
$$b_jd_j+d_{j-1}b_{j-1}-1\in \mathbb{K}_B(\mathpzc{E}_j),$$
for all $j$ (for $b_j=0$ and $d_j=0$ if $j\notin\{1,\ldots, N+1\}$). If $(\mathpzc{E}_\bullet,d_\bullet,b_\bullet)$ is a $B$-Fredholm sequence with 
$$b_jd_j+d_{j-1}b_{j-1}=1,$$
for all $j$, we say that $(\mathpzc{E}_\bullet,d_\bullet,b_\bullet)$ is acyclic.

Two $B$-Fredholm sequences of length $N+1$ are said to be isomorphic if there are unitary isomorphisms of the $B$-modules intertwining the operators $d_\bullet$ and $b_\bullet$. 
\end{definition}

The reader should note that prescribing a length $2$ $B$-Fredholm sequence is the same prescribing a $B$-Fredholm map $\mathpzc{E}_1\to \mathpzc{E}_2$ and an inverse modulo compact operators. There is a notion of direct sum of $B$-Fredholm sequences
$$(\mathpzc{E}_\bullet,d_\bullet,b_\bullet)+(\mathpzc{E}_\bullet',d_\bullet',b_\bullet')=(\mathpzc{E}_\bullet\oplus \mathpzc{E}'_\bullet,d_\bullet\oplus d_\bullet',b_\bullet\oplus b_\bullet').$$
We say that two $B$-Fredholm sequences $(\mathpzc{E}_\bullet,d_\bullet,b_\bullet)$ and $(\mathpzc{E}_\bullet',d_\bullet',b_\bullet')$ are homotopic if there is a $C([0,1],B)$-Fredholm sequence $(\tilde{\mathpzc{E}}_\bullet,\tilde{d}_\bullet,\tilde{b}_\bullet)$ such that 
$$(\tilde{\mathpzc{E}}_\bullet,\tilde{d}_\bullet,\tilde{b}_\bullet)\otimes_{\mathrm{ev}_0} B\cong (\mathpzc{E}_\bullet,d_\bullet,b_\bullet)\quad\mbox{and}\quad (\tilde{\mathpzc{E}}_\bullet,\tilde{d}_\bullet,\tilde{b}_\bullet)\otimes_{\mathrm{ev}_1} B\cong (\mathpzc{E}'_\bullet,d'_\bullet,b'_\bullet)$$
For notational convenience, we also write 
$$0\to\mathpzc{E}_1{}_{d_1}\leftrightarrows^{b_1}\mathpzc{E}_2{}_{d_2}\leftrightarrows^{b_2}\cdots \leftrightarrows \mathpzc{E}_N{}_{d_N}\leftrightarrows^{b_N}\mathpzc{E}_{N+1}\to 0,$$ 
for a $B$-Fredholm sequence $(\mathpzc{E}_\bullet,d_\bullet,b_\bullet)$.

\begin{definition}
For a $C^*$-algebra $B$, let $\tilde{K}_0(B)$ denote the set of equivalence classes of $B$-Fredholm sequence under the equivalence relation generated by homotopy and declaring acyclic $B$-Fredholm sequences equivalent to the zero complex. 
\end{definition}

\begin{theorem}
For a $C^*$-algebra $B$, the mapping 
\begin{align*}
KK_0(\C,B)&\to \tilde{K}_0(B), \\
& \left[\left(\begin{matrix}\mathpzc{E}_+\\\oplus\\ \mathpzc{E}_-\end{matrix}, \begin{pmatrix} 0& F_+\\ F_- &0\end{pmatrix}\right)\right]\mapsto \left[\left(0\to\mathpzc{E}_-{}_{F_+}\leftrightarrows^{F_-}\mathpzc{E}_+\to 0\right)\right],
\end{align*}
is a natural isomorphism of abelian groups. 
\end{theorem}

\begin{proof}[Sketch of proof]
The proof of this theorem is by standard techniques, so we only sketch it. We note that the mapping is well defined since it maps homotopies to homotopies. To show that the mapping is bijective, an inverse is constructed as follows. If $(\mathpzc{E}_\bullet,d_\bullet,b_\bullet)$ is a $B$-Fredholm sequence, one sets $\mathpzc{E}_+:=\bigoplus_k \mathpzc{E}_{2k}$, $\mathpzc{E}_-:=\bigoplus_k \mathpzc{E}_{2k+1}$ and $F_0:=\sum d_{2k+1}+b_{2k}:\mathpzc{E}_-\to \mathpzc{E}_+$. The operator $F_0$ is $B$-Fredholm, with inverse modulo compacts of the form $\sum d_{2k}+b_{2k+1}$. Functional calculus techniques applied to $F_0$ produces an almost unitary $F_+:\mathpzc{E}_-\to \mathpzc{E}_+$ homotopic to $F_0$ in the space of $B$-Fredholm operators. Modulo homotopy and acyclic $B$-Fredholm sequences, this procedure produces an inverse to the mapping in the theorem.
\end{proof}

\begin{corollary}
\label{indeineodmfrom}
The mapping 
$$\tilde{K}_0(B)\to K_0(B), \quad (\mathpzc{E}_\bullet,d_\bullet,b_\bullet)\mapsto \ind\left(\sum d_{2k+1}+b_{2k}:\bigoplus_k \mathpzc{E}_{2k+1}\to \bigoplus_k \mathpzc{E}_{2k}\right),$$
is an isomorphism of abelian groups. 
\end{corollary}

We are now in a position to define the symbol class in $K$-theory of a graded Rockland sequence, recall its definition from Section \ref{ex:BBBBBBGGGGGG2}. 

\begin{definition}
Consider a graded Rockland sequence
$$0\to C^\infty(X;\pmb{E}_1)\xrightarrow{D_1}C^\infty(X;\pmb{E}_2)\xrightarrow{D_2}\cdots \xrightarrow{D_{N-1}}C^\infty(X;\pmb{E}_{N})\xrightarrow{D_{N}} C^\infty(X;\pmb{E}_{N+1})\to 0,$$
of order $\pmb{m}=(m_1,\ldots, m_N)$. Let $d_j\in \tilde{\Sigma}^{m_j}_{H,{\rm gr}}(X;\pmb{E}_j,\pmb{E}_{j+1})$ be lifts of $\sigma_{H,{\rm gr}}^{m_j}(D_j)$ and $b_j\in \tilde{\Sigma}^{-m_j}_{H,{\rm gr}}(X;\pmb{E}_{j+1},\pmb{E}_{j})$ be lifts of $\sigma_{H,{\rm gr}}^{m_j}(B_j)$, where  $B_j\in \Psi_{H,{\rm gr}}^{-m_j}(X;\pmb{E}_{j+1},\pmb{E}_j)$ is as in Theorem \ref{odnoaindaonafofbnaon}.

Letting $D_\bullet$ denote the Rockland sequence, we define the symbol class 
$$[\sigma_H(D_\bullet)]\in K_0(C^*(T_HX)),$$ 
as the image under the isomorphism in Corollary \ref{indeineodmfrom} of the class of the $C^*(T_HX)$-Fredholm sequence:
\begin{align}
\label{symbolcomomok}
0\to &\pmb{\mathpzc{E}}^{s_1}_{\rm gr}(X;\pmb{E}_1)_{t=0}{}_{d_1}\leftrightarrows^{b_1}\pmb{\mathpzc{E}}^{s_2}_{\rm gr}(X;\pmb{E}_2)_{t=0}{}_{d_2}\leftrightarrows^{b_2}\cdots \\
\nonumber
&\cdots{}_{d_{N-1}}\leftrightarrows^{b_{N-1}}\pmb{\mathpzc{E}}^{s_{N}}_{\rm gr}(X;\pmb{E}_{N})_{t=0}{}_{d_N}\leftrightarrows^{b_N}\pmb{\mathpzc{E}}^{s_{N+1}}_{\rm gr}(X;\pmb{E}_{N+1})_{t=0}\to 0,
\end{align}
where $s_1,\ldots, s_N,s_{N+1}$ are such that $s_j-s_{j+1}=m_j$ for any $j$.
\end{definition}

We note that the symbol complex \eqref{symbolcomomok} is a well defined $C^*(T_HX)$-Fredholm sequence by Lemma \ref{actionofbolddd} and the properties of $B_1,\ldots, B_N$ given in Theorem \ref{odnoaindaonafofbnaon}.

\subsection{The $K$-homology class of a Rockland sequence}

We now turn to the convoluted business of associating a $K$-homology class with a graded Rockland sequence. For a graded vector bundle $\pmb{E}$, we write $\pmb{E}_{{\rm tr}}$ for the associated trivially graded vector bundle. We fix families of order reducing operators $\mathfrak{D}_{\pmb{E},s}\in \Psi^s_{H,{\rm gr}}(X;\pmb{E}_{{\rm tr}},\pmb{E})$ - invertible operators constructed as direct sums of suitable powers of operators as in Proposition \ref{existenceofposiskdd}.

\begin{proposition}
Consider a graded Rockland sequence
\begin{align}
\label{gradedonaodnaodn}
0\to C^\infty(X;\pmb{E}_1)\xrightarrow{D_1}&C^\infty(X;\pmb{E}_2)\xrightarrow{D_2}\cdots\\
\nonumber
& \xrightarrow{D_{N-1}}C^\infty(X;\pmb{E}_{N})\xrightarrow{D_{N}} C^\infty(X;\pmb{E}_{N+1})\to 0,
\end{align}
of order $\pmb{m}=(m_1,\ldots, m_N)$. Take $s_1,\ldots, s_N,s_{N+1}$ such that $s_j-s_{j+1}=m_j$ for any $j$. 

Let $\pmb{E}_{j,{\rm tr}}$ denote $\pmb{E}_{j}$ equipped with the trivial grading. Define the operators 
$$D_{j,{\rm tr}}:=\mathfrak{D}_{\pmb{E}_{j+1},s_{j+1}} D_j\mathfrak{D}_{\pmb{E}_j,-s_j}\in \Psi^0_H(X;\pmb{E}_{j,{\rm tr}},\pmb{E}_{j+1,{\rm tr}}).$$
Then the sequence 
\begin{align}
\label{gradedonaodnaodntrigially}
0\to C^\infty(X;\pmb{E}_{1,{\rm tr}})\xrightarrow{D_{1,{\rm tr}}}&C^\infty(X;\pmb{E}_{2,{\rm tr}})\xrightarrow{D_{2,{\rm tr}}}\cdots \\
\nonumber
&\xrightarrow{D_{N-1,{\rm tr}}}C^\infty(X;\pmb{E}_{N,{\rm tr}})\xrightarrow{D_{N,{\rm tr}}} C^\infty(X;\pmb{E}_{N+1,{\rm tr}})\to 0,
\end{align}
is a Rockland sequence.
\end{proposition}

Denoting the graded Rockland sequence \eqref{gradedonaodnaodn} by $D_\bullet$, we call the Rockland sequence \eqref{gradedonaodnaodntrigially} the order reduction of $D_\bullet$ and we denote this by $D^{\rm tr}_\bullet$. We note the following proposition which is immediate from the definition of $D^{\rm tr}_\bullet$.

\begin{proposition}
\label{gradedundeunagaidn}
Consider a graded Rockland sequence $D_\bullet$ as in \eqref{gradedonaodnaodn} and its order reduction $D^{\rm tr}_\bullet$ as in \eqref{gradedonaodnaodntrigially}. Then it holds that 
$$[\sigma_H(D_\bullet)]=[\sigma_H(D_\bullet^{\rm tr})]\in K_0(C^*(T_HX)).$$
\end{proposition}

\begin{definition}
Consider a graded Rockland sequence $D_\bullet$ as in \eqref{gradedonaodnaodn} on a compact Carnot manifold $X$. Take operators $B_{j,{\rm tr}}\in \Psi_{H}^{0}(X;\pmb{E}_{j+1,{\rm tr}},\pmb{E}_{\rm tr})$ as in Theorem \ref{odnoaindaonafofbnaon} defined from the order reduction $D^{\rm tr}_\bullet$. 

We define the $K$-homology class $[D_\bullet]\in K_0^{\rm an}(X)$ as the class of the even $K$-cycle 
$$\left(L^2(X;\oplus_k \pmb{E}_{k,{\rm tr}}), F\right),$$
where $L^2(X;\oplus_k \pmb{E}_{k,{\rm tr}})$ is graded by the parity of $k$ and $F$ is obtained from polar decomposition of 
$$F_0:=\begin{pmatrix} 0& F_{0,+}\\ F_{0,+}^* &0\end{pmatrix},\quad\mbox{for}\quad F_{0,+}:=\sum_k D_{2k+1,{\rm tr}}+B_{2k,{\rm tr}}.$$
\end{definition}

In the case of the BGG-sequence for the complex semi-simple Lie group $SL(3,\C)$, an associated $SL(3,\C)$-equivariant $K$-homology class was constructed by Yuncken \cite{yunckensl3}. If the class constructed \cite{yunckensl3} can be reconciled with the construction above, it would provide a general method to construct $\gamma$-elements for complex semi-simple Lie groups. 

\begin{theorem}
\label{theoreomom}
Consider a graded Rockland sequence $D_\bullet$ on a compact Carnot manifold $X$. Then it holds that 
$$[D_\bullet]=\mathsf{PD}^{\rm an}_H([\sigma_H(D_\bullet)])\in K_0^{\rm an}(X).$$
In particular, the $K$-homology class $[D_\bullet]$ only depends on the Rockland sequence and not on auxiliary choices.
\end{theorem}

\begin{proof}
By Proposition \ref{gradedundeunagaidn} and the definition of $D_\bullet$, we can assume that $D_\bullet=D_\bullet^{\rm tr}$. The theorem now follows from Proposition \ref{chooseanopheisenberg}.
\end{proof}

\begin{remark}
The abstract methods of Section \ref{sec:indexingener} applies to graded Rockland sequences due to Theorem \ref{theoreomom}. We nevertheless believe it to be of interest to further clarify the analytic and geometric structure of the $K$-homology class $[D_\bullet]\in K_0^{\rm an}(X)$ and the symbol class $[\sigma_H(D_\bullet)]\in K_0(C^*(T_HX))$.

In principle, the localization and approximation methods of Section \ref{sec:ktheoomon} apply to the symbol class $[\sigma_H(D_\bullet)]\in K_0(C^*(T_HX))$ for an $\pmb{F}$-regular Carnot manifold. However, we believe it would be interesting if there was a more direct method of localization for the symbol class $[\sigma_H(D_\bullet)]$ to an $I_X$–Fredholm complex in case of an $\pmb{F}$-regular manifold, similar to Lemma \ref{tehcnilmultplleme} but for symbols of Rockland sequences. The concrete BGG-complexes arising on $X=G/B$, where $B$ is the Borel subgroup, as in the original work of Bernstein-Gelfand-Gelfand \cite{bggoriginal}, or the examples computed in \cite{GenericDave,Dave_Haller1} can prove to be interesting testing grounds.
\end{remark}

\end{document}